\documentclass[12pt,reqno]{amsart}

\usepackage{extsizes,amsrefs}

\usepackage{etex}
\usepackage{hyperref}

\usepackage{txfonts,amsmath,amstext,amsthm,amscd,amsopn,verbatim,amssymb,amsfonts,mathtools,enumitem}
\usepackage{fullpage}

\usepackage{todonotes}

\usepackage{adjustbox}

\usepackage[bbgreekl]{mathbbol}

\newcommand{\tru}{\mathrm{trunc}}

\usepackage{tikz}

\usetikzlibrary{matrix}
\usetikzlibrary{patterns, shapes}
\usetikzlibrary{arrows}
\usetikzlibrary{calc,3d}
\usetikzlibrary{decorations,decorations.pathmorphing, decorations.pathreplacing}
\usetikzlibrary{through}
\tikzset{ext/.style={circle, draw,inner sep=1pt},int/.style={circle,draw,fill,inner sep=1pt},nil/.style={inner sep=1pt}}
\tikzset{exte/.style={circle, draw,inner sep=3pt},inte/.style={circle,draw,fill,inner sep=3pt}}
\tikzset{diagram/.style={matrix of math nodes, row sep=3em, column sep=2.5em, text height=1.5ex, text depth=0.25ex}}
\tikzset{diagram2/.style={matrix of math nodes, row sep=0.5em, column sep=0.5em, text height=1.5ex, text depth=0.25ex}}
\tikzset{every picture/.style={baseline=-.65ex}}
\tikzstyle{every loop}=[draw]
\tikzstyle{rloop}=[ out=10, in=-10, loop, distance=3em] 
\tikzstyle{aloop}=[ out=100, in=80, loop, distance=3em] 
\usepackage{tikz-cd}

\theoremstyle{plain}

\newtheorem{thm}{Theorem}[section]

\newtheorem{defn}[thm]{Definition}
\newtheorem{prop}[thm]{Proposition}
\newtheorem{cor}[thm]{Corollary}
\newtheorem{lemma}[thm]{Lemma}

\theoremstyle{definition}

\newtheorem{ex}[thm]{Example}
\newtheorem{rem}[thm]{Remark}

\hypersetup{pdfauthor={Thomas Willwacher}}

\newcommand{\Hom}{\mathop{Hom}}

\newcommand{\R}{{\mathbb{R}}}

\newcommand{\Z}{{\mathbb{Z}}}

\newcommand{\Q}{{\mathbb{Q}}}

\newcommand{\HGC}{{\mathrm{HGC}}}

\newcommand{\GG}{\mathbf{HG}}
\newcommand{\GGC}{\mathbf{HGC}}
\newcommand{\tGGC}{\widetilde{\GGC}}
\newcommand{\fGGC}{\GGC^\tvv}
\newcommand{\fGG}{\GG^\tvv}

\newcommand{\Graphs}{{\mathsf{Graphs}}}
\newcommand{\tvv}{\#}
\newcommand{\fGraphs}{\Graphs^{\tvv}}
\newcommand{\nfGraphs}{\overline{\Graphs}^{\tvv}}

\newcommand{\hgr}{\hat{\gr}}

\newcommand{\mF}{\mathcal{F}}

\newcommand{\bbS}{\mathbb{S}}

\newcommand{\osp}{\mathfrak{osp}}

\newcommand{\Gra}{{\mathsf{Gra}}}

\newcommand{\FF}{{\mathbf{F}}} 

\newcommand{\Def}{\mathrm{Def}}

\newcommand{\Poiss}{\mathsf{Pois}}

\newcommand{\Lie}{\mathsf{Lie}}

\renewcommand{\Bar}{{\mathtt{B}}}

\newcommand{\gl}{\mathfrak{gl}}

\newcommand{\Ass}{\mathsf{Assoc}}
\newcommand{\Com}{\mathsf{Com}}

\newcommand{\cR}{\mathcal{R}}
\newcommand{\IG}{\mathsf{IG}}
\newcommand{\pICG}{\mathsf{pICG}}
\newcommand{\pIG}{\mathsf{pIG}}
\newcommand{\ppIG}{\mathsf{pIG}} 

\newcommand{\GCp}{\mathrm{GCp}}
\newcommand{\GCex}{\GC^{\mathrm ex}}
\newcommand{\GCpex}{\GCp^{\mathrm ex}}
\newcommand{\GCx}{\mathrm{GCx}}
\newcommand{\GCpx}{\mathrm{GCpx}}

\newcommand{\FM}{\mathcal{F}}

\newcommand{\Exp}{\mathrm{Exp}}

\newcommand{\bpm}{\begin{pmatrix}}
\newcommand{\epm}{\end{pmatrix}}

\newcommand{\GC}{\mathrm{GC}}

\renewcommand{\Hom}{\mathrm{Hom}}
\newcommand{\MC}{\mathsf{MC}}

\newcommand{\mU}{\mathcal{U}}
\newcommand{\mV}{\mathcal{V}}
\newcommand{\hotimes}{\mathbin{\hat\otimes}}

\DeclareMathOperator{\End}{End}
\DeclareMathOperator{\sgn}{sgn}

\newcommand{\e}{\mathsf{e}}
\newcommand{\Op}{\mathcal{O}\mathrm{p}}
\newcommand{\La}{\Lambda}
\newcommand{\LS}{\mathsf{LS}}
\newcommand{\Hopf}{\mathcal{H}\mathrm{opf}}

\newcommand{\Free}{\mathbb{F}}

\newcommand{\Seq}{\mathcal{S}\mathit{eq}}

\newcommand{\Map}{\mathrm{Map}}
\newcommand{\Mor}{\mathrm{Mor}}
\newcommand{\Aut}{\mathrm{Aut}}

\newcommand{\gr}{\mathrm{gr}}
\newcommand{\dg}{\mathit{dg}}

\newcommand{\fc}{{\mathfrak c}}

\newcommand{\oW}{\mathring{W}} 

\newcommand{\beq}[1]{\begin{equation}\label{#1}}
\newcommand{\eeq}{\end{equation}}

\newcommand{\SL}{\mathfrak{S}\Lie}

\newcommand{\dgcAlg}{\dg\mathcal{C}\mathit{om}}
\newcommand{\ev}{\mathit{ev}}

\newcommand{\mG}{{\mathcal G}}

\newcommand{\Res}{\mathrm{Res}}
\newcommand{\Ind}{\mathrm{Ind}}
\newcommand{\coInd}{\mathrm{coInd}}
\newcommand{\Fr}{\mathrm{Fr}}

\newcommand{\ar}{\mathrm{ar}}
\newcommand{\Mod}{\mathcal{M}\mathrm{od}}
\newcommand{\Modc}{\Mod^c}
\newcommand{\dgModc}{\dg\Mod^c}
\newcommand{\dgHModc}{\dg\Hopf\Mod^c}
\newcommand{\gLaHModc}{\mathrm{g}\Hopf\La\Mod^c}
\newcommand{\dgLaHModc}{\dg\Hopf\La\Mod^c}

\newcommand{\iHom}{{\mathcal{H}om}}

\DeclareMathOperator*{\colim}{\mathrm{colim}}

\newcommand{\sset}{\mathit{s}\mathcal{S}\mathit{et}}

\newcommand{\TOp}{\mathsf{T}}

\newcommand{\fg}{\mathfrak{g}}
\newcommand{\fh}{\mathfrak{h}}
\newcommand{\curv}{\mathrm{curv}}
\newcommand{\cone}{\mathrm{cone}}
\newcommand{\sSet}{\sset}
\newcommand{\conf}{{\mathrm{conf}}}
\newcommand{\FreeMod}{\Free}
\newcommand{\HG}{\mathsf{HG}}
\newcommand{\coRes}{\mathrm{coRes}}
\newcommand{\GL}{\mathrm{GL}}
\newcommand{\Conf}{\mathrm{Conf}}

\newcommand{\sSetSeq}{\sSet\Seq}

\DeclareMathAlphabet{\mathsfit}{OT1}{cmss}{m}{sl}
\DeclareMathOperator{\AOp}{\mathsfit{A}}
\DeclareMathOperator{\BOp}{\mathsfit{B}}
\DeclareMathOperator{\COp}{\mathsfit{C}}
\DeclareMathOperator{\DOp}{\mathsfit{D}}

\DeclareMathOperator{\MOp}{\mathsfit{M}}
\DeclareMathOperator{\NOp}{\mathsfit{N}}
\DeclareMathOperator{\POp}{{\mathsfit{P}}}
\DeclareMathOperator{\QOp}{{\mathsfit{Q}}}

\DeclareMathOperator{\SOp}{\mathsfit{S}}

\DeclareMathOperator{\sSetOp}{\sSet\Op}
\DeclareMathOperator{\TopOp}{\Top\Op}

\DeclareMathOperator{\dgOpc}{\dg\Op^c}

\DeclareMathOperator{\dgHOpc}{\dg\Hopf\Op^c}

\DeclareMathOperator{\dgLaHOpc}{\dg\Hopf\Lambda\Op^c}

\DeclareMathOperator{\vdim}{\mathrm{dim}}

\renewcommand{\Top}{\mathcal{T}{op}}
\newcommand{\eql}{\mathrm{eq}}
\newcommand{\coeql}{\mathrm{coeq}}
\newcommand{\cC}{{\mathcal C}}

\newcommand{\Mfd}{{\mathbf{M}}} 
\newcommand{\Mfdd}{{\mathbf{N}}} 

\DeclareMathOperator{\id}{\mathit{id}}

\newcommand{\hCom}{\widehat{\Com}}
\newcommand{\Fc}{\FF}

\newcommand{\Uc}{U^c}

\begin{document}
\title{Models for configuration spaces of points via obstruction theory I}



\author{Thomas Willwacher}
\address{Department of Mathematics \\ ETH Zurich \\
R\"amistrasse 101 \\
8092 Zurich, Switzerland}
\email{thomas.willwacher@math.ethz.ch}




\begin{abstract}
  We derive rational (Sullivan) models for configuration spaces of points on manifolds purely from algebraic considerations via obstruction theory, essentially without the use of analytic or geometric techniques.
\end{abstract}

\maketitle


\sloppy

\tableofcontents

\section{Introduction}
Let $\Mfd$ be a manifold of dimension $n\geq 2$ and consider the space 
\[
  \conf_\Mfd(r)=\{(x_1,\dots,x_r) \mid x_j\in \Mfd, x_i\neq x:i \text{ for }i\neq j\}
\] 
of $r$ distinguishable points on $\Mfd$. Let $\FM_M(r)$ be the Fulton-MacPherson-Axelrod-Singer bordification of $\conf_\Mfd(r)$, see \cite{AxelrodSinger1994}.
Assume furthermore that $\Mfd$ is parallelized, that is, we are given a trivialization of its tangent bundle. Then the collection $\FM_\Mfd$ is equipped with a natural operadic right action of the Fulton-MacPherson version of the little disks operad $\FM_n$.
The operadic right $\FM_n$-modules $\FM_\Mfd$ are of considerable interest in algebraic topology due to their appearance in the Goodwillie-Weiss embedding calculus \cite{GW, GK, WBdB, WBdB2}. In particular, one is interested in their rational homotopy theory.
From an algebraists standpoint, there are in particular the following three algebraic problems arising from the embedding calculus:
\begin{itemize}
\item Problem 1: Find "tractable" models in the sense of rational homotopy theory for the pair of operad and operadic module $(\FM_n, \FM_\Mfd)$.
Concretely, this amounts to finding a cooperad $\COp$ and a right $\COp$-comodule $\MOp$ in differential graded commutative algebras, such that the pair $(\COp, \MOp)$ can be related to (a version of) differential forms on $(\FM_n, \FM_\Mfd)$ via a zigzag of quasi-isomorphisms. 
We refer to section \ref{sec:operads modules} below for details and more precise definitions.
\item Problem 2: Given such a right $\COp$-comodule $\MOp$ modeling $\FM_\Mfd$, we want to study the homotopy automorphism simplicial monoid $\Aut^h(\MOp)$ in the category of right $\COp$-comodules. 
\item Problem 3: Given two such models $\MOp$ and $\NOp$, say for $\FM_\Mfd$ and $\FM_\Mfdd$, with $\Mfdd$ another manifold, we want to study the derived mapping spaces $\Map^h(\MOp, \NOp)$ in the category of right $\COp$-comodules.
\end{itemize}

In this paper we shall describe algebraic methods to solve or simplify the first two problems. The computation of mapping spaces is left to a follow-up work.
Let us begin with an overview of previous work, starting with Problem 1.
The rational homotopy theory of configuration spaces of points on manifolds has received considerable attention in the literature in the last 60 years.
Major achievements are the determination of the real or rational homotopy types of configuration spaces of points in $\R^n$ by Arnold \cite{Arnold1969} and Cohen \cite{Cohen1976}, and that of points in smooth projective varieties in works of Fulton-MacPherson \cite{FultonMacPherson1994}, Kriz \cite{Kriz} and Totaro \cite{Totaro}.
These earlier works did not study the operadic actions.
More recently, the rational homotopy type of the little disks operad, and hence of $\FM_n$ has been determined, with the result that it is rationally formal \cite{FW, Frbook2}.
That is, we may take for the model $\COp$ of $\FM_n$ above the cohomology cooperad 
\[
\COp = \e_n^c := H^\bullet(\FM_n),  
\]
or a weakly equivalent object.
Furthermore, Idrissi \cite{Idrissi2018b}, and Campos and the author \cite{CW} found real models for configuration spaces of closed parallelized manifolds, that also capture the operadic $\FM_n$-module structure. 
Notably, they construct explicit maps from a (candidate) graphical model $\Graphs_\Mfd$ of $\FM_M$ into a version of differential forms $\Omega_{PA}(\FM_\Mfd)$ on $\FM_\Mfd$,
\[
  \Graphs_\Mfd\to \Omega_{PA}(\FM_\Mfd).
\]
This map is essentially given by applying the Feynman rules of a topological field theory, and involves configuration space integrals. In particular, it is defined only over the reals.

In contrast, in this paper we propose an alternative, completely algebraic approach to Problem 1 using obstruction theory.
The central new notion is that of (co)operadic (co)modules of \emph{configuration space type}. Concretely, we say that a right operadic $\FM_n$-module $\MOp$ is of \emph{configuration space type} if the natural maps 
\[
\MOp(1)^{\times r} \to (\Ind^{\Com,h}_{E_n}\MOp )(r) 
\]
from copies of $\MOp(1)$ into the homotopy induced $\Com$-module are weak equivalences, with $\Com$ being the topological operad such that $\Com(r)$ is a point for all $r$. As the name suggests, the configuration spaces $\FM_\Mfd$ are of configuration space type, see Proposition \ref{prop:FM configuration type} below.
Dualizing the above, one also has an analogous notion of being of \emph{configuration space type} for right $\e_n^c$-modules, see Definition \ref{def:config space type c} below, so that the desired models for $\FM_\Mfd$ are indeed of configuration space type, see Corollary \ref{cor:FM Hopf csc}.



As a second ingredient we use comodule variants of the Kontsevich graphs cooperad $\Graphs_n$, that have appeared similarly in \cite{CW}.
More precisely, for $V$ a finite dimensional graded vector space, elements of $\Graphs_{V,n}(r)$ are formal linear combinations of graphs with $r$ numbered "external" vertices and an arbitrary finite number of internal vertices, decorated by elements of $V$.
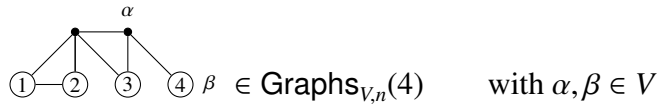
\[
\begin{tikzpicture}
\node[ext] (v1) at (0,0) {$\scriptstyle 1$};
\node[ext] (v2) at (.7,0) {$\scriptstyle 2$};
\node[ext] (v3) at (1.4,0) {$\scriptstyle 3$};
\node[ext,label=0:{$\scriptstyle \beta$}] (v4) at (2.1,0) {$\scriptstyle 4$};
\node[int] (i1) at (0.7,.7) {};
\node[int,label=90:{$\scriptstyle \alpha$}] (i2) at (1.4,.7) {};
\draw (v1) edge (v2) edge (i1) (v2) edge (i1) edge (i1) (v3) edge (i1) edge (i2) (v4) edge (i2) (i1) edge (i2);
\end{tikzpicture}
\in \Graphs_{V,n}(4)
\quad\quad \text{with }\alpha,\beta\in V
\]
For the precise definition we refer to section \ref{sec:graphical cooperads} below.
These objects $\Graphs_{V,n}(r)$ are right Hopf $\e_n^c$-comodules. Furthermore, they come with an action of a big graphical dg Lie algebra $\GGC_{V,n}$, whose elements are series of graphs with external legs (or hairs), with the legs decorated by elements of $V_1:=V\oplus \Q1$ and vertices decorated by elements of $V^*$.
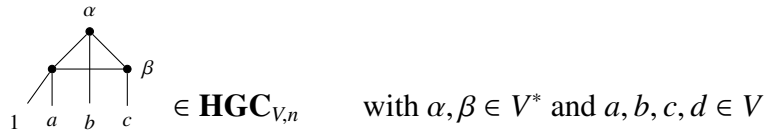
\[
  \begin{tikzpicture}[yshift=-.5cm]
    \node[int] (i1) at (.5,1) {};
    \node[int,label=90:{$\scriptstyle \alpha$}] (i2) at (1,1.5) {};
    \node[int,label=0:{$\scriptstyle \beta$}] (i3) at (1.5,1) {};
    \node (e1) at (0,.3) {$\scriptstyle 1$};
    \node (e2) at (.5,.3) {$\scriptstyle a$};
    \node (e3) at (1,.3) {$\scriptstyle b$};
    \node (e4) at (1.5,.3) {$\scriptstyle c$};
    \draw (i1) edge (i2) edge (i3) edge (e1) edge (e2)
    (i2) edge (i3) edge (e3)
    (i3) edge (e4);
  \end{tikzpicture}  
  \in 
  \GGC_{V,n}
  \quad\quad \text{with }\alpha,\beta\in V^* \text{ and } a,b,c,d\in V
\]
For the precise definition and the combinatorial description of the action of $\GGC_{V,n}$ on $\Graphs_{V,n}$ we again refer to section \ref{sec:graphical cooperads}.
For now, we just note that for any Maurer-Cartan element $Z\in \GGC_{V,n}$, we may twist $\Graphs_{V,n}$ by $Z$ to obtain a right cooperadic Hopf $\e_n^c$-comodule $\Graphs_{V,n}^Z$. Then our main result is as follows.

\begin{thm}\label{thm:main_intro}
  Let $n\geq 2$ and let $\MOp$ be a right $\La$ Hopf $\e_n^c$-comodule of configuration space type, such that $H(\MOp(1))=:\Q 1\oplus \bar H$ is connected and finite dimensional. 
  Then there is a Maurer-Cartan element $Z\in \GGC_{\bar H,n}$ and a quasi-isomorphism of right $\La$ Hopf $W\e_n^c$-comodules
  \[
  \Phi : \Graphs_{\bar H,n}^Z \to W\MOp,
  \]
  with $W\e_n^c$ and $W\MOp$ being suitable resolutions of $\e_n^c$ and $\MOp$ respectively. 
\end{thm}

Practically, this means that in order to determine the rational homotopy type of the configuration space of points on a manifold it suffices to identify the correct Maurer-Cartan element $Z\in \GGC_{\bar H,n}$. 
In many cases, this is relatively easy to do by formal degree counting arguments.
For example we show in section \ref{sec:applications} below that the Theorem above can be used to describe the rational homotopy type of the configuration spaces of points on parallelized compact manifolds without boundary of dimension at least 4.

We emphasize that Theorem \ref{thm:main_intro} above is a completely algebraic result that is shown by obstruction theoretic methods. It does not refer to or use configuration spaces of points on manifolds, or any analytic techniques like configuration space integrals.
Of course, the application to the rational homotopy of configuration spaces uses the existence of a rational homotopy theory of configuration spaces, albeit only to the point that an algebraic model exists, and is of configuration space type.

Furthermore, the models $\Graphs_{V,n}^Z$ of Theorem \ref{thm:main} obtained via twisting with the Maurer-Cartan element $Z$ automatically come equipped with an action of the twisted dg Lie algebra $\GGC_{V,n}^Z$. The dg Lie algebra $\GGC_{V,n}^Z$ is not pro-nilpotent, so the Lie algebra action does not readily integrate to an action of a group. However, we may restrict to a pro-nilpotent dg Lie subalgebra $\GGC_{V,n}^{Z,nil}\subset \GGC_{V,n}^Z$. 
Then we show by essentially the same computational techniques underlying Theorem \ref{thm:main} that the corresponding group action exhausts the "Torelli" part of the homotopy automorphisms of $\Graphs_{V,n}^Z$, or more precisely:

\begin{thm}[see Corollary \ref{cor:ExpAut}]\label{thm:intro ExpAut}
 Let $V$ be a finite dimensional positively graded vector space
 and let $Z\in \mG^1\GGC_{V,n}$ be a Maurer-Cartan element. Then we have a weak equivalence of simplicial monoids 
\[
  \Exp_\bullet(\GGC_{V,n}^{nil,Z}) \simeq \Aut^h(\Graphs_{V,n}^Z)_{[1]},
\]
where the left-hand side is the exponential (simplicial) group and the right-hand side consists of the connected components of the homotopy automorphism simplicial monoid of that correspond to morphisms inducing the identity on cohomology.
\end{thm}

In section \ref{sec:applications} we use Theorems \ref{thm:main_intro} and \ref{thm:intro ExpAut} together to compute the Torelli parts of the homotopy automorphism groups of rational models of configuration spaces of "nice enough" manifolds.
In particular, this is applicable to highly connectd manifolds.





\section*{Acknowledgements}
This paper has benefitted from discussions with numerous colleagues whom I heartfully thank, including Alexander Berglund, Benoit Fresse, Sander Kupers, Oscar Randal-Williams, Victor Turchin and Bruno Vallette.
Furthermore the author is grateful to the Swiss National Science Foundation, who supported this work through the NCCR Swissmap. 

\section{Notation and conventions}
\subsection{General conventions}
We generally work in cohomological conventions, that is, differentials on graded vector spaces have degree $+1$.
We abbreviate the term "differential graded" by dg, or omit it altogether since almost all our algebraic objects will be differential graded.
Generally, all vector spaces we consider will be $\Q$-vector spaces, and the cohomology of topological spaces will be taken with $\Q$-coefficients unless otherwise noted.
The cohomology of a differential graded vector space $(V,d)$ or a topological space $X$ will be denoted by $H(V)$ or $H(X)$ respectively.
In particular, note that 
\[
  H(X) := H^\bullet(X; \Q),
\] 
using the more common notation from the literature on the right-hand side.

For a graded vector space $V$ we denote by $V[k]$ the degree shifted version: If $x\in V$ is homogeneous of degree $p$ then the same element is of degree $p-k$ in $V[k]$.

We will often consider filtered dg vector spaces 
\[
\cdots \supset \mF^p V \supset \mF^{p+1}V\subset \cdots .  
\]
In this case we consider two versions of the associated graded spaces:
\begin{align*}
  \gr V &= \bigoplus_p \mF^p V / \mF^{p+1} V
  &
  \hgr V &= \prod_p \mF^p V / \mF^{p+1} V.
\end{align*}

For a category $\mathcal C$ we denote the set of morphisms between two objects $A$ and $B$ by $\Mor_{\mathcal C}(A,B)$. If the category is enriched over (dg) vector spaces, then we use the notation $\iHom(A,B)$ for the homomorphism (dg) vector spaces.
Also, we use the noation $\Map(A,B)$ to denote the mapping spaces or simplicial sets.

A symmetric sequence $\AOp$ in a category $\mathcal C$ is a sequence of objects $\AOp(r)$ with a right action of the symmetric group $S_r$ for $r=0,1,\dots$.
For two symmetric sequences $\AOp$, $\BOp$ we denote by $\AOp\circ \BOp$ the plethystic product, see \cite[section 5.1.6]{LV}. Mind that this notation conflicts with the use of "$\circ$" for the composition of morphisms, but we hope that no confusion arises. 

\subsection{Lie-admissible algebras and modules}\label{sec:prelie}
A Lie-admissible algebra structure on a dg vector space $\fg$ is a binary operation 
\[
* : \fg\otimes \fg \to \fg
\]
such that the associated commutator bracket 
\[
[x,y] = x* y - (-1)^{|x||y|} y* x
\]
satisfies the Jacobi-identity, and is hence a Lie bracket on $\fg$.
This is equivalent to the statement that the associator 
\[
A(x,y,z) = (x* y) * z - x* (y * z)
\]
has vanishing antisymmetric component:
\begin{multline*}
A(x,y,z) 
+ (-1)^{|x|(|y|+|z|)} A(y,z,x) 
+ (-1)^{|z|(|x|+|y|)} A(z,x,y)\\
- (-1)^{|x||y|} A(y,x,z) 
- (-1)^{|x||z|+ |y||x|+|y||z|} A(z,y,x) 
- (-1)^{|y||z|} A(x,z,y)
=0\, .
\end{multline*}


A Lie-admissible (left) $\fg$-module is a dg vector space $V$ together with a binary operation 
\[
\cdot  \colon \fg \otimes V \to V
\]
that satisfies 
\begin{align*}
  x\cdot (y\cdot v) - (x* y)\cdot v
  &=
  (-1)^{|x||y|}
  \left( y\cdot (x\cdot v) - (y* x)\cdot v \right) 
\end{align*}
for homogeneous $x,y\in\fg$, $v\in V$.
This is equivalent to the statement that $\cdot$ defines a left action of the dg Lie algebra $\fg$ on $V$.

The notions of Lie admissible algebra and Lie admissible module readily dualize to the notions of Lie admissible coalgebra and  Lie admissible comodule.
In particular, a Lie admissible coalgebra structure on a dg vector space $\fc$ is defined by a morphism
\[
\Delta : \fc \to \fc \otimes \fc   
\]
such that the corresponding coassociator
\[
A^c := (\Delta \otimes \id ) \circ \Delta - (\id \otimes \Delta) \circ \Delta \colon \fc \to \fc^{\otimes 3}
\]
has vanishing antisymmetric part,
\[
  \sum_{\sigma\in S_3} (-1)^\sigma \sigma A^c =0.
\]
The associated cobracket 
\[
\Delta - \tau \circ \Delta : \fc \to \fc\otimes \fc,
\]
with $\tau:\fc\otimes \fc\to \fc\otimes \fc$ the transposition of factors, then defines a dg Lie coalgebra structure on $\fc$.
Finally, a Lie admissible $\fc$ comodule is a comodule for the dg Lie coalgebra $\fc$.

\subsection{Fulton-MacPherson-Axelrod-Singer operad $\FM_n$}
For an integer $n\geq 1$ we consider the configuration space of $r$ numbered points in $\R^n$, $\Conf_{\R^n}(r)$.
It is acted upon by the group $\R_{>0}\ltimes \R^n$ of rescalings and translations.
One has a map 
\begin{align*}
 \Conf_{\R^n}(r)/\R_{>0}\ltimes \R^n &\to (S^{n-1})^{n \choose 2} \times (\R\mathbb{P}^1)^{n \choose 3} \\
 (x_1,\dots,x_r) &\mapsto 
 \left(\dots, \frac{x_i-x_j}{\|x_i-x_j\|},\dots, \left[\|x_i-x_j\|:\|x_j-x_k\|:\|x_k-x_i\|\right] , \dots \right).
\end{align*}
We denote the closure of the image of this map as 
\[
  \FM_n(r) := \overline{\Conf_{\R^n}(r)/\R_{>0}\ltimes \R^n }.
\]
The (compact) spaces $\FM_n(r)$ can be seen to be an iterated bordification of the configuration spaces.
They assemble into an operad, the Fulton-MacPherson-Axelrod-Singer operad $\FM_n$.

This compactification of configuration spaces is due to Axelrod-Singer \cite{AxelrodSinger1994} and Fulton-MacPherson, while the operad structure on $\FM_n$ has been considered by Getzler-Jones \cite{GetzJones} and Kontsevich \cite{KontsevichDefQ}.
The operad $\FM_n$ is weakly equivalent to the little $n$-disks operad \cite{Salvatore}, and shall serve as a substitute of the latter for the purposes of this work.

We will also consider the framed Fulton-MacPherson-Axelrod-Singer operad $\FM_n^{fr}=O(n)\ltimes \FM_n$, defined such that 
\[
  \FM_n^{fr}(r) = \FM_n(r) \times O(n)^r.
\]
It is weakly equivalent to the framed little $n$-disks operad.

\subsection{Fulton-MacPherson-Axelrod-Singer operadic module $\FM_\Mfd$}
Let $\Mfd\subset \R^N$ be a smooth $n$-dimensional submanifold.
Then the configuration space $\Conf_\Mfd(r)$ of $r$ numbered points on $\Mfd$ comes with a map
\begin{align*}
  \Conf_{\Mfd}(r) &\to \Mfd^r\times (S^{N-1})^{r \choose 2} \times (\R\mathbb{P}^1)^{r \choose 3} \\
  (x_1,\dots,x_r) &\mapsto 
  \left(\dots,x_i,\dots, \frac{x_i-x_j}{\|x_i-x_j\|},\dots, \left[\|x_i-x_j\|:\|x_j-x_k\|:\|x_k-x_i\|\right] , \dots \right).
 \end{align*}
We denote the closure of the image of $\Conf_{\Mfd}(r)$ under the above map by 
\[
  \FM_\Mfd(r) = \overline{\Conf_{\Mfd}(r)}.
\]
If $\Mfd$ is compact, then $\FM_\Mfd(r)$ is compact as well, otherwise $\FM_\Mfd(r)$ is only a partial compactification of the configuration space. 
The spaces $\FM_\Mfd(r)$ have been introduced by Axelrod-Singer \cite{AxelrodSinger1994} and Fulton-MacPherson \cite{FultonMacPherson1994}, for a more recent discussion and recollection of their properties we refer to \cite{Sinha}.

If $\Mfd$ is parallelized, then the spaces $\FM_\Mfd(r)$ assemble to form an operadic right $\FM_n$-module, see \cite{Turchin4}.
Furthermore, let $\Fr_\Mfd$ be the frame bundle of $\Mfd$.
Then we may consider framed versions 
\[
\FM_\Mfd^{fr}(r) = \FM_\Mfd(r) \times_{\Mfd^r} \Fr_\Mfd
\]
of the spaces $\FM_n(r)$, that then combine into an operadic right $\FM_n^{fr}$-module $\FM_\Mfd^{fr}$.

\subsection{Cohomology of $\FM_n$}\label{sec:en def}
We denote the rational cohomology cooperad of the little $n$-disks operad, by $\e_n^c=H(\FM_n)$.
For $n=1$ we have that $\e_n^c=\Ass^c$ is the coassociative cooperad. For $n\geq 2$, $\e_n^c=\Poiss_n^c$ is identified with the $n$-Poisson cooperad.
The cooperad $\Poiss_n^c$ is cogenerated by a cocommutative coproduct operation of degree 0 and a Lie cobracket operation of degree $n-1$.
Furthermore, as a graded commutative algebra
\[
  \Poiss_n^c(r) 
  =
  \Q[\omega_{ij}\mid 1\leq i\neq j\leq r]/
  \sim
\]
with the algebra generators $\omega_{ij}$ of degree $n-1$, and the relations $\omega_{ij}^2=0$, $\omega_{ij}=(-1)^n\omega_{ji}$, and the Arnold relation
\begin{equation}\label{equ:arnold rel}
  \omega_{ij}\omega_{jk}+
  \omega_{jk}\omega_{ki}+
  \omega_{ki}\omega_{ij}=0\,.
\end{equation}

In this paper we shall almost exclusively consider the case $n\geq 2$ and we prefer the notation $\e_n^c$ over $\Poiss_n^c$ for brevity.

\section{Operads and operadic modules and their rational homotopy theory}\label{sec:operads modules}
Here we briefly recall facts about operads and operadic modules and their rational homotopy theory.
We shall use the statements of this section only in so far that there exists a well-defined rational homotopy theory,
i.e., model structures on the appropriate categories and suitable adjuctions.
Concretely, we need below that the homotopy (co)induction functor is well defined for operadic right (co)modules.
We shall hence be relatively brief, but provide numerous pointers to the literature, where more details can be found.

\subsection{$\La$ operads and $\La$ modules}
For technical reasons we will always work with operads and modules without operations in arity zero.
However, a single arity zero operation can be represented as a $\La$ structure, following Fresse \cite{Frbook2}.
More precisely, let $\POp$ be an operad in simplicial sets or topological spaces without operations in arity zero.
Then we say that $\POp$ is a $\La$ operad if it is endowed with a collection of maps 
\begin{equation}\label{equ:mujdef}
\mu_j : \POp(r) \to  \POp(r-1)
\end{equation}
for $r=2,3,\dots$ and $j=1,\dots, r$ such that 
\begin{equation}\label{equ:Pstar def}
\POp_*(r)=
\begin{cases}
  * &\text{for $r=0$} \\
  \POp(r) &\text{for $r>0$} 
\end{cases}
\end{equation}
forms an operad with $\mu_j$ the operation $(-)\circ_j *$ of composing at the $j$-th slot with $*\in \POp_*(0)$.
Similarly, let $\MOp$ be a right $\POp$-module, without arity zero part $\MOp(0)=\emptyset$.
Then $\MOp$ is a $\Lambda$-module if it is equipped with a collection of morphisms 
\[
  \mu_j : \MOp(r) \to  \MOp(r-1)
\]
for $r=2,3,\dots$ and $j=1,\dots,r$ such that 
\[
\MOp_*(r)=
\begin{cases}
  * &\text{for $r=0$} \\
  \MOp(r) &\text{for $r>0$} 
\end{cases}
\]
is a right $\POp_*$-module, with the right action of $*$ being defined via the morphisms $\mu_j$.

In particular, we consider the Fulton-MacPherson-Axelrod-Singer operads $\FM_n$, $\FM_n^{fr}$ and their right modules $\FM_\Mfd$, $\FM_\Mfd^{fr}$ to not possess operations of arity zero, 
\[
  \FM_n(0)=\FM_n^{fr}(0)=\FM_\Mfd(0)=\FM_\Mfd^{fr}(0)=\emptyset.
\]
However, we do consider $\FM_n$, $\FM_n^{fr}$ (resp. $\FM_\Mfd$, $\FM_\Mfd^{fr}$) as $\La$-operads (resp. $\La$-modules), and thus capture the equivalent data to a nullary operation.

\subsection{$\La$ Hopf cooperads and comodules}
Let $\dgcAlg$ be the category of differential (non-negatively) graded commutative $\Q$-algebras.
A \emph{Hopf cooperad} is a cooperad in the catgeory $\dgcAlg$. Similarly, for $\COp$ a Hopf cooperad, a Hopf right $\COp$-module $\MOp$ is a right module for $\COp$ in the category $\dgcAlg$. In other words, we require that the $\MOp(r)$ are dg commutative algebras, and the cooperadic coaction morphisms
\[
\MOp(s_1+\cdots+s_r) \to \MOp(r)\otimes \COp(s_1)\otimes   \cdots\otimes \COp(s_r)
\] 
are morphisms of dg commutative algebras.

We generally assume that our cooperads and comodules have no cooperations in arity zero.
However, dually to $\La$-operads and $\La$-modules we also consider Hopf $\La$ cooperads and Hopf $\La$ comodules.
For cooperads a $\La$-structure consists of a collection of morphisms
\[
\mu_j^c : \COp(r) \to \COp(r+1)  
\]
such that 
\[
\COp_*(r)=
\begin{cases}
  \Q &\text{for $r=0$} \\
  \COp(r) &\text{for $r>0$} 
\end{cases}
\]
is a Hopf cooperad with the operations $\mu_j^c$ being the cocompositions $\COp_*(r)\to \COp_*(r+1)\otimes \COp_*(0)$.
Similarly the Hopf $\COp$-comodule $\MOp$ is a Hopf $\La$ comodule if it is equipped with operations 
\begin{equation}\label{equ:mujc mod}
\mu_j^c : \MOp(r) \to \MOp(r+1)  
\end{equation}
such that
\[
\MOp_*(r)=
\begin{cases}
  \Q &\text{for $r=0$} \\
  \MOp(r) &\text{for $r>0$} 
\end{cases}
\]
is a right Hopf $\COp_*$-comodule with $\mu_j^c$ being the cocompositions 
\[
  \MOp_*(r)\to \MOp_*(r+1)\otimes \COp_*(0).
\]

If the context allows we will often abbreviate the notation "$\La$ Hopf $\COp$-comodule" to $\COp$-comodule.

\newcommand{\TopSeq}{\Top\Seq}
\newcommand{\G}{\mathrm{G}}
\newcommand{\GS}{\mathrm{S}}
\subsection{Homotopy theory of topological and simplicial $\La$-operads}

The homotopy theory of $\La$-operads in the categories of simplicial sets and topological spaces has been developed in \cite{FrII, FresseExtended}. We briefly recall here the main results.
We denote by $\sSet$ and $\Top$ the model categories of simplicial sets and topological ($k$-)spaces\footnote{We follow the conventions of \cite[section 1.2]{DFT}.}
We use the standard Quillen equivalence 
\begin{equation}\label{equ:sset top adj}
  |-| : \sSet \rightleftarrows \Top : \GS\, .
\end{equation}
between topological spaces and simplicial sets, with $\GS$ and $|-|$ the singular simplicial complex and realization functors respectively, referring to \cite[section 8.3]{FrII} for the definition of and details on the model category structures.
We denote the categories of symmetric sequences (resp. $\La$ sequences) in $\Top$ or $\sSet$ by $(\La) \sSetSeq$ and $(\La) \TopSeq$.
The following result can then be obtained from  \cite[Theorems 8.3.19, 8.3.20]{FrII}, \cite[section 1]{DFT}.

\begin{thm}[Fresse, Ducoulombier, Turchin]
The categories $\La\sSetSeq$ and $\La \TopSeq$ can be equipped with cofibrantly generated model category structures such that the weak equivalences are the arity-wise homotopy equivalences, and the cofibrations are those morphisms that are cofibrations in the projective model structure on $\Sigma$-sequences.
The Quillen equivalence \eqref{equ:sset top adj}
induces, by objectwise application, a Quillen equivalence 
\begin{equation}\label{equ:LaSeq adj}
|-| : \La\sSetSeq \rightleftarrows \La\TopSeq : \GS  
\end{equation}
between these model categories.
\end{thm}
\begin{proof}
The statements about the existence of the model structures can be found in \cite[Theorems 8.3.19, 8.3.20]{FrII}, \cite[section 1]{DFT}. We just check that the above adjunction is a Quillen equivalence. First note that analogous adjunction 
\[
|-| : \sSetSeq \rightleftarrows \TopSeq : \GS  
\]
between the corresponding projective model categories of $\Sigma$ sequences is a Quillen adjunction because $\GS$ trivially preserves the (acyclic) fibrations. But since the (acyclic) cofibrations of the category of the category of $\La$ sequences are inherited from $\Sigma$-sequences they are also preserved by $|-|$. Hence \eqref{equ:LaSeq adj} is a Quillen adjunction. The fact that it is an equivalence trivially follows from \eqref{equ:sset top adj} being an equivalence.
\end{proof}

  The categories $\La\sSetSeq$ and $ \La\TopSeq$ are categories of functors $\La \to \sSet$, $\La\to\Top$ from a generalized Reedy category $\Lambda$ (see again \cite[section 8.3]{FrII}), and the above model structures are the corresponding Reedy model structures.

One then defines model structures on the categories of $\La$ operads $\La\sSetOp$ and $\La\TopOp$ by model categorial transfer. 
We call the resulting model structure also the Reedy model structure on $\La$ operads.

\begin{thm}[Fresse]\label{thm:model structure operads}
The categories $\La\sSetOp$ and $\La\TopOp$ of simplicial and topological $\La$ operads carry well defined cofibrantly generated model structures obtained via right transfer of model structure along the free/forgetful adjunctions 
\begin{align*} 
 \Free: \La\sSetSeq &\rightleftarrows \La\sSetOp : U
&
\Free: \La\TopSeq &\rightleftarrows \La\TopOp : U.
\end{align*}
The object-wise extension of the Quillen equivalence \eqref{equ:sset top adj} yields a Quillen equivalence 
\begin{equation}\label{equ:LaOp adj}
  |-| : \La\sSetOp \rightleftarrows \La\TopOp : \G.
\end{equation}

A morphism $f:\POp\to \QOp$ in $\La\sSetOp$ or $\La\TopOp$ is a cofibration iff the morphism $f$ is a cofibration in the category of symmetric (i.e., non-$\La$) operads.
\end{thm}
Concretely, these Reedy model structures have the following classes of distinguished morphisms.
\begin{itemize}
\item The weak equivalences are the object-wise weak equivalences.
\item The fibrations are those morphisms that are fibrations in the category of $\La$ sequences with respect to the Reedy model structure.
\item The cofibrations are those morphisms that have the left lifting property with respect to the acyclic fibrations.
\end{itemize}
\begin{proof}[Proof of Theorem \ref{thm:model structure operads}]
For the existence of the transferred model structures we refer again to \cite{FrII}. In particular, the final statement of the Theorem is \cite[8.4.12]{FrII}.

The adjunction \eqref{equ:LaOp adj} is Quillen since by construction $\G$ preserves (acyclic) fibrations since it is the right-adjoint in the Quillen adjunction \eqref{equ:LaSeq adj}.
Furthermore the fact that \eqref{equ:LaOp adj} is a Quillen equivalence follows again directly from the underlying functors \eqref{equ:sset top adj} being a Quillen equivalence.
\end{proof}

We shall also recall the following well known result from the literature.

\begin{prop}[Getzler-Jones \cite{GetzJones}, Salvatore {\cite[Corollary 3.8]{Salvatore}}]\label{prop:FMn cofibrant}
The Fulton-MacPherson-Axelrod-Singer operads $\FM_n$ are cofibrant objects in $\La \TopOp$ equipped with the Reedy model category structure.
\end{prop}
\begin{proof}
To be precise, the references show that $\FM_n$ (without nullary operation) is cofibrant in the category of symmetric operads equipped with the projective model structure.
However, by the last assertion of Theorem \ref{thm:model structure operads} above this implies that $\FM_n$ is cofibrant in $\La \TopOp$ with respect to the Reedy model structure.
\end{proof}

\subsection{Homotopy theory of topological and simplicial right modules over ($\La$-)operads}

The category of right modules over an operad can be endowed with several model category structures. For example, the projective model structure has been constructed in \cite{Frbook2} for fairly general base categories.
Here we shall need a different model category structure well suited for $\La$ operads and modules, the Reedy model structure, as defined in \cite{DFT}, and summarized in the following result.

\begin{thm}[{see \cite[Theorems 3.1 and 3.6]{DFT}}]\label{thm:model str La modules}
Let $\POp \in \La\TopOp$ and $\QOp\in \La\sSetOp$ be well pointed\footnote{This means that the inclusion of the unit in arity one is a cofibration. The condition is void in the simplicial setting.} $\La$ operads.
Then the categories $\La\Mod_{\POp}$ and $\La\Mod_{\QOp}$ of right $\La$ modules over $\POp$ and $\QOp$ carry well defined model structures obtained by right transfer along the free/forgetful adjunctions
\begin{align*} 
  \Free_{\POp}: \La\TopSeq &\rightleftarrows \La\Mod_{\POp} : U
 &
 \Free_{\QOp}: \La\sSetSeq &\rightleftarrows \La\Mod_{\QOp} : U.
 \end{align*}
 The object-wise extension of the Quillen equivalence \eqref{equ:sset top adj} yields a Quillen equivalence 
 \begin{equation}\label{equ:LaOpMod adj}
   |-| : \La\Mod_{\QOp} \rightleftarrows \La\Mod_{|\QOp|} : \G.
 \end{equation}

 A morphism $f:\MOp\to \NOp$ in $\La\Mod_{\POp}$ or $\La\Mod_{\POp}$ is a cofibration if $f$ is a cofibration in the category of symmetric (i.e., non-$\La$) operadic right modules.
\end{thm}
\begin{proof}
Although the result is essentially shown in \cite{DFT} several remarks are in order. First, \cite[Theorem 3.1]{DFT} covers the existence of the model structure over the base category of topological spaces, and in the more general situation of operadic bimodules. But since operadic right modules are a special case of bimodules, with the left action trivial, our situation is covered as well, for topological operads and modules.
The proof of well-definedness of the model structure in the topological setting given in \cite{DFT} can be transcribed to the simplicial category, using the Quillen equivalence between $\Top$ and $\sSet$. More precisely, the proof in \cite{DFT} relies on general transfer theorem by Berger and Moerdijk (\cite[Theorem 1.1]{DFT} or \cite[section 2.5]{BergerMoerdijk}), using essentially the existence of a fibrant replacement functor and functorial path objects for fibrant objects.
The fibrant replacement $\MOp\to \MOp^f$ constructed in \cite[section 3.1.1]{DFT} for topological modules may be made into a fibrant replacement in $\Mod_{\QOp}$ defined as $\NOp\to \GS |\NOp|^f$.
This is again a fibrant replacement in $\Mod_{\QOp}$ since $\GS$ preserves fibrations and and weak equivalences of $\La$ sequences by adjunction.
Furthermore, the construction of the path object in the proof of \cite[Theorem 3.1]{DFT} can easily be transcribed to the simplicial setting by just replacing the topological 1-simplex by the symplicial 1-simplex and the topological mapping space by the simplicial mapping space.
This then allows for applying again the model categorial transfer Theorem and establishes the existence of the Reedy model structure in the simplicial setting, i.e., on $\Mod_{\QOp}$. 

Finally, one turns to the Quillen equivalence \eqref{equ:LaOpMod adj}. First, the fact that the objectwise applications yield well defined functors on the categories of modules follows from monoidality of the underlying functors \eqref{equ:sset top adj} on simplicial sets and topological spaces. (For the right adjoint in \eqref{equ:LaOpMod adj} one also needs to use the adjunction unit $\QOp \to S|\QOp|$ to restrict the module structure.)
It is furthermore clear that the functors still form an adjoint pair, since the adjointness condition in terms of unit and counit, cf. \cite[Theorem 3.1.5 (2)]{Borceux1}, is agnostic of the additional module structures, and holds iff it holds in the base categories of simplicial sets or topological spaces.
Next, the right adjoint in \eqref{equ:LaOpMod adj} preserves fibrations and weak equivalences, since they are created by the forgetful functor, and since $\GS$ preserves weak equivalences and fibrations of $\La$ sequences by the Quillen adjunction \eqref{equ:LaSeq adj}.
Likewise, the property of being an equivalence is inherited from the underlying Quillen equivalence \eqref{equ:LaSeq adj} on $\La$ sequences.
\end{proof}

\begin{prop}[Turchin \cite{Turchin4}]\label{prop:FMM cofibrant}
For $M\subset \R^N$ an $n$-dimensional smooth manifold the Fulton-MacPherson-Axelrod-Singer right $\FM_n^{fr}$-module $\FM_M^{fr}$ is a cofibrant object in $\La\Mod_{\FM_n^{fr}}$ with respect to the above Reedy model structure.
If $M$ is parallelized, then $\FM_M$ is likewise a cofibrant object in $\La\Mod_{\FM_n}$.
\end{prop}
\begin{proof}
  It is shown in \cite[Lemma 2.3]{Turchin4} that $\FM_M^{fr}$ (without nullary operations) is a cofibrant symmetric right $\FM_n^{fr}$-module with respect to the projective model structure. This implies cofibrancy of the $\La$ module $\FM_M^{fr}$ in $\La\Mod_{\FM_n^{fr}}$ by the last assertion of Theorem \ref{thm:model str La modules}.

  The statement about $\FM_M$ does not explicitly appear in \cite{Turchin4}, but can be obtained from the proof of \cite[Lemma 2.3]{Turchin4} by just omitting framings.
\end{proof}

\subsection{Induction and restriction for right operadic modules}

Let $f:\POp\to \QOp$ be a morphism of (simplicial or topological) $\La$-operads. 
For $\MOp$ a right $\POp$-module we consider the restriction $\Res_{\POp}^{\QOp}$ along $f$, which is the same $\La$ sequence $\MOp$, but equipped with the right $\QOp$ module structure pulled back along $f$.
Similarly, for $\NOp$ a right $\QOp$-module we may consider the induced module $\Ind_{\QOp}^{\POp}\MOp$. Note that any $\QOp$-module may be written as a (reflexive) coequalizer of free objects
 $\NOp\cong \coeql\left(\Free_{\QOp}\Free_{\QOp} \NOp \rightrightarrows \Free_{\QOp}\NOp\right)$. We may then define $\Ind_{\QOp}^{\POp}\MOp$ as the corresponding coequalizer 
\[
  \Ind_{\QOp}^{\POp}\MOp = \coeql\left(\Free_{\POp}\Free_{\QOp} \NOp \rightrightarrows \Free_{\POp}\NOp\right).
\]
Note that the morphism $f$ is suppressed from our notation, even though both functors $\Res_{\POp}^{\QOp}$ and $\Ind_{\QOp}^{\POp}\MOp$ depend on $f$, of course. The reason is that in later applications the choice of $f$ will always be fixed and clear from the context.

The induction and restriction functors have been studied in detail in \cite{Frbook2}. 
In particular, they form an adjoint pair 
\begin{equation}\label{equ:ind res adj}  
  \Ind_{\QOp}^{\POp}: 
  \La\Mod_{\QOp} \rightleftarrows \La\Mod_{\POp}: \Res_{\POp}^{\QOp}.
\end{equation}

\begin{prop}[see {\cite[Theorem 16.B]{Frbook2}}, {\cite[Theorem 3.8]{DFT}}]
The functors \eqref{equ:ind res adj} form a Quillen pair.
If the underlying morphism $f:\POp\to \QOp$ is a weak equivalence between componentwise cofibrant\footnote{Componentwise cofibrancy means that each $\POp(r)$, $\QOp(r)$ is cofibrant in $\TOp$ or $\sSet$. In $\sSet$ the condition is void since all objects are cofibrant.} operads then \eqref{equ:ind res adj} is a Quillen equivalence.
\end{prop}
\begin{proof}
We shall note that \cite[Theorem 16.B]{Frbook2} shows the same statement, but for the projective model structures, not the Reedy structures we consider here. 
However, we can still take from loc. cit. that \eqref{equ:ind res adj} is an adjunction. Since $\Res_{\POp}^{\QOp}$ is the identity on objects it trivially creates and hence preserves weak equivalences and fibrations, which are those of the underlying category of $\La$ sequences for both $\Mod_{\POp}$ and $\Mod_{\QOp}$. Hence \eqref{equ:ind res adj} is a Quillen adjunction. 
The statement about Quillen equivalence is \cite[Theorem 3.8]{DFT}.
\end{proof}



\subsection{Rational homotopy theory of $\La$ operads}
The rational homotopy theory of operads and $\La$-operads has been developed by Fresse in \cite{FrII, FresseExtended}.
Here we summarize the main results.

First one can equip the categories of dg cooperads $\dgOpc$ with a model category structure with the following classes of distinguished morphisms.
\begin{itemize}
\item The weak equivalences are the object-wise quasi-isomorphisms.
\item The cofibrations are the morphisms that are injective in positive cohomological degrees.
\item The fibrations are defined via the right-lifting property with respect to acyclic cofibrations.
\end{itemize}

\begin{thm}[{Fresse \cite[Theorem 1.4]{FresseExtended}}]
The above classes of distinguished morphisms induce a well-defined model category structure on $\dgOpc$. This model structure is cofibrantly generated, and the generating (acyclic) cofibrations may be taken to be the (acyclic) cofibrations between cooperads of overall finite (resp. countable) dimension, concentrated in a finite range of arities. 
\end{thm}

This model category structure can then be extended to dg Hopf cooperads and dg $\La$ Hopf cooperads by right transfer.

\begin{thm}[{Fresse \cite[Theorem 1.6 and section 4.2]{FresseExtended}}]
There are well defined cofibrantly generated model category structures on the categories of dg Hopf cooperads $\dgHOpc$ and of dg $\La$ Hopf cooperads $\dgLaHOpc$ obtained via right model categorial transfer along the free/forgetful adjunctions 
\[ 
  \dgLaHOpc \rightleftarrows \dgHOpc \rightleftarrows \dgOpc.
\] 
\end{thm}

Concretely, this means that in both categories $\dgLaHOpc$ and $\dgHOpc$ the weak equivalences are the quasi-isomorphisms, the fibrations are those morphisms that are fibrations in $\dgOpc$, and the cofibrations are obtained via the lifting property.

To go further, we recall the (Sullivan) adjunction between simplicial sets and dg commutative algebras of rational homotopy theory, see \cite[section 7.2]{FrII} 
\[
\Omega : \sset \rightleftarrows (\dgcAlg)^{op} : \G   
\]
The functor $\G$ is lax monoidal, and as a consequence the arity-wise application of $\G$ yields a functor from simplicial operads to dg Hopf cooperads, that we shall also denote by $\G$.

\begin{thm}[{Fresse \cite[sections 2.1, 4.2]{FresseExtended}}]
The arity-wise application of $\G$ can be extended to Quillen pairs 
\begin{align*}
  \Omega_\sharp : \sSetOp &\rightleftarrows (\dgHOpc)^{op} : \G  
  &
  \Omega_\sharp : \La\sSetOp &\rightleftarrows (\dgLaHOpc)^{op} : \G. 
\end{align*}
The left adjoint functors $\Omega_\sharp$ furthermore have the following properties, for $\POp$ and operad in $\sSetOp$ or $\La\sSetOp$.
\begin{itemize}
\item There is a natural morphism of sequences of dg commutative algebras $\Omega_\sharp(\POp)\to \Omega(\POp)$.
\item If $\POp$ is cofibrant and $\POp(1)$ is connected and of finite rational homology type then the morphisms $\Omega_\sharp(\POp)(r)\to \Omega(\POp(r))$ are quasi-isomorphisms for any $r$.
\end{itemize}
\end{thm}

Concretely, $\Omega_\sharp(\POp)$ can be defined as follows.
Any operad can be written as a (reflexive) coequalizer of free objects
\[
\POp\cong \coeql\left(\Free \Free \POp \rightrightarrows \Free \POp \right),
\]
with $\Free$ the free operad functor.
Using that $\Omega_\sharp$ sends colimits to limits by adjunction, and that on free objects one has $\Omega_\sharp(\Free X)=\Free\Omega(X)$ we then find 
\[
  \Omega_\sharp(\POp)\cong \eql\left( \Free^c \Omega(\POp) \rightrightarrows \Free^c \Omega(\Free \POp) \right).
\]

We say that $\Omega_\sharp(\POp)$ is the dg Hopf cooperad model of the simplicial operad $\POp$. 
We define the rationalization $\POp^\Q$ of a cofibrant simplicial operad to be $\G (\Omega_\sharp(\POp))^{cof}$, where $(\Omega_\sharp(\POp))^{cof}$ is a cofibrant replacement of $\Omega_\sharp(\POp)$.

We also extend these notions to topological operads by applying the singular chains functor $\GS$ and the realization respectively.

\subsection{Rational homotopy theory of operadic modules}
The rational homotopy theory of operadic modules can be developed along the lines of Fresse's rational homotopy theory of operads outlined in the preceding subsection. We sketch here the construction, leaving some technical verifications to the forthcoming work \cite{FWMod}, where the rational homotopy theory of operadic modules is developed in more detail.

First, let $\COp$ be a fixed dg cooperad. The category of right $\COp$ dg comodules is denoted by $\dgModc_{\COp}$.
One can equip the categories of dg cooperads $\dgOpc$ with a model category structure with the following classes of distinguished morphisms.
\begin{itemize}
\item The weak equivalences are the arity-wise quasi-isomorphisms.
\item The cofibrations are the morphisms that are injective in positive cohomological degrees.
\item The fibrations are defined via the right-lifting property with respect to acyclic cofibrations.
\end{itemize}

\begin{thm}[\cite{FWMod}]
The above classes of distinguished morphisms induce a well-defined model category structure on $\dgModc_{\COp}$. It is cofibrantly generated. The generating (acyclic) cofibrations may be taken to be the cofibrations between cooperads of overall finite (resp. countable) dimension, concentrated in a finite range of arities. 
\end{thm}

Next suppose that $\COp$ is a dg ($\La$) Hopf cooperad and consider the categories of dg $\La$ Hopf comodules $\dgHModc$, resp, $\dgLaHModc$.

\begin{thm}[{\cite{FWMod}}]\label{thm:modc model cat}
There are well defined cofibrantly generated model category structures on the categories of dg Hopf $\COp$-comodules $\dgHModc_{\COp}$ and of dg $\La$ Hopf $\COp$-comodules $\dgLaHModc_{\COp}$ obtained via right model categorial transfer along the free/forgetful adjunctions 
\[ 
  \dgLaHModc_{\COp} \rightleftarrows \dgHModc_{\COp} \rightleftarrows \dgModc_{\COp}.
\] 
\end{thm}

In both categories $\dgLaHModc_{\COp}$ and $\dgHModc_{\COp}$ the weak equivalences are the quasi-isomorphisms, the fibrations are those morphisms that are fibrations in $\dgModc_{\COp}$, and the cofibrations are obtained via the lifting property.

The functor $\Omega_\sharp$ of the preceding subsection also has a natural extension to operadic modules.

\begin{thm}[\cite{FWMod}]\label{thm:FWMod main}
  Let $\POp$ be a simplical operad, or a simplicial $\La$ operad and set $\COp:=\Omega_\sharp(\POp)$.
The arity-wise application of $\G$ can be extended to Quillen adjunctions 
\begin{align*}
  \Omega_\sharp : \La\Mod_{\POp} &\rightleftarrows (\dgLaHModc_{\COp})^{op} : \G. 
\end{align*}
Let $\MOp$ be a right $\POp$-module.
Assume further that $\POp$ and $\MOp$ are cofibrant and that $\POp(1)$ and $\MOp(1)$ are connected and of finite rational homology type.
Then the morphisms $\Omega_\sharp(\MOp)(r)\to \Omega(\MOp(r))$ are quasi-isomorphisms for any $r$.
\end{thm}

Concretely, $\Omega_\sharp(\MOp)$ can be defined as follows.
We write $\MOp$ as a (reflexive) coequalizer 
\[
\MOp\cong \coeql\left(\Free_{\POp} \Free_{\POp} \MOp \rightrightarrows \Free_{\POp} \MOp \right).
\]
We then find 
\[
  \Omega_\sharp(\MOp)\cong \eql\left( \Free^c_{\Omega_\sharp(\POp)} \Omega(\MOp) \rightrightarrows \Free^c_{\Omega_\sharp(\POp)} \Omega(\Free_{\POp} \MOp) \right).
\]

We say that $\Omega_\sharp(\MOp)$ is the dg Hopf cooperadic comodule model of the simplicial $\POp$-module $\MOp$. 
We also extend this notion to topological operadic modules by applying the singular chains functor $\GS$.





\subsection{(Co)induction and (co)restriction}
Let $f:\COp\to \DOp$ be a morphism of (dg) $\La$ Hopf cooperads and suppose that $\MOp$ is a right $\La$ Hopf comodule over $\COp$.
Then $\MOp$ can be equipped with a right $\La$ Hopf comodule structure over $\DOp$ by corestriction, i.e., by composing the cocomposition with $f$. We shall denote the resulting right $\DOp$ comodule by $\coRes_{\COp}^{\DOp}\MOp$, suppressing $f$ from the notation.

\begin{prop}[\cite{FWMod}]\label{prop:res ind adjunction}
For $f:\COp\to \DOp$ a morphism of $\La$ Hopf cooperads the corestriction functor fits into a Quillen adjunction
\begin{equation}\label{equ:cores coind adj}
  \coRes_{\COp}^{\DOp} : \dgLaHModc_{\COp} \rightleftarrows\dgLaHModc_{\DOp} : \coInd^{\COp}_{\DOp}.
\end{equation}
If $f$ is a weak equivalence (i.e., a quasi-isomorphism) then the above adjunction is a Quillen equivalence.
\end{prop}
The right adjoint functor $\coInd^{\COp}_{\DOp}$ (coinduction) has the following explicit description.
Any right $\DOp$-comodule $\NOp$ can be written as a (reflexive) equalizer
\[
  \NOp = \eql\left( \Free_{\DOp}^c\NOp \rightrightarrows  \Free_{\DOp}^c\Free_{\DOp}^c\NOp \right).
\]
Correspondingly (implicitly using that a right adjoint preserves limits) we have that $\coInd^{\COp}_{\DOp}\NOp$ is the equalizer
\begin{equation}\label{equ:coind def}
\coInd^{\COp}_{\DOp}\NOp
=
\eql 
\left( \Free_{\COp}^c\NOp \rightrightarrows  \Free_{\COp}^c \Free_{\DOp}^c \NOp\right).
\end{equation}

\begin{proof}
One verifies explicitly that \eqref{equ:coind def} defines a right adjoint to the corestriction functor. 
To this end let us abbreviate $\Mor_{\DOp}(\dots)=\Mor_{\dgLaHModc_{\DOp}}(\dots)$ and write $\Mor_{\La}(\dots)$ for the morphisms in the category of dg Hopf $\La$ sequences. 

Concretely, we have the following chain of natural bijections, for $\MOp$ a $\COp$-comodule and $\NOp$ a $\DOp$-comodule.
\begin{align*}
  \Mor_{\DOp}\left( \coRes_{\COp}^{\DOp} \MOp, \NOp \right)
  &\cong
  \Mor_{\DOp}\left( \coRes_{\COp}^{\DOp} \MOp
  , \eql\left( \Free^c_{\DOp} \NOp \rightrightarrows \Free^c_{\DOp}\Free^c_{\DOp}\NOp \right) \right)
  \\
  &\cong
  \eql\left(\Mor_{\DOp}\left(  \coRes_{\COp}^{\DOp} \MOp,\Free^c_{\DOp} \NOp\right)
  \rightrightarrows \Mor_{\DOp}\left(\coRes_{\COp}^{\DOp}\MOp, \Free^c_{\DOp}\Free^c_{\DOp} \NOp \right)\right)
  \\
  &\cong  \eql\left(\Mor_{\La}\left( \MOp, \NOp\right)
  \rightrightarrows \Mor_{\La}\left(\MOp, \Free^c_{\DOp} \NOp \right)\right)
  \\ &\cong  \eql\left(\Mor_{\COp}\left( \MOp, \Free^c_{\COp}\NOp\right)
  \rightrightarrows \Mor_{\COp}\left(\MOp, \Free^c_{\COp}\Free^c_{\DOp} \NOp \right)\right)
  \\&\cong 
  \Mor_{\COp}\left(\MOp , \eql\left(
    \Free^c_{\COp} \NOp\rightrightarrows \Free^c_{\COp} \Free^c_{\DOp} \NOp
  \right) \right)
  \\&\cong 
  \Mor_{\COp}\left(\MOp ,  \coInd_{\DOp}^{\COp}\NOp\right)
\end{align*}

To show that \eqref{equ:cores coind adj} is a Quillen adjunction we will take a slight detour.
First, the same construction and argument as above shows that there is a Quillen adjunction 
\begin{equation}\label{equ:cores coind adj 2}
  \coRes_{\COp}^{\DOp} : \dgModc_{\COp} \rightleftarrows\dgModc_{\DOp} : \coInd^{\COp}_{\DOp}
\end{equation}
between the categories of dg comodules over $\COp$ and $\DOp$.
Note in particular that the right adjoint functor $\coInd^{\COp}_{\DOp}$ agrees with the one constructed above on dg $\La$ Hopf comodules since limits in $\dgLaHModc_{\COp}$ are created in $\dgModc_{\COp}$.
It is clear that \eqref{equ:cores coind adj 2} is a Quillen adjunction since the left adjoint $\coRes_{\COp}^{\DOp}$ is objectwise the identity functor and hence preserves weak equivalences (i.e., quasi-isomorphisms) and cofibrations (i.e., morphisms injective in positive degree).
Hence $\coInd^{\COp}_{\DOp}$ sends weak equivalences and fibrations in $\dgModc_{\DOp}$ to weak equivalences and fibrations in $\dgModc_{\COp}$.
However, the weak equivalences and fibrations in $\dgLaHModc_{\DOp}$ (resp. $\dgLaHModc_{\COp}$) are inherited from $\dgModc_{\DOp}$ (resp. $\dgModc_{\COp}$).
Hence $\coInd^{\COp}_{\DOp}$ as the right adjoint in \eqref{equ:cores coind adj} also sends (acyclic) fibrations to (acyclic) fibrations and hence \eqref{equ:cores coind adj} is a Quillen adjunction.


Next suppose that $f$ as in the Proposition is a weak equivalence. We want to show that the adjunction is a Quillen equivalence. Since $\coRes_{\COp}^{\DOp}$ creates weak equivalences, it suffices to show the following (see \cite[Lemma 3.3]{ErdalIlhan}).

{\bf Claim 1:} If $\NOp\in \Modc_{\DOp}$ is fibrant, then the adjunction counit $\coRes_{\COp}^{\DOp} \coInd^{\COp}_{\DOp} \NOp \to \NOp$ is a weak equivalence.


To show the claim, one notes that coinduction and corestriction are created in the underlying category of plain (i.e., non-$\La$-Hopf) dg cooperadic comodules.
Furthermore, consider some fibrant replacement $\NOp\xrightarrow{\sim}\hat \NOp$ and the commutative diagram
\[
\begin{tikzcd}
  \coRes_{\COp}^{\DOp} \coInd^{\COp}_{\DOp} \NOp \ar{r}\ar{d}{\sim}
  & \NOp \ar{d}{\sim}
  \\
  \coRes_{\COp}^{\DOp} \coInd^{\COp}_{\DOp} \hat \NOp \ar{r} & \hat \NOp
\end{tikzcd}\, .
\] 
The left-hand vertical arrow is a weak equivalence since by Ken Brown's Lemma the morphism $\coInd^{\COp}_{\DOp} \NOp\to \coInd^{\COp}_{\DOp} \hat \NOp$ is, and since corestriction preserves weak equivalences. Hence the upper horizontal arrow is a weak equivalence iff so is the lower horizontal arrow.
Taking for $\hat N=\Bar \Bar^c\NOp$ the standard bar-cobar construction, we hence have reduced Claim 1 to the following statement.

{\bf Claim 2:} The adjunction counit $\coRes_{\COp}^{\DOp} \coInd^{\COp}_{\DOp} \Bar \Bar^c\NOp \to \Bar \Bar^c\NOp$ is a weak equivalence.

To check Claim 2 note that the underlying map on dg symmetric sequences is, explicitly,
\begin{equation}\label{equ:claim2}
(\NOp \circ \Bar^c \DOp \circ \COp, d) 
\to 
(\NOp \circ \Bar^c \DOp \circ \DOp, d),
\end{equation}
with $d$ generically denoting the differentials, the left-hand side being the underlying dg symmetric sequence of $\coRes_{\COp}^{\DOp} \coInd^{\COp}_{\DOp} \Bar \Bar^c\NOp$, and the right-hand side that of $\Bar \Bar^c\NOp$.
The map is given by sending the $\COp$-factor to the corresponding $\DOp$-factor via $f$.
The differentials have several pieces arising from the differentials on $\COp$, $\DOp$, the cobar construction, the coaction on $\NOp$ and the cocomposition. 
Now consider increasing, exhaustive, bounded below filtrations on both sides of \eqref{equ:claim2} by the total cohomological degree of the factors $\NOp \circ \Bar^c \DOp$.
The only piece of the differentials $d$ on either side of \eqref{equ:claim2} that does not raise that degree is the part induced from the differentials on the right-most factors $\COp$ and $\DOp$ respectively. Hence it is clear by the K\"unneth formula that the map \eqref{equ:claim2} induces a quasi-isomorphism on the associated graded complexes and hence a quasi-isomorphism. This shows Claim 2 and hence the Proposition.
We stress that in the proof of Claim 2 it was essential that we work with non-negatively graded cochain complexes. Otherwise, our filtration would not necessarily be bounded below, and the associated spectral sequences would hence not necessarily converge to the cohomology.
\end{proof}

Finally, we note that (co)induction is compatible with taking differential forms, as follows: 
\begin{prop}\label{prop:coind rationalization commutation}
Suppose that $\SOp\to \TOp$ is a morphism of $\La$ operads in $\sSet$ or $\Top$ and let $\MOp$ be a right $\SOp$-module. 
Then there is a natural isomorphism of right $\La$ Hopf $\Omega_\sharp(\TOp)$ comodules
\begin{equation}\label{equ:omega in coind}
\Omega_\sharp(\Ind_{\SOp}^{\TOp}\MOp)
\cong
\coInd_{\Omega_\sharp(\SOp)}^{\Omega_\sharp(\TOp)} \Omega_\sharp(\MOp).
\end{equation}
\end{prop}

\begin{proof}
We have that $\Omega_\sharp(\MOp)$ is defined as the equalizer
\[
  \Omega_\sharp(\MOp) = 
  \eql \left( \Free^c_{\Omega_\sharp(\SOp)}\Omega(M) \rightrightarrows
  \Free^c_{\Omega_\sharp(\SOp)}\Omega(\Free_{\SOp} M) \right).
\]
As a right adjoint the coinduction functor commutes with limits and hence the right-hand side of \eqref{equ:omega in coind} equals 
\begin{multline*}
  \eql \left( \coInd_{\Omega_\sharp(\SOp)}^{\Omega_\sharp(\TOp)} \Free^c_{\Omega_\sharp(\SOp)}\Omega(M) \rightrightarrows
  \coInd_{\Omega_\sharp(\SOp)}^{\Omega_\sharp(\TOp)} \Free^c_{\Omega_\sharp(\SOp)}\Omega(\Free_{\SOp} M) \right)
  =
  \eql \left( \Free^c_{\Omega_\sharp(\TOp)}\Omega(M) \rightrightarrows
  \Free^c_{\Omega_\sharp(\TOp)}\Omega(\Free_{\SOp} M) \right).
\end{multline*}
On the other hand we have that 
\[
  \Ind_{\SOp}^{\TOp}\MOp= 
  \coeql \left(
    \Free_{\TOp}\Free_{\SOp}\MOp \rightrightarrows \Free_{\TOp}\MOp
    \right).
\]
But since $\Omega_\sharp$ sends colimits to limits by adjunction
\[
  \Omega_\sharp(\Ind_{\SOp}^{\TOp}\MOp)
  =
  \eql \left(
    \Omega_\sharp(
    \Free_{\TOp}\MOp) \rightrightarrows \Omega_\sharp(\Free_{\TOp}\Free_{\SOp}\MOp)
    \right)
  =
  \eql \left(
    \Free^c_{\Omega_\sharp(\TOp)}\Omega(\MOp)\rightrightarrows \Free^c_{\Omega_\sharp(\TOp)}\Omega(\Free_{\SOp}\MOp)
  \right).
\]

Hence both sides of \eqref{equ:omega in coind} agree.
\end{proof}

\subsection{W construction for $\La$ Hopf cooperads and comodules}
\label{sec:W construction}
The Boardman-Vogt W construction is a particular cofibrant resolution for topological operads \cite{BoardmanVogt}.
It has been generalized to operads in general symmetric monoidal model categories by Berger and Moerdijk \cite{BergerMoerdijk}.
A dual version of the $W$ construction for $\La$ Hopf cooperads has been introduced in \cite{FTW}, and this has been extended to $\La$ Hopf cooperadic comodules in \cite{CDI}. 

For this paper the technical definition of the $W$ construction is not important. We will only use the existence of a fibrant resolution of $\La$ Hopf cooperads and comodules with particularly good properties, as described by the following proposition.

\begin{prop}[\cite{FTW}, \cite{CDI}]
  \label{prop:W construction}
Let $\COp$ be a reduced $\La$ Hopf cooperad. Then there is a fibrant resolution 
\[
  \COp\xrightarrow{\sim}W\COp,
\]
 the $W$ construction, with the following properties:
\begin{itemize}
\item As a dg cooperad $W\COp=\Bar\oW \COp$ is the operadic bar construction of a dg operad $\oW \COp$.
In particular, $W\COp$ is cofree as a graded cooperad.
\item The dg operad $\oW \COp$ comes with a quasi-isomorphism of operads $\Bar^c\COp\xrightarrow{\sim}\oW\COp$ from the cobar construction of $\COp$ as a dg cooperad.
\end{itemize}
Let next $\MOp$ be a $\La$ Hopf $\COp$ comodule.
Then there is a $\La$ Hopf $W\COp$ comodule $W\MOp$ that is a fibrant resolution of $\MOp$ in the category of $\La$ Hopf $W\COp$ comodules
\[
\MOp \xrightarrow{\sim} W\MOp.
\]
We call it the $W$ construction of $\MOp$. It has the following properties:
\begin{itemize}
\item As a dg cooperad $W\MOp=\Bar_{\oW\COp}\oW \MOp$ is the module bar construction of a $\oW\COp$ module $\oW \MOp$.
In particular, $W\MOp$ is cofree as a graded $W\COp$ comodule.
\item The dg module $\oW \MOp$ comes with a quasi-isomorphism of $\oW\COp$ modules $\Bar_{\COp}^c\MOp\xrightarrow{\sim}\oW\MOp$ from the cobar construction of $\MOp$.
\end{itemize}
\end{prop}

\subsection{Mapping spaces}
Let 
$$
\Omega(\Delta^r)=\Q[t_1,\dots,t_r,dt_1,\dots,dt_r]
$$
be the dg commutative algebra of polynomial differential forms on the $r$-simplex. These algebras assemble into a simplicial dg commutative algebra $\Omega(\Delta^\bullet)$.

Let $\COp$ and $\DOp$ be $\La$ Hopf cooperads.
Then $\COp\otimes \Omega(\Delta^\bullet)$ and $\DOp\otimes \Omega(\Delta^\bullet)$ form $\La$ Hopf cooperads over $\Omega(\Delta^\bullet)$. In other words, we extend our ground ring from $\Q$ to $\Omega(\Delta^\bullet)$.
Then we define the mapping space between $\COp$ and $\DOp$ as the simplicial set
\[
\Map_{\dgLaHOpc}(\COp, \DOp)
:=
\Mor_{\dgLaHOpc/\Omega(\Delta^\bullet)}(\COp\otimes \Omega(\Delta^\bullet), \DOp\otimes \Omega(\Delta^\bullet) )
\]
consisting of the $\La$ Hopf cooperad morphisms over the extended ground rings $\Omega(\Delta^\bullet)$.

Similarly, let $\MOp$ and $\NOp$ be two right $\COp$-comodules.
Then we define the mapping space between $\MOp$ and $\NOp$ to be the simplicial set 
\[
  \Map_{\dgLaHModc_{\COp}}(\MOp, \NOp)
:=
\Mor_{\dgLaHModc_{\COp\otimes \Omega(\Delta^\bullet)}/\Omega(\Delta^\bullet)}(\MOp\otimes \Omega(\Delta^\bullet), \NOp\otimes \Omega(\Delta^\bullet) )
\]
consisting of the morphisms of right $\La$ Hopf $\COp\otimes \Omega(\Delta^\bullet)$ comodules over the extended ground ring $\Omega(\Delta^\bullet)$.

Unfortunately, the above definitions of mapping spaces do not readily make the categories of Hopf cooperads or Hopf comodules into simplicial model categories -- the required axioms are not all satisified, only slightly weaker versions, see \cite{FWAut} for a more detailed discussion. However, it is shown in \cite{FWAut, FWMod} that the above definitions of the mapping spaces are nevertheless weakly equivalent to the "true" mapping spaces defined for any model category \cite[II.2]{FrII}, \cite{DwyerKan}.

\subsection{A twisting construction for (co)operads}\label{sec:optwist simple}
Let $\COp$ be a dg cooperad and $\POp=\COp^*$ the dual dg operad.
As for any operad, the unary operations $\POp(1)$ from a dg algebra, and in particular a dg Lie algebra. This dg Lie algebra acts on the operad by operadic derivations by the formula, for $x\in \POp(1)$, $y\in\POp(r)$,
\[
x\cdot y = x\circ_1 y - (-1)^{|x||y|} \sum_{j=1}^r y \circ_j x.
\] 
Dually, $\POp(1)$ aso acts on $\COp$ by cooperadic coderivations.
The dual action may be defined using the same formula as above, but with $x\circ_1 (-)$  and $(-)\circ_j x$ interpreted as the compositions
\[
\COp(r) \to \COp(1) \otimes \COp(r) \xrightarrow{x\otimes \id} \COp(r) ,
\]
with the first arrow the cooperadic cocomposition dual to $\circ_1$ or $\circ_j$.

In particular, let $m\in \POp(1)$ be a Maurer-Cartan element, that is, a degree 1 element such that 
\[
dm + m\circ_1 m =0.
\]
Then $\POp$ with the twisted differential $d+m\cdot$ and unchanged composition is also a dg operad, which we denote by $\POp^m$.
Similarly, $\COp$ with twisted differential $d+m\cdot$ is also a dg cooperad, denoted by $\COp^m$.

Next, let $\COp$ be a Hopf cooperad. Then $\POp(1)$ comes with a commutative coproduct $\Delta:\POp(1)\to (\COp(1)\otimes \COp(1))^*$.
We claim that if our Maurer-Cartan element $m$ is primitive, that is 
\[
\Delta m = m\otimes \epsilon+ \epsilon \otimes m \in   \POp(1)\otimes \POp(1) \subset 
(\COp(1)\otimes \COp(1))^*,
\]
with $\epsilon\in \POp(1)$ the operadic unit,
then the action $m\cdot$ on $\COp(r)$ is also a derivation with respect to the commutative algebra structure.
To see this it suffices to check that each composition 
\[
\COp(r) \to \COp(1) \otimes \COp(r) \xrightarrow{m\otimes \id} \COp(r)  
\]
is a derivation, with the first arrow any of the possible cooperadic cocompostions. Let $y,z\in \COp(r)$ and use Sweedler notation to denote the cocompositions, e.g., $y\to \sum y'\otimes y''$.
Then we verify, using that the cocompositions are commutative algebra maps and the fact that $\sum \epsilon(x')x''=x$:
\begin{align*}
m\cdot (xy) &= \sum m((xy)') (xy)''
=
\sum (-1)^{|y'||x''|} m(x'y') x''y''
\\
&=
\sum \left(  (-1)^{|y'||x''|} m(x')\epsilon(y') x'' y''
+
(-1)^{|y'||x''|+|x'|}\epsilon(x')m(y')x'' y''
\right)
\\
&=\sum ( m(x') x'' y + (-1)^{|x|} x m(y')y'')
=
(m\cdot x)y+ (-1)^{|x|} x(m\cdot y).
\end{align*}

In particular, we conclude that $\COp^m$ is still a dg Hopf cooperad.
This construction can be extended to the case of cooperadic right comodules. Let $\MOp$ be a right cooperadic Hopf comodule for $\COp$.
Then we define the right action of $x\in \POp(1)$ on $a\in \MOp(r)$ to be 
\begin{equation}\label{equ:optwist module}
 a\cdot x = \sum_{j=1}^r a \circ_j x
\end{equation}
using similar notation "$\circ_j$" as in the cooperad case.
If $m$ is a primitive Maurer-Cartan element, then we define the twisted $\COp^m$-comodule $\MOp^m$ to be $\MOp$ with unchanged coaction and commutative algebra structure, but differential $d-(-)\cdot m$.

\section{Modules and comodules of configuration space type}

\subsection{A cofree $\Com$ module}
Let $X$ be any space or simplicial set. Then we consider the right $\La$ $\Com$-module $F_X$ such that 
\[
F_X(r) = X^{\times r}  
\]
for $r=1,2,\dots$.
The $\La$-operations are simply defined by forgetting factors,
\begin{align*}
  \mu_j: X^{\times r}  =F_X(r) \to F_X(r-1)=X^{\times r-1}   \\
  (x_1,\dots, x_r) \mapsto (x_1,\dots, \widehat{x_j},\dots, x_r).
\end{align*}
The right $\Com$-action is defined using the diagonal map $\Delta: X\to X\times X$. For example, the right action of the binary product on the first "slot" is 
\begin{align*}
  X^{\times r}  =F_X(r) \to F_X(r+1)=X^{\times r+1}   \\
  (x_1,\dots, x_r) \mapsto (x_1,x_1,x_2,\dots, x_r),
\end{align*}
and via equivariance and operadic axioms, this formula fully determines the action.

The construction $X\to F_X$ is clearly functorial. Furthermore, we have the following result.
\begin{lemma}
The functor $F$ above fits into the following adjunction
\begin{align*}
  U : \La\Mod_{\Com} &\rightleftarrows \Top : F
\end{align*}
where the left adjoint $U$ assigns to a $\La$ $\Com$ module $\MOp$ its unary part $U(\MOp)=\MOp(1)$.
In the simplicial setting one similarly has an adjunction
\begin{align*}
  U : \La\Mod_{\Com} &\rightleftarrows \sSet : F.
\end{align*}
\end{lemma}
\begin{proof}
  First note that $U\circ F$ is the identity.
Let $\MOp$ be a right $\La$ $\Com$-module and $X$ a simplicial set.
Suppose that $g:U(\MOp)\to X$ is a morphism of simplicial sets. We then check that there is a unique morphism of right $\La$ $\Com$ modules $G:\MOp\to F_X$ such that $U(G)=g$.
We first construct $G$ as a morphism of $\La$ sequences. It is defined on $G(r)$ as the composition 
\[
 \MOp(r)\xrightarrow{\prod_{j=1}^r \lambda_j}
 \prod_{j=1}^r \MOp(1) \xrightarrow{\prod_{j=1}^r g}  
\prod_{j=1}^r X= F_X(r),  
\]
where $\lambda_j: \MOp(r)\to \MOp(1)$ is the $\La$ operation coming from the map of sets $[1]\to [j]$ such that $1\mapsto j$.
We check that this map $G$ respects the $\La$-structures.
Given that $S_r$-equivariance is clear, it suffices to consider compatibility with the generating $\La$-morphisms $\mu_1$ representing an action of a nullary operation on the first slot.
We have for $x\in \MOp(r)$
\begin{align*}
\mu_1(G(x))
&= \mu_1(g(\lambda_1(x)),\dots, g(\lambda_r(x)))
=(g(\lambda_2(x)),\dots, g(\lambda_r(x)))
\\&=(g(\lambda_1(\mu_1(x))),\dots, g(\lambda_{r-1}(\mu_1(x))))
=G(\mu_1(x)).
\end{align*}
Furthermore, note that $G$ is already uniquely defined at this stage by the requiremenet that it intertwines the $\La$ operations $\lambda_j$.

It hence suffices to check that $G$ also respects the right $\Com$ action. Again by $S_r$-equivariance it suffices to consider the generating such action $c_{12}=(-)\circ_1 m$ acting with the binary product generator $m\in \Com(2)$ on the first slot.
We compute
\begin{align*}
c_{12}(G(x))
&= 
c_{12}(g(\lambda_1(x)),g(\lambda_2(x)),\dots, g(\lambda_r(x)))
=(g(\lambda_1(x)),g(\lambda_1(x)),g(\lambda_2(x)),\dots, g(\lambda_r(x)))
\\&=
(g(\lambda_1(c_{12}(x))),g(\lambda_2(c_{12}(x))),g(\lambda_3(c_{12}(x))),\dots, g(\lambda_{r+1}(c_{12}(x))))
=
G(c_{12}(x)),
\end{align*}
using the compatibility of the $\La$ structure and $\Com$-action for the penultimate equality.
\end{proof}

\subsection{Modules of configuration space type}
Any topological or simplicial $\La$ operad $\POp$ comes with a canonical map of $\La$ operads $\POp\to \Com$.
Hence, via restriction, the right module $F_X$ above can be made a right module over any $\La$ operad.

\begin{defn}
  \begin{itemize}
    \item We say that a right $\La$ $\Com$ module $\MOp$ is of \emph{power type} if the adjunction counit $F\circ U(X)\to X$ is a weak equivalence.
    \item Let $\POp$ be a well pointed $\La$ operad and $\MOp$ a right $\POp$ module.
    Then we say that $\MOp$ is of \emph{configuration space type} if the homotopy induced right module $\Ind_{\POp}^{\Com,h}\MOp$ is of power type.
  \end{itemize}
\end{defn}
In other words, the right $\POp$ module $\MOp$ is of configuration space type if the map 
\[
 \Ind_{\POp}^{\Com,h}\MOp \to F_{(\Ind_{\POp}^{\Com,h}\MOp)(1)}
\]
is a weak equivalence. In the special case that $\POp(1)=*$ is trivial we furthermore have that $(\Ind_{\POp}^{\Com,h}\MOp)(1)=\MOp(1)$.
We shall also note that the functors $F$ and $U$ preserve all weak equivalences, and hence the adjunction counit appearing in the above definition can also be interpreted as the derived adjunction counit.

We extend the above definitions to operads and right modules in topological spaces by using the equivalence between $\Top$ and $\sSet$.
The following examples justify the definition.

\begin{prop}\label{prop:FM configuration type}
  Let $\Mfd$ be an $n$-dimensional manifold.
  \begin{itemize}
    \item The Fulton-MacPherson-Axelrod-Singer framed configuration space $\FM_\Mfd^{fr}$ as a right $\La$ module over $\FM_n^{fr}$ is of configuration space type.
    \item Suppose that $\Mfd$ is parallelized. Then the Fulton-MacPherson-Axelrod-Singer compactified configuration space $\FM_\Mfd$ as a right $\La$ module over $\FM_n$ is of configuration space type.
  \end{itemize}
\end{prop}
\begin{proof}
The right $\FM_n$ modules $\FM_M$ is (Reedy) cofibrant by Proposition \ref{prop:FMM cofibrant} Hence
\[
  \Ind_{\FM_n}^{\Com,h}\FM_\Mfd =\Ind_{\FM_n}^\Com\FM_\Mfd.
\]
The right hand object is easily seen to be $F_\Mfd=F_{U(\FM_M)}$.

Similarly, $\FM_\Mfd^{fr}$ is also Reedy cofibrant in the category of right $\FM_n^{fr}$ modules, and hence 
\[
  \Ind_{\FM_n^{fr}}^{\Com,h}\FM_\Mfd^{fr} =\Ind_{\FM_n^{fr}}^\Com\FM_\Mfd^{fr}.
\]
The right-hand side is again easily seen to be equal to $F_M$, and hence of power type.
\end{proof}

For later use we shall also record some properties of modules of configuration space type.

\begin{lemma}\label{lem:csc we invariant}
  Let $\POp$ be a simplicial or well pointed topological $\La$ operad,
  and $f:\MOp\to \NOp$ a weak equivalence of $\La$ modules over $\POp$. Then $\MOp$ is of configuration space type iff so is $\NOp$.
\end{lemma}
\begin{proof}
To compute the homotopy induced module, we pick cofibrant resolutions $\hat \MOp \to \MOp$ and $\hat \NOp\to \NOp$, and a lift of $f$, $\hat f: \hat \MOp\to \hat \NOp$.
We then have a commutative diagram 
\[
\begin{tikzcd}
  \Ind_{\POp}^\Com \hat \MOp \ar{r}\ar{d}{\sim}& F_{\MOp(1)} \ar{d}{\sim}
  \\
  \Ind_{\POp}^\Com \hat \NOp \ar{r}& F_{\NOp(1)}
\end{tikzcd}
\]
in which the vertical arrows are weak equivalences. Hence the upper horizontal arrow is a weak equivalence iff so is the lower.
\end{proof}

\begin{lemma}\label{lem:csc op change invariant}
Let $g:\POp\to \QOp$ be a morphism of (simplicial or well pointed topological) $\La$ operads, and let $\MOp$ be a $\La$ $\POp$ module.
Then $\MOp$ is of configuration space type iff so is $\Ind_{\POp}^{\QOp, h}\MOp$.
\end{lemma}
\begin{proof}
By using Lemma \ref{lem:csc we invariant} we may replace $\MOp$ by a cofibrant $\POp$-module, or just assume that $\MOp$ is cofibrant from the start.
Then by adjunction $\Ind_{\POp}^{\QOp}\MOp$ is also cofibrant.
We furthermore have 
\[
  \Ind_{\POp}^{\Com}\MOp \cong \Ind_{\QOp}^{\Com}\Ind_{\POp}^{\QOp}\MOp,
\]
and hence if one of these modules is of power type so is the other.
\end{proof}

\subsection{Construction: A free $\Com^c$-comodule}
\label{sec:Fc construction}
Let $A$ be a (unital) dg commutative algebra.
Then we define a right cooperadic $\Com^c$ Hopf $\La$-comodule $\Fc_A$ such that 
\[
\Fc_A(r) =  A^{\otimes r}
\] 
for $r=1,2,\dots$. The $\La$-operations
\[
  \Fc_A(r) \to \Fc_A(r+s)
\]
are obtained by tensoring the identity map $r$ times with the unit ($s$ times). In particular, 
\[
\mu_j^*(a_1\otimes\cdots \otimes a_{r-1})
=
(a_1\otimes\cdots \otimes 1\otimes \cdots\otimes a_{r-1}),
\]
with the unit $1$ at the $j$-th position of the tensor product.
The right $\Com^c$-coaction is given by the algebra product in a natural way.
Furthermore the construction $F_A^c$ is obviously functorial in $A$.

\begin{lemma}\label{lem:F Com adjunction}
The functor $\Fc: A\to F_A^c$ fits into an adjunction
\[
  \Fc: \dgcAlg \leftrightarrows \dgLaHModc_{\Com^c}: \Uc,
\]
with the right adjoint $\Uc$ being the restriction to the arity one component, $\Uc(\MOp) = \MOp(1)$.
\end{lemma}
\begin{proof}
  First it is clear that $\Uc\circ \Fc=\mathit{id}_{\dgcAlg}$.
Let $\MOp$ be a right $\Com^c$ Hopf $\La$-comodule and $A$ be a dg commutative algebra.
Let $g: A\to \Uc(\MOp)$ be a morphism of dg commutative algebras.
Then we will show that there is a unique morphism $G:\Fc_A\to \MOp$ of $\Com^c$-comodules such that $\Uc(G)=g$. This then shows the claim.
To construct $G$ consider the $\Lambda$-operations 
\[
  \lambda_j: \MOp(1) \to \MOp(r)
\]
coming from the map of sets $[1]\to [j]$ such that $1\mapsto j$.
Then we define $G$ on $a_1\otimes \cdots \otimes a_r \in \Fc_A(r)$ such that 
\begin{equation}\label{equ:G def}
G(a_1\otimes \cdots \otimes a_r) 
:= \lambda_1(g(a_1))\cdots \lambda_r(g(a_r)).
\end{equation}
It is clear that $G$ is a morphism of dg commutative algebras since so are the $\lambda_j$ and $g$, and it is also easy to check that $G$ respects the $\La$-structure.
Hence $G$ is a map of $\La$ Hopf sequences, and already unique as such, because $\Fc_A(r)$ is generated by $\Fc_A(1)=A$ via the $\La$ operations and the product.
We hence only need to check that $G$ also respects the $\Com^c$-coaction.
To this end it is enough to verify commutativity of the diagram 
\[
  \begin{tikzcd}
    \Fc_A(r+1) \ar{r}{G}  \ar{d}{c_{12}} & \MOp(r+1) \ar{d}{c_{12}}\\
    \Fc_A(r) \ar{r}{G} & \MOp(r)  
  \end{tikzcd},
\]
where the vertical arrows $c_{12}$ are the $\Com^c$-coaction from the map $[r+1]\to [r]$ sending $j\mapsto \min(j-1,1)$. 
We compute 
\begin{align*}
  G(c_{12}(a_1\otimes \cdots \otimes a_{r+1}))
  &=
  G(a_1a_2\otimes \cdots \otimes a_{r+1})
  =
  \lambda_1(g(a_1a_2))\lambda_2(g(a_3))\cdots \lambda_r(g(a_{r+1}))
  \\&=
  \lambda_1(g(a_1))\lambda_1(g(a_2))\lambda_2(g(a_3))\cdots \lambda_r(g(a_{r+1}))
  \\&=
  c_{12}(\lambda_1(g(a_1))\lambda_2(g(a_2))\lambda_3(g(a_3))\cdots \lambda_{r+1}(g(a_{r+1})))
  \\&=
  c_{12}(G(a_1\otimes \cdots \otimes a_{r+1})),
\end{align*}
where we use the compatibility of the $\Com^c$ ocoaction and the $\La$ structure for the fourth equality.
\end{proof}

\subsection{Comodules of configuration space type}

\begin{defn}\label{def:config space type c}
  \begin{itemize}
    \item We say that a right $\La$ $\Com$ Hopf comodule $\MOp$ is of \emph{power type} if the adjunction unit $X\to \Fc\circ \Uc(X)$ is a weak equivalence.
    \item Let $\COp$ be a $\La$ Hopf cooperad and $\MOp$ a right $\COp$ comodule.
    Then we say that $\MOp$ is of \emph{configuration space type} if the homotopy coinduced right module $\coInd_{\POp}^{\Com,h}\MOp$ is of power type.
  \end{itemize}
\end{defn}
Again we note that the functors $\Fc$ and $\Uc$ preserve weak equivalences, so that the adjunction unit can also be seen as the derived adjuncition unit.

\begin{prop}
  If $\MOp$ is a simplicial or topological cofibrant right module of configuration space type, then $\Omega_\sharp\MOp$ is a right comodule of configuration space type.
\end{prop}
\begin{proof}
This follows since by Proposition \ref{prop:coind rationalization commutation} the functor $\Omega_{\sharp}$ intertwines the induction and coinduction functors, on cofibrant objects.
\end{proof}

One also has dual versions to Lemmas \ref{lem:csc we invariant} and \ref{lem:csc op change invariant}, that are proven in the same manner.
\begin{lemma}\label{lem:comodule csc basic props}
  \begin{itemize}
\item If $\MOp\to \NOp$ is a weak equivalence of dg $\La$ (Hopf) $\COp$ comodules, then $\MOp$ is of configuration space type iff so is $\NOp$.
\item If $\MOp$ is dg $\La$ (Hopf) $\DOp$ comodule and $\COp\to \DOp$ is a morphism of dg $\La$ Hopf cooperads, then $\coInd_{\DOp}^{\COp}\MOp$ is of configuration space type iff so is $\MOp$.
  \end{itemize}
\end{lemma}

Combining the above with Proposition \ref{prop:FM configuration type} we obtain:
\begin{cor}\label{cor:FM Hopf csc}
  Let $\Mfd$ be an $n$-dimensional manifold.
 Then any (dg Hopf) cooperadic comodule model for $\FM_\Mfd^{fr}$ considered as a right $\La$ $\FM_n^{fr}$-module is of configuration space type. If $\Mfd$ is parallelized then the same is true for any (dg Hopf) cooperadic comodule model of $\FM_\Mfd$, considered as a right $\La$ $\FM_n$-module.
\end{cor}

\subsection{Freeness as $\COp$-comodule}
Let $\COp$ be a dg $\La$ Hopf cooperad. The algebra units then endow it with a canonical map of  Hopf $\La$ cooperads
\[
  \Com^c \to \COp,
\]
or in other words, $\Com^c$ is an initial object in the category of Hopf $\La$ cooperads.
This means in particular that any right cooperadic $\Com^c$ (Hopf $\La$-)comodule is canonically a $\COp$ comodule by corestriction.
In particular, the $\Com^c$-comodule $\Fc_A$ of the previous subsection is a $\COp$-comodule for any $\COp$.

Also note that as before we have a forgetful functor 
\[
\Uc : \dgLaHModc_{\COp} \to \dgcAlg
\]
such that $\Uc(\MOp)=\MOp(1)$.

Then we can strengthen Lemma \ref{lem:F Com adjunction} as follows.
\begin{prop}\label{prop:freeness as C mod}
Suppose that $\COp$ is a reduced $\La$ Hopf cooperad, i.e., $\COp(1)\cong \Q$.
Then the functor $U$ fits into an adjunction
\[
 \Fc : \dgcAlg \leftrightarrows \dgLaHModc_{\COp} : \Uc
\]
with the left adjoint $\Fc$ being the functor defined in section \ref{sec:Fc construction}.
\end{prop}
\begin{proof}
We may follow the lines of the proof of Lemma \ref{lem:F Com adjunction}.
Again, given a map of dgca $g:A\to \Uc(\MOp)$, with $\MOp$ a $\COp$-comodule, we construct a map of $\COp$-comodules $G:F_A\to \MOp$ such that $\Uc(G)=g$.
This map is automatically unique, since as shown in the proof of Lemma \ref{lem:F Com adjunction} is is unique even as a map of $\La$ Hopf sequences.
Furthermore, it is necessarily given by \eqref{equ:G def}, but we have to verify that this map $G$ respects the full $\COp$-coaction.
To this end it suffices to check commutativity of the diagrams
\begin{equation}\label{equ:freeness C proof}
  \begin{tikzcd}
    \Fc_A(r+s-1) \ar{r}{G}  \ar{d}{c} & \MOp(r+s-1) \ar{d}{c}\\
    \Fc_A(r)\otimes \COp(s) \ar{r}{G\otimes \mathit{id_{\COp}}} & \MOp(r) \otimes \COp(s),
  \end{tikzcd},
\end{equation}
where now the vertical arrows are the cooperadic coactions arising from the tree
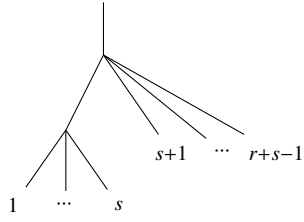
\[
\begin{tikzpicture}
\coordinate (v) at (0,1);
\coordinate (w) at (-.5,0);
\node (e1) at (-1.2,-1) {$\scriptstyle 1$};
\node (ed) at (-.5,-1) {$\scriptstyle \cdots$};
\node (es) at (.2,-1) {$\scriptstyle s$};
\node (es1) at (.9,-.3) {$\scriptstyle s+1$};
\node (edd) at (1.6,-.3) {$\scriptstyle \cdots$};
\node (esr) at (2.3,-.3) {$\scriptstyle r+s-1$};
\draw (v) edge +(0,.7) edge (w) edge (es1) edge (edd) edge (esr)
(w) edge (e1) edge (es) edge (ed)
;
\end{tikzpicture}
\]
Note that by construction of the right $\COp$-coaction on $\Fc_A$ by corestriction the left-hand vertical arrow in \eqref{equ:freeness C proof} factors through $\Fc_A(r)\otimes \Com^c(s)$.
We compute
\begin{align}
  (G\otimes \mathit{id}_{\COp})(c(a_1\otimes \cdots \otimes a_{r+s-1}))
  &=
  G(a_1\cdots a_s\otimes a_{s+1}\otimes \cdots \otimes a_{r+s-1}) \otimes 1
  \\&=
  \lambda_1(g(a_1\cdots a_s))\lambda_2(g(a_{s+1}))\cdots \lambda_r(g(a_{r+s-1}))\otimes 1
  \\&=
  \lambda_1(g(a_1))\cdots \lambda_1(g(a_s))\lambda_2(g(a_{s+1}))\cdots \lambda_r(g(a_{r+s-1}))\otimes 1 \label{equ:padj1}
\end{align}

On the other hand
\begin{align}
  c(G(a_1\otimes \cdots \otimes a_{r+s-1})) 
  &=
  c(\lambda_1(g(a_1))\cdots \lambda_{r+s-1}(g(a_{r+s-1})))
  \\&=
  c(\lambda_1(g(a_1))) \cdots c(\lambda_{r+s-1}(g(a_{r+s-1}))),\label{equ:padj2}
\end{align}
using the compatibility of the cooperadic coaction with the commutative algebra structure.
We furthermore need to use the compatibility of the cooperadic coaction with the $\La$-structure.
Consider first the case that $j>s$. Then the compatibility is expressed as the commutativity of the diagram
\[
  \begin{tikzcd}
    \MOp(1) \ar{r}{\lambda_j}  \ar{d}{\mathit{id}} & \MOp(r+s-1) \ar{d}{c}\\
    \MOp(1) \ar{r}{\lambda_{j-s+1}\otimes 1} & \MOp(r) \otimes \COp(s)
  \end{tikzcd}.
\]
Here we write $1$ for the inclusion of the unit in $\COp(s)$. In particular
\[
c(\lambda_j(g(a_j))) = \lambda_{j-s+1}(g(a_j))) \otimes 1.
\]
In the case that $j\leq s$ the compatibility reads
\[
  \begin{tikzcd}
    \MOp(1) \ar{r}{\lambda_j}  \ar{d} & \MOp(r+s-1) \ar{d}{c}\\
    \MOp(1) \otimes \COp(1) \ar{r}{\lambda_1\otimes \tilde \lambda_j} & \MOp(r) \otimes \COp(s)
  \end{tikzcd}.
\]
where $\tilde \lambda_j :\COp(1) \to \COp(s)$ is the appropriate $\La$-operation on $\COp$.
Now since $\COp$ is reduced, i.e., that $\COp(1)=\Q$, the factor $\COp(1)$ on the lower left can be omitted, the left-hand vertical arrow is the identity, and the operation $\tilde \lambda_j$ is again given by the inclusion of the unit. Hence
\[
c(\lambda_j(g(a_j))) = \lambda_{1}(g(a_j))) \otimes 1.
\]
Plugging these formulas into \eqref{equ:padj2} we see that the end result agrees with \eqref{equ:padj2}, thus finishing the proof of the compatibility of $G$ with the right $\COp$-coaction.
\end{proof}

\section{Properties of comodules of configuration space type}

\subsection{Equivalent conditions}
\begin{prop}\label{prop:alt condition}
Let $\MOp$ be a $\La$ Hopf $\COp$ comodule.
Let $\MOp\xrightarrow{\sim}\hat \MOp$ be a fibrant replacement in the category of (plain) $\COp$ comodules.
Then $\MOp$ is of configuration space type if and only if the composition 
\[
\FF_{\MOp(1)}
\to 
\coInd^{\Com^c}_{\COp}\MOp 
\to 
\coInd^{\Com^c}_{\COp}\hat\MOp 
\]
is a quasi-isomorphism.
\end{prop}

\begin{proof}
First suppose that $\hat \MOp$ is fibrant in the category of $\La$ Hopf $\COp$ comodules. Then it is also fibrant in the category of plain $\COp$ comodules by adjunction. Furthermore, since limits are created in the category of plain comodules it does not matter in which category the coinduction is taken.
Hence in this case the statement of the Proposition reduces to the definition of being of configuration space type.

Thus we only need to show that the condition in the proposition is independent of the fibrant replacement $\hat M$ chosen.
Equivalently, we take one fixed fibrant replacement
\[
\MOp \xhookrightarrow{\sim} \MOp',
\]
for which we may assume that the map is a cofibration.
Then by lifting we obtain a commutative diagram
\[
\begin{tikzcd}
  \MOp \ar{r}{\sim}\ar[hookrightarrow]{d}{\sim} & \hat \MOp \ar[twoheadrightarrow]{d} \\
  \MOp' \ar{r} \ar[dashed]{ur}{\sim}& *.
\end{tikzcd}
\]
But then we have a corresponding diagram 
\[
  \begin{tikzcd}
    \FF_{M(1)} \ar{r} & 
    \coInd^{\Com^c}_{\COp}\MOp \ar{r}\ar{d} & \coInd^{\Com^c}_{\COp}\hat \MOp  \\
    & \coInd^{\Com^c}_{\COp}\MOp' \ar{ur}{\sim}& ,
  \end{tikzcd}
\]
and the upper composition is a quasi-isomorphism iff so is the lower one.
\end{proof}

The case of most interest to us is the following:
\begin{cor}\label{cor:alt condition}
Let $\MOp$ be a $\La$ Hopf $\e_n^c$ comodule. Then $\MOp$ is of configuration space type if and only if the map
\[
\FF_{\MOp(1)}
\to 
\coInd^{\Com^c}_{\e_n^c}
\Bar_{\e_n^c} \Bar^c_{\e_n^c}\MOp
\cong
\Bar_{\Com^c} \Bar^c_{\e_n^c}
\MOp
\]
is a quasi-isomorphism, where we use the bar and cobar construction of comodules, see \cite{Frbook2}.
\end{cor}
\begin{proof}
  The bar-cobar resolution of $\MOp$ is fibrant, so we can apply Proposition \ref{prop:alt condition}.
For the final isomorphism one uses the explicit formula for the coinduction.
\end{proof}

There is yet another simplification possible. Note that one has a morphism of (plain) graded cooperads (for $n\geq 2$)
\[
  \e_n^c \to \Lie_n^c,
\]
where $\Lie_n^c$ is the cooperad governing shifted Lie coalgebras with Lie cobracket of degree $n-1$.
In particular any $\e_n^c$ comodule $\MOp$ is also a $\Lie_n^c$-comodule.
Furthermore, one has a morphism of dg symmetric sequences
\begin{equation}\label{equ:bar cobar to bar Lie}
\Bar_{\Com^c} \Bar_{\e_n^c}^c \MOp
\cong (\MOp\circ \Com\{n\}\circ\Lie_0 \circ \Com^c, d)
\to 
\Bar^c_{\Lie_n^c} \MOp
\cong (\MOp\circ \Com\{n\},d)
\end{equation}
obtained from the projection of $\Lie_0 \circ \Com^c$ to its unary part.
In particular, note the composition 
\[
\FF_{\MOp(1)} \to 
\Bar_{\Com^c} \Bar_{\e_n^c}^c \MOp
\to
\Bar^c_{\Lie_n^c} \MOp
\]
sends any $x\in \FF_{\MOp(1)}(r)$ to its image in $\MOp(r)\subset \Bar^c_{\Lie_n^c} \MOp$, and this yields a well defined map because the $\Lie_n^c$-coaction on $\FF_{\MOp(1)}$ is trivial.

\begin{prop}\label{prop:Cobar Lie criterion}
For $n\geq 2$, a $\La$ Hopf $\e_n^c$ comodule is of configuration space type iff the map of dg symmetric sequences
\[
  \FF_{\MOp(1)} \to \Bar^c_{\Lie_n^c} \MOp
\]
is a quasi-isomorphism.
\end{prop}
\begin{proof}
It suffices to show that the map \eqref{equ:bar cobar to bar Lie} is a quasi-isomorphism, in each arity separately. Endow the left-hand side of \eqref{equ:bar cobar to bar Lie} with a filtration by the total arity of the factor $\MOp\circ \Com\{n\}$. 
This filtration is obviously bounded in each arity, so that the spectral sequence converges to the cohomology.
The associated graded symmetric sequence has the form 
\[
  (\Bar^c_{\Lie_n^c} \MOp) \circ (\Lie_0 \circ \Com^c,d),
\]
and the right-hand factor $(\Lie_0 \circ \Com^c,d)$ is the operadic Koszul complex of the commutative cooperad. By Koszulness of the latter, it is acyclic in arities $>1$, and the Proposition easily follows. 
\end{proof}

\subsection{Algebraic Cohen-Taylor-type spectral sequence}
\label{sec:CohenTaylorSS}

There are several constructions of spectral sequences computing the cohomology of configuration spaces of "nice" manifolds in the literature \cite{CohenTaylor, Totaro, Kriz}
\[
E^1 = \LS_{H(M)} \Rightarrow H(\FM_M)  
\]
whose $E^1$-page is identified with $\LS_{H(M)}:=\FF_{H(\MOp(1))} \circ \Lie_n^c$.
We shall note that such a spectral sequence can also easily be constructed purely algebraically in a general setting.
Concretely, let $\MOp$ is a right $\e_n^c$-comodule of configuration space type. 
We consider the bar-cobar resolution 
\[
\hat M=\Bar_{\e_n^c}\Bar^c_{\e_n^c}\MOp.  
\]
The $\e_n^c$ comodule $\hat M$ is quasi-free, and $\e_n^c$ has a natural grading by the number of Lie cobrackets. Here we assume $n\geq 2$.
This grading induces an ascending, bounded below, exhaustive filtration on $\hat M$
\[
\mF^0\hat \MOp \subset \mF^1\hat \MOp \subset \cdots.
\]
More precisely, if we write explicitly $\hat \MOp=(\MOp \circ \e_n\{n\}\circ \e_n^c, d)$, then $\mF^p\hat \MOp$ is spanned by elements that contain $\leq p$ Lie cobrackets in the factor $\e_n^c$.
Note that the filtration is bounded since cobrackets carry positive degree, and $\hat M$ is non-negatively graded. Hence the associated spectral converges to the cohomology. The $E^0$-page (associated graded) has the form 
\[
(\Bar_{\Com^c}\Bar^c_{\e_n^c}\MOp)\circ \Lie_n^c, 
\]
and by Corollary \ref{cor:alt condition} its cohomology is 
\begin{equation}\label{equ:E1LS}
E^1=\FF_{H(\MOp(1))} \circ \Lie_n^c.
\end{equation}


For the moment, we shall only need the following application, whose proof uses the above spectral sequence, up to the $E^1$-page.
\begin{prop}\label{prop:configtype qiso condition}
  Let $\COp$ be a $\La$ Hopf cooperad quasi-isomorphic to $\e_n^c$ for some $n\geq 2$.
  Let $f:\MOp\to \NOp$ be a morphism of $\La$ Hopf $\COp$-comodules of configuration space type. 
  Suppose that the arity one part of the map $f$ 
  \[
  U(f): \MOp(1) \to \NOp(1)
  \]
  is a quasi-isomorphism.
  Then $f$ is a quasi-isomorphism as well.
\end{prop}
\begin{proof}
First, since the categories of $\e_n^c$ and of $\COp$ comodules are Quillen equivalent, it suffices to consider the case of $\e_n^c$ comodules.
Next, by bar-cobar-resolving $\MOp$ and $\NOp$ as right (non-Hopf) $\e_n^c$ comodules it suffices to show that the lower arrow in the diagram 
\[
\begin{tikzcd}
  \MOp\ar{d}{\sim} \ar{r}{f} & \NOp \ar{d}{\sim}\\
  \Bar \Bar^c \MOp \ar{r}{\Bar \Bar^cf} & \Bar \Bar^c \NOp
\end{tikzcd}
\] 
is a weak equivalence.
But we may equip both sides with the filtration $\mF(-)$ described above and consider the associated spectral sequence.
By the remarks before the Proposition the $E^1$ pages on either side agree with \eqref{equ:E1LS}, using that $H(\MOp(1))\cong H(\NOp(1))$ by assumption. Also, one easily checks that the induced map by $\Bar \Bar^cf$ is an isomorphism on the $E^1$-pages.
Hence one concludes that $\Bar \Bar^cf$ and hence $f$ is a quasi-isomorphism.
\end{proof}

\subsection{The diagonal element}
\begin{defn}\label{def:diagonal element alg}
Let $A$ be a differential graded commutative algebra. Then we say that an element $\Delta=\sum \Delta' \otimes \Delta''\in A\otimes A$ of degree $n$ is a \emph{diagonal element (of degree $n$)} for $A$ if the following conditions are satisfied. 
\begin{enumerate}
  \item $\Delta$ is closed, i.e., $d\Delta=0$.
  \item $\Delta$ is graded (anti)symmetric, 
  \[
  \tau \Delta = (-1)^n\Delta,
  \] 
  with $\tau:A\otimes A\to A\otimes A$ switching the factors.
  \item For all $a\in A$ we have
  \begin{equation}\label{equ:a Delta commute}
    (a\otimes 1) \cdot \Delta = (1\otimes a) \cdot \Delta \in A\otimes A.
  \end{equation}
\end{enumerate}
The Euler class associated to the diagonal element $\Delta$ is 
\[
E(\Delta)=\sum\Delta'\Delta'' \in A.
\]
\end{defn}
Note that we do not require any non-degeneracy from $\Delta$ -- the element $\Delta=0$ is also a diagonal element according to our definition. Furthermore, for odd $n$ the Euler class $E(\Delta)$ always vaishes by symmetry.

A diagonal element defines a degree $n$ coproduct on $A$ by the formula 
\[
a\mapsto \sum \sum a\Delta' \otimes \Delta'',
\]
and in fact a graded version of a Frobenius algebra structure. Note also that any multiple of a diagonal element is again a diagonal element.

We shall see that to a $\La$ Hopf $\e_n^c$ comodule of configuration space type we can associate a diagonal element on the cohomology of the arity one part, at least up to a scale factor.

\begin{lemma}
Let $\COp$ be a dg $\La$ Hopf cooperad and let $\MOp$ be a $\La$ Hopf $\COp$ comodule.
Then there is a long exact sequence of cohomology groups
\begin{multline}\label{equ:coIndlongexact}
  \cdots \to H^p(\coInd_{\COp}^{\Com^c,h}(\MOp)(2))
  \to H^p( \MOp(1)) \oplus H^p(\MOp(2)) 
  \to H^p(\MOp(1)\otimes \COp(2))
  \\\xrightarrow{c} H^{p+1}(\coInd_{\COp}^{\Com^c}(\MOp)(2))
  \to \cdots.
\end{multline}
Furthermore, all morphisms in this sequence are compatible with the $S_2$-action, and with the multiplicative action of $H^q(\coInd_{\COp}^{\Com^c}(\MOp)(2))$ in the sense that for any $x\in H^q(\coInd_{\COp}^{\Com^c}(\MOp)(2))$ one has a commutative diagram

\adjustbox{scale=.7,center}{%
\begin{tikzcd}[scale=.5]
  \cdots \ar{r} 
  & H^p(\coInd_{\COp}^{\Com^c,h}(\MOp)(2))
  \ar{r}\ar{d}{x\cdot}
  & H^p( \MOp(1)) \oplus H^p(\MOp(2)) \ar{r}\ar{d}{x\cdot}
  & H^p(\MOp(1)\otimes \COp(2))\ar{r}\ar{d}{x\cdot}
  & H^{p+1}(\coInd_{\COp}^{\Com^c,h}(\MOp)(2)) \ar{r}\ar{d}{x\cdot}
  & \cdots
  \\
  \cdots \ar{r} 
  & H^{p+q}(\coInd_{\COp}^{\Com^c,h}(\MOp)(2))\ar{r}
  & H^{p+q}( \MOp(1)) \oplus H^p(\MOp(2)) \ar{r}
  & H^{p+q}(\MOp(1)\otimes \COp(2))\ar{r}
  & H^{p+q+1}(\coInd_{\COp}^{\Com^c,h}(\MOp)(2)) \ar{r}
  & \cdots
\end{tikzcd}\, .
}
\end{lemma}
\begin{proof}
We may assume that $\MOp$ is fibrant, since \eqref{equ:coIndlongexact} is unaltered upon replacing $\MOp$ by a fibrant resolution.
By construction the arity 2 part of the coinduced comodule fits into a pullback square in the category of dg commutative algebras $\dgcAlg$:
\begin{equation}\label{equ:coind2pullback}
  \begin{tikzcd}
    \coInd_{\COp}^{\Com^c}(\MOp)(2) 
    \ar{r} \ar{d} & \MOp(1)=\MOp(1)\otimes \Com^c(2)\ar{d} \\
    \MOp(2) \ar{r} & \MOp(1) \otimes \COp(2) 
  \end{tikzcd} .
\end{equation}
Note that limits in $\dgcAlg$ are created in dg vector spaces, so this diagram is also a pullback diagram in dg vector spaces.
 
We claim that the diagram is a homotopy pullback diagram.
From this it then immediately follows that we have the long exact sequence \eqref{equ:coIndlongexact}.
The compatibility with the multiplicative action of $H^q(\coInd_{\COp}^{\Com^c}(\MOp)(2))$ also follows readily since \eqref{equ:coind2pullback} is a commutative diagram of dg commutative algebras.

To show the claim we may replace $\MOp$ by a(nother) fibrant replacement. We may even take this replacement in the category of plain (non-Hopf) dg $\COp$-comodules, since limits in the category of dg Hopf $\COp$-comodules (such as $\coInd_{\COp}^{\Com^c}(\MOp)$) are created in the category of plain dg $\COp$-comodules. But taking the bar-cobar construction as a fibrant replacement we explicitly see that the lower horizontal arrow in \eqref{equ:coind2pullback} is a fibration (i.e., a surjection), and hence the claim follows.
\end{proof}


Next assume that $\MOp$ is of configuration space type. Then the connecting homomorphism $c$ above in particular induces maps
\begin{equation}\label{equ:prediagcomp}
  H^p(\COp(2)) \xrightarrow{1\otimes (-)} H^p(\MOp(1)\otimes \COp(2))
  \xrightarrow{c}
    H^{p+1}(\coInd_{\COp}^{\Com^c,h}(\MOp)(2))
    =
    H^p(\MOp(1)\otimes \MOp(1)).
\end{equation}

\begin{defn}\label{def:diagonal element}
Let $\MOp$ be a $\La$ Hopf $\COp$-comodule of configuration space type, with $\COp$ a $\La$ Hopf cooperad such that $H(\COp)=\e_n^c$.
Then we say that the \emph{diagonal element} 
\[
  \Delta_{\MOp} \in H^{n}(\MOp(1)\otimes \MOp(1))
\]
of $\MOp$ is the image of the cobracket $B\in \e_n^c(2)=H(\COp(2))$ under the composition \eqref{equ:prediagcomp}, that is
\[
  \Delta_{\MOp} 
  = c(1\otimes B)\, .
\]
\end{defn}

\begin{lemma}
The element $\Delta_{\MOp} \in H^{n}(\MOp(1)\otimes \MOp(1))$
is indeed a diagonal element of degree $n$ for $H(\MOp(1))$ in the sense of definition \ref{def:diagonal element alg}.
Its Euler class is zero, $E(\Delta_{\MOp})=0$.
\end{lemma}
\begin{proof}
Closedness is void in this case.
The symmetry follows from the $S_2$-equivariance of \eqref{equ:coIndlongexact}.
To show \eqref{equ:a Delta commute} let $a\in H(\MOp(1))$ and set 
\[
  \tilde a := (a\otimes 1)-(1\otimes a)
  \in H(\MOp(1))\otimes H(\MOp(1))
  =H(\coInd_{\COp}^{\Com^c,h}(\MOp)(2)).
\]
We hence have to show $\tilde a\Delta_{\MOp} = 0$.
But by compatibility of the long exact sequence \eqref{equ:coIndlongexact} with the multiplicative action of $H(\MOp(1))\otimes H(\MOp(1))$ we have 
\[
  \tilde a\Delta_{\MOp}
  = 
  c(\underbrace{\tilde a\cdot (1\otimes B)}_{=0} )
  =0,
\]
since the multiplicative action on $H(\MOp(1))$ factors through the multiplication map $H(\MOp(1))\otimes H(\MOp(1))\to H(\MOp(1))$, which kills $\tilde a$.

Finally, by exactness of the sequence \eqref{equ:coIndlongexact} the Euler class is in the kernel of the map 
\[
  H^n(\coInd_{\COp}^{\Com^c,h}(\MOp)(2))
  =
  H^n(\MOp(1)\otimes \MOp(1))
  \to H^n( \MOp(1)) \oplus H^n(\MOp(2)).
\]
In particular the map to $H^n( \MOp(1))$ is the multiplication of both factors, and hence the Euler class of $\Delta_{\MOp}$ vanishes, $E(\Delta_{\MOp})=0$.
\end{proof}


The terminology is derived from the following main example.
\begin{ex}\label{ex:diag of config space}
  We shall work over the reals in this example.
  Let $\Mfd$ be an orientable compact $n$-dimensional manifold. The map $H_c(M)\to H(M)$ from compactly supported cohomology of $\Mfd$ to cohomology of $\Mfd$ is represented by an element
  \[
  \Delta_{\Mfd} \in \bigoplus_p (H^p_c(\Mfd) )^* \otimes H^p(\Mfd)  \cong \bigoplus_p H^{n-p}(\Mfd)\otimes H^p(\Mfd),
  \]
using Poincar\'e duality in the last step.
We call $\Delta_\Mfd\in H^n(\Mfd\times \Mfd)$ the diagonal element of $\Mfd$. If $\Mfd$ is closed, $\Delta_\Mfd$ is the Poincar\'e dual of the diagonal inside $\Mfd\times \Mfd$, hence the name.
Note however, that in general $\Delta_\Mfd$ need not be non-degenerate in any sense, for example we have $\Delta_{\R^n}=0$ for $n\geq 1$.

Suppose next that $\Mfd$ is parallelized. Then a de Rham representative of $\Delta_\Mfd$ may be computed as follows. First we have the inclusion of the unit sphere bundle of the tangent bundle 
\[
ST\Mfd \cong \Mfd\times S^{n-1} \to \FM_\Mfd(2)   
\]
as the part of the boundary corresponding to both points in the configuration being infinitely close.
Let $\eta\in \Omega_{dR}^{n-1}(S^{n-1})$ be the standard volume form on $S^{n-1}$, and let $\beta\in \Omega^{n-1}_{dR}(ST\Mfd)$ be its pullback to $ST\Mfd\cong \Mfd\times S^{n-1}$.
Let $\alpha \in \Omega^{n-1}_{dR}(\FM_\Mfd(2))$ be an $n-1$-form with support in a tubular neighborhood of $ST\Mfd$ extending $\beta$. We may also assume for simplicity that $d\alpha=0$ in a (smaller) neighborhood of $ST\Mfd$. 
In particular $d\alpha$ represents a cohomology class in $H^n(\Mfd\times \Mfd)$, which we claim is the diagonal $\Delta_M$ as defined above.
Indeed, we check that for $\gamma\in \Omega_{dR,c}(M)$ we have 
\begin{align*}
\int_{y\in M} d\alpha(x,y) \gamma(y) 
&=
\int_{\FM_\Mfd(2)\to M} d\alpha(x,y) \gamma(y)
=
\int_{ST\Mfd\to \Mfd} \alpha(x,y) \gamma(y)
=
\gamma(x).
\end{align*}
Here we first integrate over the factor $\Mfd$.
Then we replace this integral by a fiber integral over the fiber of the projection $\pi_1:\FM_\Mfd(2)\to \Mfd$.
We abusively write $\gamma(y):=\pi_2^*\gamma$ for the pullback of $\gamma$ to 
$\FM_\Mfd(2)$ under the second projection $\pi_2:\FM_\Mfd(2)\to \Mfd$.
For the second equality we apply Stokes' Theorem. For the third equality we use that $\alpha$ agrees with $\beta$ on $ST\Mfd$ by construction.

But our construction of a de Rham representative of $\Delta_\Mfd$ is the same as the construction of the image of $\beta$ under the connecting homomorphism in the long exact sequence namely
\[
  \cdots \to H^p(\Mfd\times \Mfd)
  \to H^p( \Mfd) \oplus H^p(\FM_\Mfd(2)) 
  \to H^p(\Mfd\times S^{n-1})
  \xrightarrow{c} H^{p+1}(\Mfd\times \Mfd)
  \to \cdots.
\]
But this agrees with the real version of the long exact sequence \eqref{equ:coIndlongexact} in the case of $\MOp=\Omega_\sharp(\FM_\Mfd)$.
This shows that (the image in real cohomology of) our algebraic diagonal element according to Definition \ref{def:diagonal element} agrees with the diagonal element of the manifold $\Mfd$ in this example, hence justifying the name. Also note that the derivation shows that the diagonal is independent of the choice of parallelization of $\Mfd$.
  
\end{ex}

\subsection{Example: Lambrechts-Stanley type comodules}
Let $A$ be a dg commutative algebra and let $\Delta\in A\otimes A$ be a diagonal element for $A$ of degree $n$ of vanishing Euler class $E(\Delta)=0$ in the sense of definition \ref{def:diagonal element alg}.
Then we consider a right $\La$ Hopf $\e_n^c$ comodule $\LS_{A,n}^\Delta$ as follows. As a graded commutative algebra
\[
  \LS_{A,n}^\Delta(r)
  =
  \FF_A(r) \otimes \e_n^c(r)/{\sim}
  \, =
  A^{\otimes r} \otimes \e_n^c(r)/{\sim}, 
\]
with the additional relations
\begin{equation}\label{equ:a omega rel}
  a_i \omega_{ij} = a_j\omega_{ij},
\end{equation}
where $\omega_{ij}$ is a cobracket generator of $\e_n^c$, see section \ref{sec:en def}, and for $a\in A$ and $i=1,\dots, r$ we set
\[
a_j := 1\otimes \cdots \otimes \underbrace{a}_{j\text{-th}}\otimes \cdots \otimes 1 \in A^{\otimes r}.
\]
We define the differential to be that inherited by $A$ on the factor $\FF_A$, and additionally
\begin{equation}\label{equ:d Delta def}
d \omega_{ij} = \Delta_{ij}:=\sum \Delta_i'\Delta_j''.
\end{equation}
Note that the defining condition \eqref{equ:a Delta commute} for the diagonal element is needed so that the relation \eqref{equ:a omega rel} is closed under the differential. Furthermore, we leave it to the reader to check that the differential of the Arnold relation \eqref{equ:arnold rel} is also zero, given \eqref{equ:a omega rel}.

Next we turn to cooperadic $\La$ $\e_n^c$ comodule structure on $\LS_{A,n}^\Delta$.
First note that the tensor product $\FF_A\otimes \e_n^c$ inherits a $\La$ Hopf $\e_n^c$ comodule structure from its factors.
More specifically, since $\FF_A$ is a cooperadic $\Com^c$-comodule and $\e_n^c$ is a cooperadic $\e_n^c$-comodule, $\FF_A\otimes \e_n^c$ is a comodule for $\Com^c\otimes \e_n^c = \e_n^c$.

We just need to check that the additional relation \eqref{equ:a omega rel} and the differential \eqref{equ:d Delta def} are compatible with the $\La$ structure and $\e_n^c$ coaction.
Concretely, the differential needs to commute with those operations, and \eqref{equ:a omega rel} defines a sequence of ideals (one in each arity) that is closed under coactions. 

Both statements are clear for the $\La$ operations, since applying a $\La$ operation to \eqref{equ:a omega rel} or \eqref{equ:d Delta def} just produces a similar relation in higher arity.  

The same holds for the coaction, unless the coaction is such that it affects both indices $i$ and $j$, i.e., both indices are in the subset of indices collapsed by some coaction.
In this case compatibility of \eqref{equ:d Delta def} with the coaction boils down to the statement
\[
\sum \Delta_i'\Delta_i'' = E(\Delta)_i = 0,
\]
that we ensure by requiring that the Euler class of $\Delta$ vanishes.
For \eqref{equ:a omega rel} we just need to check that $a_i-a_i=0$, which is obviously always satisfied.

We call $\LS_{A,n}^\Delta$ a $\La$ Hopf $\e_n^c$ comodule of Lambrechts-Stanley type, since Lambrechts and Stanley have used very similar objects in \cite{LS}, in the special case that $A$ is a Poincar\'e duality algebra and $\Delta$ is its diagonal. In their case they also did not consider the right $\e_n^c$ coaction, and hence did not need the condition on the vanishing of the Euler class.

Note also that we always have the isomorphism of graded symmetric sequences,
\[
  \LS_{A,n}^\Delta = \Com^c \circ A \circ \Lie_n^c.
\]

\begin{prop}
The $\La$ Hopf $\e_n^c$ comodule $\LS_{A,n}^\Delta$ is of configuration space type.
\end{prop}
\begin{proof}
By Proposition \ref{prop:Cobar Lie criterion} it suffices to check that the morphism
\[
\FF_A \to \Bar^c_{\Lie_n^c}\LS_{A,n}^\Delta
\]
is a quasi-isomorphism. 
To see this we take a descending filtration on both sides by the combined degree of the $A$-factors. This filtration is bounded, simply because all other degree carrying factors are non-negatively graded.
On the associated graded one does not see the differential on $\FF_A$, hence the associated graded map is identified with 
\[
  \FF_A^\flat \to \Bar^c_{\Lie_n^c}(\LS_{A,n}^\Delta)^\flat,
\]
with $(-)^\flat$ indicating that we consider the onject with trivial differential.
But since $(\LS_{A,n}^\Delta)^\flat$ is the free $\Lie_n^c$ comodule cogenerated by $\FF_A^\flat$, this morphism is a quasi-isomorphism, and hence the Proposition is shown. 
\end{proof}

\begin{rem}
Let $\MOp$ be a right $\e_n^c$-comodule of configuration space type with diagonal $\Delta$, and consider again the spectral sequence of section \ref{sec:CohenTaylorSS}. Then the $E^1$-page \eqref{equ:E1LS} of that spectral sequence agrees, as a graded symmetric sequence, with $\LS_{H(\MOp(1)), n}^\Delta$. 
In fact, one may check, using our main Theorem \ref{thm:main}, that the differential on the $E^1$-page is that of $\LS_{H(\MOp(1)), n}^\Delta$. We will not use this fact elsewhere in the paper and hence leave the proof to the reader or future work.
\end{rem}

\section{$\SL_\infty$ algebras, Maurer-Cartan spaces and obstruction theory}

\subsection{Filtered complete curved $\SL_\infty$-algebras}
Following \cite{DolRog} we will work with a degree shifted version of (curved) $L_\infty$-algebras, defined such that Maurer-Cartan elements have cohomological degree 0 instead of $+1$. To emphasize this degree shift, we call these algebras (curved) $\SL_\infty$-algebras.

A curved $\SL_\infty$-algebra structure on a graded vector space $\fg$ is a collection of maps
\[
\{ \mu_k : S^k(V) \to V[-1] \}_{k\geq 0}
\]
that satisfy a system of quadratic equation of the form
\[
\sum_{k+l=n+1} \sum_{\sigma\in Sh(l,k+1)} \pm \mu_k(\mu_l(x_{\sigma(1)},\dots,x_{\sigma(l)}), x_{\sigma(l+1)},\dots, x_{\sigma(n+1)})=0,   
\]
where the inner sum is over all $(k,l)$-shuffle permutations, and the sign is the Koszul sign corresponding to the permutation $\sigma$ of the elements $x_1,\dots,x_{n+1}$.
The curved $\SL_\infty$-structure is a(n honest) $\SL_\infty$-structure if the curvature $\mu_0$ vanishes.

A filtered complete curved $\SL_\infty$-algebra is a $\SL_\infty$-algebra $\fg$ together with a descending complete filtration 
\[
\mF^1\fg = \fg \supset \mF^2\fg\supset\mF^3\fg \supset \cdots  
\]
compatible with the $\SL_\infty$-operations $\{\mu_k\}_{k\geq 0}$ in the sense that 
\begin{equation}\label{equ:Linfty filtration compat}
  \mu_k(\mF^{p_1}\fg,\dots,\mF^{p_k}\fg)
  \subset \mF^{p_1+\dots+p_k}\fg.
\end{equation}
A Maurer-Cartan element in the filtered complete curved $\SL_\infty$-algebra $\fg$ is a degree 0 element $m$ such that 
\[
\curv(m) := \sum_{k\geq 0} \frac 1 {k!} \mu_k(m,\dots,m)=0.  
\]
We denote the set of Maurer-Cartan elements by 
\[
\MC(\fg) = \{ x \in \fg_0 \mid \curv(x)=0 \}.
\]

Suppose that $m\in \fg$ is any element of degree 0. Then we can build a filtered complete curved $\SL_\infty$-algebra $\fg^m$ with underlying graded vector space $\fg$ but $\SL_\infty$-operations
\[
  \mu_k^m(x_1,\dots,x_k)
  :=
  \sum_{l\geq 0} \frac 1 {l!} \mu_{k+l}(m,\dots,m,x_1,\dots,x_k).  
\]
We say that $\fg^m$ is obtained from $\fg$ by twisting with $m$.
In particular note that $\mu_0^m=\curv(m)$, so that if $m$ is a Maurer-Cartan element the curved $\SL_\infty$-algebra $\fg^m$ is in fact a(n honest, i.e., non-curved) $\SL_\infty$-algebra.
In this case $\mu_1^m$ is a differential, i.e., $\mu_1^m(\mu_1^m(x))=0$ for all $x\in \fg$.

\subsection{Schematic formalism for $\SL_\infty$-algebras}
\label{sec:Linfty_formalism}

There are several equivalent definitions of $L_\infty$- or $\SL_\infty$-algebras.
Here we will use the formulation of \cite[section 4]{FWAut} that is well adapted to our purposes.
Let $\fg$ be a filtered complete $\SL_\infty$-algebra as above. 
Take any graded commutative algebra $\cR$ and consider the completed tensor product 
$$
\fg \hotimes \cR =\lim_\leftarrow \left(\fg\otimes \cR/\mF^p\fg\otimes \cR\right).
$$
This is then a (curved) $\SL_\infty$-algebra over $\cR$. Furthermore, by polarization, it is sufficient to know for each $\cR$ and $x\in \fg \hotimes \cR$ of degree 0 the value of the function
\[
\mU^{\fg}_{\cR}(x) := \curv(x) = \sum_{k\geq 0} \frac 1 {k!} \mu_k(x,\dots,x).
\]
We often treat the extension of the ground ring to $\cR$ as implicit and write $\mU^{\fg}=\mU^{\fg}_{\cR}$, or even $\mU=\mU^{\fg}_{\cR}$ if $\fg$ is clear from the context. 
This then encodes all $\mu_k$ for arguments of arbitrary degrees. In the non-curved case one requires that the series starts at $k=1$.
Furthermore, the $\SL_\infty$-relations can be neatly encoded as the statement that 
\begin{equation}
  \label{equ:Linfty structure series}
 D\mU(x)[\mU(x)]=0  
\end{equation}
for all $x$ and $\cR$ as above, where $D$ denotes the derivative with respect to $x$.
Equivalently, avoiding the derivative this can be formulated as saying that 
\begin{equation}
  \label{equ:Linfty structure series 2}
  \mU(x+\epsilon\mU(x)) =\mU(x)
\end{equation}
where $\epsilon$ is a formal variable of degree $-1$ that we adjoin to our ring $\cR$.

A curved $\SL_\infty$-morphism between curved $\SL_\infty$-algebras $F:\fg\to \fh$ is a collection of linear maps 
\[
 \{ F_k : S^k(V) \to V \}_{k\geq 0}
\]
that are compatible with the filtrations in the sense that 
\begin{equation}\label{equ:Linfty morphism filtration compat}
  F_k(\mF^{p_1}\fg,\dots,\mF^{p_k}\fg)
  \subset \mF^{p_1+\dots+p_k}\fh.
\end{equation}
and that satisfy a polynomial system of equations.
Again we may encode the morphism $F$ as a series 
\[
F(x) = \sum_k \frac 1 {k!} F_k(x,\dots,x)
\]
defined on degree zero elements $x\in \fg \hotimes \cR$.
Then the system of equations to be satisfied by $F$ is just
\[
  \mU^{\fh}(F(x))=
DF(x)[\mU^{\fg}(x)]
\]
for all $x$ and $\cR$ as above.
Again this may be cast in the derivative-free form
\begin{equation}\label{equ:Linfty morphism series}
  F(x)+\epsilon \mU^{\fh}(F(x))
  =
  F(x+\epsilon \mU^{\fg}(x))
\end{equation}
with $\epsilon$ a degree $-1$ variable that squares to zero.

A (flat) $\SL_\infty$-morphism is a curved $\SL_\infty$-morphism with $F_0=0$.

\subsection{The Maurer-Cartan space of a filtered complete $\SL_\infty$-algebra}\label{sec:Linfty_MC}

One defines the Maurer-Cartan space of a filtered complete curved $\SL_\infty$-algebra $\fg$ to be the simplicial set 
\[
  \MC_\bullet(\fg) = \MC(\fg\hotimes \Omega(\Delta^\bullet)),
\]
where $\Omega(\Delta^k)$ is the dg commutative algebra of polynomial differential forms on the $k$-simplex.
We shall recall a few properties of this space.
First, $\pi_0\MC_\bullet(\fg)$ is identified with the gauge-equivalence classes of Maurer-Cartan elements.
Let $\tau\in \MC(\fg)$ be a Maurer-Cartan element.
Then it is shown in \cite[Theorem 5.5]{Berglund} that the higher homotopy groups may be computed as the homology groups of the twisted $\SL_\infty$-algebra
\[
\pi_k(\MC_\bullet(\fg)) \cong H_k(\fg^\tau) := H^{-k}(\fg^\tau)
\]
for $k\geq 1$.
For $k\geq 2$ the group operation here is just addition, while for $k=1$ it is defined using the Campbell-Baker-Hausdorff formula, cf. loc. cit.
Clearly, if $\fg$ is of finite type, or at least each $\fg^\tau$ is of finite homological type, than so are all homotopy groups of any connected component of $\MC_\bullet(\fg)$.
In this case $\pi_1$ is automatically a nilpotent group and it acts in a nilpotent manner on the higher homotopy groups due to the compatibility relations \eqref{equ:Linfty filtration compat}. Hence (each component of) $\MC_\bullet(\fg)$ is a nilpotent space.

Furthermore, the cohomology of $\MC_\bullet(\fg)$ is computed by a version of the Chevalley-Eilenberg complex of $\fg$.
Typically, for this one needs to require that $\fg$ is nilpotent of finite type, see \cite[Corollary 1.3]{Berglund}.
However, in most of our examples one is lucky in that $\fg$ is in fact the dual of an $\SL_\infty$-coalgebra $\fc$, $\fg=\fc^*$, and furthermore, it has homology of finite type.
In this case one can show (see  \cite{WillwacherRecollections}) that the cohomology of the connected component $\MC_\bullet(\fg)_\tau$ is computed by the the (co-)Chevalley-Eilenberg complex
\[
C_{CE}(\tru(\fc^m)) 
\]
of the degree (co)truncation of the twist $\fc^m$ of $\fc$. More concretely, $\tru(\fc^m)$ is the quotient of $\fc^m$ by all elements of negative cohomological degree and the exact elements in degree zero.

\subsection{$\SL_\infty$ obstruction theory}
We recall here basic obstruction theory in the context of curved $\SL_\infty$-algebras. We focus on the simplest possible case.
For a comprehensive overview and more general versions of these techniques we refer to \cite{Silvan}.

We are interested in finding conditions for the existence of a Maurer-Cartan element in a curved filtered complete $\SL_\infty$-algebra.
To do this we assume that $\fg$ is equipped with a second complete filtration 
\[
\mG^0\fg = \fg \supset \mG^1\fg \supset \cdots 
\]
also compatible with the $\SL_\infty$-operations.
Explicitly, we allow this filtration, used for the obstruction theory, to be different from the filtration $\mF^p$ above needed to handle convergence issues.

Suppose that the curvature of $\fg$ satisfies $\mu_0\in \mG^1\fg$.
Then, although $\mu_1$ does not define a differential on $\fg$, it still defines a differential on the associated graded $\hgr_{\mG}\fg$ since 
\[
  \mu_1(\mu_1(x))=\pm \mu_2(\mu_0, x)\in \mG^{q+1} \fg
\]
for all $x\in \mG^q\fg$. 
One furthermore has the following result.

\begin{lemma}\label{lem:preMC}
Let $\fg$ be a filtered complete curved $\SL_\infty$-algebra as above with $\mu_0\in \mG^p\fg$ for some $p\geq 1$ such that 
\[
H^1(\hgr_{\mG} \fg)=0 .
\]
Then there is a degree zero element $m\in\mG^{p}\fg$ such that $\curv(m)\in \mG^{p+1}\fg$.
\end{lemma}
\begin{proof} 
The curvature $\mu_0$ of the curved $\SL_\infty$-algebra is always a closed degree 1 element by the $\SL_\infty$-relations, 
\[
\mu_1(\mu_0)=0.  
\]
Hence by the vanishing of the first cohomology of the associated graded there is some $m\in \mG^{p}\fg$ of degree 0 such that $\mu_0+\mu_1(m)\in \mG^{p+1}\fg$.
Hence 
\[
\curv(m)=\mu_0+\mu_1(m)+\frac 1 2 \mu_2(m,m)+\cdots
\in \mG^{p+1}\fg,
\]
using that $p\geq 1$.
\end{proof}
\begin{cor}\label{cor:MCexistence1}
  Let $\fg$ be a filtered complete curved $\SL_\infty$-algebra as above such that $\mu_0\in \mG^1\fg$ and 
  \[
H^1(\hgr_{\mG} \fg)=0 .
\]
  Then there is a Maurer-Cartan element in $\mG^{1}\fg$.
\end{cor}
\begin{proof}
  We show the result by constructing a sequence of elements $m_0,m_1,m_2,\dots \in \mG^{1}\fg$
  such that 
  \begin{align}
    \curv(m_k)&\in \mG^{k+1}\fg 
    \\ 
    \label{equ:MCkk1}
    \text{and} \quad\quad  x_{k}:=m_{k}-m_{k-1}&\in \mG^{k}\fg.
  \end{align}
  To start with we set $m_0=0$.
  Suppose by induction that $m_{k-1}$ as above has been constructed.
  Then the twisted curved $L_\infty$-algebra $\fg^{m_{k-1}}$ satisfies the assumptions of Lemma \ref{lem:preMC} for $p=k$. (Note that the twist does not contribute to the differential on the associated graded.)
By the Lemma we hence obtain an element $x_{k}\in \mG^{k}\fg$ such that 
\[
\curv_{\fg^{m_{k-1}}}(x_{k})=\curv_{\fg}(m_{k-1}+x_{k})\in \mG^{k+1}\fg.
\]
We hence set $m_{k}:=m_{k-1}+x_k$.
Finally, the desired Maurer-Cartan element in $\mG^1\fg$ is
\[
m = \lim_{k\to \infty} m_k= \sum_{k=1}^\infty x_k,  
\]
using that the series is convergent by \eqref{equ:MCkk1}.



\end{proof}

\section{Deformation complexes for Hopf cooperadic comodules}
\label{sec:def cx}
It is a recurring theme of algebraic deformation theory that the set of maps between two (sufficiently nice) algebraic objects can be identified with the Maurer-Cartan elements of a dg Lie or $L_\infty$-algebra, which we will call the deformation dg Lie algebra, or deformation complex.
It has been shown by Fresse and the author in \cite{FWAut} that this theme can be realised for Hopf cooperads, or $\La$ Hopf cooperads. Here we describe a variant of their construction for $\La$ Hopf comodules.

\subsection{Construction of the deformation complex}
\label{subsec:def cx}
We fix a reduced dg $\La$ Hopf cooperad $\COp$.
Suppose that we are given two $\La$ Hopf comodules $\MOp$ and $\NOp$ that satisfy the following properties.
\begin{enumerate}
\item $\NOp$ is quasi-cofree as a dg $\COp$-comodule, cogenerated by the symmetric sequence $W$, so that $\NOp = (\FreeMod_{\COp}^c W, d)$.
\item $\MOp$ is quasi-free as a $\La$ Hopf collection, generated by the symmetric sequence $V$. Concretely, this means the following. Let $\hat V$ be the free $\La$ Hopf collection generated by $V$, i.e., 
\[
\hat V(r) \cong \oplus_{f: [k]\hookrightarrow [r]} V(k),  
\]
where the direct sum is over all order preserving injections.
Then $\MOp(r)$ is quasi-free as a dg commutative algebra, generated by $\hat V(r)$, so that $\MOp(r)=S(\hat V(r))$.
We will write for short 
\[
S_\La V= S(\hat V).  
\]
\item In addition we require that the generating symmeric sequence $V$ is equipped with an exhaustive ascending fitration of finite dimensional symmetric sequences
\[
0=\mF^0V(r)\subset \mF^1V(r)\subset \cdots
\]
that is compatible with the $\La$ Hopf $\COp$ comodule structure in the following sense.
First we prolong our filtration on the generating symmetric sequence $V$ to $\hat V$ and then to $\MOp$ by defining  
\[
  \mF^p\MOp(r) = 
  \mF^p(S_\La V)(r)
  =
  \sum_{p_1+\dots+p_m\leq p \atop m\geq 0}
  \mathrm{im}\left(\mF^{p_1}\hat V(r)\otimes \cdots\otimes \mF^{p_m}\hat V(r) \right).
\]

\begin{itemize}
\item Then we require first that the differential on $\MOp$ is compatible with this filtration in the sense that $d\mF^p\MOp(r)\subset \mF^p\MOp(r)$.
\item Second we require that the comodule coaction is compatible with the filtration in the sense that for $x\in \mF^p\MOp(r)$ the images of the coaction are contained in $\mF^p\MOp(r-s+1)\otimes \COp(s)$.
\end{itemize}
Note that the filtration is automatically compatible with the commutative algebra structure and the $\La$ structure by construction.
\end{enumerate}

We call a pair $(\MOp, \NOp)$ of $\La$ Hopf $\COp$ comodules satisfying the above technical conditions a \emph{good pair of $\La$ Hopf $\COp$ comodules}.
Generally, for a differential graded object $X$ we denote by $X^\flat$ the underlying graded object, i.e., $X^\flat$ is obtained from $X$ by setting the differential to zero.

\begin{lemma}\label{lem:freecofree}
Let $(\MOp, \NOp)$ be a good pair of $\La$ Hopf $\COp$ comodules and let $\cR$ be any graded commutative algebra. Then the inclusion of generators and projection to cogenerators induces a bijection
\[
\Mor_{\gLaHModc_{\cR\otimes \COp^\flat}/\cR}(\cR\otimes \MOp^\flat,\cR\otimes \NOp^\flat) 
\to
\Mor_{\bbS/\cR}(\cR\otimes V, \cR\otimes W).
\]
Here $\Mor_{\gLaHModc_{\cR\otimes \COp^\flat}/\cR}(\cdots)$ denotes the set of morphisms of right $\La$-Hopf comodules over $A\otimes \COp^\flat$, with the ground ring $\cR$, and $\Mor_{\bbS/\cR}(\dots)$ denotes the set of morphisms of symmetric sequences in $\cR$-modules.
\end{lemma}
\begin{proof}
  Denote by $\iota:V\to \MOp$ the inclusion of generators and by $\pi: \NOp\to W$ the projection to cogenerators.
  The map of the Lemma hence associates to the module morphism $F:\cR\otimes \MOp^\flat\to \cR\otimes \NOp^\flat$ the morphism of symmetric sequences $f=\pi\circ F\circ \iota:\cR\otimes V^\flat\to \cR\otimes W^\flat$.

  Suppose that some map $f:\cR\otimes V\to \cR\otimes W$ of symmetric sequences is given. Then we will construct a corresponding map of $\La$ Hopf $\COp^\flat$ comodules $F$ such that $\pi\circ F\circ \iota=f$, showing surjectivity. 
  To simplify notation, let us furthermore drop $\cR$ from the notation, assuming tacitly that we always work of the ground ring $\cR$, and all objects have been replaced by their $\cR$-extension.

  We then construct $F$ by induction, using the filtration on $\MOp$. 
  To this end, also note that by our conditions $\mF^p \MOp^\flat$ is a graded $\COp^\flat$-comodule and we can write
  \[
  \MOp^\flat = \colim \mF^p \MOp^\flat.
  \]
  First, $\mF^0 \MOp^\flat=\Com^c$ is spanned by the units in each arity. Hence necessarily $F$ is defined on $\mF^0 \MOp^\flat$ to be the composition 
  \[
    \mF^0 \MOp^\flat=\Com^c\to \NOp^\flat.
  \] 
  This map is a map of graded $\La$ Hopf $\COp^\flat$-comodules.
  Next suppose that we have defined a map of $\La$ Hopf $\COp^\flat$-comodules
  \[
    F: \mF^{p-1} \MOp^\flat\to \NOp^\flat,
  \]
  such that $\pi\circ F\circ \iota=f\mid_{\mF^{p-1} V}$ and we wish to extend this map to $\mF^{p} \MOp^\flat$.
  We do this in turn by a nested induction on arity. We denote the arity truncation by $\mF^{p} \MOp^\flat_{\leq q}$. We suppose that we already have defined a map of truncated $\La$ Hopf $\COp$-comodules 
  \[
    F : \mF^{p} \MOp^\flat_{\leq r-1}\to \NOp^\flat,
  \]
  such that $\pi\circ F\circ \iota$ agrees with the appropriate restriction of $f$, and we desire to extend $F$ to arity $r$.
  (The base case of the induction $r=0$ is empty.)
  We do this by first defining a suitable map of symmetric sequences
  \[
    g : \mF^{p} \MOp^\flat_{\leq r} \to W,
  \]
  to play the role of $\pi\circ F$.
  To this end let $x=x_1\cdots x_k\in \mF^{p} \MOp^\flat(r)$ with $x_1,\dots,x_k\in\hat V(r)$. If $k\geq 2$, then necessarily each $x_j$ must be an element of $\mF^{p-1}\hat V$, and we use the inductively defined $F$ and set $g(x)=\pi( F(x_1)\cdots F(x_k))$. If $k=1$ we have that $x=\lambda(x')$, where $x'\in V$, and $\lambda$ is some $\La$-operation.
  If $\lambda$ is the identity operation then we set $g(x)=f(x)$.
  Otherwise $x'\in V(r')$ for some $r'<r$ and we set $g(x)=\pi(\lambda(F(x'))$, using that $F$ is defined on lower arities by induction.
  Finally, we define $F:\mF^{p} \MOp^\flat_{\leq r} \to \NOp^\flat$ to be the unique map of graded $\COp^\flat$-comodules such that $\pi\circ F=g$. In other words, we send each $x\in \mF^{p} \MOp^\flat_{\leq r}$ to the unique $y\in \MOp^\flat(r)$ such that $\pi(y)=g(x)$ and for each coaction $\Delta$ we have 
  \[
  \Delta(F(x)) = (F\otimes \id_{\COp} )(\Delta ) .
  \]
  In particular, note that if $x=x_1\cdots x_k$ is a nontrivial product, then the element $y=F(x_1)\cdots F(x_k)$ satisfies these equations by induction hypothesis and compatibility of product and coaction. Hence our $F$ is compatible with the product. By a similar argument it is also compatible with the $r$-truncated $\La$ operations. Furthermore $F$ obviously respects the $\COp$-coaction by definition. Also, it is clear that $\pi\circ F\circ \iota$ agrees with $f$.
Finally it can easily be seen that the $F$ thus defined indeed restricts in lower arity and fitration degree to the $F$ of previous induction stages.

Thus we have shown in particular the surjectivity of the map of the Lemma. However, tracing the above construction, the definition at every step was forced by the requirements that $F$ has to be a morphism $\La$ Hopf $\COp^\flat$-modules and $\pi\circ F\circ \iota=f$. Hence it follows that the map of the Lemma is also injective, and hence a bijection as claimed. 
\end{proof}

\begin{rem}
  Note that the exhaustive ascending filtration on $V$ induces a descending complete filtration on the internal homomorphism space $\iHom_{\bbS}(V,W)$ 
  \[
    \iHom_{\bbS}(V,W)= \mF^1\iHom_{\bbS}(V,W) \supset \mF^2\iHom_{\bbS}(V,W) \supset \cdots 
  \]
  such that 
  \[
  \mF^p\iHom_{\bbS}(V,W) 
  \]
  consists of those morphisms that vanish on $\mF^{p-1}V$.
In the setting of the Lemma we then also have 
\begin{equation}\label{equ:pre def tensor out}
  \Mor_{\bbS/\cR}(\cR\otimes V, \cR\otimes W)
  \cong \left( \cR\hotimes \iHom_{\bbS}(V,W) \right)_0,  
\end{equation}
where the right-hand side is the degree 0 subspace of the tensor product, completed with respect to the filtration above. Note that for this statement we use the finite dimensionality of the $\mF^pV$.
\end{rem}


So far we have ignored the differentials on our comodules. It turns out that the differentials give rise to a complete curved $\SL_\infty$-structure.

\begin{prop}\label{prop:Def Linfty}
  Let $\COp$ be a reduced $\La$ Hopf cooperad and let $(\MOp, \NOp)$ be a good pair of $\La$ Hopf $\COp$ comodules with generators and cogenerators $V$ and $W$ as above.
Then there is a filtered complete curved $\SL_\infty$-structure on the graded vector space
\[
  \Def(\MOp, \NOp):= \iHom_{\bbS}(V,W)
\]
such that for any dgca $\cR$ the Maurer-Cartan elements in $\cR\hotimes \Def(\MOp, \NOp)$ are in one-to-one correpsondence with the right $\La$ Hopf $\COp$-comodule morphisms over $\cR$
\[
  \cR\otimes \MOp\to  \cR\otimes \NOp.
\]
This one-to-one correpsondence is more precisely given by the map of Lemma \ref{lem:freecofree}, i.e., the Maurer-Cartan element $f$ corresponds to a module morphism $F$ whose restriction-projection to (co)generators is $f$.
\end{prop}
\begin{proof}
  The proof is formal given Lemma \ref{lem:freecofree} and identical to the proof of the analogous statement for Hopf cooperads \cite[Theorem 21]{FWAut}. 

  In particular, using the schematic formalism for $\SL_\infty$-algebras of section \ref{sec:Linfty_formalism} we define the generating function encoding the curved $\SL_\infty$-operations to be
  \begin{equation}\label{equ:Def gen func}
\mU(f) = \pi\circ( d_{\NOp}\circ \Phi(f) - \Phi(f) \circ d_{\MOp}  ) \circ\iota ,
  \end{equation}
with $\pi$ and $\iota$ the projection to cogenerators and includion of generators respectively, and $\Phi$ the inverse isomorphism of Lemma \ref{lem:freecofree}.
For the verification that this enerating function defines a $\SL_\infty$-algebra structure we refer to the proof of Proposition \ref{prop:Lifty semidirect} below, where a more general statement is shown.

It is however clear that the Maurer-Cartan equation $\mU(f)=0$ is equivalent to the statement that $\Phi(f)$ intertwines the differentials, and is hence a morphism of dg $\La$ Hopf $\COp$ comodules.

\end{proof}


\subsection{Dg Lie algebra action}\label{sec:dg Lie action}
Next let $(\MOp, \NOp)$ be a good pair of $\La$ Hopf $\COp$-comodules as above and let $\fg$ be a filtered complete dg Lie algebra.
Suppose that $\fg$ acts on $\MOp$ by biderivations.
By this we mean that for homogeneous $x\in \fg$ and $m,m'\in \MOp$ we have the compatibility equations
\begin{equation}\label{equ:bider relations}
\begin{aligned}
x\cdot (mm')&= (x\cdot m)m' + (-1)^{|x||m|}m(x\cdot m') \\
x\cdot \mu_j^c(m) &= \mu_j^c (x\cdot m) \\
\Delta (x\cdot m) &= ((x\cdot) \otimes \id_{\COp} )\Delta(m),
\end{aligned}
\end{equation}
where $x\cdot$ denotes the action of $x$, $\mu_j^c$ are the defining $\La$-operations as in \eqref{equ:mujc mod} and $\Delta$ denotes the $\COp$-coaction. 
We shall also assume that the action of $\fg$ on $\MOp$ is compatible with the filtrations on $\fg$ and $\MOp$ in the sense that 
\begin{equation}\label{equ:action filtr compat}
\mF^p\fg \cdot \mF^q\MOp \subset \mF^{q-p}\MOp.  
\end{equation}

\begin{prop}\label{prop:Lifty semidirect}
In the setting above there is a filtered complete curved $\SL_\infty$-algebra structure on 
\[
  \fg \ltimes \Def(\MOp, \NOp):= 
  \fg[1] \oplus \Def(\MOp, \NOp)
\] 
satisfying the following properties:
\begin{itemize}
  \item For any graded commutative algebra $\cR$ the Maurer-Cartan elements $m=Z+f\in \MC(\cR\hotimes (\fg \ltimes \Def(\MOp, \NOp)))$ are in 1-1 correspondence to pairs consisting of a Maurer-Cartan element in $Z\in \cR\hotimes \fg$ and a morphism of $\La$ Hopf right $\COp$-comodules over $\cR$
\[
\Phi(f) : (\MOp\otimes \cR)^Z \to \NOp\otimes \cR.
\]  
\item 
The $\SL_\infty$-operations $(\nu_k)$ of $\fg \ltimes \Def(\MOp, \NOp)$ extend those on both factors in the following sense:
If all inputs to the $\nu_k$ are in $\fg$ and the output is projected to $\fg$ then the operations agrees with the dg Lie structure on $\fg$.
Furthermore, if all inputs are in $\Def(\MOp, \NOp)$ then so is the output and the structure agrees with the curved $\SL_\infty$-structure on $\Def(\MOp, \NOp)$ of Proposition \ref{prop:Def Linfty}.
Finally, the $\SL_\infty$-operations with mixed inputs (some in $\fg$, some in $\Def(\MOp, \NOp)$) take values in $\Def(\MOp, \NOp)$. \footnote{This is the reason for (ab)using the semidirect product notation.}
\end{itemize}
\end{prop}
\begin{proof}
We will again use the formalism for $\SL_\infty$-algebras outlined in section \ref{sec:Linfty_formalism} and define the generating function to be 
\[
\mU(Z+f) = \mU^{\fg}(Z) + \pi\circ( d_{\NOp}\circ \Phi(f) - \Phi(f) \circ (d_{\MOp} + Z\cdot)  ) \circ\iota ,
\]
where we have abbreviated
\[
  \mU^{\fg}(Z) = dZ+\frac 1 2 [Z,Z],
\]
with the bracket and differential that of $\fg$.
Let us verify the $\SL_\infty$-relations in the form \eqref{equ:Linfty structure series}.
We will neeed use that for a $\La$ Hopf comodule map $F$ we have $\Phi(\pi\circ F\circ \iota)=F$.
In particular, we have that for a derivation $D$ of the map $\Phi(f)$
\[
  \Phi(f+\epsilon \pi \circ D\circ \iota)
  =
  \Phi(\pi\circ (\Phi(f) + \epsilon D )\circ \iota)
  =
  \Phi(f) + \epsilon D.
\]
We then compute
\begin{align*}
  \mU(Z+f + \epsilon \mU(Z+f))
&=
\mU^{\fg}(Z+\epsilon \mU^{\fg}(Z))
+
\pi\circ( d X -
 X (d + Z\cdot+\epsilon \mU^{\fg}(Z)\cdot ) 
) \circ \iota
\end{align*}
with 
\[
X :=
\Phi(f + \epsilon \pi \underbrace{(d\Phi(f)-\Phi(f)(d+Z\cdot))}_{\text{derivation of $\Phi(f)$}} \iota )
= 
\Phi(f)+\epsilon (d\Phi(f)-\Phi(f)(d+Z\cdot)).
\]
Using that $d^2=0$ and abbreviating $F:=\Phi(f)$ we compute
\begin{align*}
  d X 
  - X (d + Z\cdot+\epsilon \mU_{\fg}(Z)\cdot )
  &=
  d F 
  - F (d + Z\cdot )
  +\epsilon F
  ((dZ)\cdot + (Z\cdot)(Z\cdot)- \mU^{\fg}(Z)\cdot)
  \\&=
  d F 
  - F (d + Z\cdot ).
\end{align*}
Inserting, we obtain 
\begin{align*}
  \mU(Z+f + \epsilon \mU(Z+f))
  &=
  \mU^{\fg}(Z)
  + 
  \pi (d F 
  - F (d + Z\cdot ))\iota
= \mU(Z+f)
\end{align*}
as desired.

Then the stated properties of the $\SL_\infty$-structure are clear from the construction and the discussion above.
\end{proof}

\subsection{Dg Lie algebra actions and mapping spaces}
Let $\fg$ be a filtered complete dg Lie algebra. 
The exponential group of $\fg$ is the simplicial group 
\[
  \Exp_\bullet(\fg) = Z_0(\fg\hotimes \Omega(\Delta^\bullet))
\]
consisting of the closed degree zero elements $Z_0(-)$, with the composition given by the Baker-Campbell-Hausdorff product $*$.
That is, for $x,y\in Z_0(\fg\hotimes \Omega(\Delta^\bullet))$ we have 
\[
x*y = \log(e^xe^y) = x+y+\frac12 [x,y]+\cdots.  
\]

Let again $\MOp$ be a right $\La$ Hopf $\COp$-comodule and suppose that $\fg$ acts on $\MOp$ by biderivations.
Then in particular we obtain a map of simplicial monoids
\[
\Exp_\bullet(\fg) \to \Map(\MOp, \MOp),
\]
hence an action of the exponential group of $\fg$ on $\MOp$. 
Given a map of right $\La$ Hopf $\COp$-comodules $f:\MOp\to \NOp$ we also obtain a map of simplicial sets 
\begin{equation}\label{equ:Exp to M to N}
  \Exp_\bullet(\fg) \to \Map(\MOp, \NOp)
\end{equation}
by acting on (i.e., composing with) $f$.

Now suppose we are in the setting of subsection \ref{subsec:def cx}, i.e., that $\MOp, \NOp$ form a good pair. Then we can write the mapping space as the Maurer-Cartan space of the deformation $\SL_\infty$-algebra
\[
\Map(\MOp, \NOp) = \MC_\bullet(\Def(\MOp, \NOp)).
\]
Abusing notation, we denote the Maurer-Cartan element corresponding to the map $f$ by $f\in \MC(\Def(\MOp, \NOp))$ as well.
Furthermore, we denote by 
\[
  \Def(\MOp\xrightarrow{f} \NOp) := \Def(\MOp, \NOp)^f  
\]
the twisted $\SL_\infty$-algebra. Note that this is an honest $\SL_\infty$-algebra, not a curved one. The twisting does not alter the Maurer-Cartan space, but $\MC_\bullet(
  \Def(\MOp \xrightarrow f \NOp) )$ has a distinguished basepoint given by the zero MC element, corresponding to $f$.
Below we will then need the following result.
\begin{prop}\label{prop:linfty from g on M action}
Let $\MOp, \NOp$ be a good pair of $\La$ Hopf $\COp$ comodules, $f:\MOp\to \NOp$ a map of $\La$ Hopf $\COp$ comodules and $\fg$ a dg Lie algebra with a good action on $\MOp$ in the sense of section \ref{sec:dg Lie action}.
Let $\fg_{ab}$ be the abelianization of $\fg$, considered as a $\SL_\infty$ algebra.
Then there is a $\SL_\infty$-morphism 
\[
\Phi : \fg_{ab} \to \Def(\MOp \xrightarrow f \NOp)
\]
compatible with the filtrations such that the 
map of simplicial sets \eqref{equ:Exp to M to N} 
agrees with the morphism on Maurer-Cartan spaces 
\[
\Exp_\bullet(\fg)=\MC_\bullet(\fg_{ab})
\to 
\MC_\bullet(
\Def(\MOp \xrightarrow f \NOp) )
=
\Map(\MOp, \NOp).
\]
\end{prop} 

This is just the version for comodules of its cooperadic counterpart \cite[Proposition 24]{FWAut}.
However, given our Proposition \ref{prop:Lifty semidirect} above we can also formally obtain a self-contained proof by the following Lemma.
\begin{lemma}\label{lem:linfty from g on h action}
Let $\fg$ be a dg Lie algebra that acts on the (possibly curved) $\SL_\infty$-algebra $\fh$ by (possibly curved) $\SL_\infty$-derivations.
We suppose that both $\fg$ and $\fh$ are equipped with compatible descending complete filtrations 
\begin{align*}
  \fg = \mF^1\fg \supset \mF^2\fg \supset \cdots \\
  \fh = \mF^1\fh \supset \mF^2\fh \supset \cdots 
\end{align*}
and that the action of $\fg$ on $\fh$ is compatible with the filtrations.
Let $\alpha\in \MC(\fh)$ be a Maurer-Cartan element.
Then there is a (possibly curved) $\SL_\infty$-morphism
\[
\fg_{ab} \to \fh
\]
compatible with the filtrations
from the abelianization of $\fg$ such that the induced map on Maurer-Cartan spaces
\[
\Exp_\bullet(\fg) = \MC_\bullet(\fg_{ab})\to \MC_\bullet(\fh)
\]
agrees with the map obtained by acting with the exponential group on the given MC element $\alpha$.
\end{lemma}
\begin{proof}
Let us first spell out the action of $\fg$ on $\fh$ explicitly.
The action is given by operations $\nu_k(x;y_1,\dots,y_k)$ for $k\geq 0$, such that for given $x\in \fg$, the (curved, pre-) $\SL_\infty$-derivation associated to $x$ is 
\[
x\cdot = \{\nu_k(x; -) \}_{k\geq 0}.
\]
These derivations from a Lie algebra and the $\SL_\infty$-structure on $\fh$ is itself encoded by a $\SL_\infty$-derivation $D$.
Then having $\fg$ act on $\fh$ means exlicitly that one has 
\begin{align}\label{equ:g action on h}
  [D,x\cdot]&= (dx)\cdot &
  [x\cdot, y\cdot ]&=[x,y]\cdot .  
\end{align}

Now we again use the schematic formalism of section \ref{sec:Linfty_formalism} to encode the $\SL_\infty$-morphism.
Let $\cR$ be a graded commutative algebra and $x\in \fg \hotimes \cR$ a degree zero element.
Then we can exponentiate the $\SL_\infty$-derivation $x\cdot$ to a (pre-)$\SL_\infty$-morphism 
\[
\exp(x\cdot) = \sum_{k\geq 0} \frac 1 {k!} (x\cdot)^k.
\]
By the first equation of \eqref{equ:g action on h} this satisfies 
\[ 
(1+\epsilon D)\circ \exp(x\cdot)\circ(1-\epsilon D) = \exp(x\cdot +\epsilon dx\cdot),
\]
for $\epsilon$ a variable of degree $-1$.
We then define our $\SL_\infty$-morphism via the generating function 
\[
F(x) = \exp(x\cdot)(\alpha)
\]
by putting $\alpha$ in all slots of the 
$\SL_\infty$-morphism.
In this language the generating function for the $\SL_\infty$-structure on $\fh$ is just 
\[
\mU^{\fh}(y) = D(y),
\]
and that of $\fg_{ab}$ is 
\[
\mU^{\fg_{ab}}(x) = dx,
\]

We check the $\SL_\infty$-relations in the form \eqref{equ:Linfty morphism series}:
\begin{align*}
  F(x +\epsilon dx)
  &=
  \exp(x\cdot+\epsilon dx\cdot)(\alpha)
  \\
  &=(1+\epsilon D)\circ \exp(x\cdot)\circ(1-\epsilon D)(\alpha) 
  \\
  &=(1+\epsilon D)( \exp(x\cdot)(\alpha) )
  \\
  &=
  F(x) + \epsilon \mU_{\fh}(F(x)),
\end{align*}
where in the last but one line we used that $D(\alpha)=0$ since $\alpha$ is a Maurer-Cartan element.
\end{proof}

\begin{proof}[Proof of Proposition \ref{prop:linfty from g on M action}]
Proposition \ref{prop:Lifty semidirect} provides a $\SL_\infty$-structure on $\fg\oplus \Def(\MOp\xrightarrow{f} \NOp)$.
Looking at the definition in the proof of that Proposition we see that this $\SL_\infty$-structure encodes an action of $\fg$ on $\Def(\MOp\xrightarrow{f} \NOp)$. Hence we can invoke Lemma \ref{lem:linfty from g on h action} to obtain the desired $\SL_\infty$-morphism.
Tracing the definitions one quickly checks that the induced map of Maurer-Cartan spaces is indeed given by the action of the exponential group.
\end{proof}

\section{Graphical (co)operads and (co)modules}
\label{sec:graphical cooperads}
The purpose of this section is to introduce various cooperads and comodules defined using (Feynman) diagrams. Mostly, these objects have already appeared in the literature. The cooperad $\Graphs_n$ whose construction is recalled in section \ref{sec:Graphsn} was defined by Kontsevich in \cite{K2}. The definition of the comodule $\Graphs_{V,n}$ appeared (essentially) in \cite{CW}, and the Lie algebra $\GGC_{V,n}$ can be found in a restricted form in \cite{Felder}.

Generally, graph complexes can be defined in one of two ways, either using general algebraic and operadic constructions, or combinatorially.
The former approach has the advantage that most formal properties, like the existence of algebraic structures, or the equation $d^2=0$ for the differential follow immediately from the constructions.
The latter, combinatorial approach to graph complexes is much more elementary and likely preferred by non-experts in operad theory.
However, in the combinatorial picture identities like $d^2=0$ are usually more tedious to show explicitly, and some formulas (e.g., for the differentials) seemingly come from thin air.

Here we will take a hybrid approach, introducing the graph complexes in a mostly combinatorial form, but yet utilizing the algebraic structures to minimize the amount of identities to be checked by hand.
This requires us to temporarily consider somewhat larger graph complexes than the ones we will eventually use.
More concretely, we will temporarily consider graph complexes with no valency restrictions on vertices, while we later always restrict to (internally) $\geq 3$-valent graphs, resulting in much smaller complexes.
We will indicate this by a superscript "$\tvv$", for example $\Graphs_{V,n}^{\tvv}$ will be the intermediate graph complex with no valency restrictions, while the object we actually use later is $\Graphs_{V,n}$.

\subsection{The Kontsevich $\La$ Hopf cooperad $\Graphs_n$}\label{sec:Graphsn}
For $n$ a fixed integer we define the free graded commutative algebras 
\[
  \Gra_n(r) = \Q[\omega_{ij}=(-1)^n\omega_{ji}\mid 1\leq i,j\leq r],
\]
with the generators $\omega_{ij}$ of degree $n-1$.
Elements of $\Gra_n(r)$ can be understood as linear combinations of graphs with $r$ numbered vertices and an edge between vertices $i$ and $j$ for every factor $\omega_{ij}$, for example 
\[
\omega_{12}^2\omega_{23} \omega_{45}\omega_{55}
\leftrightarrow 
\begin{tikzpicture}
  \node[ext] (v1) at (0,0) {$\scriptstyle 1$};
  \node[ext] (v2) at (0.7,0) {$\scriptstyle 2$};
  \node[ext] (v3) at (1.4,0) {$\scriptstyle 3$};
  \node[ext] (v4) at (2.1,0) {$\scriptstyle 4$};
  \node[ext] (v5) at (2.8,0) {$\scriptstyle 5$};
  \draw (v2) edge[bend left] (v1) edge[bend right] (v1) edge (v3) (v4) edge (v5) 
  (v5) edge[loop above] (v5);
\end{tikzpicture}\ .
\]
Similarly, we may build the symmetric sequences of dg commutative algebras
\[
  \fGraphs_n(r) = \bigoplus_{k\geq 0} (\Gra_n(r+k)\otimes \Q[n]^{\otimes k})_{S_k}.
\]
Elements of $\fGraphs_n(r)$ may be understood as linear combinations of graphs with $r$ numbered "external" vertices and an arbitrary (finite) number $k$ of unnumbered "internal" vertices. We shall draw the latter black in pictures.
\[
\begin{tikzpicture}
  \node[ext] (v1) at (0,0) {$\scriptstyle 1$};
  \node[ext] (v2) at (.7,0) {$\scriptstyle 2$};
  \node[ext] (v3) at (1.4,0) {$\scriptstyle 3$};
  \node[int] (i1) at (.35,.7) {};
  \node[int] (i2) at (1.05,.7) {};
  \draw (v1) edge (i1) 
  (i1) edge (i2) edge (v2) 
  (i2) edge (v2) edge (v3)
  (v2) edge[bend right] (v3);
\end{tikzpicture}
\in \fGraphs_{n}(3) 
\] 
Let us also make explicit the degree and sign conventions that are contained in the definition above. The cohomological degree of a graph is 
\[
  (n-1)(\#\text{edges})-n (\#\text{internal vertices}).  
\]
In other words, the internal vertices carry degree $-n$, and edges carry degree $n-1$. 
The coinvariants under the symmetric group actions translate into the following sign conventions for graphs, depending on the parity of $n$:
\begin{itemize}
\item For $n$ even, we shall think of a graph as coming with an ordering of the set of edges, stemming from the ordering of the factors in the corresponding monomial in $\omega_{ij}$.
Changing this ordering of edges by a permutation $p$ alters the sign by $\sgn(p)$.
\item For $n$ odd the graph comes with an ordering of the set of internal vertices, from the factors $\Q[n]$, and changing this ordering by a permutation $p$ changes the sign by $\sgn(p)$.
Furthermore, since $\omega_{ij}=-\omega_{ji}$, the edges have to be oriented, and we identify two different orientations up to sign.
\[
\begin{tikzpicture}
  \draw[->] (0,0) -- (.7,0);
\end{tikzpicture}
=
-\, 
\begin{tikzpicture}
  \draw[<-] (0,0) -- (.7,0);
\end{tikzpicture}
\]
\end{itemize}
Note that we often omit these orientation data (ordering of edges, vertices, and edge directions) from pictures, thus leaving the sign undefined.

Each space $\fGraphs_{n}(r)$ is a free graded commutative algebra, 
the commutative product is obtained by gluing to graphs at the $r$
external vertices, for example:
\begin{equation}\label{equ:Graphs product pic}
  \left(
  \begin{tikzpicture}
    \node[ext] (v1) at (0,0) {$\scriptstyle 1$};
    \node[ext] (v2) at (.7,0) {$\scriptstyle 2$};
    \node[ext] (v3) at (1.4,0) {$\scriptstyle 3$};
    \node[int] (i1) at (.35,.7) {};
    \node[int] (i2) at (1.05,.7) {};
    \draw (v1) edge (i1) 
    (i1) edge (i2) edge (v2) 
    (i2) edge (v2) edge (v3);
  \end{tikzpicture}
  \right)
  \wedge 
  \left(
  \begin{tikzpicture}
    \node[ext] (v1) at (0,0) {$\scriptstyle 1$};
    \node[ext] (v2) at (.7,0) {$\scriptstyle 2$};
    \node[ext] (v3) at (1.4,0) {$\scriptstyle 3$};
    \node[int, white] (i1) at (.35,.7) {};
    \draw 
    (v2) edge[bend right] (v3);
  \end{tikzpicture}
  \right)
  =
  \begin{tikzpicture}
    \node[ext] (v1) at (0,0) {$\scriptstyle 1$};
    \node[ext] (v2) at (.7,0) {$\scriptstyle 2$};
    \node[ext] (v3) at (1.4,0) {$\scriptstyle 3$};
    \node[int] (i1) at (.35,.7) {};
    \node[int] (i2) at (1.05,.7) {};
    \draw (v1) edge (i1) 
    (i1) edge (i2) edge (v2) 
    (i2) edge (v2) edge (v3)
    (v2) edge[bend right] (v3);
  \end{tikzpicture}\, .
\end{equation}
To fix the signs in this picture, and in similar pictures later, one has to specify the orientation data on the right-hand graph.
We do this by retaining the natural ordering of edges or vertices on the graphs on the left-hand side, and just juxtapposing their orderings.

The collection of spaces $\fGraphs_n$ furthermore comes with the structure of a graded (symmetric) $\La$-cooperad.
First, the right $S_r$-action on $\fGraphs_n$ is defined by permuting the numbering on the external vertices.
To define the cooperadic cocomposition it is sufficient to specify the reduced cocompositions of the form 
\[
  \Delta_s : \fGraphs_n(r) \to \fGraphs_n(r-s+1)\otimes \fGraphs_n(s)
\]
corresponding to the subset $\{1,\dots,s\}\subset \{1,\dots,r\}$.
Then, for a graph $\Gamma\in \fGraphs_n(r)$,
\[
\Delta_s (\Gamma)  
=
\sum_{\gamma\subset \Gamma\atop 1,\dots,s\in \gamma}
\pm
(\Gamma/\gamma) \otimes \gamma,
\]
with the sum over all subgraphs\footnote{A subgraph is a subset of the sets of vertices and edges of $\Gamma$, so that for each edge contained in the subgraph $\gamma$, its endpoints are also in $\gamma$.} $\gamma\subset\Gamma$ that contain the external vertices $1,\dots,s$, and no other external vertices, and with $\Gamma/\gamma$ the graph obtained by contracting $\gamma$ to one new external vertex numbered 1 and renumbering the external vertices $s+1,\dots,r$ to $2,\dots,r-s+1$.
The sign is the sign of the unshuffle permutation moving the edges or vertices of $\gamma$ to the right, relative to the ordering of edges/vertices in $\Gamma$. 
Finally, the $\La$-operations $\fGraphs_n(r)\to \fGraphs_n(r+1)$ are defined by adding one zero-valent external vertex to the graph.

We define a differential on $\fGraphs_n$ by edge contraction.
To simplify later proofs, we define it in two pieces.
The first piece is defined on a graph $\Gamma\in \fGraphs_n$ by
\[
  d_c' \Gamma = \sum_{e} \pm \Gamma /e,
\]
with the sum over all edges between two distinct internal vertices, and $\Gamma /e$ the graph obtained by contracting edge $e$.
\begin{align}\label{equ:dcp pic}
  d_c'
  \begin{tikzpicture}[baseline=-.65ex]
  \node[int] (v) at (0,0) {};
  \node[int](w) at (0,.5) {};
  \draw (v) edge +(-.5,.5) edge +(.5,.5) edge (w) (w) edge +(-.2,.5) edge +(.2,.5);
  \end{tikzpicture}
  &=
  \begin{tikzpicture}[baseline=-.65ex]
  \node[int] (v) {};
  \draw (v) edge +(-.5,.5) edge +(-.2,.5) edge +(.2,.5) edge +(.5,.5);
  \end{tikzpicture}
  \end{align}
If the order of edges and vertices is such that $e$ is the first edge, connecting the first internal vertex in the ordering to the second, then the sign is "+", with the ordering of the remaining edges and vertices in $\Gamma /e$ being the same as in $\Gamma$.
We leave it to the reader to check that $(d_c')^2=0$, and $d_c'$ respects the $\La$ Hopf cooperad structure on $\fGraphs_n$. That is, $d_c'$ is a coderivation with respect to the $\La$ cooperad structure, and a derivation with respect to the commutative algebra structure.

Next, we want to apply the twisting construction of section \ref{sec:optwist simple}.
To this end, we consider the dual operad $(\fGraphs_n)^*$, whose elements can naturally be identified with graphs as well.
We denote the dual differential to $d_c'$ by $\delta_{split}'$.
(It is given combinatorially by splitting internal vertices.)
We define the element 
\[
M_\mid = 
\begin{tikzpicture}
\node[ext] (v) at (0,0) {$\scriptstyle 1$};  
\node[int] (w) at (0,0.7) {};
\draw (v) edge (w);  
\end{tikzpicture}\in (\fGraphs_n(1))^*
\]
\begin{lemma}
The element $-M_\mid\in (\fGraphs_n(1))^*$ is a Maurer-Cartan element, i.e.,
\[
  \delta_{split}'M_\mid -M_\mid \circ_1 M_\mid =0.
\]
\end{lemma}
\begin{proof}
We compute 
\[
  \delta_{split}'M_\mid
  =
  \begin{tikzpicture}
    \node[ext] (v) at (0,0) {$\scriptstyle 1$};  
    \node[int] (w) at (0,0.5) {};
    \node[int] (ww) at (0,0.9) {};
    \draw (v) edge (w) (w) edge (ww);  
  \end{tikzpicture}
\]
with the ordering such that edges and vertices are ordered and oriented top down. Furthermore,
\[
  M_\mid \circ_1 M_\mid
  =
  \underbrace{
    \begin{tikzpicture}
      \node[ext] (v) at (0,0) {$\scriptstyle 1$};  
      \node[int] (w) at (-.5,0.7) {};
      \node[int] (ww) at (0.5,0.7) {};
      \draw (v) edge (w) edge (ww);  
    \end{tikzpicture}
  }_{=0}
  \, +\, 
  \begin{tikzpicture}
    \node[ext] (v) at (0,0) {$\scriptstyle 1$};  
    \node[int] (w) at (0,0.5) {};
    \node[int] (ww) at (0,0.9) {};
    \draw (v) edge (w) (w) edge (ww);  
  \end{tikzpicture}.
\]
The first graph on the right-hand side is zero by symmetry.
In the second, the vertices and edges are ordered and oriented top-down.
\end{proof}

Hence, as in section \ref{sec:optwist simple}, we may endow $\fGraphs_n$ with the differential $d:=d_c' - M_\mid \cdot$ to obtain a 
dg $\La$ Hopf cooperad.

So far we have been working with the relatively large cooperad $\fGraphs_n$. We then define the quotient
\[
  \Graphs_n(r) = \fGraphs_n(r) /J(r)
\]
by the subspace $J(r)$ spanned by graphs that either have a connected component of internal vertices only or contain internal vertices of valence $\leq 2$.
We leave it to the reader to check that the dg $\La$ Hopf cooperad structure descends to this quotient.

Let us only note that combinatorially, the differential $d$ acts by contracting edges, with the part $d_c''=- M_\mid \cdot$ contracting edges between an internal and an external vertex.
\begin{align*}
d_c''
\begin{tikzpicture}[baseline=-.65ex]
\node[ext] (v) at (0,0) {$\scriptstyle j$};
\node[int](w) at (0,.5) {};
\draw (v) edge +(-.5,.5) edge +(.5,.5) edge (w) (w) edge +(-.2,.5) edge +(.2,.5);
\end{tikzpicture}
&=
\begin{tikzpicture}[baseline=-.65ex]
\node[ext] (v) {$\scriptstyle j$};
\draw (v) edge +(-.5,.5) edge +(-.2,.5) edge +(.2,.5) edge +(.5,.5);
\end{tikzpicture}
\end{align*}



There is a natural map of $\La$ Hopf cooperads from $\Graphs_n$ to the $n$-Poisson cooperad
\begin{equation}\label{equ:Graphsen}
\Graphs_n\to \Poiss_n^c 
\end{equation}
sending any graph with internal vertices to zero and any edge between vertices $i$ and $j$ to the algebra generators $\omega_{ij}$ of $\Poiss_n^c$, see section \ref{sec:en def}.
Also, note again that for $n\geq 2$ we have $\e_n^c=\Poiss_n^c$.
We just quote the following result.

\begin{thm}[Kontsevich, Lambrechts-Voli\'c \cite{K2}, \cite{LVformal}]
  The map \eqref{equ:Graphsen} is a quasi-isomorphism.
\end{thm}

\subsection{The $\La$ Hopf $\fGraphs_n$ comodule $\fGraphs_{V,n}$}
Let $V$ be a finite dimensional positively graded vector space and $n$ an integer.
Then we define collection of graded commutative algebras
\[
  \Gra_{V,n}(r) = \Gra_n(r) \otimes S(V)^{\otimes r}.
\]
Elements can be understood as graphs with $r$ numbered vertices, that can furthermore be decorated by zero or more elements of $V$, or alternatively by one element of $S(V)$ each.
Next consider 
\[
\fGraphs_{V,n}(r) :=  \bigoplus_{k\geq 0} (\Gra_{V,n}(r+k)\otimes \Q[n]^{\otimes k})_{S_k}.
\]
Elements of this space are linear combinations of graphs with $r$ numbered external vertices and an arbitrary number of unnumbered internal vertices. We again draw the latter as black dots in pictures.
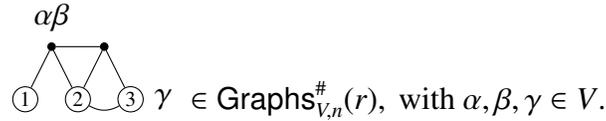
\[
\begin{tikzpicture}
  \node[ext] (v1) at (0,0) {$\scriptstyle 1$};
  \node[ext] (v2) at (.7,0) {$\scriptstyle 2$};
  \node[ext,label=0:{$\gamma$}] (v3) at (1.4,0) {$\scriptstyle 3$};
  \node[int,label=90:{$\alpha\beta$}] (i1) at (.35,.7) {};
  \node[int] (i2) at (1.05,.7) {};
  \draw (v1) edge (i1) 
  (i1) edge (i2) edge (v2) 
  (i2) edge (v2) edge (v3)
  (v2) edge[bend right] (v3);
\end{tikzpicture}
\in \fGraphs_{V,n}(r), \text{ with $\alpha,\beta,\gamma\in V$.} 
\] 
We define a graded $\La$ Hopf right cooperadic $\fGraphs_{n}$-comodule structure on $\fGraphs_{V,n}$ as follows.
First, the commutative algebra structure is defined as for $\fGraphs_{n}$ by gluing graphs at external vertices, analogous to \eqref{equ:Graphs product pic}.
Furthermore, by subgraph contraction 
$\fGraphs_{n,V}$ is equipped with a right $\fGraphs_n$-cooperadic comodule structure.
More precisely, denote by 
\[
  \Delta_s : \fGraphs_{V,n}(r) \to \fGraphs_{V,n}(r-s+1)\otimes \fGraphs_n(s)
\]
the cooperadic coactions corresponding to the subset $\{1,\dots,s\}\subset \{1,\dots,r\}$.
Then, for a graph $\Gamma\in \fGraphs_{V,n}(r)$,
\[
\Delta_s (\Gamma)  
=
\sum_{\gamma\subset \Gamma\atop 1,\dots,s\in \gamma}
\pm
(\Gamma/\gamma) \otimes \gamma,
\]
with the sum running over all subgraphs $\gamma\subset\Gamma$ that contain the external vertices $1,\dots,s$, and no other external vertices, and with $\Gamma/\gamma$ the graph obtained by contracting $\gamma$ to one new external vertex numbered 1 and renumbering the external vertices $s+1,\dots,r$ to $2,\dots,r-s+1$.

The $\La$-operations $\fGraphs_{V,n}(r)\to \fGraphs_{V,n}(r+1)$ are defined by adding one zero-valent external vertex to the graph.

The differential on $\fGraphs_{V,n}$ is defined (for now) just as the operation $d_c'$ contrating an edge between a pair of internal vertices as in \eqref{equ:dcp pic}.

\subsection{The graded Lie algebra $\fGGC_{V,n}$}\label{sec:HGC upper 1}
We next define Lie coalgebras $\fGG_{V,n}$ and $\GG_{V,n}$, generalizing similar constructions of \cite{Felder}. 
To prepare, let $V_1:=V\oplus \Q 1$ be obtained by adding a one-dimensional vector space to $V$, spanned by the symbol $1$ of degree 0. Pick a basis $e_1,\dots,e_N$ of $V$ and extend it to a basis $e_0:=1,e_1,\dots,e_N$ of $V_1$. Let $f_0,\dots,f_N$ be the corresponding dual basis of $V_1^*$.

Furthermore let $\nfGraphs_{n,V}$ be the quotient of $\fGraphs_{n,V}$ obtained by setting to zero all graphs that have an external vertex of valency $\neq 1$, and all graphs that are not connected.
Then as a dg vector space we define
\[
\fGG_{V,n}=
\bigoplus_{r\geq 1}\nfGraphs_{n,V}(r)\otimes_{S_r} (V_1^*)^{\otimes r}.
\]
Elements can be seen as graphs with legs (or hairs) that are each decorated by a single element of $V_1^*$.
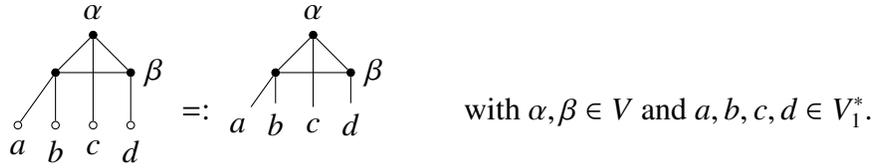
\[
\begin{tikzpicture}[yshift=-.5cm]
  \node[int] (i1) at (.5,1) {};
  \node[int,label=90:{$\alpha$}] (i2) at (1,1.5) {};
  \node[int,label=0:{$\beta$}] (i3) at (1.5,1) {};
  \node[ext,label=-90:{$a$}] (e1) at (0,.3) {};
  \node[ext,label=-90:{$b$}] (e2) at (.5,.3) {};
  \node[ext,label=-90:{$c$}] (e3) at (1,.3) {};
  \node[ext,label=-90:{$d$}] (e4) at (1.5,.3) {};
  \draw (i1) edge (i2) edge (i3) edge (e1) edge (e2)
  (i2) edge (i3) edge (e3)
  (i3) edge (e4);
\end{tikzpicture}
=:
\begin{tikzpicture}[yshift=-.5cm]
  \node[int] (i1) at (.5,1) {};
  \node[int,label=90:{$\alpha$}] (i2) at (1,1.5) {};
  \node[int,label=0:{$\beta$}] (i3) at (1.5,1) {};
  \node (e1) at (0,.3) {$a$};
  \node (e2) at (.5,.3) {$b$};
  \node (e3) at (1,.3) {$c$};
  \node (e4) at (1.5,.3) {$d$};
  \draw (i1) edge (i2) edge (i3) edge (e1) edge (e2)
  (i2) edge (i3) edge (e3)
  (i3) edge (e4);
\end{tikzpicture}
\quad
\quad 
\text{ with $\alpha,\beta\in V$ and $a,b,c,d\in V_1^*$.}
\]
In the following, we will generally not draw the (external) vertices at the end of the hairs, as is indicated in the picture. The exception are the graphs with a single external vertex, which cannot be meaningfully drawn without the vertex.
\begin{equation}\label{equ:onevert HG}
  \begin{tikzpicture}
    \node[ext,label=-90:{$a$},label=90:{$\alpha$}] (e1) at (0,0) {};
  \end{tikzpicture}
  \quad\quad \text{with $\alpha\in V$, $a\in V_1^*$.}
\end{equation}
These one-vertex graphs span a copy of $V_1^*\otimes V\cong \iHom(V_1,V)$ inside $\fGG_{V,n}$.

We will next define a Lie admissible coproduct on $\fGG_{V,n}$ and a Lie admissible coaction of $\fGG_{V,n}$ on $\fGraphs_{V,n}$.
\begin{defn}
We define a \emph{cut} of a
graph $\Gamma\in \fGG_{V,n}$
or $\Gamma\in \fGraphs_{V,n}$ to be a subset $I$ of the set of half-edges and decorations of $\Gamma$ that satisfes the following conditions:
\begin{itemize}
  \item If we detach all half-edges and decorations in $I$ from their vertices, then the half-edges and decorations are contained in a single connected component, that contains none of the vertices incident to the half edges. We call this connected component the upper part, and its complement the lower part.
  \item The lower part contains at least one external vertex.
\end{itemize}
We say that a cut is connected if the lower part  is connected, and we say that the cut is internal, if the lower part contains all the external vertices. We write $I\subset \Gamma$ to indicate that $I$ is a cut of $\Gamma$.
\end{defn}
Note that if any decoration is contained in $I$, then the conditions are only satisfiable if $I$ consists of that single decoration and nothing else. We call this the degenerate cut. Otherwise $I$ is a subset of the set of half-edges. We call the elements of $I$ the cut half edges.
A cut may be graphically represented by drawing a horizontal line through the graph that intersects all cut half-edges and no other edges, such that the vertices incident to the cut half-edges are below the line. The following are examples of cuts for graphs in $\fGG_{V,n}$ on $\fGraphs_{V,n}$.

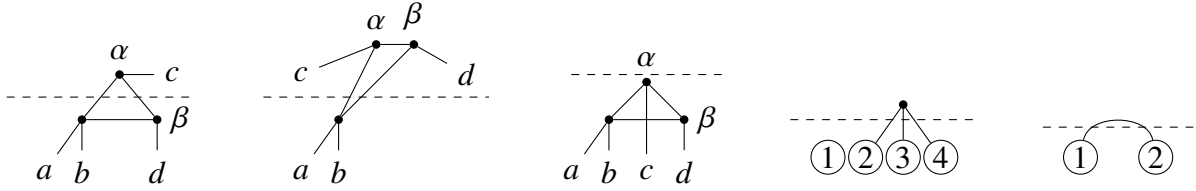
\begin{gather*}
  \begin{tikzpicture}[yshift=-.5cm]
    \node[int] (i1) at (.5,1) {};
    \node[int,label=90:{$\alpha$}] (i2) at (1,1.6) {};
    \node[int,label=0:{$\beta$}] (i3) at (1.5,1) {};
    \node (e1) at (0,.3) {$a$};
    \node (e2) at (.5,.3) {$b$};
    \node (e3) at (1.7,1.6) {$c$};
    \node (e4) at (1.5,.3) {$d$};
    \draw (i1) edge (i2) edge (i3) edge (e1) edge (e2)
    (i2) edge (i3) edge (e3)
    (i3) edge (e4);
    \draw[dashed] (-.5,1.3) -- (2, 1.3);
  \end{tikzpicture}
  \quad\quad 
  \begin{tikzpicture}[yshift=-.5cm]
    \node[int] (i1) at (.5,1) {};
    \node[int,label=90:{$\alpha$}] (i2) at (1,2) {};
    \node[int,label=90:{$\beta$}] (i3) at (1.5,2) {};
    \node (e1) at (0,.3) {$a$};
    \node (e2) at (.5,.3) {$b$};
    \node (e3) at (0,1.6) {$c$};
    \node (e4) at (2.2,1.6) {$d$};
    \draw (i1) edge (i2) edge (i3) edge (e1) edge (e2)
    (i2) edge (i3) edge (e3)
    (i3) edge (e4);
    \draw[dashed] (-.5,1.3) -- (2.5, 1.3);
  \end{tikzpicture}
  \quad \quad 
  \begin{tikzpicture}[yshift=-.5cm]
    \node[int] (i1) at (.5,1) {};
    \node[int,label=90:{$\alpha$}] (i2) at (1,1.5) {};
    \node[int,label=0:{$\beta$}] (i3) at (1.5,1) {};
    \node (e1) at (0,.3) {$a$};
    \node (e2) at (.5,.3) {$b$};
    \node (e3) at (1,.3) {$c$};
    \node (e4) at (1.5,.3) {$d$};
    \draw (i1) edge (i2) edge (i3) edge (e1) edge (e2)
    (i2) edge (i3) edge (e3)
    (i3) edge (e4);
    \draw[dashed] (0,1.6) -- (2, 1.6);
  \end{tikzpicture}
  \quad \quad 
  \begin{tikzpicture}[baseline=-.65ex]
    \node[ext] (v1) at (0,0) {1};
    \node[ext] (v2) at (0.5,0) {2};
    \node[ext] (v3) at (1,0) {3};
    \node[ext] (v4) at (1.5,0) {4};
    \node[int] (w2) at (1.0,.7) {};
    \draw   (v2)  edge (w2) (v3)  edge (w2) (v4) edge (w2);
    \draw[dashed] (-.5,.5) -- (2, .5);
    \end{tikzpicture}
    \quad \quad 
    \begin{tikzpicture}[baseline=-.65ex]
      \node[ext] (v1) at (0,0) {1};
      \node[ext] (v2) at (1,0) {2};
      \draw  (v1) edge[out=80,in=100] (v2);
      \draw[dashed] (-0.5,.4) -- (1.5, .4);
      \end{tikzpicture}
\end{gather*}

Now we use the following construction to produce from a connected cut $I$ of $\Gamma\in \fGG_{V,n}$ an element of $\fGG_{V,n}\otimes  \fGG_{V,n}$, i.e., two hairy graphs.
\begin{enumerate}
\item We detach the cut half-edges from their incident vertices, i.e., we "cut along the dashed line".
\item We then add a univalent external vertex to each cut half-edge, and decorate it with $f_j\in V_1^*$.
We decorate the vertex the half-edge was attached to with $e_j$, and sum over all $j=0,\dots,N$, for each cut half-edge. 
In other words, we formally replace the cut half-edges with a diagonal in $V_1^*\otimes V_1$.
Hereby, we identify a decoration by the special element $1=e_0\in V_1$ on $\Gamma'$ by no decoration at all. The following examples shall illustrate the procedure:
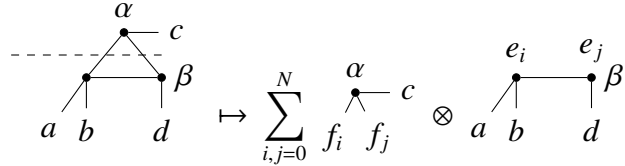
\[ 
  \begin{tikzpicture}[yshift=-.5cm]
    \node[int] (i1) at (.5,1) {};
    \node[int,label=90:{$\alpha$}] (i2) at (1,1.6) {};
    \node[int,label=0:{$\beta$}] (i3) at (1.5,1) {};
    \node (e1) at (0,.3) {$a$};
    \node (e2) at (.5,.3) {$b$};
    \node (e3) at (1.7,1.6) {$c$};
    \node (e4) at (1.5,.3) {$d$};
    \draw (i1) edge (i2) edge (i3) edge (e1) edge (e2)
    (i2) edge (i3) edge (e3)
    (i3) edge (e4);
    \draw[dashed] (-.5,1.3) -- (1.5, 1.3);
  \end{tikzpicture}
\, \mapsto \, 
\sum_{i,j=0}^N
\begin{tikzpicture}
  \node[int,label=90:{$\alpha$}] (i2) at (0,.3) {};
  \node (e1) at (-.3,-.3) {$f_i$};
  \node (e2) at (.3,-.3) {$f_j$};
  \node (e3o) at (.7,.3) {$c$};
  \draw (i2) edge (e1) edge (e2) edge (e3o);
\end{tikzpicture}
\otimes 
\begin{tikzpicture}[yshift=-.5cm]
  \node[int,label=90:{$e_i$}] (i1) at (.5,1) {};
  \node[int,label=0:{$\beta$},label=90:{$e_j$}] (i3) at (1.5,1) {};
  \node (e1) at (0,.3) {$a$};
  \node (e2) at (.5,.3) {$b$};
  \node (e4) at (1.5,.3) {$d$};
  \draw (i1) edge (i3) edge (e1) edge (e2)
  (i3) edge (e4);
\end{tikzpicture}
\]
\item In the special case of a degenerate cut $I$, consisting of a decoration only, the same procedure applies, creating a one-vertex graph of the form \eqref{equ:onevert HG} with vertex decorated by $\alpha$ and $f_j$.
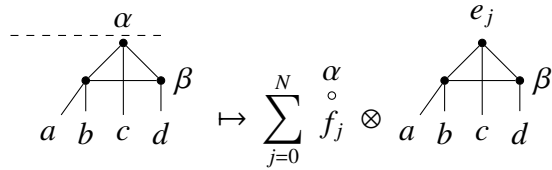
\[
  \begin{tikzpicture}[yshift=-.5cm]
    \node[int] (i1) at (.5,1) {};
    \node[int,label=90:{$\alpha$}] (i2) at (1,1.5) {};
    \node[int,label=0:{$\beta$}] (i3) at (1.5,1) {};
    \node (e1) at (0,.3) {$a$};
    \node (e2) at (.5,.3) {$b$};
    \node (e3) at (1,.3) {$c$};
    \node (e4) at (1.5,.3) {$d$};
    \draw (i1) edge (i2) edge (i3) edge (e1) edge (e2)
    (i2) edge (i3) edge (e3)
    (i3) edge (e4);
    \draw[dashed] (-.5,1.6) -- (1.5, 1.6);
  \end{tikzpicture}
  \, \mapsto \, 
  \sum_{j=0}^N
  \begin{tikzpicture}
    \node[ext,label=-90:{$f_j$},label=90:{$\alpha$}] (e1) at (0,.3) {};
  \end{tikzpicture}
  \otimes 
  \begin{tikzpicture}[yshift=-.5cm]
    \node[int] (i1) at (.5,1) {};
    \node[int,label=90:{$e_j$}] (i2) at (1,1.5) {};
    \node[int,label=0:{$\beta$}] (i3) at (1.5,1) {};
    \node (e1) at (0,.3) {$a$};
    \node (e2) at (.5,.3) {$b$};
    \node (e3) at (1,.3) {$c$};
    \node (e4) at (1.5,.3) {$d$};
    \draw (i1) edge (i2) edge (i3) edge (e1) edge (e2)
    (i2) edge (i3) edge (e3)
    (i3) edge (e4);
  \end{tikzpicture}
\]
\item We denote the graphs thus produced by the Sweedler-type notation 
\[
\sum \Gamma^I \otimes \Gamma_I, 
\]
and we call $\Gamma^I$ an upper graph of the cut and $\Gamma_I$ a lower graph.
\end{enumerate}
The Lie-admissible coproduct\footnote{See section \ref{sec:prelie} for the definition of Lie-admissible algebras.} a graph $\Gamma\in \fGG_{V,n}$ is then given by the sum over all connected cuts
\begin{equation}\label{equ:DeltadefpreLie}
  \Delta \Gamma = \sum_{I\subset \Gamma \atop \text{ conn. cut}} \sum \Gamma^I \otimes \Gamma_I
  \in \fGG_{V,n}\otimes \fGG_{V,n}.
\end{equation}
The Lie admissible coaction on some $\Gamma\in \Graphs_{n,V}$ is defined by the exact same formulas except that one sums over all internal cuts, i.e., those cuts for which the upper part does not include external vertices.
\begin{equation}\label{equ:DeltadefpreLiemod}
  \Delta_m \Gamma = \sum_{I \subset \Gamma \atop \text{int. cut}} \sum \Gamma^I \otimes \Gamma_I
  \in \fGG_{V,n}\otimes \fGraphs_{V,n}.
\end{equation}

\begin{prop}
The formulas \eqref{equ:DeltadefpreLie} and \eqref{equ:DeltadefpreLiemod} define a dg Lie-admissible coalgebra structure on $\fGG_{V,n}$ and a dg Lie-admissible coaction on $\fGraphs_{V,n}$ compatible with the dg $\La$ Hopf comodule structure, with differential $d_c'$. In particular, $\fGG_{V,n}$ is a dg Lie coalgebra, and its dual space $\fGGC_{V,n}=(\fGG_{V,n})^*$ is a dg Lie algebra, acting on $\fGraphs_{V,n}$, preserving the $\La$ Hopf $\fGraphs_n$-comodule structure.
\end{prop}

\begin{proof}
For $\Gamma\in \fGG_{V,n}$ we compute the co-associator 
\[
  A(\gamma) = (\Delta\otimes \id)\circ \Delta(\Gamma) -   (\id \otimes  \Delta)\circ \Delta(\Gamma)
\]
and verify that its antisymmetrization vanishes.
We have 
\begin{align}\label{equ:HG assoc1}
  (\Delta\otimes \id)\circ \Delta(\Gamma)
  =
  \sum_{I \subset \Gamma \atop \text{conn. cut}}
  \sum_{J \subset \Gamma^{I} \atop \text{conn. cut}}
  \sum 
  (\Gamma^{I})^{J} \otimes (\Gamma^{I})_{J} \otimes  \Gamma_{I}.
\end{align}
while 
\begin{align}\label{equ:HG assoc2}
  (\id \otimes  \Delta)\circ \Delta(\Gamma)
  =
  \sum_{J' \subset \Gamma \atop \text{conn. cut}}
  \sum_{I' \subset \Gamma_{J'} \atop \text{conn. cut}}
  \sum 
  \Gamma^{J'} \otimes (\Gamma_{J'})^{I'} \otimes  \Gamma_{I'}.
\end{align}
Each of the sums is over pairs of cuts that cut the original connected graph $\Gamma$ into three connected pieces. The difference between \eqref{equ:HG assoc1} and \eqref{equ:HG assoc2} is that in \eqref{equ:HG assoc1} the second cut goes through the upper part, while in \eqref{equ:HG assoc2} the second cut goes through the lower part.
We may classify these double cuts by the incidence relations between the three pieces of the graph produced, with two pieces being incident if they have a common interface along they can be cut in one combination of cuts. These incidence patterns correspond to the connected directed acyclic graphs with three vertices, and look as follows:
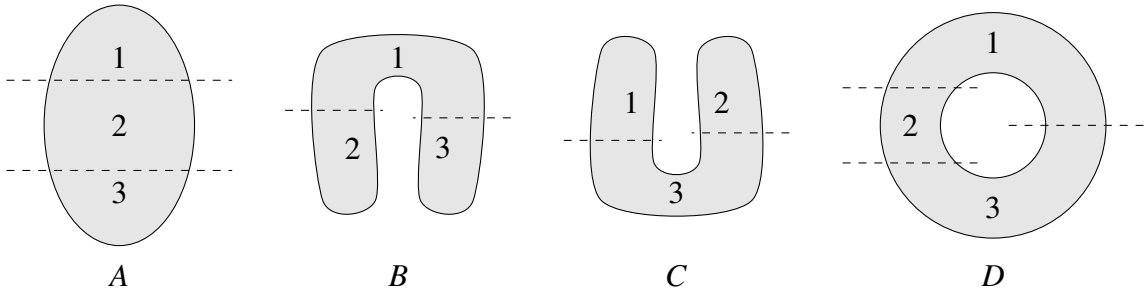
\begin{align*}
  &
\begin{tikzpicture}
\draw[fill=black!10] (0,0) ellipse (1cm and 1.6cm);
\draw[dashed] (-1.5,-.6) -- (1.5,-.6); 
\draw[dashed] (-1.5,.6) -- (1.5,.6); 
\node at (0,.9) {$1$};
\node at (0,0) {$2$};
\node at (0,-.9) {$3$};
\node at (0,-2) {$A$};
\end{tikzpicture}
& &
\begin{tikzpicture}
  \draw[fill=black!10] plot [smooth cycle] coordinates { (-1,1) (1,1) (1,-1) (.3,-1) (.3,.5)(-.3,.5) (-.3,-1) (-1,-1)};
  \draw[dashed] (-1.5,0.2) -- (-.2,0.2); 
  \draw[dashed] (1.5,0.1) -- (.2,.1); 
  \node at (0,.9) {$1$};
  \node at (-.6,-.3) {$2$};
  \node at (.6,-.3) {$3$};
  \node at (0,-2) {$B$};
  \end{tikzpicture}
  & & 
  \begin{tikzpicture}[yscale=-1]
    \draw[fill=black!10] plot [smooth cycle] coordinates { (-1,1) (1,1) (1,-1) (.3,-1) (.3,.5)(-.3,.5) (-.3,-1) (-1,-1)};
    \draw[dashed] (-1.5,0.2) -- (-.2,0.2); 
    \draw[dashed] (1.5,0.1) -- (.2,.1); 
    \node at (0,.9) {$3$};
    \node at (-.6,-.3) {$1$};
    \node at (.6,-.3) {$2$};
    \node at (0,2) {$C$};
    \end{tikzpicture}
    & & 
    \begin{tikzpicture}
      \draw [fill=black!10,even odd rule] (0,0) circle[radius=1.5cm] circle[radius=.7cm];
      \draw[dashed] (-2,0.5) -- (-.2,0.5); 
      \draw[dashed] (-2,-0.5) -- (-.2,-.5); 
      \draw[dashed] (2,0) -- (.2,0); 
      \node at (0,1.1) {$1$};
      \node at (-1.1,0) {$2$};
      \node at (0,-1.1) {$3$};
      \node at (0,-2) {$D$};
      \end{tikzpicture}
\end{align*}
For example, here is an example graph with two cuts corresponding to the first pattern:
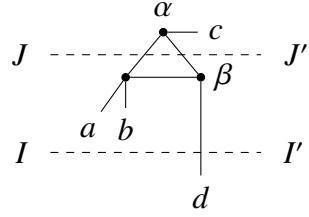
\[
    \begin{tikzpicture}[yshift=-.5cm]
      \node[int] (i1) at (.5,1) {};
      \node[int,label=90:{$\alpha$}] (i2) at (1,1.6) {};
      \node[int,label=0:{$\beta$}] (i3) at (1.5,1) {};
      \node (e1) at (0,.3) {$a$};
      \node (e2) at (.5,.3) {$b$};
      \node (e3) at (1.7,1.6) {$c$};
      \node (e4) at (1.5,-.6) {$d$};
      \draw (i1) edge (i2) edge (i3) edge (e1) edge (e2)
      (i2) edge (i3) edge (e3)
      (i3) edge (e4);
      \draw[dashed] (-.5,1.3) -- (2.3, 1.3);
      \draw[dashed] (-.5,0) -- (2.3, 0);
      \node[label=180:{$J$}] at (-.5,1.3) {};
      \node[label=180:{$I$}] at (-.5,0) {};
      \node[label=0:{$J'$}] at (2.3,1.3) {};
      \node[label=0:{$I'$}] at (2.3,0) {}; 
    \end{tikzpicture}
\]

We consider the double cuts corresponding to each pattern $A,B,C,D$ in turn.
For pattern $A$ the cuts are produced both in the sum \eqref{equ:HG assoc1} and \eqref{equ:HG assoc1}, since one can cut first below, then above, or in the other order. Hence in the coassociator, which is the difference of \eqref{equ:HG assoc1} and \eqref{equ:HG assoc1}, those terms disappear.
For pattern $D$ the same conclusion holds similarly.
For pattern $B$ the corresponding cuts occur only in the sum \eqref{equ:HG assoc1}, since it is not possible to exercise the second cut in the lower part. These terms will hence contribute to the associator.
However, since in \eqref{equ:HG assoc1} each pair of cuts can be performed in two orders, the corresponding contribution to the coassociator is symmetric under interchange of the second and third tensor factor. Hence the antisymmetrization vanishes.
Similarly, the contributions of terms corresponding to incidence pattern $C$ all come from \eqref{equ:HG assoc2}, but are symmetric under interchange of the first and second tensor factor.
Overall, this shows that the antisymmetrization of the coassociator vanishes. 

Similarly, we consider the coaction of $\HG_{V,n}$ on $\Graphs_{n,V}$.
We need to show that the comodule coassociator is symmetric in the first two slots. To this end we compare, for $\Gamma\in \Graphs_{n,V}$,
\begin{align}\label{equ:HG assoc1m}
  (\Delta\otimes \id)\circ \Delta_m(\Gamma)
  =
  \sum_{I \subset \Gamma \atop \text{int. cut}}
  \sum_{J \subset \Gamma^{I} \atop \text{conn. cut}}
  \sum 
  (\Gamma^{I})^{J} \otimes (\Gamma^{I})_{J} \otimes  \Gamma_{I}.
\end{align}
while 
\begin{align}\label{equ:HG assoc2m}
  (\id \otimes  \Delta_m)\circ \Delta_m(\Gamma)
  =
  \sum_{J' \subset \Gamma \atop \text{int. cut}}
  \sum_{I' \subset \Gamma_{J'} \atop \text{int. cut}}
  \sum 
  \Gamma^{J'} \otimes (\Gamma_{J'})^{I'} \otimes  \Gamma_{I'}.
\end{align}

Again both sums run over (certain) double cuts of $\Gamma$, with the caveat that for the internal cuts all external vertices are required to lie in the lower part of the cut.
Again we have the same 4 (a priori) possible incidence patterns depicted above between the three pieces of the graph produced by the cut.
For patterns $A$, $C$, $D$ the same argument as above shows that the respective contribution to the coassociator is either zero or symmetric in the first two factors.
For pattern $B$, however, the previous argument is not applicable: Since we only sum over internal cuts in the first sum of \eqref{equ:HG assoc1m}, there is only one permissible order of cutting, not two:
The external vertices are required to lie in the component 3, hence the unique permissible order of cutting is to first separate the union of components 1 and 2 from 3, and then cut 1 and 2. 
Fortunately, there occur no double cuts in the double sum \eqref{equ:HG assoc1m} with the incidence pattern $B$: The only external vertices in $\Gamma^I$ stem from the cut by definition of "internal cut". But to realize incidence pattern $B$ the second cut $J$ would need to be such that all external vertices lie in the upper part $(\Gamma^I)^J$, and none in $(\Gamma^I)^J$. But this is not permissible by our definitin of cut -- the lower part needs to contain at least one external vertex.

The arguments above show that we have a well-defined graded Lie-admissible coalgebra structure on $\GG_{V,n}$ and a Lie-admissible coaction on $\Graphs_{V,n}$.




\end{proof}

Let us describe the dual Lie algebra $\fGGC_{V,n}=(\fGG_{V,n})^*$ combinatorially. Generators of $\fGGC_{V,n}$ may be identified with graphs with hairs, that may carry one or more decorations in $V^*$ at internal vertices, and each hair carries one decoration in $V_1$.
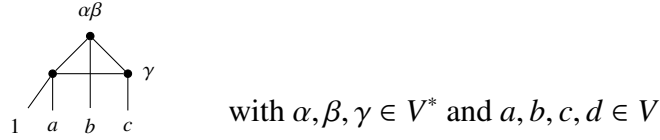
\[
  \begin{tikzpicture}[yshift=-.5cm]
    \node[int] (i1) at (.5,1) {};
    \node[int,label=90:{$\scriptstyle \alpha\beta$}] (i2) at (1,1.5) {};
    \node[int,label=0:{$\scriptstyle \gamma$}] (i3) at (1.5,1) {};
    \node (e1) at (0,.3) {$\scriptstyle 1$};
    \node (e2) at (.5,.3) {$\scriptstyle a$};
    \node (e3) at (1,.3) {$\scriptstyle b$};
    \node (e4) at (1.5,.3) {$\scriptstyle c$};
    \draw (i1) edge (i2) edge (i3) edge (e1) edge (e2)
    (i2) edge (i3) edge (e3)
    (i3) edge (e4);
  \end{tikzpicture}  
  \quad\quad \text{with }\alpha,\beta,\gamma\in V^* \text{ and } a,b,c,d\in V
\]
The Lie-admissible product $\Gamma_1 * \Gamma_2$ of two such graphs is obtained by summing over all partial matchings of hairs in $\Gamma_1$ to decorations in $\Gamma_2$, gluing the hair to the decorated vertex and multiplying by the natural pairing of the decorations. The decoration $1$ at hairs is treated specially: These hairs of $\Gamma_1$ can be connected to any vertex of $\Gamma_2$.
Here is a schematic picture of the product:
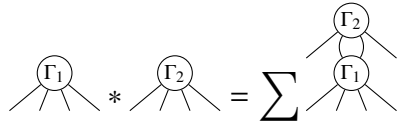
\[
\begin{tikzpicture}
\node[ext] (v) at (0,.3) {$\scriptstyle \Gamma_1$};
\draw (v) edge +(-.6,-.5) edge +(-.2,-.5) edge +(.2,-.5) edge +(.6,-.5) ;
\end{tikzpicture}
*
\begin{tikzpicture}
  \node[ext] (v) at (0,.3) {$\scriptstyle \Gamma_2$};
  \draw (v) edge +(-.6,-.5) edge +(-.2,-.5) edge +(.2,-.5) edge +(.6,-.5) ;
  \end{tikzpicture}
=
\sum 
\begin{tikzpicture}
  \node[ext] (v1) at (0,.3) {$\scriptstyle \Gamma_1$};
  \node[ext] (v2) at (0,1) {$\scriptstyle \Gamma_2$};
  \draw (v2) edge[bend left] (v1) edge[bend right] (v1) edge +(-.6,-.5)  edge +(.6,-.5) 
  (v1) edge +(-.6,-.5) edge +(-.2,-.5) edge +(.2,-.5) edge +(.6,-.5) ;
  \end{tikzpicture}
\]
The Lie bracket is the commutator $[\Gamma_1,\Gamma_2]=\Gamma_1*\Gamma_2- (-1)^{|\Gamma_1||\Gamma_2|}\Gamma_2*\Gamma_1$.

The differential on $\fGG_{V,n}$ is given by the operation $d_c'$ contracting an edge between two internal vertices.
The dual differential on $\fGGC_{V,n}$ is the dual operation 
\[
\delta_{split}' = (d_c')^* . 
\]
Combinatorially, this can be identified with the operation of splitting a vertex in all possible ways,
\begin{align}\label{equ:dsplit pic}
  \delta_{split}'
  \begin{tikzpicture}[baseline=-.65ex]
    \node[int] (v) {};
    \draw (v) edge +(-.5,.5) edge +(-.2,.5) edge +(.2,.5) edge +(.5,.5);
    \end{tikzpicture}
  &=
  \sum
  \begin{tikzpicture}[baseline=-.65ex]
    \node[int] (v) at (0,0) {};
    \node[int](w) at (0,.5) {};
    \draw (v) edge +(-.5,.5) edge +(.5,.5) edge (w) (w) edge +(-.2,.5) edge +(.2,.5);
    \end{tikzpicture}\, .
\end{align}

\subsection{The differential and the Lie algebra $\GGC_{V,n}$} \label{sec:HGC upper 2}

We note that any Maurer-Cartan element $Z\in \fGGC_{V,n}$ can be used to twist $\fGraphs_{V,n}$ to some $\La$ Hopf $\fGraphs_n$ comodule $(\fGraphs_{V,n})^Z$.
Concretely, the underlying graded $\La$ Hopf $\fGraphs_n$ comodule is the same, but the differential in $(\fGraphs_{V,n})^Z$ is modified to $d+Z\cdot$.

We next identify one specific Maurer-Cartan element 
\begin{lemma}
The element 
\[
M_2 = -
\sum_{j=0}^N
\begin{tikzpicture}
  \node (v) at (0,-.3) {$e_j$};
  \node[int, label=90:{$f_j$}] (w) at (0,.4) {};
  \draw (v) edge (w);
\end{tikzpicture}
\in \fGGC_{V,n}
\]
is a Maurer-Cartan element, that is,
\[
  \delta_{split}'M_2 + \frac 12 [M_2,M_2] = 0.
\]
\end{lemma}
\begin{proof}
  It is convenient to write 
  \[
M_2 = -
\underbrace{
\begin{tikzpicture}
  \node (v) at (0,-.3) {$1$};
  \node[int] (w) at (0,.4) {};
  \draw (v) edge (w);
\end{tikzpicture}
}_{A}
-
\underbrace{
\sum_{j=1}^N
\begin{tikzpicture}
  \node (v) at (0,-.3) {$e_j$};
  \node[int, label=90:{$f_j$}] (w) at (0,.4) {};
  \draw (v) edge (w);
\end{tikzpicture}
}_{B}\, .
\]
Then we compute 
  \begin{align*}
    \delta_{split}' A &=
    \begin{tikzpicture}[yshift=-.3cm]
      \node (v) at (0,0) {$1$};  
      \node[int] (w) at (0,0.7) {};
      \node[int] (ww) at (0,1.1) {};
      \draw (v) edge (w) (w) edge (ww);  
    \end{tikzpicture}
    &
    \delta_{split}'M_2
    &=
    -\sum_{j=1}^N
    \left(
    \begin{tikzpicture}[yshift=-.3cm]
      \node (v) at (0,0) {$e_j$};  
      \node[int, label=0:{$f_j$}] (w) at (0,0.7) {};
      \node[int] (ww) at (0,1.1) {};
      \draw (v) edge (w) (w) edge (ww);  
    \end{tikzpicture}
+
\begin{tikzpicture}[yshift=-.3cm]
  \node (v) at (0,0) {$e_j$};  
  \node[int] (w) at (0,0.7) {};
  \node[int, label=90:{$f_j$}] (ww) at (0,1.1) {};
  \draw (v) edge (w) (w) edge (ww);  
\end{tikzpicture}
\right)
  \end{align*}
  with the ordering such that edges and vertices are ordered and oriented top down. Furthermore, $\frac 12 [M_2, M_2]
  =
  M_2 * M_2$, and we have
  \begin{align*}
    A* A &=
    \begin{tikzpicture}[yshift=-.3cm]
      \node (v) at (0,0) {$1$};  
      \node[int] (w) at (0,0.7) {};
      \node[int] (ww) at (0,1.1) {};
      \draw (v) edge (w) (w) edge (ww);  
    \end{tikzpicture}
    &
    A* B &=
    \sum_{j=1}^N
    \begin{tikzpicture}[yshift=-.3cm]
      \node (v) at (0,0) {$e_j$};  
      \node[int, label=0:{$f_j$}] (w) at (0,0.7) {};
      \node[int] (ww) at (0,1.1) {};
      \draw (v) edge (w) (w) edge (ww);  
    \end{tikzpicture}
    &
    B * A &= 0
    &
    B* B &=
    \sum_{j=1}^N
    \begin{tikzpicture}[yshift=-.3cm]
      \node (v) at (0,0) {$e_j$};  
      \node[int] (w) at (0,0.7) {};
      \node[int, label=90:{$f_j$}] (ww) at (0,1.1) {};
      \draw (v) edge (w) (w) edge (ww);  
    \end{tikzpicture}.
  \end{align*}
  Hence the Maurer-Cartan equation follows.
\end{proof}

We may hence twist the dg Lie algebra $\fGGC_{V,n}$ by this Maurer-Cartan element $M_2$. 
We define the dg Lie algebra 
\[
\GGC_{V,n} \subset (\fGGC_{V,n})^{M_2}
\]
to be the dg Lie subalgebra generated by graphs all of whose internal vertices have valence $\geq 3$.

\begin{lemma}\label{lem:HGC well defined}
the dg Lie subalgebra $\GGC_{V,n} \subset (\fGGC_{V,n})^{M_2}$ is well-defined, i.e., differentials and Lie brackets of graphs with $\geq 3$-valent internal vertices are linear combinations of graphs with only $\geq 3$-valent internal vertices.
\end{lemma}
\begin{proof}
The Lie bracket is defined as the dual operation to cutting as above. It glues two graphs and cannot reduce the valence of vertices in the process.

However, the differential $\delta_{split}'+[M_2,-]$ contains terms that can potentially create vertices of valence one and two, and we have to check that these terms cancel in pairs.

Let $\Gamma\in \GGC_{V,n}$ be a graph with only $\geq 3$-valent vertices. 
Then the terms in $\delta_{split}'\Gamma$ that are graphs with univalent vertices are all obtained by attaching a univalent vertex to any of the internal vertices.
These are precisely cancelled by the terms of $[M_2,\Gamma]=M_2 * \Gamma - (-1)^{|\Gamma|} \Gamma * M_2$ (with $*$ denoting the Lie admissible product) arising from the summand 
\[
  \begin{tikzpicture}
    \node (v) at (0,-.3) {$e_0$};
    \node[int, label=90:{$ $}] (w) at (0,.4) {};
    \draw (v) edge (w);
  \end{tikzpicture}
  \, * \Gamma.  
\]

Next we consider the potential terms in  $\delta_{split}'\Gamma+[M_2,\Gamma]$ with bivalent vertices.
We distinguish four cases:
\begin{enumerate}
\item Bivalent vertices connected to two internal vertices.
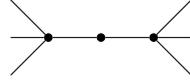
\[
  \begin{tikzpicture}
    \node[int] (v) at (-.7,0) {};
    \node[int] (w) at (0,0) {};
    \node[int] (u) at (.7,0) {};
    \draw (v) edge (w) (w) edge (u) 
    (v) edge +(-.5,0) edge +(-.5,0.5) edge +(-.5,-0.5)
    (u) edge +(.5,0) edge +(.5,0.5) edge +(.5,-0.5);
  \end{tikzpicture}
\]
These terms are produced by $\delta_{split}'$. However, they always come in pairs, from the splitting of the left-hand and right-hand vertex, and the two terms come with opposite signs. For example, for even $n$ we can write explicitly the numbering of the edges and one obtains the cancellation (where we show only the terms corresponding to one edge in $\Gamma$):
\[
  \delta_{split}'
  \begin{tikzpicture}
    \node[int] (v) at (-.3,0) {};
    \node[int] (u) at (.3,0) {};
    \draw (v) edge node[above]{$\scriptstyle j$} (u) 
    (v) edge +(-.5,0) edge +(-.5,0.5) edge +(-.5,-0.5)
    (u) edge +(.5,0) edge +(.5,0.5) edge +(.5,-0.5);
  \end{tikzpicture}
  =
  \underbrace{
  \begin{tikzpicture}
    \node[int] (v) at (-.7,0) {};
    \node[int] (w) at (-.2,0) {};
    \node[int] (u) at (.7,0) {};
    \draw (v) edge node[above]{$\scriptstyle 1$} (w) (w) edge node[above]{$\scriptstyle j+1$} (u) 
    (v) edge +(-.5,0) edge +(-.5,0.5) edge +(-.5,-0.5)
    (u) edge +(.5,0) edge +(.5,0.5) edge +(.5,-0.5);
  \end{tikzpicture}
  +
  \begin{tikzpicture}
    \node[int] (v) at (-.7,0) {};
    \node[int] (w) at (0.2,0) {};
    \node[int] (u) at (.7,0) {};
    \draw (v) edge node[above]{$\scriptstyle j+1$} (w) (w) edge node[above]{$\scriptstyle 1$} (u) 
    (v) edge +(-.5,0) edge +(-.5,0.5) edge +(-.5,-0.5)
    (u) edge +(.5,0) edge +(.5,0.5) edge +(.5,-0.5);
  \end{tikzpicture}
  }_{=0}
  +\dots.
\]
For odd $n$ the cancellation is similar.
\item Bivalent vertices connected to one internal and one external vertex. 
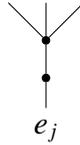
\[
  \begin{tikzpicture}
    \node (v) at (0,-.3) {$e_j$};
    \node[int, label=90:{$ $}] (w1) at (0,.4) {};
    \node[int, label=90:{$ $}] (w2) at (0,.9) {};
    \draw (v) edge (w1) 
    (w2) edge (w1) edge +(-.5,.5) edge +(0,.5) edge +(.5,.5);
  \end{tikzpicture}
\]
Each such term arise twice in the differential, once from splitting the upper internal vertex, and once from the term 
\[
  \Gamma *
  \begin{tikzpicture}
    \node (v) at (0,-.3) {$e_j$};
    \node[int, label=90:{$f_j$}] (w) at (0,.4) {};
    \draw (v) edge (w);
  \end{tikzpicture}
\]
that occurs in $[M_2,\Gamma]$.

\item Bivalent vertices connected to two external vertices.
Again, the cancellation is similar to the previous case.
\item Bivalent vertices connected to one internal or external vertex, and carrying one decoration.
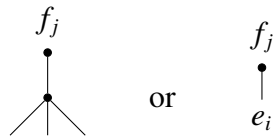
\[
  \begin{tikzpicture}
    \node[int, label=90:{$ $}] (w1) at (0,0) {};
    \node[int, label=90:{$f_j$}] (w2) at (0,.6) {};
    \draw 
    (w1) edge (w2) edge +(-.5,-.5) edge +(0,-.5) edge +(.5,-.5);
  \end{tikzpicture}
  \quad\quad
  \text{or}
  \quad\quad
  \begin{tikzpicture}
    \node (v) at (0,-.3) {$e_i$};
    \node[int, label=90:{$f_j$}] (w) at (0,.4) {};
    \draw (v) edge (w);
  \end{tikzpicture}
\]
These terms are potentially produced by the differential from graphs 

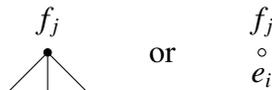
\[
  \begin{tikzpicture}
    \node[int, label=90:{$f_j$}] (w1) at (0,0) {};
    \draw 
    (w1) edge +(-.5,-.5) edge +(0,-.5) edge +(.5,-.5);
  \end{tikzpicture}
  \quad\quad
  \text{or}
  \quad\quad
  \begin{tikzpicture}
    \node[ext, label=90:{$f_j$},label=-90:{$e_i$}] (w) at (0,0) {};
  \end{tikzpicture}
\]
But as before, each is obtained in two ways that cancel, from splitting the vertex and $M_2*\Gamma$ on the left, and from $M_2*\Gamma$ and $\Gamma *M_2$ on the right.
\end{enumerate}

\end{proof}

Let us give an explicit combinatorial description of the resulting differential $\delta=\delta_{split}'+[M_2,-]$ on $\GGC_{V,n}$.  We decompose it into two parts
\[
\delta = \delta_{split} + \delta_{join}.
\]
The piece $\delta_{split}$ acts by splitting internal vertices 
\begin{align*}
  \delta_{split}
  \begin{tikzpicture}[baseline=-.65ex]
    \node[int] (v) {};
    \draw (v) edge +(-.5,.5) edge +(-.2,.5) edge +(.2,.5) edge +(.5,.5);
    \end{tikzpicture}
  &=
  \sum
  \begin{tikzpicture}[baseline=-.65ex]
    \node[int] (v) at (0,0) {};
    \node[int](w) at (0,.5) {};
    \draw (v) edge +(-.5,.5) edge +(.5,.5) edge (w) (w) edge +(-.2,.5) edge +(.2,.5);
    \end{tikzpicture}\, ,
\end{align*}
but in contrast to the similar $\delta_{split}'$ we have only such splittings on the right-hand side that do not produce bivalent or univalent vertices. The bivalent or univalent terms are killed by matching terms of $[M_2,-]$ as seen in the previous proof.
The remaining terms of $[M_2,-]$ make up $\delta_{join}$.
Concretely, $\delta_{join}$ acts by fusing a subset of the $1$-decorated hairs together with zero or one $a\in V$-decorated hair into one new $1$-decorated or $a$-decorated hair.
\begin{align*}
  \delta_{join}
 \begin{tikzpicture}[baseline=-.8ex]
 \node[draw,circle] (v) at (0,.3) {$\Gamma$};
 \node (w1) at (-.7,-.5) {};
 \node (w2) at (-.25,-.5) {};
 \node (w3) at (.25,-.5) {};
 \node (w4) at (.7,-.5) {};
 \draw (v) edge (w1) edge (w2) edge (w3) edge (w4);
 \end{tikzpicture} 
 = 
 \sum_{B} 
 \begin{tikzpicture}[baseline=-.8ex]
 \node[draw,circle] (v) at (0,.3) {$\Gamma$};
 \node (w1) at (-.7,-.5) {};
 \node (w2) at (-.25,-.5) {};
 \node[int] (i) at (.4,-.5) {};
 \node (w4) at (.4,-1.3) {$1$ or $a$};
 \draw (v) edge (w1) edge (w2) edge[bend left] (i) edge (i) edge[bend right] (i) (w4) edge (i);
 \end{tikzpicture} \, .
 \end{align*}

\subsection{The $\La$ Hopf $\Graphs_n$ comodule $\Graphs_{V,n}$}
Summarizing, we have so far defined the following objects:
\begin{itemize}
\item The dg cooperad $\fGraphs_n$ with differential $d_c'$.
The Maurer-Cartan element $-M_1\in (\fGraphs_n(1))^*$ 
then allows us to twist the differential to $d_c=d_c'-M_1\cdot$ (see section \ref{sec:optwist simple}) and pass to the quotient dg cooperad $\Graphs_n$ spanned by $\geq 3$-valent graphs.
\item The dg Lie algebra $\fGGC_{V,n}$ with differential $\delta_{split}'$ dual to $d_c'$.
Using the Maurer-Cartan element $M_2\in \fGGC_{V,n}$ we twisted the differential to $\delta=\delta_{split}'+[M_2,-]$ and could pass to the dg Lie subalgebra $\GGC_{V,n}$ spanned by $\geq 3$-valent graphs.
\item The $\La$ Hopf $\fGraphs_n$-comodule $\fGraphs_{n,V}$ with differential $d_c'$ that comes equipped with a compatible action of the dg Lie algebra $(\fGGC_{V,n},\delta_{split}')$.
\end{itemize}

The Maurer-Cartan element $M_2$ allows us furthermore to twist the action of $\fGGC_{V,n}$ on $\fGraphs_{n,V}$ to an action of $(\fGGC_{V,n})^{M_2}$ on the twisted cooperadic $(\fGraphs_n, d_c')$ comodule $(\fGraphs_{n,V})^{M_2}$.
This dg Lie action of course restricts to the dg Lie subalgebra $\GGC_{V,n}\subset (\fGGC_{V,n})^{M_2}$. 

In addition, we may twist by the Maurer-Cartan element $-M_1$ as in section \ref{sec:optwist simple} so as to obtain 
a $\Graphs_{V,n}$-comodule $(\fGraphs_{V,n})^{-M_1+M_2}$.
More precisely, this twisted module has the same underlying graded Hopf $\La$ sequence, and the same formula for the action and coaction, but differential $d_c'-M_1\cdot +M_2\cdot$.

We next define the quotient 
\[
  \Graphs_{V,n}(r) := (\fGraphs_{V,n})^{-M_1+M_2}(r) / I(r),
\]
where $I(r)$ is spanned by all graphs with connected components without external vertices and by graphs with internal vertices of valency $\leq 2$, with a $V$-decoration counted as contributing 1 to the valency.
Similarly to Lemma \ref{lem:HGC well defined} one checks that the differential descends to the quotient.
The $\La$ Hopf cooperadic $\Graphs_n$ comodule structure also descends to the quotient, since neither the commutative product nor the $\La$ structure nor the $\Graphs_n$ coaction can remove vertices of low valency.
Finally, the action of the dg Lie algebra $\GGC_{V,n}$ also descends to an action on $\Graphs_{V,n}$ compatible with the $\La$ Hopf $\Graphs_n$ comodule structure, since again this action cannot remove lower valent vertices.

Since the object $\Graphs_{V,n}$ plays a central role for this paper let us summarize here its various algebraic structures.
\begin{itemize}
\item The space $\Graphs_{V,n}(r)$ consists of formal linear combinations of graphs with $r$ numbered external and an arbitrary number of internal vertices. All vertices may be decorated by zero, one or more elements of $V$.
All internal vertices must be at least trivalent, with each decoration in $V$ counting +1 to the valency.
To fix signs, graphs come with an orientation, depending on $n$. 
\item Each space $\Graphs_{V,n}(r)$ is a graded commutative algebra, with the product being the union of graphs, identifying the external vertices.
 \begin{equation*}
  \left(
  \begin{tikzpicture}
    \node[ext, label=90:{$\scriptstyle \beta$}] (v1) at (0,0) {$\scriptstyle 1$};
    \node[ext] (v2) at (.7,0) {$\scriptstyle 2$};
    \node[ext] (v3) at (1.4,0) {$\scriptstyle 3$};
    \node[int] (i1) at (.35,.7) {};
    \node[int, label=90:{$\scriptstyle \alpha$}] (i2) at (1.05,.7) {};
    \draw (v1) edge (i1) 
    (i1) edge (i2) edge (v2) 
    (i2) edge (v2) edge (v3)
    (i1) edge[loop above] (i1);
  \end{tikzpicture}
  \right)
  \wedge 
  \left(
  \begin{tikzpicture}
    \node[ext, label=90:{$\scriptstyle \gamma$}] (v1) at (0,0) {$\scriptstyle 1$};
    \node[ext] (v2) at (.7,0) {$\scriptstyle 2$};
    \node[ext] (v3) at (1.4,0) {$\scriptstyle 3$};
    \node[int, white] (i1) at (.35,.7) {};
    \draw 
    (v2) edge[bend right] (v3);
  \end{tikzpicture}
  \right)
  =
  \begin{tikzpicture}
    \node[ext, label=90:{$\scriptstyle \beta\gamma$}] (v1) at (0,0) {$\scriptstyle 1$};
    \node[ext] (v2) at (.7,0) {$\scriptstyle 2$};
    \node[ext] (v3) at (1.4,0) {$\scriptstyle 3$};
    \node[int] (i1) at (.35,.7) {};
    \node[int, label=90:{$\scriptstyle \alpha$}] (i2) at (1.05,.7) {};
    \draw (v1) edge (i1) 
    (i1) edge (i2) edge (v2) 
    (i2) edge (v2) edge (v3)
    (v2) edge[bend right] (v3)
    (i1) edge[loop above] (i1);
  \end{tikzpicture}\, .
\end{equation*}
\item The symmetric group $S_r$ acts on $\Graphs_{V,n}(r)$ by permuting the numbering of the external vertices. There is a $\La$ structure, with the $\La$ operations adding zero-valent external vertices to graphs.
\[
  \mu_4^c : 
  \begin{tikzpicture}
    \node[ext, label=90:{$\scriptstyle \beta$}] (v1) at (0,0) {$\scriptstyle 1$};
    \node[ext] (v2) at (.7,0) {$\scriptstyle 2$};
    \node[ext] (v3) at (1.4,0) {$\scriptstyle 3$};
    \node[int] (i1) at (.35,.7) {};
    \node[int, label=90:{$\scriptstyle \alpha$}] (i2) at (1.05,.7) {};
    \draw (v1) edge (i1) 
    (i1) edge (i2) edge (v2) 
    (i2) edge (v2) edge (v3)
    (i1) edge[loop above] (i1);
  \end{tikzpicture}
  \mapsto
  \begin{tikzpicture}
    \node[ext, label=90:{$\scriptstyle \beta$}] (v1) at (0,0) {$\scriptstyle 1$};
    \node[ext] (v2) at (.7,0) {$\scriptstyle 2$};
    \node[ext] (v3) at (1.4,0) {$\scriptstyle 3$};
    \node[ext] (v4) at (2.1,0) {$\scriptstyle 4$};
    \node[int] (i1) at (.35,.7) {};
    \node[int, label=90:{$\scriptstyle \alpha$}] (i2) at (1.05,.7) {};
    \draw (v1) edge (i1) 
    (i1) edge (i2) edge (v2) 
    (i2) edge (v2) edge (v3)
    (i1) edge[loop above] (i1);
  \end{tikzpicture}
\]
\item There is an operadic right $\Graphs_n$ comodule structure on $\Graphs_{V,n}$ by subgraph contraction, for example:
\[
  \begin{tikzpicture}
    \node[ext, label=90:{$\scriptstyle \alpha$}] (v1) at (0,0) {$\scriptstyle 1$};
    \node[ext, label=90:{$\scriptstyle \beta$}] (v2) at (.7,0) {$\scriptstyle 2$};
    \node[ext] (v3) at (1.4,0) {$\scriptstyle 3$};
    \node[int, white] (i1) at (.35,.7) {};
    \draw (v1) edge (v2)
    (v2) edge (v3);
  \end{tikzpicture}
  \mapsto
  \begin{tikzpicture}
    \node[ext, label=90:{$\scriptstyle \alpha\beta$}] (v1) at (0,0) {$\scriptstyle 1$};
    \node[ext] (v3) at (.7,0) {$\scriptstyle 2$};
    \draw (v1) edge (v3)
    (v1) edge[loop below] (v1);
  \end{tikzpicture}
  \otimes 
  \begin{tikzpicture}
    \node[ext] (v1) at (0,0) {$\scriptstyle 1$};
    \node[ext] (v3) at (.7,0) {$\scriptstyle 2$};
  \end{tikzpicture}
  +
  \begin{tikzpicture}
    \node[ext, label=90:{$\scriptstyle \alpha\beta$}] (v1) at (0,0) {$\scriptstyle 1$};
    \node[ext] (v3) at (.7,0) {$\scriptstyle 2$};
    \draw (v1) edge (v3);
  \end{tikzpicture}
  \otimes 
  \begin{tikzpicture}
    \node[ext] (v1) at (0,0) {$\scriptstyle 1$};
    \node[ext] (v3) at (.7,0) {$\scriptstyle 2$};
    \draw (v1) edge (v3);
  \end{tikzpicture}
\]
By corestriction along the canonical quasi-isomorphism $\Graphs_n\to \Poiss_n^c$ we also obtain a $\Poiss_n^c$ comodule structure on $\Graphs_{V,n}$, and hence an $\e_n^c$-comodule structure for $n\geq 2$.
\item The differential $d_c$ on $\Graphs_{V,n}$ acts by contracting an edge between a pair of distinct vertices, at least one of which must be internal.
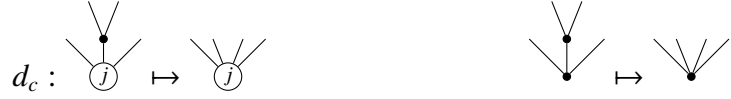
\begin{align*}
  d_c:
  \begin{tikzpicture}[baseline=-.65ex]
  \node[ext] (v) at (0,0) {$\scriptstyle j$};
  \node[int](w) at (0,.5) {};
  \draw (v) edge +(-.5,.5) edge +(.5,.5) edge (w) (w) edge +(-.2,.5) edge +(.2,.5);
  \end{tikzpicture}
  &\mapsto
  \begin{tikzpicture}[baseline=-.65ex]
  \node[ext] (v) {$\scriptstyle j$};
  \draw (v) edge +(-.5,.5) edge +(-.2,.5) edge +(.2,.5) edge +(.5,.5);
  \end{tikzpicture}
  &
  \begin{tikzpicture}[baseline=-.65ex]
  \node[int] (v) at (0,0) {};
  \node[int](w) at (0,.5) {};
  \draw (v) edge +(-.5,.5) edge +(.5,.5) edge (w) (w) edge +(-.2,.5) edge +(.2,.5);
  \end{tikzpicture}
  &\mapsto
  \begin{tikzpicture}[baseline=-.65ex]
  \node[int] (v) {};
  \draw (v) edge +(-.5,.5) edge +(-.2,.5) edge +(.2,.5) edge +(.5,.5);
  \end{tikzpicture}
  \end{align*}
  In particular, note that the more cumbersome description of the differential above as a sum of three terms yields a combinatorially relatively simple formula. However, our approach has the advantage that automatically $d_c$ squares to zero and is compatible with all other algebraic structure.
  \item There is an action of the dg Lie algebra $\GGC_{V,n}$ on $\Graphs_{V,n}$.
  The action of $\gamma\in \GGC_{V,n}$ on $\Gamma\in \Graphs_{V,n}$ is defined combinatorially by cutting off a subgraph
  from $\Gamma$ (which then lives in $\GG_{V,n}$) and then pairing that subgraph with $\gamma\in \GGC_{V,n}= \GG_{V,n}^*$. 
  For example, the action of 
  \begin{align*}
    \gamma&=\begin{tikzpicture}
      \node (v1) at (0,0) {$\scriptstyle 1$};
      \node (v2) at (.7,0) {$\scriptstyle e_1$};
      \draw (v1) edge (v2);
    \end{tikzpicture}
    \in \GGC_{V,n}
& &\text{on} &
\Gamma &=
\begin{tikzpicture}
  \node[ext] (v1) at (0,0) {$\scriptstyle 1$};
  \node[ext] (v2) at (.7,0) {$\scriptstyle 2$};
  \node[ext] (v3) at (1.4,0) {$\scriptstyle 3$};
  \draw (v1) edge (v2);
\end{tikzpicture}
\in \Graphs_{V,n}(3)
  \end{align*}
  is computed as follows.
  \begin{align*}
    \begin{tikzpicture}
      \node[ext] (v1) at (0,0) {$\scriptstyle 1$};
      \node[ext] (v2) at (.7,0) {$\scriptstyle 2$};
      \node[ext] (v3) at (1.4,0) {$\scriptstyle 3$};
      \draw (v1) edge (v2);
    \end{tikzpicture}
    & \xrightarrow{\text{cut}}
    \begin{tikzpicture}
      \node[ext] (v1) at (0,0) {$\scriptstyle 1$};
      \node[ext] (v2) at (.7,0) {$\scriptstyle 2$};
      \node[ext] (v3) at (1.4,0) {$\scriptstyle 3$};
      \draw (v1) edge[out=80, in=105] (v2);
      \draw[dashed] (-.2,.25) -- (1.6,.25); 
    \end{tikzpicture}
    =
    \sum_{i,j=0}^N 
    \begin{tikzpicture}
      \node[ext,label=90:{$\scriptstyle e_i$}] (v1) at (0,0) {$\scriptstyle 1$};
      \node[ext,label=90:{$\scriptstyle e_j$}] (v2) at (.7,0) {$\scriptstyle 2$};
      \node[ext] (v3) at (1.4,0) {$\scriptstyle 3$};
    \end{tikzpicture}
    \otimes 
    \begin{tikzpicture}
      \node (v1) at (0,0) {$\scriptstyle f_i$};
      \node (v2) at (.7,0) {$\scriptstyle f_j$};
      \draw (v1) edge (v2);
    \end{tikzpicture}
    \in 
    \Graphs_{V,n}(3) \otimes \GG_{V,n}
    \\
    &\xrightarrow{\text{pair with $\gamma$}
    }
    \begin{tikzpicture}
      \node[ext,label=90:{$\scriptstyle e_1$}] (v1) at (0,0) {$\scriptstyle 1$};
      \node[ext] (v2) at (.7,0) {$\scriptstyle 2$};
      \node[ext] (v3) at (1.4,0) {$\scriptstyle 3$};
    \end{tikzpicture}
    \pm
    \begin{tikzpicture}
      \node[ext] (v1) at (0,0) {$\scriptstyle 1$};
      \node[ext,label=90:{$\scriptstyle e_1$}] (v2) at (.7,0) {$\scriptstyle 2$};
      \node[ext] (v3) at (1.4,0) {$\scriptstyle 3$};
    \end{tikzpicture}
  \end{align*}

  \item The algebraic structures above are all compatible with each other.
\end{itemize}

  


\subsection{Quasi-freeness of $\Graphs_{V,n}$ as $\La$ Hopf $\Com^c$ comodule}\label{sec:Graphs freeness}
The main reason why the comodule $\Graphs_{V,n}$ is convenient to use is that it is free as a graded $\La$ Hopf $\Com^c$ comodule, as we shall explain in this subsection.
First, let 
\[
  \IG_{V,n}(r) \subset \Graphs_{V,n}(r)
\] 
be the subspace spanned by internally connected graphs, that is, graphs that remain connected after deleting all external vertices.
For example, in the following list the first two graphs are internally connected, the last one is not.
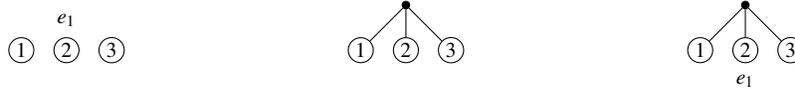
\begin{align*}
&\begin{tikzpicture}
  \node[ext] (v1) at (0,0) {$\scriptstyle 1$};
  \node[ext,label=90:{$\scriptstyle e_1$}] (v2) at (.6,0) {$\scriptstyle 2$};
  \node[ext] (v3) at (1.2,0) {$\scriptstyle 3$};
\end{tikzpicture}
&
&
\begin{tikzpicture}
  \node[int] (w) at (.6,.6) {};
  \node[ext] (v1) at (0,0) {$\scriptstyle 1$};
  \node[ext] (v2) at (.6,0) {$\scriptstyle 2$};
  \node[ext] (v3) at (1.2,0) {$\scriptstyle 3$};
\draw (w) edge (v1) edge (v2) edge (v3);
\end{tikzpicture}
&
&
\begin{tikzpicture}
  \node[int] (w) at (.6,.6) {};
  \node[ext] (v1) at (0,0) {$\scriptstyle 1$};
  \node[ext,label=-90:{$\scriptstyle e_1$}] (v2) at (.6,0) {$\scriptstyle 2$};
  \node[ext] (v3) at (1.2,0) {$\scriptstyle 3$};
\draw (w) edge (v1) edge (v2) edge (v3);
\end{tikzpicture}
\end{align*}
Every graph splits uniquely as a union of internally connected components. For example, the third of the three graphs above is the union (identifying the external vertices) of the previous two.
Since the union is by definition the product on $\Graphs_{V,n}(r)$,
we conclude that $\Graphs_{V,n}(r)$ is free as a graded commutative algebra,
\[
  \Graphs_{V,n}(r)=S(\IG_{V,n}(r)),
\]
generated by the internally connected graphs. 
Next, let 
\[
  \ppIG_{V,n}(r)\subset \IG_{V,n}(r)   
\]
be spanned by the subspace spanned by those graphs all of whose external vertices have valency one. For example, the second graph of the three above is in $\ppIG_{V,n}(3)$.

Then we observe that  collection of internally connected graphs $\IG_{V,n}$ furthermore is free as a right $\La$ $\Com^c$-comodule, generated by the sub-collection $\ppIG_{V,n}$.
Indeed, let $\Gamma\in \IG_{V,n}$ be some graph.
Then it can be obtained from a graph $\Gamma'\in\ppIG_{V,n}$ by the following operations:
\begin{itemize}
\item Adding zero-valent external vertices to $\Gamma'$ for every zero-valent external vertex in $\Gamma$ using the $\La$-structure.
\item Fusing subsets of external vertices of $\Gamma'$ for every external vertex in $\Gamma$ of valency $\geq 2$, using the $\Com^c$-coaction.
\item Renumbering external vertices using the symmetric group action if necessary.
\end{itemize}
The following picture illustrates the process:
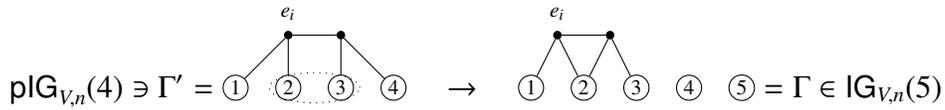
\[
  \ppIG_{V,n}(4)
  \ni
  \Gamma'=
  \begin{tikzpicture}
    \node[ext] (v1) at (0,0) {$\scriptstyle 1$};
    \node[ext] (v2) at (.7,0) {$\scriptstyle 2$};
    \node[ext,] (v3) at (1.4,0) {$\scriptstyle 3$};
    \node[ext,] (v4) at (2.1,0) {$\scriptstyle 4$};
    \node[int,label=90:{$\scriptstyle e_i$}] (i1) at (.7,.7) {};
    \node[int] (i2) at (1.4,.7) {};
    \draw (v1) edge (i1) 
    (i1) edge (i2) edge (v2) 
    (i2) edge (v3) edge (v4);
    \draw [dotted] (1.05,0) ellipse (.6 and .2);
  \end{tikzpicture}
\quad\to\quad
  \begin{tikzpicture}
    \node[ext] (v1) at (0,0) {$\scriptstyle 1$};
    \node[ext] (v2) at (.7,0) {$\scriptstyle 2$};
    \node[ext,] (v3) at (1.4,0) {$\scriptstyle 3$};
    \node[ext,] (v4) at (2.1,0) {$\scriptstyle 4$};
    \node[ext,] (v5) at (2.8,0) {$\scriptstyle 5$};
    \node[int,label=90:{$\scriptstyle e_i$}] (i1) at (.35,.7) {};
    \node[int] (i2) at (1.05,.7) {};
    \draw (v1) edge (i1) 
    (i1) edge (i2) edge (v2) 
    (i2) edge (v2) edge (v3);
  \end{tikzpicture}
  =
  \Gamma\in \IG_{V,n}(5)
\]
Here the dotted ellipse indicates the vertices that are fused by a $\Com^c$-coaction. Vertices $4$ and $5$ in $\Gamma$ are added by the $\La$ operations.
One observes that $\Gamma$ determines the graph $\Gamma'$ and the $\La$ Hopf $\Com^c$ coaction needed to obtain $\Gamma$ from $\Gamma'$ uniquely, up to renumbering vertices.
Hence we conclude that $\Graphs_{V,n}$ is free as a graded $\La$ Hopf $\Com^c$ comodule, generated by $\ppIG_{V,n}$.
We also refer to Proposition \ref{prop:restr 1} and its proof below for a more formal formulation of this freeness statement.

\subsection{The filtration $\mF$}
\label{sec:FFiltration}
There are two filtrations we will consider on the comodule $\Graphs_{V,n}$ and on the dg Lie algebra $\GGC_{V,n}$.
Recall that $V$ is a finite dimensional positively graded vector space by assumption. We set $|V|$ to be the maximum degree of a nonzero element in $V$, or $|V|:=1$ if $V=0$.

Then we consider the increasing exhaustive filtration by $3|V|$ times the number of edges plus the total degree of decorations in $V$, that is 
\[
\mF^p \Graphs_{V,n} 
\]
is spanned by all graphs $\Gamma$ such that 
\[
3 |V| \# \text{edges} + \# \text{total decoration degree} \leq p.  
\]
For example, $\mF^0\Graphs_{V,n}\cong \Com^c$.
Furthermore it is clear that every graph is contained in some $\mF^p\Graphs_{V,n}$ by finiteness of the graph and hence the filtration is exhaustive. Finally, it is easy to see that the filtration is compatible with the dg $\La$ Hopf comodule structure. More precisely, the product preserves the number of edges and the decorations and hence
\[
\mF^p \Graphs_{V,n} 
\wedge 
\mF^q \Graphs_{V,n} \subset \mF^{p+q} \Graphs_{V,n}.
\]
The same holds for the $\La$ structure. The differential and $\Poiss_n^c$-coaction at worst reduce the number of edges, and are hence also compatible with the filtration.
Finally, note that there are only finitely many graphs $\Gamma$ with a given number of edges and decorations.
Hence each subspace $\mF^p \Graphs_{V,n}(r)$ is finite dimensional.

We similarly consider the dg Lie coalgebra $\GG_{V,n}$, and impose on it the filtration $\mF$ by $3|V|$ times the number of edges plus decoration degree, minus the decoration degree at hairs.
More precisely, a graph $\Gamma$ is in $\mF^p\GG_{V,n}$ if 
\[
  3 \# \text{edges} + \# \text{total decoration degree} \leq p,
\]
where we compute the total decoration degree such that $V$-decorations are counted positively and $V^*$-decorations of negatively.
(This is consistent with the convention that dualization flips the sign of the cohomological degrees.) 
As before, the filtration is compatible with the differential.
Furthermore, the filtration is compatible with the Lie cobracket, and the coaction of $\GG_{V.n}$ on $\Graphs_{V,n}$ - these operations preserve the number of edges and the total decoration degree.

We furthermore impose the dual filtration on the dual space $\GGC_{V,n}$ such that 
\[
\mF^p \GGC_{V,n}
=\{
x\in \GGC_{V,n}=\HG_{V,n}^* \mid 
x(\mF^{p-1}\GG_{V,n})=0 
\}.
\]
Note in particular that $\mF^{-|V|}\GGC_{V,n}=\GGC_{V,n}$.
Also note that $\GGC_{V,n}$ is not pro-nilpotent, though $\mF^1 \GGC_{V,n}$ is. 

Collecting the observations above together with the freeness result of the previous subsection we have shown the following Lemma.
\begin{lemma}
The operadic right $\COp=\Graphs_n$ $\La$ Hopf comodule $\MOp =\Graphs_{V,n}$ with the filtration $\mF$ above satisfies the conditions b and c of section \ref{subsec:def cx}.

Furthermore, the action of $\fg=\mF^1\GGC_{V,n}$ on $\MOp=\Graphs_{V,n}$ satisfies the conditions of section \ref{sec:dg Lie action}, namely:
\begin{itemize}
\item $\mF^1\GGC_{V,n}$ acts on the operadic right $\COp=\Graphs_n$ $\La$ Hopf comodule $\Graphs_{V,n}$ by biderivations, i.e., equations \eqref{equ:bider relations} hold.
\item The action respects the filtrations in that \eqref{equ:action filtr compat} holds.
\end{itemize}
\end{lemma}

\subsection{The filtration $\mG$}\label{sec:Gfiltration}
We similarly consider the filtration $\mG$ on $\Graphs_{V,n}$ by the number of edges. Concretely, $\mG^p \Graphs_{V,n}$
is spanned by all graphs $\Gamma$ such that 
\[
\# \text{edges} \leq p.  
\]
Similarly, consider the analogous filtration on $\GG_{V,n}$ such that $\mG^p \GG_{V,n}$ is spanned by hairy graphs with at most $p$ edges.
The filtrations are compatible with the $\La$ Hopf comodule structure on $\Graphs_{V,n}$, with the dg Lie coalgebra structure on $\GG_{V,n}$, and with the coaction of $\GG_{V,n}$ on $\Graphs_{V,n}$. (The argument is a slight simplification of that of the previous subsection.)
The spaces $\mG^p\Graphs_{V,n}$ are not finite dimensional in general.

Again we dualize the ascending exhaustive filtration $\mG$ on $\GG_{V,n}$ to a descending complete filtration on $\GGC_{V,n}$.
We have 
\[
  \GGC_{V,n} / \mG^1\GGC_{V,n} \cong \gl_V' := \iHom(V, V_1). 
\]


Furthermore, note that
\[
\mG^0 \Graphs_{n,V} \cong \Fc_{S(V)},
\]
with $\Fc_{S(V)}$ defined as in section \ref{sec:Fc construction} above.
For later use we shall record the following technical results.
\begin{lemma}\label{lem:grG0 map pre}
The canonical inclusion $\FF_{S(V)}\to \Graphs_{n,V}^\flat$ has a one-sided inverse 
in the category of graded $\La$ Hopf $\Com^c$-comodules
\[
\tilde F:  \Graphs_{n,V}^\flat \to \Fc_{S(V)}.
\]
This morphism is defined by sending all graphs to zero that contain at least one edge, and is the identity map on $\mG^0\Graphs_{n,V} = \Fc_{S(V)}$.
\end{lemma}
\begin{proof}
  This is clear.
\end{proof}

\begin{lemma}\label{lem:grG0 map}
Let $\MOp$ be a $\La$ Hopf $\e_n^c$ comodule and $V$ a positively graded vector space. Let $\phi:V\to \MOp(1)$ be a morphism of dg vector spaces. It induces a morphism of dg commutative algebras $\phi:S(V)\to \MOp(1)$.
Then there is a morphism of $\La$ Hopf $(W\e_n^c)^\flat$ comodules
\[
F: \Graphs_{n,V}^\flat\to W\MOp^\flat
\]
that agrees on $\mG^0\Graphs_{n,V} = \FF_{S(V)}$ with the canonical morphism 
\[
  \FF_{S(V)} \to  W\MOp
\]
of $\La$ Hopf $W\e_n^c$ comodules that is obtained from the adjunction of Proposition \ref{prop:freeness as C mod} applied to the given morphism of dg commutative algebras $\phi:S(V)\to \MOp(1)=W\MOp(1)$. In particular, $F$ respects the differentials when restricted to $\mG^0\Graphs_{V,n}=\FF_{S(V)}$.
\end{lemma}
\begin{proof}
  By Lemma \ref{lem:freecofree} there is a unique $F$ as above such that $\pi\circ F \circ \iota=\pi\circ \tilde F \circ \iota$ where $\tilde F$ is the morphism of Lemma \ref{lem:grG0 map pre}, and $\iota$ and $\pi$ are the inclusion of generators and projection to cogenerators respectively, as in Lemma \ref{lem:freecofree}.
  Note here that $F\neq \tilde F$, since the map $\tilde F$ does not respect the $W\e_n^c$ coaction. (Though it does respect this coaction on $\mG^0\Graphs_{n,V}$.) 
  

  Next, using the freeness result Proposition \ref{prop:freeness as C mod} we see that the restriction of $F$ on $\FF_{S(V)}$ is uniquely determined by the restriction of $F$ to the arity one piece $S(V)$, and by freeness of the latter by the restriction to $V$. But by construction we send $V$ to cocycles in $W\MOp(1)=\MOp(1)$, hence the map $S(V)\to W\MOp(1)$ is compatible with the differentials. Hence so is $F$ when restricted to $\mG^0\Graphs_{V,n}=\FF_{S(V)}$ by Proposition \ref{prop:freeness as C mod}.
\end{proof}

\subsection{The edge spectral sequence, and configuration space type}
Let $Z\in \mG^1\GGC_{V,n}$ be a Maurer-Cartan element. There we study the spectral sequence on $\Graphs_{V,n}^Z$ arising from the above filtration $\mG^p\Graphs_{V,n}^Z$ by the number of edges.
We first note that the $Z$ consists of graphs with at least one edge,
\[
Z=Z_{1} + Z_{\geq 2},
\]
where $Z_{1}$ is a linear combination of graphs with exactly one edge, and $Z_{\geq 2}$ is a linear combination of graphs with two or more edges. Concretely, $Z_1$ has the following schematic form
\begin{equation}\label{equ:Z1AB}
Z_1 
=
\underbrace{\sum (\text{const})
\begin{tikzpicture}[baseline=-.65ex]
\node[int,label=90:{$\alpha_1\cdots \alpha_k$}] (v) at (0,0.5) {};
\node (w) at (0,-.5){$\beta$};
\draw (v) edge (w); 
\end{tikzpicture}}_{A}
+
\underbrace{\sum (\text{const})
\begin{tikzpicture}[baseline=-.65ex]
  \node (v) at (0,0) {$\alpha$};
  \node (w) at (1.3,0) {$\beta$};
  \draw (v) edge (w);
\end{tikzpicture}}_{B}
\end{equation}

Also note that the Lie bracket on $\GGC_{V,n}$ preserves the number of edges, and the differential creates exactly one edge.
It follows from these observations that the element $Z_1$ is itself a Maurer-Cartan element,
\[
dZ_1 +\frac 1 2 [Z_1,Z_1]=0.
\]
One can then easily check that the Maurer-Cartan equation for $Z_1$ states that the 1-vertex part $A$ encodes a commutative product on $V_1$, and the zero-vertex part $B$ encodes a diagonal element $\Delta\in V_1\otimes V_1$ for this product, see Definition \ref{def:diagonal element alg}.
To make the dependence of the diagonal element on $Z$ explicit we will also denote it by $\Delta(Z):=\Delta$ below.

The twisted $\La$ Hopf right $\e_n^c$ comodule $\Graphs_{V,n}^{Z_1}$ agrees (at least as symmetric sequence) with the $E^1$ page of the spectral sequence associated to the edge filtration on $\Graphs_{V,n}^Z$.
One has a map of right $\La$ Hopf $\e_n^c$ comodules
\[
  \Graphs_{V,n}^{Z_1} \to \LS_{V_1,n}^\Delta,
\]
that sends all graphs with internal vertices to zero, an edge between external vertices $i$ and $j$ to the generator $\omega_{ij}$ and a decoration $v$ at vertex $j$ to the generator $v_j$.

One then has the following result.
\begin{prop}\label{prop:Graphs Z1 LS qiso}
The map
\begin{equation}\label{equ:prop Z1}
\Graphs_{V,n}^{Z_1}\to \LS_{V_1}^\Delta
\end{equation}
is a quasi-isomorphism.
\end{prop}
\begin{proof}
First note that the differential on the left-hand side has three terms
\[
d = d_{c} + A\cdot +B\cdot,
\]
where $A,B\in \GGC_{V,n}$ are the to parts of the MC element $Z_1=A+B$ as in \eqref{equ:Z1AB} and $d_c$ acts by contracting an edge.
The terms $d_{c}$ and $A\cdot$ each remove one internal vertex and one edge, while $B\cdot$ removes one edge and leaves the number of internal vertices constant.
In particular, the differential is homogeneous of degree $-1$ with respect to the number of vertices, and both sides of \eqref{equ:prop Z1} have a grading by the number 
\[
(\text{cohomological degree}) +\#\text{edges}.
\]
It suffices to consider the subcomplexes for which this grading is fixed.
On this subcomplexes we further consider the bounded filtration by the number 
\[
p(\Gamma) = \#\text{edges}-\#(\text{internal vertices}).
\]
On the associated graded complexes the part $B\cdot$ of the differential is not visible, and the statement of the proposition becomes equivalent to the statement that the map 
\[
  \Graphs_{V,n}^{A}\to \LS_{V_1,n}^0
\]
is a quasi-isomorphism. Also mind that the right-hand side has zero differential, so practically we have to show that the cohomology of the left-hand side agrees with the right-hand side.
This is done by an induction on the arity, following similar proofs in \cite{LVformal, CW}.
We suppose that we know the statement in arity $r-1$. Then we split
\[
  \Graphs_{V,n}^{A}(r)
  =
  \begin{tikzcd}[column sep=.5em]
    W_0 \ar[loop below,looseness=8]{}
    & \oplus & 
    W_1  \ar[loop below,looseness=8]{}
    \ar[bend left]{rr}
    & \oplus & 
    W_{\geq 2} 
    \ar[loop below,looseness=8]{}
  \end{tikzcd},
\]
where $W_0\cong \Graphs_{V,n}^A(r-1)$ is the subcomplex spanned by graphs in which the external vertex 1 has valency zero, $W_1$ is the subspace spanned by graphs with vertex one univalent, and $W_{\geq 2}$ spanned by graphs with vertex one at least bivalent.
(Here each $V$ decoration is counted as contributing +1 to the valency.)
The arrows indicate pieces of the differential. In particular, the piece of the differential 
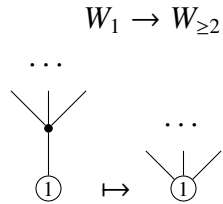
\begin{align*}
  W_1 \to W_{\geq 2} \\
  \begin{tikzpicture}[baseline=-.65ex]
    \node[ext] (v) at (0,0) {$\scriptstyle 1$};
    \node[int] (w) at (0,.8) {};
    \draw (w) edge (v) edge +(-.5,.5) edge +(0,.5) edge +(.5,.5);
    \node at (0, 1.6) {$\cdots$};
    \end{tikzpicture}
    \, \mapsto \,
    \begin{tikzpicture}[baseline=-.65ex]
      \node[ext] (v) at (0,0) {$\scriptstyle 1$};
      \draw  (v) edge +(-.5,.5) edge +(0,.5) edge +(.5,.5);
      \node at (0, .8) {$\cdots$};
    \end{tikzpicture}
\end{align*}

contracts the edge at vertex 1 if possible, and can easily be checked to be a surjective map. Its kernel is spanned by graphs in which vertex 1 has either (i) a single $V$-decoration or (ii) an edge to another external vertex.
From these observations and the induction it is easy to check that $H(\Graphs_{V,n}^ A)$ has a basis spanned by graphs such that the following holds.
\begin{itemize}
\item There are no internal vertices.
\item Vertex $j$ is either $V$-decorated or has at most one edge connecting it to a vertex with higher index, i.e., to one of the vertices $j+1,j+2,\dots,n$. 
\end{itemize}
It is well known that such graphs also form a basis for the right-hand side of \eqref{equ:prop Z1}, see e.g. \cite{SinhaLD}.
To see this directly one can also realize that $\LS_{V,n}^0$ is the quotient of $\Graphs_{V,n}^A$ by the subspace of all graphs with at least one internal vertex, and differentials of graphs with precisely one internal vertex. Hence it in fact suffices to show that each cohomology class in $H(\Graphs_{V,n}^A)$ has a representative without internal vertices, and that we just showed.
\end{proof}

Since $\LS_{V_1,n}^\Delta$ is of configuration space type, we obtain the following corollary.
\begin{cor}\label{cor:GraphsZ1 config space type}
The $\La$ Hopf $\e_n^c$ comodule $\Graphs_{V,n}^{Z_1}$ is of configuration space type.
\end{cor}

\begin{cor}\label{cor:GraphsZ config space type}
For any Maurer-Cartan element $Z\in \mG^1 \HGC_{V,n}$ the $\La$ Hopf $\e_n^c$ comodule $\Graphs_{V,n}^{Z}$ is of configuration space type. 
Furthermore, the canonical map of dg vector spaces
\begin{equation}\label{equ:V1 to Graphs}
 V_1\to \Graphs_{V,n}^{Z}(1)
\end{equation}
is a weak equivalence, and the diagonal class of $\Graphs_{V,n}^Z$ in the sense of Definition \ref{def:diagonal element} is $\Delta(Z)\in V_1\otimes V_1$ as defined above.
\end{cor}
\begin{proof}
By Proposition \ref{prop:Cobar Lie criterion} we need to check that the morphism of symmetric sequences
\begin{equation}\label{equ:GraphsZ cs type}
  \FF_{\Graphs_{V,n}^{Z}(1)} \to \Bar^c_{\Lie_n^c}\Graphs_{V,n}^{Z}
\end{equation}
is a quasi-isomorphism. We take on both sides the filtrations by the number of edges. 
The morphism induce on the $E^1$ pages with respect to this filtration is identified with the map
\[
  \FF_{\Graphs_{V,n}^{Z_1}(1)} \to \Bar^c_{\Lie_n^c}\Graphs_{V,n}^{Z_1},
\]
and this is a quasi-isomorphism by Corollary \ref{cor:GraphsZ1 config space type}. Hence by standard results on spectral sequences the map \eqref{equ:GraphsZ cs type} is a quasi-isomorphism as well.

Next, the fact that \eqref{equ:V1 to Graphs} is a quasi-isomorphism follows from the explicit computation of the cohomology of the right-hand side above.

Finally, to compute the diagonal element recall definition \ref{def:diagonal element}.
We start with the cobracket generator 
\[
  b \in H^{n-1}(\e_n^c(2)), 
\]
and obtain the element
\[
1\otimes b\in H^{n-1}(\MOp(1)\otimes \e_n^c(2)) .
\]
To compute the value of the connecting homomorphism, we need to find a preimage under the cocomposition morphism, for which we may take the graph 
\[
L=
\begin{tikzpicture}
  \node[ext] (v) at (0,0) {$\scriptstyle 1$};
  \node[ext] (w) at (1,0) {$\scriptstyle 2$};
  \draw (v) edge (w);
\end{tikzpicture}
\in \Graphs_{V,n}^Z(2).
\]
Its differential is 
\[
\sum
\begin{tikzpicture}
  \node[ext,label=90:{$\Delta'$}] (v) at (0,0) {$\scriptstyle 1$};
  \node[ext,label=90:{$\Delta''$}] (w) at (1,0) {$\scriptstyle 2$};
\end{tikzpicture}
\in \Graphs_{V,n}^Z(2).
\]
This is the image of the element $\Delta$ above under the inclusion $\FF_M(2)\to \Graphs_{V,n}^Z(2)$, and hence $\Delta$ is indeed the diagonal element of $\Graphs_{V,n}^Z$ as desired.
\end{proof}

\subsection{Digression: Homotopy commutative algebras}\label{sec:hcom algebras}
We will use two models for homotopy commutative algebras.
First, let $\Com_\infty=\Bar^c \Com^{\vee}$ be the minimal resolution of the commutative operad $\Com$, with $\Com^{\vee}=\Lie_0^c$ the Koszul dual cooperad. Also let $\hCom_\infty=\Bar^c\Bar\Com$
be the bar cobar resolution of $\Com$.
Both operad $\Com_\infty$ and $\hCom_\infty$ inherit natural $\La$-operad structures from $\Com$ and we have canonical quasi-isomorphisms of $\La$ operads
\[
  \Com_\infty \to \hCom_\infty \to \Com.
\]

Let $V$ be a dg vector space. Then a $\Com_\infty$- (respectively $\hCom_\infty$-)structure on $V$ is a left action of $\Com_\infty$ (respectively $\hCom_\infty$) on $V$.

A $\Com_\infty$- or $\hCom_\infty$-structure on the dg vector space $V_1$ equipped with the canonical cocycle $1\in V_1$ is (strongly) unital, if the left action extends to a left action of $(\Com_\infty)_1$ or $(\hCom_\infty)_1$ (see \eqref{equ:Pstar def}) by sending the nullary operation to $1$.
Concretely, for a $\Com_\infty$-structure this means that $1$ is the unit with respect to the commutative product, and inserting $1$ into any slot of a higher $\Com_\infty$-homotopy produces zero. For a $\hCom_\infty$-structure the higher homotopy do not 
necessarily vanish upon insertion of $1\in V_1$, but one 
has compatibility relations determined by the $\La$ structure.



We shall also note that a $\Com_\infty$-structure (resp. $\hCom_\infty$-structure) on $V$ is the same data as a Maurer-Cartan element in the convolution dg Lie algebra 
\begin{align*}
\fh_V &:= \Hom_{\bbS}(\overline{\Com^\vee}, \End_V) & &\text{or} &
\hat \fh_V &:= \Hom_{\bbS}(\overline{\Bar \Com}, \End_V),
\end{align*}
see \cite[section 10.1]{LV}.
Furthermore, note that elements of $(\Bar \Com(r))^*$ can be considered as linear combinations of rooted trees with $r$ leaves, and elements of $(\Com^\vee(r)^*)$ similarly as linear combinations of Lie trees, i.e., rooted binary trees modulo the Jacobi (IHX) relations.
Hence, for finite dimensional $V$, elements of $\fh_V$ (resp. $\hat \fh_V$) may be considered as linear combinations of 
such trees with leaves decorated by elements of $V^*$, and the root decorated by an element of $V$.
There is a natural map of dg Lie algebras $\hat \fh_V\to \fh_V$ reflecting the map of operads $\Com_\infty\to \hCom_\infty$.

Unital $\Com_\infty$- and $\hCom_\infty$-structures on $V_1=V\oplus\Q 1$ can also be encoded as Maurer-Cartan elements in suitable dg Lie algebras $\fh_{V,1}$ and $\hat \fh_{V,1}$ respectively.
To describe these dg Lie algebras first consider the Maurer-Cartan element $m_0 \in \hat \fh_{V_1}$ (respectively $m_0 \in \fh_{V_1}$) that corresponds to the "minimal" commutative algebra structure for which $1$ is the unit. Concretely, in this commutative algebra structure the product with $1$ is the identity, but all other products vanish.
Then $\fh_{V,1}\subset \fh_{V_1}^{m_0}$ and $\hat \fh_{V,1}\subset \hat \fh_{V_1}^{m_0}$ are dg Lie subalgebras of the non-unital versions, twisted by $m_0$.
Assuming that $V$ is finite dimensional for simplicity, 
the subalgebra $\fh_{V,1}^{m_0}$ is simply the subspace of linear combinations of trees with legs decorated by $V^*\cong 1^\perp \subset V_1^*$.
Similarly, $\hat \fh_{V,1}$ can also be identified with linear combinations of trees with legs decorated by $V^*$, However, the embedding $\hat \fh_{V,1}\to \hat \fh_{V_1}^{m_0}$ is given by summing over all ways of adding augmentations in $V_1^*$ to vertices.

We shall also note that by convention we consider here the differential as data associated to the dg vector space $V$, and not as part of the homotopy commutative structure.


\subsection{The complexity filtration and the homotopy type of the unary part}
\label{sec:complexity filtration}

Let $\Gamma\in \GGC_{V,n}$ be a graph with $e$ edges and $v$ vertices. Then we say that the \emph{complexity} of $\Gamma$ is the number $e-v$. Alternatively, the complexity is the number of loops in the graph obtained fro $\Gamma$ by gluing together all hairs. 
The complexity is preserved by the differential and additive under the Lie bracket and hence defines a complete grading on the dg Lie algebra $\GGC_{V,n}$. 
We shall denote the corresponding descending complete filtration by $\cC^\bullet \GGC_{V,n}$, that is, $\cC^p \GGC_{V,n}\subset \GGC_{V,n}$ is the subspace generated by graphs of complexity $\geq p$.

\newcommand{\tree}{{\mathrm{tree}}}
Let us study the quotient dg Lie algebra 
\[
  \GGC_{V,n}^\tree := \GGC_{V,n} / \cC^1 \GGC_{V,n}.
\]
Elements are series of tree graphs with exactly one hair, like the following
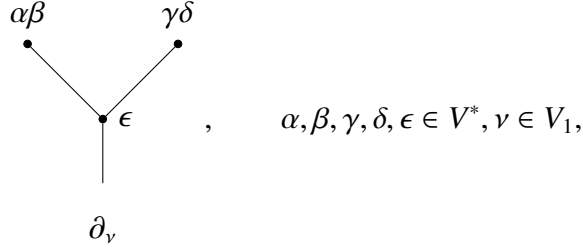
\[
  \begin{tikzpicture}
    \node[int,label=90:{$\alpha\beta$}] (v1) at (-1,1) {};
    \node[int,label=90:{$\gamma\delta$}] (v2) at (1,1) {};
    \node[int,label=0:{$\epsilon$}] (v3) at (0,0) {};
    \node[label=-90:{$\partial_{\nu}$}] (x) at (0,-1) {};
    \draw (v3) edge (v1) edge (v2) edge (x); 
  \end{tikzpicture},\quad\quad \alpha,\beta,\gamma,\delta,\epsilon \in V^*, \nu \in V_1,
\]
the differential is given by vertex splitting, and the Lie bracket is given by grafting trees.

Comparing to the previous subsection one can in fact easily see that $\mG^1\GGC_{V,n}^\tree\cong \hat\fh_{V,1}$, so in particular Maurer-Cartan elements therein determine unital $\hCom_\infty$-structures on $V_1$.

Concretely, let $Z\in \mG^1\GGC_{V,n}$ be a Maurer-Cartan element. 
%
We may decompose $Z$ into pieces of different complexity,
\[
Z=Z_0+Z_1+\cdots.  
\]
Note that here we abusively reuse the notation $Z_j$ to denote the part of complexity $j$, whereas in the previous subsection the subscript referred to the edge number.
Then the image of $Z$ in $\GGC_{V,n}^\tree$ is determined by, and may be identified with $Z_0$, which can then be interpreted as a Maurer-Cartan element in $\hat\fh_{V,1}$, and hence as a unital $\hCom_\infty$-structures on $V_1$.

Next consider the object $\Graphs_{V,n}$.
We may extend the complexity grading on $\Graphs_{V,n}$ by declaring a graph with $v$ internal vertices and $e$ edges to have complexity $e-v$. The complexity gradings on $\GGC_{V,n}$ and $\Graphs_{V,n}$ are compatible with the action of $\GGC_{V,n}$ on $\Graphs_{V,n}$ in the sense that if $\gamma\in \GGC_{V,n}$ is of complexity $p$ and $\Gamma\in \Graphs_{V,n}$ is of complexity $q$, then $\gamma\cdot \Gamma$ is of complexity $q-p$.
We may furthermore endow $\Graphs_{V,n}$ with an ascending exhaustive filtration such that $\cC^p\Graphs_{V,n}$ is spanned by graphs of complexity $\leq p$. This filtration is compatible with the action of $\HGC_{V,n}$, and the dg commutative algebra structure.
The filtration extends to any twisted version $\Graphs_{V,n}^Z$ by compatibility of the complexity grading with the action of $\GGC_{V,n}$.
We shall in particular study the dg commutative subalgebra 
\[
 \cC^0 \Graphs_{V,n}^Z(1) \subset \Graphs_{V,n}^Z(1).
\]
Elements of $\cC^0 \Graphs_{V,n}^Z(1)$ are linear combinations of tree graphs.

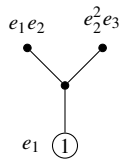
\[
\begin{tikzpicture}[baseline=-.65ex]
  \node[ext,label=180:{$\scriptstyle e_1$}] (v) at (0,0) {$\scriptstyle 1$};
  \node[int] (w) at (0,.8) {};
  \node[int, label=90:{$\scriptstyle e_1e_2$}] (w1) at (-.5,1.3) {};
  \node[int, label=90:{$\scriptstyle e_2^2e_3$}] (w2) at (.5,1.3) {};
  \draw (w) edge (w1) edge (w2) edge (v);
  \end{tikzpicture}
\]

\begin{lemma}\label{lem:V1 hCom infty}
  Let $V$ be a graded vector space and let $Z\in \mG^1\GGC_{V,n}$ be a Maurer-Cartan element. Then the following holds.
  \begin{itemize}
\item The inclusions
\begin{equation}\label{equ:V1 hCom infty}
V_1 \to \cC^0 \Graphs_{V,n}^Z(1) \to \Graphs_{V,n}^Z(1)
\end{equation}
are quasi-isomorphisms of dg vector spaces.
\item The canonical (unital) $\hCom_\infty$-structure on $V_1$ obtained from the dg commutative algebra structure on $\cC^0 \Graphs_{V,n}^Z(1)$ (or equivalently $\Graphs_{V,n}^Z(1)$) by homotopy transfer is isomorphic to the unital $\hCom_\infty$-structure encoded by the complexity zero part $Z_0$ as above.
\end{itemize}
\end{lemma}
\begin{proof}
We have already seen in Corollary \ref{cor:GraphsZ config space type} that the composition of both arrows of \eqref{equ:V1 hCom infty} is a quasi-isomorphism. 

The remaining statements can directly be concluded from the observation that $\cC^0 \Graphs_{V,n}^Z(1)$ is in fact the canonical rectification of the strongly unital $\hCom_\infty$-algebra $V_1$, defined analogously to the weakly unital version of the canonical rectification, see \cite[Theorem 6.3.2]{HirshMilles}. 
However, since the statements of \cite{HirshMilles} are formulated in a slightly different setup, we shall outline here another more elementary proof of the Lemma by applying the idea of the proof of Proposition \ref{prop:Graphs Z1 LS qiso}.

First, recall that the differential on $\Graphs_{V,n}^Z(1)$ has the form $d=d_c+d_Z$ with $d_c$ being given by edge contraction and $d_Z$ the action of $Z$. We may define a homotopy $h_c$ for the differential $d_c$ defined by adding one internal vertex, connected to the external vertex $1$ by one edge, and moving all edges and decorations from $1$ to that new vertex. Pictorially:
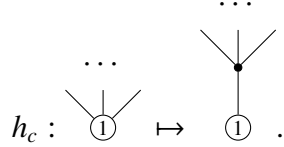
\begin{align*}
  h_c: 
  \begin{tikzpicture}[baseline=-.65ex]
    \node[ext] (v) at (0,0) {$\scriptstyle 1$};
    \draw  (v) edge +(-.5,.5) edge +(0,.5) edge +(.5,.5);
    \node at (0, .8) {$\cdots$};
  \end{tikzpicture}
    \, \mapsto \,
    \begin{tikzpicture}[baseline=-.65ex]
      \node[ext] (v) at (0,0) {$\scriptstyle 1$};
      \node[int] (w) at (0,.8) {};
      \draw (w) edge (v) edge +(-.5,.5) edge +(0,.5) edge +(.5,.5);
      \node at (0, 1.6) {$\cdots$};
      \end{tikzpicture}.
\end{align*}
This homotopy fits into homotopy transfer data 
\[
\begin{tikzcd}
  V_1 \ar[shift left]{r} [above]{\iota} 
  &
   (\cC^0 \Graphs_{V,n}^Z(1), d_c)
  \ar[shift left]{l}[below]{p_0}
  \ar[rloop]{}{h}
\end{tikzcd}.
\]
The projection $p_0$ here is the identity on graphs with zero or one decorations at vertex $1$ and no further vertices and edges, and $p_0$ on all other graphs.
Applying the Homological Perturbation Lemma \cite[Lemma 2.4]{Crainic} we then get homotopy transfer data 
\[
\begin{tikzcd}
  V_1 \ar[shift left]{r}[above]{\iota} 
  &
   (\cC^0 \Graphs_{V,n}^Z(1), d_c+d_Z)
  \ar[shift left]{l}[below]{p}
  \ar[rloop]{}{h}
\end{tikzcd}
\]
with 
\begin{align*}
  p &= p_0+p_c d_Z h_c\, .
\end{align*}

We conclude in particular that the inclusion $\iota$ is a quasi-isomorphism, and hence follows the first statement of the Lemma.
Furthermore, the above homotopy transfer data may readily be inserted into the homotopy transfer formulas.
We shall assume here familiarity with homotopy transfer, see \cite[section 10.3]{LV} for a general discussion.
Let us just note that the generators of $\hCom_\infty$ are cooperations in $\Lie_{0,\infty}^c(r)$, and may hence be identified with linear combinations of trees with $r$ leaves, $r=2,3,\dots$.
To obtain the image of some such tree $T$ in $\End(V_1)(r)$ we just re-interpret the tree as a a composition tree, with each internal edge representing a homotopy $h$, each vertec a product in $\cC^0 \Graphs_{V,n}^Z(1)$, the leaves the inclusion $\iota$ and the root the projection $p$.
Inserting the above formulas it is then elementary exercise to check that indeed the operation corresponding to the tree $T$ is the same as the term(s) of $Z$ of shape $T$. This then shows the second statement of the Lemma.
\end{proof}


 

\section{Graphs and deformation complexes of $\Com^c$ comodules}
In the previous section we have introduced the $\La$ Hopf cooperadic $\e_n^c$ comodules $\Graphs_{V,n}$.
Eventually, to show our main result Theorem \ref{thm:main_intro}, we will need to study the obstruction theory for $\La$ Hopf $\e_n^c$ comodule maps from $\Graphs_{V,n}$.
A main intermediate step towards this goal, which is taken in this section, is to study just $\La$ Hopf $\Com^c$-comodule maps, which can be done more easily, due to the quasi-freeness of $\Graphs_{V,n}$ as a $\La$ Hopf $\Com^c$ comodule, see section \ref{sec:Graphs freeness}. 
Similar arguments have been used in the literature already, see \cite{Turchin2,Turchin3,FTW3}.

\subsection{$\Com^c$ comodule maps}\label{sec:Graphs to M}

Let again $\cR$ be a graded commutative algebra. We shall study the set 
\[
\Mor_{\La HRMod-\Com^c/\cR}(\Graphs_{V,n}^\flat\otimes \cR,\MOp^\flat\otimes \cR)
\]
of morphisms of $\Lambda$-Hopf right $\Com^c$ comodules between $\Graphs_{V,n}^\flat$ and the underlying graded comodule $\MOp^\flat$ of some other $\Lambda$-Hopf right $\Com^c$ comodule $\MOp$.
As an immediate consequence of the freeness of $\Graphs_{V,n}$ as a graded $\La$ Hopf $\Com^c$ comodule, see section \ref{sec:Graphs freeness}, we have the following result.
\begin{prop}\label{prop:restr 1}
For a graded commutative algebra $\cR$ and a right $\La$ Hopf $\Com^c$ comodule $\MOp$ the map
\[
   \phi: \Mor_{\La HRMod-\Com^c/\cR}(\Graphs_{V,n}^\flat\otimes \cR,\MOp^\flat\otimes \cR)
    \to 
    \Mor_{g\bbS/ \cR}(\pIG_{V,n}^\flat\otimes \cR, \MOp^\flat\otimes \cR) 
\]
obtained by sending a morphism $F$ on the left-hand side to the composition with the inclusion of generators
\[
    \pIG_{V,n}^\flat
    \hookrightarrow
    \Graphs_{V,n}^\flat
    \xrightarrow{f} 
    \MOp^\flat
\]
is a bijection.
\end{prop}
\begin{proof}
We shall describe a two sided inverse $U$ to the map of the Proposition. So let us start with a morphism of symmetric sequences
\[
    f\in \Mor_{g\bbS/ \cR}(\pIG_{V,n}^\flat\otimes \cR, \MOp^\flat\otimes \cR)   
\]
and our goal is to construct a morphism of $\La$ Hopf $\Com^c$ comodules
\[
F=U(f): \Graphs_{V,n}^\flat\otimes \cR\to\MOp^\flat\otimes \cR.
\]
Suppose some element $x\in \Graphs_{V,n}^\flat(r)\otimes \cR$ is given. Using the freeness of $\Graphs_{V,n}^\flat(r)\otimes \cR$ as a graded comutative algebra over $\cR$ it suffices to assume $x=x_1\cdots x_k$ with each $x_j$ in the generating vector space $\IG_{V,n}(r)$.
Obviously, in this case we define $F(x)=F(x_1)\cdots F(x_k)$ as the corresponding product in $\MOp^\flat\otimes \cR$.
So we can assume directly that $x$ is an internally connected graph. Then denote the set of incident half-edges at the external vertex $j$ of $x$ by $S_j$, $j=1,\dots, r$.
Disconnecting these edges we obtain an element $x'\in \ppIG_{V,n}(S_1\sqcup \dots \sqcup S_r)\otimes \R$, and one has a unique right $\La\Com^c$-coaction $\Delta: \ppIG(S_1\sqcup \dots \sqcup S_r)\otimes \R\to \IG_n(r)\otimes \R$ so that $\Delta(x')=x$. The following picture shall illustrate the construction in one example.
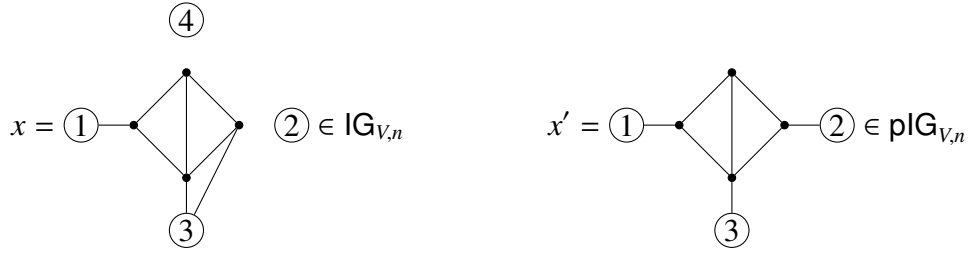
\begin{align*}
    x&=
    \begin{tikzpicture}[scale=.7,baseline=-.65ex]
        \node[int] (v1) at (-1,0){};
        \node[int] (v2) at (0,1){};
        \node[int] (v3) at (1,0){};
        \node[int] (v4) at (0,-1){};
        \node[ext] (w1) at (-2,0) {$1$};
        \node[ext] (w2) at (2,0) {$2$};
        \node[ext] (w3) at (0,-2) {$3$};
        \node[ext] (w4) at (0,2) {$4$};
        \draw (v1)  edge (v2) edge (v4) edge (w1) (v2) edge (v4) (v3) edge (v2) edge (v4) (v4) edge (w3) (v3) edge (w3);
    \end{tikzpicture}
        \in
        \IG_{V,n}
        & 
        x'=
    \begin{tikzpicture}[scale=.7,baseline=-.65ex]
        \node[int] (v1) at (-1,0){};
        \node[int] (v2) at (0,1){};
        \node[int] (v3) at (1,0){};
        \node[int] (v4) at (0,-1){};
        \node[ext] (w1) at (-2,0) {$1$};
        \node[ext] (w2) at (2,0) {$2$};
        \node[ext] (w3) at (0,-2) {$3$};
        \draw (v1)  edge (v2) edge (v4) edge (w1) (v2) edge (v4) (v3) edge (v2) edge (v4) (v4) edge (w3) (v3) edge (w2);
        \end{tikzpicture}
        \in
        \ppIG_{V,n}
\end{align*}
Then we define $F(x):= \Delta(F(x'))$, where we abusively denote by $\Delta$ also the corresponding $\La\Com^c$-coaction on $\LS_{A,n}\otimes \cR$.
Hence we can assume that $x\in \pIG_n(r)\otimes \cR$ and define $F(x)=f(x)$ in this case.

This finishes the construction of $F=U(f)$. We need to verify various properties. First, we need to verify that $F$ indeed respects that Hopf $\La$ $\Com^c$-comodule structure.
By construction $F(x_1\cdots x_k):=F(x_1)\cdots F(x_k)$ and $F$ clearly is morphism of Hopf collections. By compatibility with the Hopf structure we furthermore need to verify compatibility with the $\La$ $\Com_n^c$-comodule structure only on commutative algebra generators $\IG_n$. For those the compatibility with the $\La$ $\Com^c$-coaction again follows directly from the definition.

Next we need to verify that $U$ as constructed is indeed a two-sided inverse to $\phi$.
First, it is obvious by construction that $\phi(U(f))=f$ as desired. Second, note that every step of the construction of $F$ out of $f$ above was forced upon us by the compatibility with the Hopf $\La$ $\e_n^c$-comodule structure. Following the argument one hence sees that any Hopf $\La$ $\Com^c$-comodule morphism $F$ is uniquely determined by its restriction to generators $\pIG_{V,n}$, or in other words $\phi$ is injective, thus finishing the proof.
\end{proof}

Note Proposition \ref{prop:restr 1} only addresses the graded Hopf comodule morphisms. For such a morphism $F=U(f)$ to be an honest dg Hopf comodule morphism we need to ask in addition that it intertwines the differentials, i.e.,
\begin{equation}\label{equ:dF Fd}
    d_{\MOp} F - F d_{\Graphs_{V,n}}=0.
\end{equation}
To simplify this, note first the the expression $d_{\MOp} F - F d_{\Graphs_{V,n}}$ is a biderivation of $F$.
This means that, by slightly adjusting the argument in the proof of Proposition \ref{prop:restr 1}, equation \eqref{equ:dF Fd} is in fact equivalent to 
\[
  (d_{\MOp} F - F d_{\Graphs_{V,n}})\circ \iota =0.  
\]
In other words it suffices to check the equation on $\La$ Hopf $\Com^c$-comodule generators. Let us work out the explicit equation obtained for $F=U(f)$. For the first term, involving the differential on $\MOp$ we obtain
\[
    d_{\MOp} U(f)\circ \iota =d_{\MOp} f.
\]
Next consider the contribution of $d_{\Graphs_n}=d_c$. Recall that this differential contracts an edge in a graph.
Start with a generator $x\in \ppIG_{V,n}(r)$.
Then generally, contracting edges may result in non-internally connected graphs and
\[
d_{\Graphs_{V,n}} x = \sum x_1 \cdots x_k
\]
is a sum of products of internally connected components $x_j\in \IG_n(r)$.
Then, following the construction of $U$ above, one sees that 
\begin{equation}\label{equ:Linfty pre HGC}
    \begin{aligned}
    U(f)(d_{\Graphs_{V,n}}x)
    &=
    U(f)(\sum x_1 \cdots x_k)
    =
    \sum U(f)(x_1) \cdots U(f)(x_k)
    \\&=
    \sum \Delta(f(x_1'))) \cdots (\Delta(f(x_k')),
    \end{aligned}
\end{equation}
using the product on $\MOp$ on the very right.
Furthermore, we (ab)used here the notation $x_j'\in \pIG_n$ and $\Delta$ for the elements obtained by deconcatenating vertices, and the $\La$ $\Com^c$ coaction as in the proof of Proposition \ref{prop:restr 1}.

\subsection{The hairy graph complex and its $\SL_\infty$-structure}\label{sec:HGC}
Let $\MOp$ be a right $\La$ Hopf $\Com^c$ comodule.
As a graded vector space the hairy graph complex is defined to be
\[
\HGC_{\MOp, V,n} := \iHom_{\bbS}(\pIG_{V,n}^\flat, \MOp^\flat).
\]
The notation $\iHom_{\bbS}(\cdots)$ here refers to the graded vector space of graded morphisms of symmetric sequences.
The most important special case for us is that $\MOp= \Fc_A$ for some dg commutative algebra $A$.
In this case the elements of $\HGC_{A,V,n}:=\HGC_{\Fc_A,V,n}$ can be identified with series of $V$-decorated graphs with external legs decorates by $A$, for example
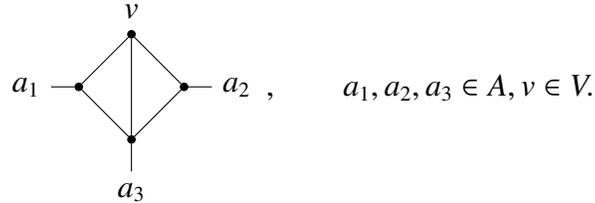
\[
\begin{tikzpicture}[scale=.7,baseline=-.65ex]
\node[int] (v1) at (-1,0){};
\node[int,label={$v$}] (v2) at (0,1){};
\node[int] (v3) at (1,0){};
\node[int] (v4) at (0,-1){};
\node (w1) at (-2,0) {$a_1$};
\node (w2) at (2,0) {$a_2$};
\node (w3) at (0,-2) {$a_3$};
\draw (v1)  edge (v2) edge (v4) edge (w1) (v2) edge (v4) (v3) edge (v2) edge (v4) (v4) edge (w3) (v3) edge (w2);
\end{tikzpicture}
\,,\quad\quad
a_1,a_2,a_3\in A, v\in V.
\]
There is also a complete descending filtration on $\HGC_{\MOp,V,,n}$ by the number of edges in graphs. It has the property that each $\mG$-graded quotient (identifiable with the subspace of graphs with a fixed number of edges) is of finite type, as long as $A$ is of finite type and non-negatively graded.
We hence see that the set of degree zero elements in $\HGC_{\MOp,V,n}\hotimes \cR$ is identified with the set $\Mor_{g\bbS/ \cR}(\pIG_n^\flat\otimes \cR, \MOp^\flat\otimes \cR)$, and via Proposition \ref{prop:restr 1} also with the set $\Mor_{\La HRMod-\Com^c/\cR}(\Graphs_{V,n}^\flat\otimes \cR,\MOp^\flat\otimes \cR)$ of morphisms of graded $\La$ Hopf $\Com^c$-comodules.
Following the previous subsection, for $f\in \HGC_{\MOp,V,n}\hotimes \cR$ of degree 0 we shall denote by $U(f): \Graphs_{V,n}^\flat\otimes \cR\to \MOp^\flat\otimes \cR$ the corresponding morphisms of graded $\La$ Hopf $\Com^c$-comodules.
The equation \eqref{equ:dF Fd} then translates into a power series without constant term for $f$:
\[
\mU(f) := [d,U(f)]\circ \iota= (d_{\MOp}U(f)-U(f)d_{\Graphs_{V,n}} ) \circ \iota.
\]
All terms are $\cR$-linear and functorial in $\cR$. Furthermore, we claim that $\mU(f)$ satisfies the $\SL_\infty$-relations in the form of \eqref{equ:Linfty structure series 2},
\[
  \mU(f+\epsilon\mU(f)) =\mU(f).
\]
Indeed, we first note that since $U(f)$ is a morphism of Hopf comodules, $[d,U(f)]$ is a biderivation of $U(f)$, i.e., with $\epsilon$ a formal variable of degree $-1$ that we adjoin to our ground ring,
\[
  U(f) +\epsilon [d,U(f)]
\]
is a morphism of Hopf comodules. But the composition of this morphism with $\iota$ is clearly $f+\epsilon\mU(f)$, and hence 
\[
  U(f+\epsilon\mU(f)) = U(f) +\epsilon [d,U(f)].
\]
Buth then we find, as desired,
\[
  \mU(f+\epsilon\mU(f)) = [d, U(f+\epsilon\mU(f)) ]\circ \iota 
  =
  [d, U(f) +\epsilon [d,U(f)]]\circ \iota 
  =
  [d, U(f)]\circ \iota  + 0 = \mU(f).
\]

We hence arrive at the following Corollary to the discussion in the previous section.
\begin{cor}\label{cor:HGC}
There is a filtered complete $\SL_\infty$-structure on $\HGC_{\MOp,V,n}$ such that the Maurer-Cartan elements in $\HGC_{\MOp,V,n}\hotimes \cR$ are in 1-1-correspondence to the dg Hopf $\Com^c$-comodule morphisms $\Graphs_{V,n}\otimes \cR\to \MOp$. In particular, taking $\cR=\Omega(\Delta^\bullet)$,
\[
\Map_{\dgLaHModc_{\Com^c}}(\Graphs_{V,n},\MOp) 
\cong 
\MC_\bullet(\HGC_{\MOp,V,n}).
\]
\end{cor}

We can also work out the combinatorial form of the $\SL_\infty$-operations on graphs, for the special case $\MOp=\Fc_A$.
To this end look again at \eqref{equ:Linfty pre HGC}.
Note that the differential only contracts one edge, and hence the internally connected graphs $x_1,\dots,x_k$ have only one external vertex in common which they all connect to, and they can connect there (possibly) with more than one half-edge. All other external vertices still have valency one, since $x\in \ppIG_{V,n}$. 
This means that combinatorially the $k$-ary $\SL_\infty$ operation $\mu_k$, for $k\geq 2$ is given on graphs $\Gamma_1,\dots,\Gamma_k\in \HGC_{A,V,n}$ by summing over all ways of fusing a subset of hairs to a single hair, in a way forming a connected graph, as shon in the following picture (for $k=2$).
\[
    \left[\begin{tikzpicture}[baseline=-.8ex]
    \node[draw,circle] (v) at (0,.3) {$\Gamma$};
    \draw (v) edge +(-.5,-.7) edge +(-.25,-.7) edge +(0,-.7) edge +(.25,-.7) edge +(.5,-.7);
    \end{tikzpicture},
    \begin{tikzpicture}[baseline=-.65ex]
    \node[draw,circle] (v) at (0,.3) {$\Gamma'$};
    \draw (v) edge +(-.5,-.7) edge +(-.25,-.7) edge +(0,-.7) edge +(.25,-.7) edge +(.5,-.7);
    \end{tikzpicture}
    \right]
    = \sum\begin{tikzpicture}[baseline=-.8ex]
    \node[draw,circle] (v) at (0,.3) {$\Gamma$};
    \node[int] (i) at (.5,-.5) {};
    \draw (v) edge +(-.5,-.7) edge +(0,-.7) edge (i) edge[bend left] (i) edge[bend right] (i);
    \node[draw,circle] (vv) at (1,.3) {$\Gamma'$};
    \draw (vv) edge (i) edge[bend left] (i) edge[bend right] (i) edge +(0,-.7) edge +(.5,-.7) (i) edge (.5,-1);
    \end{tikzpicture}\,,
\]
For $k=1$ the differential $\mu_1$ is given by the same procedure, except that there is also an additional term $d_A$ inherited from the differential on $A$, and a term splitting internal vertices inherited from the differential on $\pICG_{V,n}=(\pIG_{V,n})^*$.

\begin{rem}
So far we have encountered two very similar graph complexes, $\GGC_{V,n}$ and $\HGC_{\MOp, V, n}$.
In fact, for $\MOp(r)=V_1^{\otimes r}$ these two graded vector spaces are isomorphic, and hence we use similar notation.
However, we shall also emphasize that they carry different algebraic structures and play different roles: $\GGC_{V,n}$ is a dg Lie algebra with Lie bracket of degree zero, while $\HGC_{\MOp, V, n}$ carries a different $\SL_\infty$-algebra structure. In particular, the $\SL$-bracket is of degree $+1$.
The reader is hence advised to keep the two objects $\GGC_{V,n}$ and $\HGC_{\MOp, V, n}$ apart.
\end{rem}

\subsection{Graded version}
For technical reasons we will also consider the associated graded right $\La$ Hopf $\Poiss_n^c$ comodule
\[
\gr \Graphs_{V,n} = \gr_{\mG} \Graphs_{V,n},
\]
where we use the filtration $\mG$ by the number of edges as in section \ref{sec:Gfiltration}.
We note that $\gr \Graphs_{V,n}$ has vanishing differential and the $\Poiss_n^c$ coaction corestricts from the $\Com^c$-coaction,
\[
  \gr \Graphs_{V,n} = \coRes_{\Com^c}^{\Poiss_n^c}\gr \Graphs_{V,n}.
\]
Note also that $\gr \Graphs_{V,n}$ is free as a right $\La$ Hopf $\Com^c$ comodule, by the same arguments that show that $\Graphs_{V,n}$ is quasi-free, see above.
Similarly, for $\MOp$ a right $\La$ Hopf $\Com^c$ comodule we consider the descending complete filtration by the number of edges on the graph complexes $\HGC_{\MOp, V,n}$.
Concretely, 
\[
\mG^p \HGC_{\MOp, V,n} 
\]
consists of series of graphs with $\geq p$ edges. Note that the inherited $\SL_\infty$-structure on the associated (complete) graded $\hgr \HGC_{\MOp, V,n}$ is abelian, and the differential is merely that inherited from $\MOp$.
One then has the following graded version of Proposition \ref{prop:restr 1} and Corollary \ref{cor:HGC}.
\begin{prop}\label{prop:graded HGC}
For every graded commutative algebra $\cR$ and $\La$ Hopf $\Com^c$-comodule $\MOp$ the composition with the inclusion of generators 
\[
\Mor_{\cR\otimes\Com^c}(\cR\otimes \gr\Graphs_{V,n}^\flat, \cR\otimes \MOp^\flat )
\xrightarrow{-\circ \tilde \iota}
\left( \hgr \HGC_{\NOp, V, n} \hotimes \cR \right)_0
\] 
is a bijection, such that the closed elements on the right-hand side are precisely the images of $\La$ Hopf  
$\Com^c$ comodule morphisms that respect the differentials.
\end{prop}

Let now $\MOp$ be a right $\La$ Hopf $\Poiss_n^c$ comodule and consider the graded version of the deformation complex 
\[
\Def(\gr\Graphs_{V,n},W\MOp)
\cong
\hgr\Def(\Graphs_{V,n},W\MOp)
\]
of section \ref{sec:def cx}, where $W\MOp$ is the W fibrant resolution of $\MOp$, see section \ref{sec:W construction}. Here we implicitly prolong the descending complete filtration on the deformation complex induced by the filtration on $\Graphs_{V,n}$.

Then, using the adjunction between corestriction and coinduction, and the fact that $\gr\Graphs_{V,n}$ is the corestriction of the underlying $\Com^c$-comodule, we have the following comparison of the deformation complex and hairy graph complex.

\begin{prop}
Let $\MOp$ be a right $\La$ Hopf $\Poiss_n^c$ comodule, and let $\NOp:=\coInd_{W\Poiss_n^c}^{\Com^c}W\MOp$. Then there is an isomorphism of dg vector spaces
\[
  \Def(\gr\Graphs_{V,n},W\MOp)
  \cong \hgr \HGC_{\NOp, V, n} 
\]
\end{prop}
\begin{proof}
  We will construct a(n a priori) curved $\SL_\infty$-morphism between the abelian $\SL_\infty$-algebras $\Def(\gr\Graphs_{V,n},W\MOp)$ and $\hgr \HGC_{\NOp, V, n}$. Then we will verify that only the linear component of this $\SL_\infty$-morphism is non-vanishing. We again use the formalism of section \ref{sec:Linfty_formalism} to encode the $\SL_\infty$-morphism. 
Let $\cR$ be a graded commutative algebra, then by definition we have isomorphisms
\begin{equation}\label{equ:Phicomp}
\begin{aligned}
  \left( \Def(\gr\Graphs_{V,n},W\MOp)\hotimes \cR \right)_0
  &\xrightarrow{F, \cong} 
  \Mor_{W\e_n^c\otimes \cR}(\gr\Graphs_{V,n}\otimes \cR, W\MOp\otimes \cR)
  \\&\xrightarrow{A,\cong} 
  \Mor_{\Com^c\otimes \cR}(\gr\Graphs_{V,n}\otimes \cR, \NOp\otimes \cR)
  \\&\xrightarrow{(-)\circ\iota,\cong}
  \left( \hgr \HGC_{\NOp, V, n} \hotimes \cR \right)_0.
\end{aligned}
\end{equation}
The first morphism is obtained from Lemma \ref{lem:freecofree} (see also \eqref{equ:pre def tensor out}), the middle identification is by adjunction, and the right-hand identification is given by Proposition \ref{prop:graded HGC}.
In particular, note that $W\MOp$ is quasi-free and the coinduced module $\NOp\subset W\MOp$ can be identified with the sub-symmetric sequence consisting of elements whose cocomposition lands in $W\MOp\circ \Com^c$.
The adjunction map $A$ sends a morphism $f$ to itself, noting that the image of $f$ must be in $\NOp$ by compatibility with the $W\e_n^c$-coaction.

It is clear that all arrows above are obtained by $\cR$-linear extension of structure operations, and hence the composition defines a power series 
\[
\Phi(x) = (A(F(x)))\circ \tilde \iota  
\]
whose summands are obtained by $\cR$-linear extension of multilinear maps defined over $\Q$.
To check that $\Phi$ encodes a non-flat $\SL_\infty$-morphism we can check the $\SL_\infty$-relations in the form \eqref{equ:Linfty morphism series}.
first note that 
\[
A(F(x+\epsilon\mU(x) )) = A((1+\epsilon d) F(x))
=
(1+\epsilon d)A(F(x)).
\]
Hence 
\[
  \Phi(x+\epsilon\mU(x))
  =
  (1+\epsilon d)A(F(x))\circ \tilde \iota  
  =
  \Phi(x) + \mV(\Phi(x))
\]
as desired. We next check that in fact $\Phi(x)$ is linear in $x$, with $\mU$ the generating function for the $\SL_\infty$-algebra $\Def(\gr\Graphs_{V,n},W\MOp)$ and $\mV$ that for $\hgr \HGC_{\NOp, V, n}$. To this end note that the $k$-linear part of $F(x)$ consists of morphisms that take values in $S^k\IG_{V,n}$.
But the restriction to generators $\circ\tilde \iota$ renders all these parts zero for $k\neq 1$.
Furthermore, as we have seen above this linear part must be a isomorphism, since all three maps in \eqref{equ:Phicomp} are bijections.
\end{proof}



We shall later need the following Corollary of the above Proposition, connecting the hairy graph complex and the deformation complex.
\begin{cor}\label{cor:Def is HGC}
  Let $\MOp$ be a $\La$ Hopf $\Poiss_n^c$-comodule of configuration space type. Let
  $\NOp := \coInd_{W\Poiss_n^c}^{\Com^c}(W\MOp)$.
  Then we have the following morphisms of dg vector spaces
  \[
  \Def(\gr \Graphs_{V,n}, W\MOp)
  \cong 
  \hgr \HGC_{\NOp,V,n} \leftarrow \hgr \HGC_{\MOp(1),V,n}
  \]
  of which the left-hand morphism is an isomorphism and the right-hand arrow is a quasi-isomorphism. Here the right-hand arrow is induced by the canonical quasi-isomorphism $\FF_{\MOp(1)}\to \NOp$, and we use our shorthand notation $\HGC_{\MOp(1),V,n}=\HGC_{\FF_{\MOp(1)},V,n}$.
\end{cor}
\begin{proof}
The only thing left to check here is that the right-hand arrow is a quasi-isomorphism.
However, note that the construction $?\mapsto\hgr\HGC_{?,V,n}$ is a functor. Furthermore, it sends quasi-isomorphisms of $\La$ Hopf $\Com^c$-comodules to quasi-isomorphisms by the K\"unneth theorem.
\end{proof}



\section{Main result and proof of Theorem \ref{thm:main_intro}}
In this section we shall show the following result, which is a slight refinement of our main Theorem \ref{thm:main_intro} from the introduction.

\begin{thm}\label{thm:main}
Let $n\geq 2$ and let $\MOp$ be a right $\La$ Hopf $\e_n^c$-comodule of configuration space type, such that $H:=H(\MOp(1))$ is connected and finite dimensional. 
Then there is a Maurer-Cartan element $Z\in \mG^1\GGC_{\bar H,n}$ and a quasi-isomorphism of right $\La$ Hopf $W\e_n^c$-comodules
\[
\Phi : \Graphs_{\bar H,n}^Z \to W\MOp.  
\]
\end{thm}

\subsection{Proof of Theorem \ref{thm:main}}
We fix once and for all cohomology representatives $H\to \MOp(1)$, i.e., we consider $H$ (noncanonically) as a subspace of cocycles in $\MOp(1)$. By restriction, we also obtain a map from the positively graded part $\bar H\to \MOp(1)$.
We desire to construct a Maurer-Cartan element in the curved $\SL_\infty$-algebra of Proposition \ref{prop:Lifty semidirect}
\[
m \in  \mG^1\GGC_{H,n} \ltimes \Def(\Graphs_H,W\MOp) =:\fg.  
\]
We consider the right-hand side $\fg$ as equipped with a descending 
complete filtration by the number of edges in graphs $\fg = \mG^0\fg \supset \mG^1\fg \supset \cdots$.

We first construct a Maurer-Cartan element up to order one, i.e., a degree one element $m_0\in \fg$ such that
\[
  \curv(m_0) \in \mG^1\fg .
\]
We will take $m_0\in \Def(\Graphs_{\bar H,n},W\MOp)$, i.e., the component in $\GGC_{\bar H,n}$ is zero.
But looking at the description of Maurer-Cartan elements 
from Proposition \ref{prop:Def Linfty} we see that Maurer-Cartan elements correspond to morphisms of $\La$ Hopf $(W\e_n^c)^\flat$-comodules
\[
F :   \Graphs_{\bar H,n}^\flat \to W\MOp^\flat
\]
that intertwine the differentials $dF=Fd$, with $m_0=\pi\circ F\circ \iota$ obtained by inclusion of generators and projection to cogenerators. For Maurer-Cartan elements up to order one the equation $dF=Fd$ need only be satisfied on $\mG^0\Graphs_{\bar H,n}=\FF_{S(\bar H)}$.
But then we can invoke Lemma \ref{lem:grG0 map} to construct this map $F$, given our fixed inclusion $H\to \MOp(1)$.
Since by that Lemma the restriction of $F$ to $\mG^0\Graphs_{\bar H,n}$ is compatible with the differentials, our element $m_0$ is a Maurer-Cartan element up to order one as desired.

Next we want to use $\SL_\infty$-obstruction theory to construct a Maurer-Cartan element $m_1$ in $\mG^1\fg^{m_0}$.
First, we compute the relevant obstruction spaces, showing that they vanish.
\begin{prop}\label{prop:Hgrg vanishing}
In the setting of Theorem \ref{thm:main} and with $\fg$ as as above, we have that 
\[
  H(\hgr_{\mG} (\mG^1\fg^{m_0}))=0.
\]
\end{prop}
We postpone the proof of this Proposition to the next subsection, and continue with our proof of Theorem \ref{thm:main}.

Using Proposition \ref{prop:Hgrg vanishing} we can immediately apply Corollary \ref{cor:MCexistence1}  to conclude that the curved $\SL_\infty$-algebra $\fg^{m_0}$ contains a Maurer-Cartan element $m_1\in \mG^1\fg^{m_0}$. Equivalently, $m_0+m_1$ is a Maurer-Cartan element in $\fg$.
Let us write
\[
m_0+m_1 = Z + m'  
\]
with $Z\in \mG^1\GGC_H$ and $m'\in \Def(\Graphs_{\bar H,n},W\MOp)$.
Concretely, $m'$ then encodes a morphism of right Hopf $\La$ $W\e_n^c$ comodules
\[
 \Phi: \Graphs_{\bar H,n}^Z \to W\MOp.
\]
To finish the proof of the Theorem we need to show that $\Phi$ is a quasi-isomorphism.
But both sides are $W\e_n^c$-comodules of configuration space type.
Hence we may employ Proposition \ref{prop:configtype qiso condition} to conclude that the map is a quasi-isomorphism if so is the unary part $U(\Phi):\Graphs_H^Z(1)\to W\MOp(1)$.
But by the last assertion of Corollary \ref{cor:GraphsZ config space type} we have that the canonical map of dg vector spaces $H\hookrightarrow \Graphs_H^Z(1)$ is a quasi-isomorphism for any Maurer-Cartan element $Z\in \mG^1\GGC_H$. Furthermore, the composition 
\[
  H\hookrightarrow \Graphs_{\bar H,n}^Z(1) \to W\MOp(1)
\]
factors through $\mG^0\Graphs_{\bar H,n}^Z$, is determined by $m_0$, and agrees with the quasi-isomorphism $H\to \MOp(1) \to W\MOp(1)$ we fixed in the beginning. Hence the conditions of Proposition \ref{prop:configtype qiso condition} are satisfied and the Theorem is proven.
\hfill \qed

\subsection{Proof of Proposition \ref{prop:Hgrg vanishing}}
\label{sec:Hgrg vanishing proof}
For this proof we always use the filtrations $\mG$ by numbers of edges and hence omit the corresponding subscripts in $\hgr_{\mG}=: \hgr$.
First note that 
\[
  \hgr \mG^1\fg^{m_0} = \hgr\mG^1\GGC_{\bar H, n} \oplus \mG^1\Def(\gr \Graphs_{V.n}, W\MOp)
\]
has the differential
\begin{equation}\label{equ:propproofd}
  d = d_{act} + d_{\Def},  
\end{equation}
where 
\[
  d_{act}: \hgr\mG^1\GGC_{\bar H, n}\to \mG^1\Def(\gr \Graphs_{V.n}, W\MOp)[1]
\] 
is the part arising from the action of the hairy graphs on the Maurer-Cartan element $m_0$ and $d_{\Def}$ is the differential on the deformation complex.
Note in particular that the differential on $\hgr\mG^1\GGC_{\bar H, n}$ is zero. Hence we can rephrase the problem as showing that the map $d_{act}$ is a quasi-isomorphism of dg vector spaces.
But the deformation complex is shown in Corollary \ref{cor:Def is HGC} to be quasi-isomorphic to the hairy graph complex $\HGC_{\MOp(1),\bar H,n}$. Furthermore, by the quasi-isomorphism $H\to \MOp(1)$ chosen in our construction we obtain a chain of quasi-isomorphisms of complexes
\begin{equation}\label{equ:HGCzigzig}
  \hgr\GGC_{\bar H, n} = \hgr\HGC_{H,\bar H,n} 
  \to \hgr\HGC_{\MOp(1),\bar H,n} 
  \xrightarrow{\text{Cor. \ref{cor:Def is HGC} }}
  \Def(\gr \Graphs_{V.n}, W\MOp).
\end{equation}
Hence we just need to show that this composition agrees with $d_{act}$, or at least that it induces the same cohomology isomorphism.
To do this we first need to unpack the definition of $d_{act}$. Let $\Gamma\in \hgr\GGC_{\bar H, n}$ be an (automatically closed, since the differential is zero) element of degree $k$. Consider the graded commutative algebra $\cR=\Q[\epsilon]/\epsilon^2$ with $\epsilon$ of degree $-k$. Let $\gr F: \gr \Graphs_{V.n}\to W\MOp$ the morphism of $\La$ Hopf $W\e_n^c$ comodules encoded by our $m_0$. Concretely, $\gr F$ factors through the canonical morphism $\FF_{S(\bar H)}\to W\MOp$, and is in particular zero on all graphs with an edge. Then by definition $d_{act}\Gamma$ is such that 
\[
  m_0 + \epsilon d_{act}\Gamma
  = \pi \circ ( F\circ e^{\epsilon \Gamma\cdot}) \circ \iota
  =
  \pi \circ ( F + \epsilon F( \Gamma\cdot -) ) \circ \iota.
\]
More concretely, let us again identify the (graded vector space underlying the) hairy graph complex with
\[
\hgr\GGC_{\bar H, n} \cong \iHom_{\bbS}(\gr\ppIG_{V,n}, \FF_{H}),
\]
where we (ab)use the notation $\FF_{H}$ for the $\bbS$-module with arity $r$ component $H^{\otimes r}$. Then $F( \Gamma\cdot \iota(-))$ is the composition
\[
  \gr\ppIG_{V,n}
  \xrightarrow{\Gamma\cdot}
  \FF_{H} 
  \hookrightarrow
  \FF_{S(\bar H)}
  \to \FF_{\MOp(1))}
  \to W\MOp.
\]
But this agrees with the morphism obtained from $\Gamma$ by following \eqref{equ:HGCzigzig}, cf. Corollary \ref{cor:Def is HGC}.
\hfill\qed

\section{Automorphism spaces and proof of Theorem \ref{thm:intro ExpAut}}\label{sec:automorphisms}
\subsection{Definition, and general facts}
Generally, let us denote by $\Map'(\cdots)$ the subspace of the mapping space consisting of the connected components of morphisms that induce isomorphisms on cohomology. 
Let $\MOp$ be a $\La$ Hopf $\COp$ comodule, and let $\hat \MOp$ be a fibrant and cofibrant resolution.
Then we define the homotopy automorphism space of $\MOp$ to be
\[
\Aut^h(\MOp) = \Map'(\hat \MOp, \hat \MOp)\subset \Map(\hat \MOp, \hat \MOp).
\]
This is a simplicial monoid, whose homotopy type is independent of the choice of fibrant cofibrant resolution, see \cite[section II.2.2]{FrII}.
Furthermore, since every quasi-isomorphism in particular induces an isomorphism on cohomology, one has a map of simplicial monoids
\[
  \Aut^h(\MOp) \to \GL(H(\MOp(1))),
\]
where we consider the right-hand side as a discrete group. In particular, note that all morphisms in the same connected component of $\Aut^h(\MOp)$ induce the same map on cohomology.
We denote the preimage of the identity by 
\[
  \Aut^h(\MOp)_{[1]} \subset \Aut^h(\MOp).
\]
This should not be confused with the connected component of the identity -- $\Aut^h(\MOp)_{[1]}$ may have multiple connected components.

The homotopy automorphism group is in general difficult to compute, with one complication being that typically we do not have "simple" fibrant and cofibrant replacements available.
However, one has the following recognition principle.
\begin{prop}\label{prop:recognition}
Suppose that $f:\MOp\to \NOp$ is a quasi-isomorphism of $\La$ Hopf $\COp$ comodules, with $\MOp$ cofibrant and $\NOp$ fibrant.
Suppose that $G$ is a simplicial group acting on $\MOp$ in the sense that there is a map of simplicial monoids $G \to \Map(\MOp, \MOp)$.
\begin{itemize}
  \item Then if the induced map (by composition with $f$)
\[
G\to \Map'(\MOp, \NOp)
\]
is a weak equivalence, then $G$ is weakly equivalent to $\Aut^h(\MOp)$ as a simplicial monoid.
\item 
Let $\Map'(\MOp, \NOp)_{[f]}\subset \Map'(\MOp, \NOp)$ be composed of connected components of morphisms that induce the same map $H(\MOp(1))\to H(\NOp(1))$ as $f$. Then if the induced map 
\[
G\to \Map'(\MOp, \NOp)_{[f]}
\]
is a weak equivalence, $G$ is weakly equivalent to $\Aut^h(\MOp)_{[1]}$ as a simplicial monoid.
\end{itemize}
\end{prop}
\begin{proof}
The first part is true in any model category, and is proven in \cite{FWAut}.
The proof of the second part is a small variation.
\end{proof}



\subsection{Automorphisms of $\Graphs_{V,n}^Z$}
Let $n\geq 2$.
Recall that the dg Lie algebra $\GGC_{V,n}$ acts on the $\La$ Hopf $\e_n^c$ comodule $\Graphs_{V,n}$ by biderivations. Accordingly, for $Z\in \MC(\mG^1\GGC_{V,n})$ a Maurer-Cartan element, the twisted dg Lie algebra 
$\GGC_{V,n}^Z$ acts on $\Graphs_{V,n}^Z$ by biderivations.

However, $\GGC_{V,n}$ is generally not pro-nilpotent:
Recall that we have the semidirect product decomposition
\[
  \GGC_{V,n} = \gl_{V}' \ltimes \mG^1 \GGC_{V,n},
\]
and the part $\gl_{V}'=\iHom(V,V_1)$ is finite dimensional but not nilpotent in general.
We shall hence consider the pro-nilpotent dg Lie subalgebra 
\[
  \GGC_{V,n}^{nil} :=  (\gl_{V}')_{<0} \ltimes \mG^1 \GGC_{V,n} \subset \GGC_{V,n},
\]
where we replace $\gl_{V_1}'$ by its graded Lie subalgebra $(\gl_{V_1}')^{<0}\subset \gl_{V_1}'$ of negative cohomological degree. Note that alternatively we also have the identifaction
\[
  \GGC_{V,n}^{nil} =\mF^1 \GGC_{V,n}  
\]
as the $\mF^1$-part for the filtration $\mF$ of section \ref{sec:FFiltration}, from which pro-nilpotence is again clear. 
The dg Lie algebra $\GGC_{V,n}^{nil}$ still acts on $\Graphs_{V,n}$ by restriction, and the twisted version $\GGC_{V,n}^{nil,Z}$ acts on $\Graphs_{V,n}^Z$. By pro-nilpotence we can then pass to the exponential group and obtain a map of simplicial monoids 
\[
  \Exp_\bullet(\GGC_{V,n}^{nil,Z}) \to 
  \Aut(\Graphs_{V,n}^Z)_{[1]}.
\] 

Our main result is then the following.

\begin{thm}\label{thm:main aut}
  Let $V$ be a finite dimensional positively graded vector space, let $n\in \Z$, and let $Z\in \mG^1\HGC_{V,n} $ be a Maurer-Cartan element.
Then the action of $\GGC_{V,n}^{nil,Z}$ on $\Graphs_{V,n}^Z$ induces a weak equivalence 
\[
  \Exp_\bullet(\GGC_{V,n}^{nil,Z}) \to 
  \Map(\Graphs_{V,n}^Z, W\Graphs_{V,n}^Z)_{[f]},
\] 
where $f:\Graphs_{V,n}^Z\to W\Graphs_{V,n}^Z$ is the canonical quasi-isomorphism.
\end{thm}
Together with the second part of Proposition \ref{prop:recognition} we hence obtain the following corollary, that is identical to Theorem \ref{thm:intro ExpAut}.
\begin{cor}\label{cor:ExpAut}
We have a weak equivalence of simplicial monoids 
\[
  \Exp_\bullet(\GGC_{V,n}^{nil,Z}) \simeq \Aut^h(\Graphs_{V,n}^Z)_{[1]}.
\]
\end{cor}

The results above are formulated for the exponential groups of our dg Lie algebras. Some readers might prefer to work with classifying spaces instead. However,
for $\fg$ a filtered complete dg Lie algebra satisfying suitable conditions and $m\in \MC(\fg)$ a Maurer-Cartan element we have weak equivalences 
\[
B\Exp_\bullet(\fg^m)  
\simeq 
\MC_\bullet(\fg)_m 
\simeq 
\MC_\bullet(\tru(\fg^m)),
\]
where the truncation $\tru(\fg^m)$ consists of the elements of $\fg$ of degree $<0$ and the closed elements in degree zero.
For example, suitable conditions for the above statement to hold are that $\fg$ is the dual of a filtered Lie coalgebra $\fc$, and has finite dimensional quotients $\mF^p\fg/\mF^{p+1}\fg$, see \cite{WillwacherRecollections}.
These conditions are satisfied in our case and hence we can conclude:

\begin{cor}\label{cor:BAut}
We have that 
\[
  \MC_\bullet(\GGC_{V,n}^{nil})_Z
 \simeq  
\MC_\bullet(\tru(\GGC_{V,n}^{nil,Z}))
 \simeq  
 B\Exp_\bullet(\GGC_{V,n}^{nil,Z})
 \simeq 
 B\Aut^h(\Graphs_{V,n}^Z)_{[1]},
\]
with $\MC_\bullet(\GGC_{V,n}^{nil})_Z\subset \MC_\bullet(\GGC_{V,n}^{nil})$ being the connected component of $Z$.
\end{cor}

\subsection{Proof of Theorem \ref{thm:main aut}}

To prove Theorem \ref{thm:main aut} we will first reformulate the problem as a comparison of Maurer-Cartan spaces.
To this end we adapt the discussion of the previous section.
First note that there is a morphism of $\SL_\infty$-algebras
\[
\pi : \Def(\Graphs_{V,n}^Z\to W\Graphs_{V,n}^Z)
\to \gl_{V}'[1]
\]
by restricting the morphism to the arity one generators $V$ and projecting the result to the at most linear part in $V$. 
We then define the $\SL_\infty$-subalgebra 
\begin{multline*}
\Def(\Graphs_{V,n}^Z\to W\Graphs_{V,n}^Z)'
:=
\pi^{-1}((\gl_V')_{<0})
\\\subset 
\Def(\Graphs_{V,n}^Z\to W\Graphs_{V,n}^Z) \cong \iHom_{\bbS}(\pICG_{V,n}, \oW\Graphs_{V,n}^Z),
\end{multline*}
to consist of those morphisms $f$ such that $\pi(f)$ is of negative degree.

\begin{lemma}
One has the following commutative diagram of simplicial sets
\[
  \begin{tikzcd}
    \MC_\bullet (\Def(\Graphs_{V,n}^Z\to W\Graphs_{V,n}^Z)') \ar{r}{=} \ar{d}
    & \Map(\Graphs_{V,n}^Z,W\Graphs_{V,n}^Z)_{[f]} \ar{d}\\
    \MC_\bullet (\Def(\Graphs_{V,n}^Z\to W\Graphs_{V,n}^Z)) \ar{r}{=}
    & \Map(\Graphs_{V,n}^Z,W\Graphs_{V,n}^Z)
  \end{tikzcd}\, .
\]
\end{lemma}
\begin{proof}
  A $k$-simplex 
  \begin{align*}
    X\in
    \Map_k(\Graphs_{V,n}^Z,W\Graphs_{V,n}^Z)
    &=
    \MC_k (\Def(\Graphs_{V,n}^Z\to W\Graphs_{V,n}^Z))
    \\&=
    \MC (\Def(\Graphs_{V,n}^Z\to W\Graphs_{V,n}^Z) \hotimes \Omega(\Delta^k))
  \end{align*}
  is in the simplicial subset $
    \Map(\Graphs_{V,n}^Z,W\Graphs_{V,n}^Z)_{[f]}$
  iff the following condition holds:
  \begin{itemize}
    \item
  For some (and hence any) point $a\in \Delta^n$, the morphism 
  \[
  \ev_a X : \Graphs_{V,n}^Z\to W\Graphs_{V,n}^Z,
  \]
  obtained by restriction induces a map $[\ev_a X]$ on cohomology that satisfies $[\ev_a X]=[f]$.
  \end{itemize}
  Interpreting $\ev_a X$ as an element of $
    \MC (\Def(\Graphs_{V,n}^Z\to W\Graphs_{V,n}^Z)) \cong \Hom_{S}(\pICG_{V,n}, \oW\Graphs_{V,n}^Z)$
  this means that the linear part $\pi(\ev_a X)$ vanishes.

  But this equivalent to the 0-form part of $X$ being in 
  $\Def(\Graphs_{V,n}^Z\to W\Graphs_{V,n}^Z)'$.
  The higher form parts have coefficients in this subspace anyway by degree reasons, so that in fact 
  \[
    X\in \MC_k\Def(\Graphs_{V,n}^Z\to W\Graphs_{V,n}^Z)'). 
  \]
  Conversely, any such $X$ clearly satisfies the condition above and hence the Lemma follows.
\end{proof}

Next we consider the $\SL_\infty$-morphism 
\[
  \Phi: (\GGC_{V,n}^{Z,nil})_{ab} \to \Def(\Graphs_{V,n}^Z\to W\Graphs_{V,n}^Z)
\]
obtained from the action of $\GGC_{V,n}^{Z,nil}$ on $\Graphs_{V,n}^Z$
via Proposition \ref{prop:linfty from g on M action}.

\begin{lemma}\label{lem:restricted phi}
The $\SL_\infty$-morphism $\Phi$ above takes values in the $\SL_\infty$-subalgebra $\Def(\Graphs_{V,n}^Z\to W\Graphs_{V,n}^Z)'$, i.e., it is an $\SL_\infty$-morphism
\[
  \Phi: (\GGC_{V,n}^{Z,nil})_{ab} \to \Def(\Graphs_{V,n}^Z\to W\Graphs_{V,n}^Z)'.
\]
\end{lemma}
\begin{proof}
  Let $\gamma_1,\dots, \gamma_N\in \GGC_{V,n}^{Z,nil}$ and $x = \Phi_n(\gamma_1,\dots, \gamma_N)$ be the image of the $N$-ary part of the $\SL_\infty$-morphism.
  Then we have to check that the linear part $\pi(x)$ of $x$ has no part of degree 0.
  However, irrespective of the precise combinatorial form of $\Phi$ the image $x$ is built from $\gamma_1,\dots, \gamma_N$ using the action of $\HGC_{V,n}$ on $\Graphs_{V,n}$. This action preserves the numbers of edges in graphs.
  Hence potential contributions to the linear part of $x$ must come from the linear parts (corresponding to graphs without edges) of $\gamma_1,\dots, \gamma_N$. But $\GGC_{V,n}^{Z,nil}\subset \GGC_{V,n}^{Z,nil}$ is defined by removing graphs without edge in non-negative degree.
  Hence we may assume $\gamma_1,\dots, \gamma_N\in \GGC_{V,n}^{Z,nil}$ all have degrees $\leq 0$. Then $x$ has degree $\leq 1-N\leq 0$. Hence $\pi(x)$ is concentrated in negative degrees as desired. 
\end{proof}

\begin{prop}
  \label{prop:grPhi qiso}
Let $\Phi_1$ be the linear part of the $\SL_\infty$-morphism $\Phi$ of Lemma \ref{lem:restricted phi}.
Then the induced map on the associated edge-graded complexes
\[
\hgr\Phi_1 : \hgr_{\mG} \GGC_{V,n}^{nil,Z} \to \hgr_{\mG} \Def(\Graphs_{V,n}^Z\to W\Graphs_{V,n}^Z)'
\]
is a quasi-isomorphism.
\end{prop}
\begin{proof}
We check this separately on the pieces $\gr_{\mG}^k(-)$ with $k=0$ and with $k\geq 1$ edge.

First note that for $k\geq 1$ by construction
\[
  \gr_{\mG}^k \Def(\Graphs_{V,n}^Z\to W\Graphs_{V,n}^Z)'
  =
  \gr_{\mG}^k \Def(\Graphs_{V,n}^Z\to W\Graphs_{V,n}^Z).
\]
Furthermore we have 
\[
  \gr_{\mG}^k \GGC_{V,n}^{nil,Z}
  =
  \begin{cases}
    \gr_{\mG}^k \GGC_{V,n} & \text{for $k\geq 1$} \\
    (\gl_{V}')_{<0} & \text{for $k=0$}
  \end{cases}
\]
with vanishing differential.

Next reconsider the proof of Proposition \ref{prop:Hgrg vanishing} in section \ref{sec:Hgrg vanishing proof} above, and specialize it to $\MOp =\Graphs_{V,n}$. 
Then we have computed there that for all $k$
\[
  \gr_{\mG}^k \GGC_{V,n} \cong
  H(\gr_{\mG}^k \Def(\Graphs_{V,n}\to W\Graphs_{V,n}) ).
\]
The isomorphism is given by the action of a $\GGC_{V,n}$-element on the element of the deformation complex corresponding to the morphism 
$\gr_{\mG}\Graphs_{V,n}\to \gr_{\mG}W\Graphs_{V,n}$ that sends graphs with edges to zero and is the natural map on graphs without edges.
But this agrees with our map $\gr_{\mG}\Phi_1$.
Hence we may readily conclude that $\gr_{\mG}\Phi$ induces a quasi-isomorphism on all graded pieces $\gr_{\mG}^k(-)$ for $k\geq 1$.
It hence suffices to consider the graded pieces $\gr_{\mG}^0(-)$.
But this fits into a commutative diagram
\[
\begin{tikzcd}
  \gr_{\mG}^0 \GGC_{V,n}^{nil}
  =(\gl_{V}')_{<0}
  \ar{r}\ar{d}{\gr_{mG}^0\Phi_1}
  &
  \gr_{\mG}^0 \GGC_{V,n}
  =\gl_{V}'
  \ar{d}\ar{r}
  &
  (\gl_{V}')_{\geq 0}
  \ar{d}{=}
  \\
  \gr_{\mG}^0 \Def(\Graphs_{V,n}^Z\to W\Graphs_{V,n}^Z)'
  \ar{r}
  &
  \gr_{\mG}^0 \Def(\Graphs_{V,n}^Z\to W\Graphs_{V,n}^Z)
  \ar{r}
  &
  (\gl_{V}')_{\geq 0}
\end{tikzcd}\ .
\]
The rows are exact, and the middle vertical arrow is again a quasi-isomorphism by the proof of Proposition \ref{prop:Hgrg vanishing}.
Hence $\gr_{mG}^0\Phi_1$ must also be a quasi-isomorphism.
\end{proof}


\begin{proof}[Proof of Theorem \ref{thm:main aut}]
By Proposition \ref{prop:recognition} it suffices to check that the $\SL_\infty$-morphism $\Phi$ of Lemma \ref{lem:restricted phi} induces a weak equivalence of the associated Maurer-Cartan spaces.

To show this we would like to apply our variant of the Goldman-Millson Theorem, Theorem \ref{thm:GM tech}.
We will check first that it is applicable.
First, the $\SL_\infty$ algebras $\fg$ and $\fh$ of the Theorem will be $\fg=(\GGC_{V,n}^{nil,Z})_{ab}$ and $\fh=\Def(\Graphs_{V,n}^Z\to W\Graphs_{V,n}^Z)'$.
The filtrations $\mF^\bullet$ and $\mG^\bullet$ on these objects required by Theorem \ref{thm:GM tech} will be our filtrations $\mF^\bullet$ and $\mG^\bullet$ of sections \ref{sec:FFiltration} and \ref{sec:Gfiltration}.
The $\SL_\infty$-morphism $F$ of the Theorem is $\Phi$ as above. 
We check that all conditions of Theorem \ref{thm:GM tech} are satisfied (refer to the statement of the Theorem for the numbering):

\begin{enumerate}
  \item $\fg$ is clearly abelian and equipped with the two filtrations as required.
  The no-edge part is $\gr_{\mG}^0 \GGC_{V,n}^{nil}
  =(\gl_{V}')_{<0}$ and has no elements of degree zero, hence $
    \fg^0 = \mG^1\fg^0$ as required.
  \item $\fh$ is also equipped with two compatible complete filtrations as required.
    \item We check that $\SL_\infty$ operations $\mu_k$ satisfy \eqref{equ:mu in G1} for $k\geq 2$:
    Indeed, consider again the definition of the $\SL_\infty$ structure via the generating function $\mU(f)$ in \eqref{equ:Def gen func}.
    A possible term in a higher bracket in $\fh$ that is not in $\mG^1\fh$ corresponds to a nonzero term in $\mU(f)(\Gamma)$ that contains $f$ at least quadratically, where $\Gamma\in \Graphs_{V,n}^Z$ is an internally connected graph with zero edges. But internally connected graphs with zero edges consist just of external vertices plus one decoration in $V$. The cooperadic coaction cannot produce from such a graph one that is not internally connected. But that would be required to produce a dependence on $f$ that is at least quadratic, cf. the proof of Lemma \ref{lem:freecofree}. Hence \eqref{equ:mu in G1} follows.

    \item It is clear that $\mG^p\fg \subset \mF^p\fg$ and $\mG^p\fh \subset \mF^p\fh$, hence the filtrations are compatible in the required sense.
    \item The morphism $F=\Phi$ is compatible with the filtrations $\mF$ and $\mG$ by construction, since it uses only the $\La$ Hopf comodule structure on $\Graphs_{V,n}$ and the $\HGC_{V,n}^{nil}$-action, which are compatible with the filtrations.
    \item The linear component $F_1=\Phi_1$ induces an isomorphism on the associated $\mG$-graded complexes by Proposition \ref{prop:grPhi qiso}.
  \end{enumerate}
  Hence Theorem \ref{thm:GM tech} is applicable in our situation and we conclude that the map 
  \[
   \Exp_\bullet(\GGC_{V,n}^{nil,Z})=\MC_\bullet((\GGC_{V,n}^{nil,Z})_{ab})
   \to 
   \MC_\bullet(\Def(\Graphs_{V,n}^Z\to W\Graphs_{V,n}^Z)') =\Map(\Graphs_{V,n}^Z,W\Graphs_{V,n}^Z)
  \]
  is a weak equivalence of simplicial sets, so that Theorem \ref{thm:main aut} follows.

\end{proof}


\section{Two smaller variants of $\GGC_{\bar H,n}$: $\GC_{H,n}$ and $\GCp_{H,n}$}\label{sec:GC from HGC}

In this section we introduce two smaller versions of our main dg Lie algebra $\GGC_{\bar H,n}$, that arise in connection to the study of configuration spaces of orientable manifolds that are either closed, or have a boundary that is an $n-1$-sphere.
In these two cases either the associated diagonal element in $H(M)\otimes H(M)$, or the reduced diagonal element in $\bar H(M)\otimes \bar H(M)$ is non-degenerate.
This then allows us to simplify the associated graph complexes significantly.
In this section we shall introduce the algebraic preliminaries and the smaller graph complexes. The connection to configuration spaces is then worked out in section \ref{sec:applications} below.

\subsection{Maurer Cartan element associated to a non-degenerate diagonal class}\label{sec:ZDeltadef}
Let again $V$ be a finite dimensional positively graded vector space, $V_1:=\Q1\oplus V$ and $n\geq 2$ an integer.
We shall assume that a non-degenerate $(-1)^n$-symmetric element of degree $n$
\[
\Delta \in V_1 \otimes V_1
\]
is given, that we will call the diagonal.
The symmetry condition means that $\tau \Delta= (-1)^n\Delta$, with $\tau : V_1 \otimes V_1\to V_1 \otimes V_1$ the transposition of factors.
It shall be convenient to fix a graded basis $\{e_0:=1, e_1,\dots, e_N=:\omega\}$ of $V_1$, and let $g_{ij}$ be the matrix corresponding to the diagonal element $\Delta$, so that 
\[
\Delta = \sum_{i,j=0}^N g_{ij} e_i\otimes e_j,
\]
with $g_{ij} = (-1)^{n+|e_i||e_j|}g_{ji}$.

Note that the positivity of the grading on $V$ and the existence of the non-degenerate element $\Delta$ imply that $V$ is concentrated in degrees $1,\dots, n$, and one-dimensional in degree $n$. We assume that the numbering on the basis elements is such that $e_N=\omega$ is a basis of the degree $n$ subspace $V^n$, and one normalizes such that $g_{0N}=1$. Then, then by symmetry and degree reasons, 
\begin{align*}
  g_{N0}&=(-1)^n & g_{0j}&=g_{j0}=g_{N0}=g_{0N} = 0 \quad\text{ for $j=1,\dots,N-1$.}
\end{align*}

Let $\{f_0,\dots,f_N\}$ be the corresponding dual basis of $V_1^*$, and let $h_{ij}$ be the inverse matrix of $g_{ij}$.
To $\Delta$ we can then associate the degree 1 element 
\begin{equation}\label{equ:ZDeltapic}
Z_\Delta = 
\underbrace{
  \frac 12
\sum_{i,j=0}^N g_{ij} 
\begin{tikzpicture} 
\node (v) at (0,0) {$e_i$};
\node (w) at (1,0) {$e_j$};
\draw (v) edge[bend left] (w);
\end{tikzpicture}
}_{=: L_\Delta}
-
\underbrace{
  \frac 12
\sum_{i,j=1}^{N-1}
(-1)^{|f_j|}
h_{ij}
\begin{tikzpicture}[yshift=-.5cm]
 \node (v) at (0,0) {$\omega$};
\node[int,label=90:{$f_if_j$}] (w) at (0,1) {};
\draw (v) edge (w);
\end{tikzpicture}
 }_{=:Y_\Delta}
\in \GGC_{V,n}.
\end{equation}
To fix signs, we order here the decorations from left to right and top to bottom as displayed, and orient the edges to the right and downwards.
Note that the factor $(-1)^{|f_i|}$ ensures that for odd $n$ the summand $Y_\Delta$ is non-trivial, despite the matrix $h_{ij}$ being anti-symmetric.

\begin{lemma}
  Let $V=\oplus_{j>0} V^j$ be a positively graded finite dimensional vector space such that $V_1=\Q1\oplus V$ has vanishing Euler characteristic, i.e., $1+\sum_j(-1)^j\vdim(V_j)=0$. Let $\Delta\in V_1\otimes V_1$ be a nondegenerate element as above. Then the element $Z_\Delta\in \GGC_{V,n}$ is a Maurer-Cartan element, i.e., $\delta Z_\Delta+\frac 1 2 [Z_\Delta,Z_\Delta]=0$. 
\end{lemma}
\begin{proof}
  In the following, we shall order all decorations in pictures from left to right and top to bottom, and likewise all edges.
  Edges are oriented from top to bottom, and from left to right.
With these conventions we first compute 
\begin{align*}
\delta L_\Delta &= 
-
\begin{tikzpicture}[yshift=-.5cm]
  \node (v) at (0,0) {$\omega$};
  \node[int] (w) at (0,1) {};
  \draw (v) edge (w) (w) edge[loop above] (w);
\end{tikzpicture}
+
\sum_{i=0}^{N-1}
\begin{tikzpicture}[yshift=-.5cm]
  \node (v) at (0,0) {$e_i$};
  \node (vv) at (1,0) {$\omega$};
  \node[int,label=90:{$f_i$}] (w) at (0.5,1) {};
  \draw (v) edge (w) (w) edge (vv);
\end{tikzpicture} .
\end{align*}
while $\delta Y_\Delta = 0$.
Furthermore,
\begin{align*}
[L_\Delta, Y_\Delta] 
&=
-
\underbrace{ \frac 14 \sum_{i,j=1}^{N-1} 2g_{ij}h_{ij} (-1)^{|e_i||e_j|+|e_j|+n} }_{=(*)} 
\begin{tikzpicture}[yshift=-.5cm]
  \node (v) at (0,0) {$\omega$};
  \node[int] (w) at (0,1) {};
  \draw (v) edge (w) (w) edge[loop above] (w);
\end{tikzpicture}
+
\sum_{i,j=1}^{N-1} (\text{const.})
\underbrace{
\begin{tikzpicture}[yshift=-.5cm]
  \node (v) at (0,0) {$\omega$};
  \node (vv) at (1,0) {$\omega$};
  \node[int,label=90:{$f_i f_j$}] (w) at (0.5,1) {};
  \draw (v) edge (w) (w) edge (vv);
\end{tikzpicture} 
}_{=0 \text{ by symmetry} }
-
\sum_{i,k=1}^N
\underbrace{\frac 1 4
\sum_{j=1}^N 4 g_{ij} h_{jk}(-1)^{|f_i| + n +n+ |e_i||f_k|}
}_{(**)}
\begin{tikzpicture}[yshift=-.5cm]
  \node (v) at (0,0) {$e_i$};
  \node (vv) at (1,0) {$\omega$};
  \node[int,label=90:{$f_k$}] (w) at (0.5,1) {};
  \draw (v) edge (w) (w) edge (vv);
\end{tikzpicture} 
\\&=
\begin{tikzpicture}[yshift=-.5cm]
  \node (v) at (0,0) {$\omega$};
  \node[int] (w) at (0,1) {};
  \draw (v) edge (w) (w) edge[loop above] (w);
\end{tikzpicture}
-
\sum_{i=1}^{N-1} 
\begin{tikzpicture}[yshift=-.5cm]
  \node (v) at (0,0) {$e_i$};
  \node (vv) at (1,0) {$\omega$};
  \node[int,label=90:{$f_i$}] (w) at (0.5,1) {};
  \draw (v) edge (w) (w) edge (vv);
\end{tikzpicture} 
.
\end{align*}
For $n$ odd the tadpole graph vanishes by symmetry. For $n$ even its 
numeric prefactor $(*)$ is 
\[
  \frac 14 
  \sum_{i,j=1}^{N-1} 2g_{ij}h_{ij} 
  (-1)^{|e_i||e_j|+|e_j|} 
=
\frac 12
\sum_{i,j=1}^{N-1}
g_{ij}h_{ji} (-1)^{|e_j|}
=
\frac 12
\sum_{j=1}^{N-1}
\sum_{i=0}^{N}
g_{ij}h_{ji} (-1)^{|e_j|}
=
\frac 12
\sum_{j=1}^{N-1}
(-1)^{|e_j|}
= -1,
\]
using the Euler characteristic condition.
Similarly, the coefficient $(**)$ is 
\[
\frac14 \sum_{j=1}^N 4 g_{ij} h_{jk}(-1)^{|f_i| + n +n+ |e_i||f_k|}
=
\delta_{ik} (-1)^{2|f_i|}= \delta_{ik}.
\]
Trivially, we have that $[Y_\Delta, Y_\Delta]=0$, since $f_N$ does not appear as a decoration in $Y_\Delta$.
Hence it follows that indeed $\delta Z_\Delta+\frac 1 2 [Z_\Delta,Z_\Delta]=0$.
\end{proof}

\subsection{Definition of the dg Lie algebra $\GCx_{V, n}$}
Let $V$, $\Delta\in V_1\otimes V_1$ and the Maurer-Cartan element $Z_\Delta$ be as in the previous subsection. We consider the twisted dg Lie algebra $\GGC_{V,n}^{Z_\Delta}$, and in particular 3 specific dg Lie subalgebras.

\subsubsection{An orthosymplectic Lie algebra}

First, define $\gl_{V_1}=\iHom(V_1,V_1)$ to be the graded Lie algebra consisting of linear maps $V_1\to V_1$.
As before we use the notation
\[
\gl'_V \cong \iHom(V,V_1) \cong \{\phi\in \gl_{V_1}\mid \phi(1)=0 \} \subset \gl_{V_1}
\]
be the graded Lie subalgebra composed of linear maps that annihilate the special element $1\in V_1$. 
We understand $\gl'_{V}\subset \HGC_{V,n}$ as a graded Lie subalgebra, consisting of the graphs with no vertices and no edges, just hairs decorated from above (by $V^*$) and below (by $V_1$). 
Let 
\[
  \osp_{V} = \{\phi\in  \gl'_V\mid 
  \phi\cdot \Delta = \sum_{i,j=0}^N g_{ij}(\phi(e_i)\otimes e_j + (-1)^{|\phi||e_i|} e_i\otimes \phi(e_j))=0\}
  \subset \gl'_V\subset \HGC_{V,n}
\]
be the graded Lie subalgebra of morphisms that annihilate the diagonal element $\Delta$. Note that $\osp_{V}$ depends on the choice of $\Delta$, though we do not indicate this in the notation.

\begin{lemma}\label{lem:LDelta phi commute}
The element $L_\Delta\in \GGC_{V,n}$ defined in \eqref{equ:ZDeltapic} commutes with all elements $\phi \in \osp_V\subset \GGC_{V,n}$,
\[
  [\phi, L_\Delta] =0.
\]
\end{lemma}
\begin{proof}
The decorations of the graph in $L_\Delta$ are precisely $\Delta$, which is annihilated by $\phi$ by definition.
\end{proof}

We shall also define two notable graded Lie subalgebras of $\osp_V$.
First, the degree zero elements $\osp_V^0\subset \osp_V$ are an ordinary Lie algebra. This graded Lie algebra is isomorphic to a direct sum of an orthogonal Lie algebra $\mathfrak o(p,q)$ or a symplectic Lie algebra and general linear Lie algebras, hence the notation.
Furthermore, for the topological applications the negatively graded part 
\[
\osp_V^{nil} = \{\phi \in \osp_V\mid |\phi|<0\}  
\]
is of relevance. It is a nilpotent graded Lie algebra.

\subsubsection{Definition if $\GC_{V, n}$}
\label{sec:GCVn def}

Next, let $\tGGC_{V,n}$ be an extended version of $\GGC_{V,n}$ that includes graphs with zero hairs. 
In other words, we may repeat the definition of $\GGC_{V,n}$ in sections \ref{sec:HGC upper 1} and \ref{sec:HGC upper 2}, except that we also allow graphs without hairs.
Concretely, as graded vector space we have
\[
  \tGGC_{V,n} = \GC_{V,n}[1] \oplus \GGC_{V,n},
\]
where the new piece $\GC_{V,n}[1]$ is the summand spanned by graphs with no hairs. Again $\tGGC_{V,n}$ is a dg Lie algebra, with the differential and Lie bracket following the same combinatorial formulas as that of $\GGC_{V,n}$.

Now we split the differential on the $Z_\Delta$-twisted complex $\tGGC_{V,n}^{Z_{\Delta}}$
\[
\delta^{Z_\Delta} =\delta + [Z_\Delta,-]
=
\delta_0+\delta_1  
\]
into a piece $\delta_0$ that leaves the number of hairs constant or decreases it and a piece $\delta_1$ that adds one hair.
\begin{lemma}\label{lem:GCdeltas}
For any $\Gamma\in \GC_{V,n}[1]\subset \tGGC_{V,n}$ we have that 
\[
  \delta_1^2\Gamma = (\delta_0\delta_1+\delta_1\delta_0)\Gamma = \delta_0^2 \Gamma = 0.
\]
\end{lemma}
\begin{proof}
  We have that 
  \[
    0 = (\delta^{Z_\Delta})^2\Gamma = \delta_1^2\Gamma
    + (\delta_0\delta_1+\delta_1\delta_0)\Gamma 
    +
    \delta_0^2 \Gamma.
  \]
  The first term on the right-hand side is a linear combination of graphs with exactly two hairs, and is the only one containing such graphs.
  Similarly, note that $\delta_0$ cannot produce graphs with no hairs from graphs that have hairs. Hence $(\delta_0\delta_1+\delta_1\delta_0)\Gamma$ is a linear combination of graphs with exactly one hair. Finally $\delta_0^2 \Gamma$ only contains graphs with no hairs. Hence all three terms must vanish separately.
\end{proof}

It follows that $(\GC_{V,n},-\delta_0)$ is a dg vector space.
Furthermore, we may define a Lie bracket on $\GC_{V,0}$ by the formula
\begin{equation}\label{equ:GC bracketdef}
[x,y]_{\GC} := 
(-1)^{|x|}[\delta_1 x,y]_{\GGC}
=
[x,\delta_1 y]_{\GGC},   
\end{equation}
where the Lie bracket on the right is taken in $\tGGC_{V,n}$, as is indicated by the subscripts.
The right-most equality will be shown together with the next Proposition. 
This construction is an instance of a derived bracket, see \cite{Kos} for a general review of derived brackets.
\begin{prop}\label{prop:GC well defined}
The differential $-\delta_0$ and the bracket $[-,-]_{\GC} $ define a dg Lie algebra structure on $\GC_{V,n}$, such that the map 
\begin{equation}\label{equ:delta1}
    \delta_1 :  (\GC_{V,n},-\delta_0, [-,-]_\GC) \to  \GGC_{V,n}^{Z_\Delta}
\end{equation}
is an injective dg Lie algebra morphism.
\end{prop}
\begin{proof}
  First we show that \eqref{equ:delta1} is an injective map.
  For $\Gamma\in \GC_{V,n}$ a graph $\delta_1\Gamma$ is obtained by summing over all ways of attaching one $\omega$-decorated hair to a vertex, or replacing a decoration $a\in V^*$ by a hair labeled with the corresponding dual element $b\in V_1$ of degree $<n$.
  A one sided inverse to $\delta_1$ may hence be defined by cutting the unique $\omega$-hair if present, and dividing by the number of vertices. In particular, \eqref{equ:delta1} is injective.

  Next, the map \eqref{equ:delta1} intertwines the differentials:
  \[
  \delta^{Z_\Delta}  (\delta_1\Gamma)
  =
  (\delta_0+\delta_1)  \delta_1\Gamma
  =
  -\delta_1\delta_0\Gamma,
  \]
using Lemma \ref{lem:GCdeltas}.
Next, for $x,y\in \GC_{V,n}$ of degrees $|x|$, $|y|$ we have, by the compatibility of the differential and Lie bracket in $\tGGC_{V,n}^{Z_\Delta}$,
\[
  0
  =
(\delta_0+\delta_1)\underbrace{[x,y]_{\HGC}}_{=0}
=
[(\delta_0+\delta_1)x,y]_{\HGC}
+(-1)^{|x|+1}
[x,(\delta_0+\delta_1)y]_{\HGC}
=
[\delta_1 x,y]_{\HGC}
+(-1)^{|x|+1}
[x,\delta_1y]_{\HGC}.
\]
This shows the last equality of \eqref{equ:GC bracketdef}.
Similarly,
\[
(\delta_0+\delta_1)[x,\delta_1 y]_{\HGC}
=
[(\delta_0+\delta_1)x,\delta_1y]_{\HGC}
+(-1)^{|x|+1}
[x,(\delta_0+\delta_1)\delta_1 y]_{\HGC}
=
[(\delta_0+\delta_1)x,\delta_1y]_{\HGC}
+(-1)^{|x|}
[x,\delta_1 \delta_0y]_{\HGC}.
\]
Taking the part with exactly one hair on both sides we get 
\[
  \delta_1[x,\delta_1 y]_{\HGC}
  =
  [\delta_1 x,\delta_1 y]_{\HGC},
\]
and this shows that the map \eqref{equ:delta1} intertwines the Lie brackets. By injectivity of \eqref{equ:delta1} the well-definedness of the dg Lie algebra $\GC_{V,n}$ (i.e., the Jacobi identity and compatibility relations) follows.

\end{proof}

Alternatively, the well-definedness of the dg Lie algebra structure could also be deduced from standard results for derived brackets, see for example \cite[Theorem 1.1]{Kos}.

Let us also illustrate the combinatorial form of the differential and Lie bracket on $\GC_{V,n}$.
The differential on $\GGC_{V,n}^{Z_\Delta}$ and $\tGGC_{V,n}^{Z_\Delta}$ consists of 4 terms
\[
\delta = \delta_{split}+\delta_{Y}+ \delta_{glue} + \delta_{join} +\delta_{hair}
\]
with $\delta_{split}$ splitting a vertex into two
\[
  \delta_{split}   \begin{tikzpicture}[baseline=-.65ex]
    \node[int] (v) {};
    \draw (v) edge +(-.5,.5) edge +(-.2,.5) edge +(.2,.5) edge +(.5,.5)edge +(.2,-.5) edge +(-.2,-.5);
    \end{tikzpicture}
    =\sum
  \begin{tikzpicture}[baseline=-.65ex]
  \node[int] (v) at (0,0) {};
  \node[int](w) at (0,.3) {};
  \draw (v) edge +(.2,-.5) edge +(-.2,-.5) edge +(-.5,.5) edge +(.5,.5) edge (w) (w) edge +(-.2,.5) edge +(.2,.5) ;
  \end{tikzpicture}
\]
The piece $\delta_{Y}=-[Y_\Delta,-]$ comes from the action of the element $Y_\Delta$ in \eqref{equ:ZDeltapic}.

The piece $\delta_{glue}$ glues two decorations by an edge, according to the nondegenerate pairing given
\[
  \delta_{glue} 
  \begin{tikzpicture}[baseline=-.65ex, scale=.7]
    \node[draw, circle] (c) at (0,0) {$\cdots$};
    \node[int, label=90:{$\alpha$}] (v) at (-1,1) {};
    \node[int, label=90:{$\beta$}] (w) at (1,1) {};
    \draw (v) edge (c.west) edge (c.north) (w) edge (c.north) edge (c.east);
  \end{tikzpicture}
  = 
  (\alpha,\beta)
  \begin{tikzpicture}[baseline=-.65ex, scale=.7]
    \node[draw, circle] (c) at (0,0) {$\cdots$};
    \node[int] (v) at (-1,1) {};
    \node[int] (w) at (1,1) {};
    \draw (v) edge (c.west) edge (c.north) (w) edge (c.north) edge (c.east) edge (v);
  \end{tikzpicture}.
\]
The piece $\delta_{join}$ fuses multiple hairs into one vertex with one hair attached is not so relevant for the futher discussion.
Finally $\delta_{hair}$ attaches a hair at a decoration
\[
  \delta_{hair} \,
  \begin{tikzpicture}[baseline=-.65ex, scale=.7]
    \node[draw, circle] (c) at (0,0) {$\cdots$};
    \node[int, label=90:{$\beta$}] (w) at (1,1) {};
    \draw  (w) edge (c.north) edge (c.east);
  \end{tikzpicture} = 
  \begin{tikzpicture}[baseline=-.65ex, scale=.7]
    \node[draw, circle] (c) at (0,0) {$\cdots$};
    \node[int] (w) at (1,1) {};
    \node (h) at (2.2,1) {$\beta^*$};
    \draw  (w) edge (c.north) edge (c.east) edge (h);
  \end{tikzpicture}
  .
\]
Here $\beta^*$ is the image of $\beta\in V_1^*$ under the identification $V_1\cong V_1^*$ given by our pairing. Also it should be noted that implicity all vertices are considereed as decorated by the symbol $1$.

The differential on $\GC_{V,n}$ is given by $-\delta_{split}-\delta_{glue}-\delta_Y$, while the part $\delta_{hair}$ induces the map from $\GC_{V,n}\to \GGC_{V,n}$.

The Lie bracket on $\GC_{V,n}$ is defined similarly to the operation $\delta_{glue}$, we just apply it to two different graphs.
\[
  \left[
  \begin{tikzpicture}[baseline=-.65ex, scale=.7]
    \node[draw, circle] (c) at (0,0) {$\cdots$};
    \node[int, label=90:{$\alpha$}] (w) at (1,1) {};
    \draw (w) edge (c.north) edge (c.east);
  \end{tikzpicture}
  ,
  \begin{tikzpicture}[baseline=-.65ex, scale=.7]
    \node[draw, circle] (c) at (0,0) {$\cdots$};
    \node[int, label=90:{$\beta$}] (v) at (-1,1) {};
    \draw (v) edge (c.west) edge (c.north);
  \end{tikzpicture}
  \right]
  =
  (\alpha,\beta)\, 
  \begin{tikzpicture}[baseline=-.65ex, scale=.7]
    \node[draw, circle] (c1) at (0,0) {$\cdots$};
    \node[int] (w) at (1,1) {};
    \draw (w) edge (c1.north) edge (c1.east);
    \node[draw, circle] (c2) at (3,0) {$\cdots$};
    \node[int] (v) at (2,1) {};
    \draw (v) edge (c2.west) edge (c2.north) edge (w);
  \end{tikzpicture}
\]

\subsubsection{Incorporating $\osp_V$, and definition of $\GCx_{V,n}$}
We will next combine the graded subpaces $\osp_V\subset \GGC_{V,n}$ and $\GC_{V,n}\subset \GGC_{V,n}^{Z_\Delta}$, and consider the graded subspace 
\[
\GCx_{V,n} := \osp_V\oplus \GC_{V,n}\subset \GGC_{V,n}^{Z_\Delta}.
\]
Note that here we use the inclusion by $\delta_1: \GC_{V,n}\to \GGC_{V,n}^{Z_\Delta}$ as in Proposition \ref{prop:GC well defined} to understand $\GC_{V,n}$ as a subspace of $\GGC_{V,n}^{Z_\Delta}$.
The main result is as follows.

\begin{prop}\label{prop:GCx def}
The graded subspace $\GCx_{V,n}\subset \GGC_{V,n}^{Z_\Delta}$ is closed under the bracket and differential and hence a dg Lie subalgebra.
\end{prop}

The main computation going into the proof is the following.
For $\alpha\in V^*$ we define the graph 
\[
X_\alpha := \sum_{j=1}^{N-1} (-1)^j h_{ij} 
\begin{tikzpicture}
\node[int,label=90:{$\alpha f_i f_j$}] (v) at (0,0) {};
\end{tikzpicture}  \in \GC_{V,n}\,,
\]
with $(h_{ij})$ the matrix inverse to our pairing matrix $(g_{ij})$.
In particular, for $\phi\in \osp_V$ we are interested in the element 
\[
  X_\phi := X_{\phi^*(f_0)}.
\]
We then have:
\begin{lemma}\label{lem:osp diff}
The differential of $\phi\in \osp_V\subset \GGC_{V,n}^{Z_\Delta}$ is 
\[
\delta^{Z_\Delta} \phi = \delta_1 X_\phi.
\]
\end{lemma}
\begin{proof}
This is a somewhat tedious computation that we leave to the reader.
\end{proof}

\begin{proof}[Proof of Proposition \ref{prop:GCx def}]
  We have already seen in Proposition \ref{prop:GC well defined} that $\GC_{V,n}\subset \GGC_{V,n}^{Z_\Delta}$ is a dg Lie subalgebra, and we know that $\osp_V\subset \GGC_{V,n}^{Z_\Delta}$ is a graded Lie subalgebra.
  By Lemma \ref{lem:osp diff} we know that the differential of elements in $\osp_V$ is contained in $\GC_{V,n}$, and hence $\GCx_{V,n}$ is closed under the differential.

  It remains to be checked that the Lie bracket of any $\phi\in \osp_{V}$ and $\Gamma\in \GC_{V,n}$ is contained in $\GC_{V,n}$.
  Explicitly, suppose that $\Gamma= \delta_1\gamma$ for $\gamma\in \GC_{V,n}[1]\subset \tGGC_{V,n}$.
  By Lemma \ref{lem:LDelta phi commute} we know that $[\phi, L_\Delta]=0$ and hence 
  \[
  [\phi , [L_\Delta, \gamma]]
  =
   (-1)^{|\phi|} [L_\Delta, [\phi,\gamma]].
  \]
  Take the part on both sides with exactly one hair to obtain 
  \[
    [\phi, \delta_1\gamma] 
    =
    (-1)^{|\phi|} 
    \delta_1 [\phi,\gamma]
  \]
  The left-hand side is the bracket of $\phi$ and $\Gamma$.
  Since $[\phi,\gamma]\in \GC_{V,n}[1]\subset \tGGC_{V,n}$ the right-hand side is in $\GC_{V,n}\subset\GGC_{V,n}$.
\end{proof}

The action of $\phi$ on $\Gamma$ is the "obvious" action, by letting the dual morphism $\phi^*$ act on all decorations in our graph, and adding decorations $\phi^*(f_0)$ to vertices.


One shortcoming of the dg Lie algebra $\GCx_{V,n}$ is that it is not (pro-)nilpotent since $\osp_V$ is not.
However, for many applications it suffices to consider the slightly smaller but pro-nilpotent dg Lie subalgebra 
\begin{equation}\label{equ:GCex def}
\GCex_{V,n} = \osp_V^{nil} \oplus \GC_{V,n} \subset \GCx_{V,n},
\end{equation}
obtained by restricting to the negative degree Lie subalgebra $\osp_V^{nil}=(\osp_V)_{<0}\subset \osp_V$.
(The direct sum above is meant in the sense of graded vector spaces, not of dg Lie algebras.)

\subsection{Cohomology of $\GCx_{V,n}$ and $\GGC_{V,n}^{Z_\Delta}$}

Finally we want to compare the cohomology of our dg Lie algebras $\GCx_{V,n}$ and $\GGC_{V,n}^{Z_\Delta}$. 
To this end we first describe several elements in $\GGC_{V,n}$.
For $\alpha,\beta \in V_1$ one has the elements
\begin{equation}\label{equ:L alpha beta}
L_{\alpha\beta} =
\begin{tikzpicture}[baseline=-.65ex]
\node (v) at (0,0) {$\alpha$};
\node (w) at (1.3,0) {$\beta$};
\draw (v) edge (w);
\end{tikzpicture}\in \GGC_{V,n},
\end{equation}
i.e., two hairs decorated by $\alpha$ and $\beta$, joined by an edge. We will need specifically the element $L_{\omega\omega}$ with both decorations given by the top degree basis element $\omega\in V$. 
Then we compute (cf. \eqref{equ:ZDeltapic}) that 
\[
  \delta L_{\omega\omega}
  =
  [ L_\Delta, L_{\omega\omega}] 
  =
  0
\]
and 
\[
  [ Y_\Delta, L_{\omega\omega}] 
  =
  \pm 
  \begin{tikzpicture}[yshift=-.5cm]
    \node (v) at (0,0) {$\omega$};
    \node (vv) at (1,0) {$\omega$};
    \node[int] (w) at (0.5,1) {};
    \draw (v) edge (w) (w) edge (vv);
  \end{tikzpicture}
  =0
\]
by symmetry. Hence we have that $\delta L_{\omega\omega}+[Z_\Delta,L_{\omega\omega}]=0$ so that $L_{\omega\omega}$ is a cocycle in $\GGC_{V,n}^{Z_\Delta}$ of cohomological degree $n+1$.

Next, we consider the hedgehog graphs with $k=1,2,\dots$ vertices and $k$ hairs 
\[
H_k =
\begin{tikzpicture}[baseline=-.65ex]
  \node (v0) at (0:1) {$\cdots$};
  \foreach \a [count=\ai, evaluate=\prev using int(\ai -1)] in {60,120,180,240,300}
    { 
      \node[int] (v\ai) at (\a:1) {};
      \node (w\ai) at (\a:1.5) {$\omega$};
      \draw (v\ai) edge (w\ai);
      \draw (v\ai) edge (v\prev);
    }
  \draw (v0) edge (v5);
  \end{tikzpicture}\in\GGC_{V,n}\, .
\]
We note that for $k$ and $n$ either both even or both odd we have $H_k=0$ by symmetry.
One checks that 
\[
  \delta H_k = [L_\Delta,H_k]=[Y_\Delta,H_k] =0,
\]
and hence the $H_k$ are also cocycles in $\GGC_{V,n}^{Z_\Delta}$, of degree $2k$.

\begin{prop}\label{prop:GC HGC comparison}
The inclusion 
\begin{equation}\label{equ:GC HGC comparison}
  \GCx_{V,n}
  \oplus
  \Q L_{\omega\omega}
  \oplus 
  \bigoplus_{k\geq 1 \atop k\equiv n+1 \, \mathrm{ mod }\,  2}
  \Q H_k
  \to \GGC_{V, n}^{Z_\Delta}
\end{equation}
is a quasi-isomorphism of dg vector spaces.
It furthermore induces a quasi-isomorphism already on the associated graded complexes with respect to the filtration by the number of vertices.
\end{prop}
\begin{proof}
  First note that the map is indeed a map of dg vector spaces, since both $L_{\omega\omega}$ and the $H_k$ are cocycles.
  It suffices to show the statement about the associated graded spaces -- the fact that the map is a quasi-isomorphism then follows from standard spectral sequence arguments.
  
  We leave it to the reader to check that the cohomology of the zero-vertex graded component of $\GGC_{V,n}^{Z_\Delta}$ is precisely $\Q L_{\omega\omega} \oplus \osp^{nil}_V$. 

  We then consider the associated graded piece with $k\geq 1$ vertices.
  Note that the mapping cone of the inclusion $\GC_{V,n}\to \GGC_{V.n}^{Z_\Delta}$ is identified with the extended complex $\tGGC_{V,n}^{Z_\Delta}$. Hence we can equivalently show that the cohomology of $\gr_{vert}^k\tGGC_{V,n}^{Z_\Delta}$ is spanned by $H_k$.
  Note that the differential on the associated graded has two pieces: $\delta_{glue}$ adding an edge between decorations, and $\delta_{hair}$, adding a hair at a decoration.
  To compute the cohomology, let us take the spectral sequence for the bounded filtration by the number of internal (i.e., non-hair) edges. On the associated (edge-)graded only the piece $\delta_{hair}$ of the differential survives.
  The resulting complex can easily be seen to be a direct summand of the tensor product $\bigotimes_j V_j^{\otimes k}$ of $k$ complexes, one for each vertex. 
  Concretely, $V_j$ is spanned by the possible decorations and hairs attachable to vertex $j$.
  Suppose that vertex $j$ has $r_j$ adjacent non-hair edges. Then the cohomology of $V_j$ is easily computed with the result
  \[
  H(V_j)=
  \begin{cases}
    0 & \text{if $e_j\neq 2$} \\
    \Q[-1] & \text{if $r_j=2$}
  \end{cases}
  \]
  The surviving cohomology for $e_j=2$ is spanned by an $\omega$-decorated hair at vertex $j$. But then the only possible graphs with $e_j=2$ for all $j$ are the hedgehog diagrams $H_j$.
  Since they are closed, the inner spectral sequence abuts here and we finish the proof of the proposition.
\end{proof}

We note that the dg vector subspace of $\GGC_{V,n}^{Z_\Delta}$ on the left-hand side of \eqref{equ:GC HGC comparison} is not closed under the Lie bracket. 
However, at least the smaller dg vector subspace
\[
\Q H_1 \oplus \GC_{V,n} \subset \GGC_{V,n}^{Z_\Delta},  
\]
is closed under the Lie bracket, and hence forms a dg Lie subalgebra.
Here $H_1$ is the smallest of the hedgehog cocycles,
\[
  H_1 = 
  \begin{tikzpicture}
  \node[int] (v) at (0,.3) {};
  \node (w) at (0,-.3) {$\omega$};
  \draw (v) edge[loop above] (v) edge (w);
  \end{tikzpicture}.
\]
The element $H_1$ is of degree $2$ and $H_1=0$ if $n$ is odd by symmetry.
Let $\mF_{vert}^\bullet$ be the descending complete filtration on $\GGC_{V,n}^{Z_0}$ by the number of vertices, so that $\mF_{vert}^p\GGC_{V,n}^{Z_0}$ is spanned by graphs with $\geq p$ vertices.

\begin{cor}\label{cor:HGC GC}
  The inclusion of dg Lie algebras
  \[
    \GC_{V,n}' := \underbrace{\Q H_1}_{=0 \text{ for $n$ odd}} \oplus \ \GC_{V,n} \subset 
    \mF^1_{vert} \GGC_{V,n}^{Z_\Delta}
  \]
  induces a quasi-isomorphism in degrees $\leq 2$ on the associated graded Lie algebras with respect to the vertex filtration. Furthermore, the induced maps
  \begin{equation}\label{equ:GC HGC MC}
    \MC_\bullet(\GC_{V,n}) \to \MC_\bullet(\GC_{V,n}') \to \MC_\bullet(\mF^1_{vert} \GGC_{V,n}^{Z_\Delta})
  \end{equation}
  are weak equivalences.
\end{cor}
\begin{proof}
The quasi-isomorphism assertion follows immediately from Proposition \ref{prop:GC HGC comparison}, since all classes on the left-hand side of \eqref{equ:GC HGC comparison} that are not contained in $\GC_{V,n}'$ are either of degree $\geq 3$ or have no vertex.

Hence the second weak equivalence in \eqref{equ:GC HGC MC} is an immediate consequence of the Goldman-Millson Theorem, see \cite{DolRog}. Note that a priori the Theorem stated in \cite{DolRog} requires that the morphism induces a quasi-isomorphism on the associated graded complexes in all degrees. However, elements of degrees $\geq 3$ never appear in the proof, and the condition can be relaxed. 

Finally, one has to check that $\MC_\bullet(\GC_{V,n}) \to \MC_\bullet(\GC_{V,n}')$ is a weak equivalence. This can also be covered by a version of the Goldman-Millson Theorem with slightly relaxed conditions relative to \cite{DolRog}. However, in our case it is easiest to note that in fact $\MC_\bullet(\GC_{V,n}) = \MC_\bullet(\GC_{V,n}')$, since the additional class $H_1$ never appears in a Lie bracket or differential or in another role in the definition of the Maurer-Cartan space.
\end{proof}

\subsection{Non-degenerate reduced diagonal class and the dg Lie algebra $\GCp_{H,n}$}

Next we shall assume that we are given a non-degenerate $(-1)^n$-symmetric "diagonal" element of degree $n$
\[
  \bar \Delta\in V \otimes V.  
\]
Again we pick a graded basis $e_1,\dots, e_N$ of $V$ and define the matrix $g_{ij}$ such that 
\[
  \bar \Delta = \sum_{j=1}^N g_{ij} e_i\otimes e_j \in V\otimes V,
\]
with $g_{ji}=(-1)^{n+|e_i||e_j|} g_{ij}$.
Note that in contrast to the previous subsections $V$ must now be concentrated in degrees $1,\dots,n-1$.
We use $\bar\Delta$ to build the degree 1 element 
\begin{equation}\label{equ:Z bar Delta}
Z_{\bar \Delta} = 
\sum_{i,j=1}^N g_{ij} 
\begin{tikzpicture}
\node (v) at (0,0) {$e_i$};
\node (w) at (1,0) {$e_j$};
\draw (v) edge[bend left] (w);
\end{tikzpicture}
\in \GGC_{V,n}  
\end{equation}
\begin{lemma}
The element $Z_{\bar \Delta}\in \GGC_{V,n} $ is a Maurer-Cartan element.
\end{lemma}
\begin{proof}
Trivially $\delta Z_{\bar \Delta}=[Z_{\bar \Delta},Z_{\bar \Delta}]=0$.
\end{proof}

Just as in the previous subsection, we again consider the extended dg Lie algebra with underlying graded vector space
\[
  \tGGC_{V,n} = \GC_{V,n}[1] \oplus \GGC_{V,n},
\]
where $\GC_{V,n}[1]$ is again the summand spanned by graphs with no hairs.
We split the differential on the $Z_{\bar \Delta}$-twisted complex $\tGGC_{V,n}^{Z_{\bar \Delta}}$ 
\[
\delta^{Z_{\bar \Delta}} = \delta +[Z_{\bar \Delta},-]=  \delta_0+\delta_1  
\]
into a piece $\delta_0$ that leaves the number of hairs constant or reduces it and a piece $\delta_1$ that adds one hair.
We then obtain the dg vector space $(\GC_{V,n},-\delta_0)$ and the map of dg vector spaces
  \[
 \delta_1 :  (\GC_{V,n},-\delta_0) \to  \HGC_{V,n}^{Z_{\bar \Delta}}.
 \]
Again we may define a Lie bracket on $\GC_{V,0}$ by the formula
\[
[x,y]_{\GC} := (-1)^{|x|}[\delta_1x,y]_{\GGC}= [x,\delta_1 y]_{\GGC},   
\]
where the Lie bracket on the right is taken in $\tGGC_{V,n}$.
\begin{prop}
The differential $\delta_0$ and the bracket $[-,-]_{\GC}$ define a dg Lie algebra structure on $\GC_{V,n}$, such that the map $\delta_1$ is a dg Lie algebra morphism.
The kernel $I:=\ker\delta_1\subset \GC_{V,n}$ is generated by graphs (with no hairs) that have no decoration in $V^*$.
\end{prop}
\begin{proof}
  The proof is largely parallel to that of Proposition \ref{prop:GC well defined} above, so we shall be brief.
  First, the description of the kernel follows the proof of injectivity in Proposition \ref{prop:GC well defined}.
  Furthermore, as above, the equation $(\delta^{Z_{\bar \Delta}})^2=0$ and the compatibility of differential and bracket on $\tGGC_{V,n}^{Z_{\bar \Delta}}$ yield the following identities for $x,y\in \GC_{V,n}$:
  \begin{align*}
    \delta_0^2x &=\delta_1^2 x = (\delta_1\delta_0+\delta_0\delta_1)x = 0 \\
    [\delta_1x, y]_{\GGC} &= (-1)^{|x|}[x, \delta_1 y]_{\GGC}
    \\
    \delta_1[x, \delta_1 y]_{\GGC} &= [\delta_1 x, \delta_1 y]_{\GGC}.
  \end{align*}
  From this it immediately follows that the map $\delta_1$ respects the Lie bracket and differential.
However, since $\delta_1$ is not injective, we can not readily deduce the well-definedness of the dg Lie algebra structure on $\GC_{V,n}$ from that on $\GGC_{V,n}^{Z_{\bar \Delta}}$. Rather, we check the required identites by hand.
First, for $x,y\in \GC_{V,n}$ of degrees $|x|$, $|y|$ we have that 
\begin{align*}
  [x,y]_{\GC} &= [x,\delta_1 y]_{\GGC}
  =
  (-1)^{(|x|+1)|y|}[\delta_1 y,x]_{\GGC}
  =
  (-1)^{|x||y|} [y,x]_{\GC}.
\end{align*}
Hence the Lie bracket has the correct anti-symmetry.
The Jacobi identity is shown as follows, using the Jacobi identity on $\tGGC_{V,n}$:
\begin{align*}
[x,[y,z]_{\GC}]_{\GC}
&= 
[x,\delta_1 [y,\delta_1 z]_{\GGC}]_{\GGC}
=
[x,[\delta_1  y,\delta_1 z]_{\GGC}]_{\GGC}
\\&=
[[x,\delta_1 y]_{\GGC},\delta_1 z]_{\GGC}
+
(-1)^{(|x|+1)|y|}
[\delta_1 y,[x,\delta_1 z]_{\GGC}]_{\GGC}
\\&=
[[x,y]_{\GC},z]_{\GC}
+
(-1)^{ |x| |y|}
[y,[x,z]_{\GC}]_{\GC}.
\end{align*}
The compatibility of differential and Lie bracket reads:
  \begin{align*}
    \delta_0 [x,y]_{\GC} &= 
    (\delta-\delta_1)[ x, \delta_1 y]_{\GGC}
    =
   [ \delta_0 x, \delta_1 y]_{\GGC}
   +(-1)^{|x|+1}
   [ x, \delta_0 \delta_1 y]_{\GGC}
   \\&=
   [ \delta_0 x, y]_{\GC}
   +
   (-1)^{|x|}
   [ x, \delta_1 \delta_0 y]_{\GGC}
   =
   [ \delta_0 x, y]_{\GC}
   +
   (-1)^{|x|}
   [ x, \delta_0 y]_{\GC}
  \end{align*}
Here we used that $\delta_0=\delta-\delta_1$ distributes over the Lie bracket, since this is true for both $\delta$ and $\delta_1$.
\end{proof}

Note that the dg Lie algebra morphism $\delta_1$ factorizes over the quotient dg Lie algebra 
\[
  \GCp_{V,n} := \GC_{V,n} / \ker \delta_1,
\]
and the resulting morphism of dg Lie algebras 
\[
\delta_1 :   \GCp_{V,n}\to \GGC_{V,n}^{Z_{\bar \Delta}}
\]
is an injection.

Next, we define the graded Lie subalgebra $\osp_V\subset \GGC_{V,n}$ as in the previous subsections, i.e.,
\[
  \osp_{V} = \{\phi\in  \gl'(V_1)\mid 
  \phi\cdot \bar \Delta = \sum_{i,j=1}^N g_{ij}(\phi(e_i)\otimes e_j + (-1)^{|\phi||e_i|} e_i\otimes \phi(e_j))=0\}
  \subset \gl'(V_1)\subset \GGC_{V,n}\, .
\]

We then have the following variant of Proposition \ref{prop:GCx def}.

\begin{prop}\label{prop:GCpx def}
The inclusion $\iota: \osp_{V}\to \GGC_{V,n}^{Z_{\bar \Delta}}$ is a morphism of dg Lie algebras, i.e., every element $\phi\in \osp_{V}$ is a cocycle, $\delta^{Z_{\bar \Delta}}\phi=0$.
Furthermore, the image of the injective map  
\[
  i : \osp_{V} \oplus \GCp_{V,n} \xrightarrow{\iota\oplus \delta_1}
  \GGC_{V,n}^{Z_{\bar \Delta}}
\]
is a dg Lie subalgebra.
\end{prop}
In particular this endows 
\[
  \GCpx_{V,n} := \osp_{V} \oplus \GCp_{V,n}  
\]
with a dg Lie algebra structure, extending those on $\osp_{V}$ and $\GCp_{V,n}$. This dg Lie algebra is just the semidirect product of $\osp_{V}$ and $\GCp_{V,n}$, using the natural action of $\osp_{V}$ on $\GCp_{V,n}$ by transforming the $V^*$-decorations on vertices.

\begin{proof}[Proof of Proposition \ref{prop:GCpx def}]
First note that $[Z_{\bar \Delta}, \phi]=0$ for all $\phi\in \osp_V$, since $\phi$ annihilates the diagonal that appears in the decoration of the hairs in $Z_{\bar \Delta}$. Hence each such $\phi$ is closed under the differential.
Hence, for $x\in \GC_{V,n}$ we have that 
\[
  \delta [\phi, x ]
  =
  (-1)^{|\phi|} [\phi, \delta x ].
\]
Projecting to the subspace generated by graphs with exactly one hair we obtain
\begin{align*}
  \delta_1 [\phi, \delta x ] &=
  (-1)^{|\phi|}  [\phi, \delta_1 x ].
\end{align*}
In particular this implies that $\GCpx_{V,n}$, considered as a graded subspace of $\GGC_{V,n}^{Z_{\bar \Delta}}$, is also closed under the bracket. 
\end{proof}

Finally, we may compare the cohomologies of of dg Lie algebras, analogously to Proposition \ref{prop:GC HGC comparison}.
\begin{prop}\label{prop:GCp GGC comparison}
  The inclusion
  \[
   i: \GCpx_{V,n} \to \GGC_{V,n}^{Z_{\bar \Delta}}
  \]
  is a quasi-isomorphism of dg vector spaces.
  The spectral sequence on the mapping cone of $i$ associated to the filtration by number of vertices has trivial $E^2$-page.
  Furthermore, for $n\geq 3$ the induced map of connected components of the Maurer-Cartan spaces 
  \begin{equation}\label{equ:GCp GGC MC}
    \pi_0\MC_\bullet(\GCp_{V,n}) \to \pi_0 \MC_\bullet(\mF^1_{vert}\GGC_{V,n}^{Z_{\bar \Delta}})
  \end{equation}
  is a bijection.
  
  \end{prop}
  \begin{proof}
  We proceed as in the proof of Proposition \ref{prop:GC HGC comparison}. It is sufficient to show the statement about the spectral sequence on the mapping cone, which implies that the cone is acyclic, and hence that $d_1$ is indeed a quasi-isomorphism.
  
  On the $E^0$-page of the spectral sequence we have to compute $H(\gr^k_{vert}  \cone(i) )$ for all numbers of vertices $k$.
  For $k=0$ the complex $\gr^0_{vert}  \cone(i)$ has the form 
  \[
    \osp_V \to \gl'_V \to \mathrm{span} \{ L_{\alpha\beta} \mid \alpha,\beta \in V_1 \}\cong (V_1\otimes V_1 \otimes \sgn_2^n)_{S_2},
  \]
  with $L_{\alpha\beta}$ the elements \eqref{equ:L alpha beta}. 
  Since $\osp_V$ is the kernel of the right-hand arrow by definition, the above sequence is exact at the first and second position. The cokernel of the right-hand map is easily seen to be $\Q L_{e_0e_0}$, so that 
  \[
    H(\gr^0_{vert}  \cone(i)) \cong \Q L_{e_0e_0}.
  \]
  
  For $k\geq 1$ we have that
  \[
    \gr^k_{vert}  \cone(i) \cong (\gr^k_{vert} \tGGC_{V,n}^{Z_{\bar \Delta}} /I ,\delta_{glue}+\delta_{hair})
  \] 
  with $I$ spanned by graphs with no hairs and no decorations.
  Note that here the part $\delta_{hair}$ only adds a hair in place of a $V$-decoration, and only adds hairs decorated by $V^*$, but it does not attach a hair to an undecorated vertex as the analogous piece of the differential did in the proof of Proposition \ref{prop:GC HGC comparison}.
  It is convenient to split 
  \[
    ( \tGGC_{V,n}^{Z_{\bar \Delta}} /I, \delta_{glue}+\delta_{hair}) \cong U_0 \oplus U_{\geq 1},
  \] 
  with $U_0$ the subcomplex generated by graphs with no $V$-decorations on vertices and all hairs decorated by $e_0=1$, and with $U_{\geq 1}$ the subcomplex generated by all other graphs. In particular, the differential $\delta_{glue}+\delta_{hair}$ is zero on $U_0$.
  By the same argument as in the proof of Proposition \ref{prop:GC HGC comparison} one see that $H(\gr^k_{vert}U_{\geq 1}, \delta_{glue}+\delta_{hair})=0$ for $k\geq 1$. 
  The $E^1$-page of our spectral sequence is hence identified with 
  \[
    (U_0\oplus \Q L_{e_0e_0},\delta_{split} + \delta_{join}).
  \]
  It is shown in \cite[Proposition 1]{KWZ2} that this complex is acyclic, and hence the first two statements of the Proposition are shown.
  For completenedd, let us quickly recall from \cite{KWZ2} the reason for this acyclicity.
  The operation $\delta_{join}$ in particular has one term that joins all the hairs of the graph together:
  
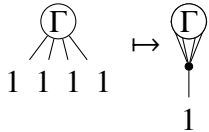
\[
  \begin{tikzpicture}
    \node[ext] (v) at (0,.3) {$\Gamma$};
    \node (w1) at (-.6,-.5) {$1$};
    \node (w2) at (-.2,-.5) {$1$};
    \node (w3) at (.2,-.5) {$1$};
    \node (w4) at (.6,-.5) {$1$};
    \draw (v) edge (w1) edge (w2) edge (w3) edge (w4) ;
  \end{tikzpicture}
  \mapsto 
  \begin{tikzpicture}
    \node[ext] (v) at (0,.3) {$\Gamma$};
    \node[int] (w) at (0,-.3) {};
    \node (w1) at (0,-1) {$1$};
    \draw (v.west) edge (w)  (v.east) edge (w) (v.north west) edge (w) (v.north east) edge (w)
    (w1) edge (w);
    \node[ext,fill=white] at (0,.3) {$\Gamma$};
  \end{tikzpicture}
  \]
  For this piece one can construct a homotopy and hence show acyclicity, see \cite{KWZ2} for details.
  
  Finally we turn to the statement about Maurer-Cartan spaces.
  We want to use a version of the Goldman-Millson theorem. However, since our spectral sequence only converges on the $E^2$-page, not the $E^1$-page as required in the standard version \cite[Theorem 1.1]{DolRog}, we have to slightly adapt the argument.
  First we note that due to the condition $n\geq 3$ the $E^1$ page of our spectral sequence is concentrated in degrees $\leq -1$.
  Indeed, 
  the degree of a graph $\Gamma$ in $\GGC_{V,n}$ with $v$ vertices, $h\geq 1$ hairs all decorated by $e_0=1$ and $e$ edges is
  \begin{align*}
    |\Gamma| &= nv -(n-1) e 
    \\& \leq -\frac12 (n-3) v  -\frac 12 (n-1) h 
    \leq -1,
  \end{align*}
  since by trivalence $e\geq \frac 32 v +\frac 12 h$ and $n\geq 3$.
  Hence the inclusion $\GCp_{V,n} \to \mF^1_{vert}\GGC_{V,n}^{Z_{\bar \Delta}}$ induces a quasi-isomorphism on the associated graded coplexes in degrees $\geq 0$.
  Looking at the proof of the $\pi_0$-part of the Goldman-Millson Theorem in \cite[section 3]{DolRog} we see that it only requires the quasi-isomorphism property in degrees 0,1,2.
  Hence this is still applicable here and yields the bijection \eqref{equ:GCp GGC MC}.
  \end{proof}

For later use we again define a pro-nilpotent variant, cf. \eqref{equ:GCex def},
\begin{equation}\label{equ:GCpex def}
\GCpex_{V,n} = \osp_V^{nil} \oplus \GCp_{V,n} \subset \GCpx_{V,n}.
\end{equation}

\section{Applications}\label{sec:applications}

In the previous sections we have studied abstract $\La$ Hopf $\e_n^c$-comodules of configuration space type in general.
In this section we shall use the results to produce models for topological configuration spaces of points on manifolds using the alagebraic results above.
We also discuss the automorphism groups.






\subsection{Rational homotopy type of configuration spaces of general manifolds}

Let $\Mfd$ be a parallelized connected manifold of dimension $n\geq 2$. Then by Proposition \ref{prop:FM configuration type} the compactified configuration spaces of points $\FM_\Mfd$ form an $\FM_n$ module of configuration space type. Furthermore, by Corollary \ref{cor:FM Hopf csc} the differential forms $\Omega_\sharp(\FM_\Mfd)$ are a right $\La$ Hopf $\Omega_\sharp(\FM_n)$-comodule of configuration space type.
Now, we use the rational formality of the little disks operad to connect $\Omega_\sharp(\FM_n)$ to its cohomology $\e_n^c$ by a zigzag of quasi-isomorphisms.
\[
  \Omega_\sharp(\FM_n) \to \bullet \leftarrow \e_n^c.
\]
We may use the Quillen equivalences of categories of comodules of Proposition \ref{prop:res ind adjunction} over these cooperads to extend our zigzag to a zigzag of quasi-isomorphisms of pairs of $\La$ Hopf cooperads and their comodules 
\[
  (\Omega_\sharp(\FM_n),\Omega_\sharp(\FM_\Mfd))
   \to \bullet \leftarrow (\e_n^c,\MOp),
\]
with $\MOp$ some $\La$ Hopf $\e_n^c$ comodule quasi-isomorphic to $\Omega_\sharp(\FM_\Mfd)$.

By Lemma \ref{lem:comodule csc basic props} we have that the right-hand object $\MOp$ is still of configuration space type.
Hence our main Theorem \ref{thm:main} tells us that there is a Maurer-Cartan element $Z\in \mG^1\GGC_{\bar H, n}$ and a quasi-isomorphism of $\La$ Hopf $W\e_n^c$-comodules
\[
\Graphs_{\bar H,n}^Z \to W\MOp \leftarrow \MOp,
\]
where $\bar H=\bar H(\Mfd)$ is the reduced cohomology of $M$.
Thus we have shown:

\begin{cor}\label{cor:config space}
  Let $\Mfd$ be an $n\geq 2$-dimensional connected parallelized manifold with finite dimensional reduced rational cohomology $\bar H=\bar H(\Mfd)$.
  Then there is a Maurer-Cartan element $Z\in \mG^1\GGC_{\bar H}$ such that $\Graphs_{\bar H,n}^Z$ is a model for $\FM_M$. Concretely, this means that there is a zigzag of quasi-isomorphisms of pairs of Hopf cooperads and their $\La$-Hopf comodules 
  \[
  (\e_n^c, \Graphs_{\bar H,n}^Z) \to \bullet \leftarrow (\Omega_\sharp(\FM_n),\Omega_\sharp(\FM_M)).  
  \]
\end{cor}

In practice, it is often difficult to determine the correct Maurer-Cartan element $Z$ appearing in the Corollary.
However, for several special cases this is possible, as we shall discuss next.

\subsection{Special case 1: Highly connected manifolds}
The simplest case is that of highly connected manifolds.

\begin{lemma}\label{lem:highconn vanishing}
Let $n\geq 3$ and $V$ be a finite dimensional positively graded vector space such that the degree $k$ part $V^k$ is zero for $k< \frac13 n$ or $k> \frac23 n-1$.
Then any graph $\Gamma\in \GGC_{V,n}$ with at least one vertex has degree $|\Gamma|\leq 0$.
\end{lemma}
\begin{proof}
Let $\Gamma\in \GGC_{V,n}$ be a graph with $e$ edges, $v\geq 1$ vertices, $u$ decorations in $V^*$ and $h$ hairs with $V_1$-decorations.
The total degree of the graph is 
\begin{equation}\label{equ:pre Gammadeg}
  |\Gamma| = 
  n v - (n-1)e + d + D 
\end{equation}
with $d\leq 0$ the total degree of the decorations in $V^*$ and $D\geq 0$ the total degree of the decorations (at the hairs) in $V_1$.
By the trivalence condition we have that $3v + h \leq 2e + u$, or equivalently $v\leq \frac23 e+\frac13 u -\frac13 h$.
Inserting this into \eqref{equ:pre Gammadeg} together with the degree bounds on $V$ we find
\begin{align*}
  |\Gamma|
  &\leq 
  (\frac 23 n-n+1) e
  +\frac n3 u - \frac n3 h 
  -\frac n3 u + (\frac 23 n-1)h
  \\
  &\leq -3(n-3) (e-h).
\end{align*}
But since any hair must be attached to an edge we have $e\geq h$, and since furthermore $n\geq 3$ we have $|\Gamma|\leq 0$ as desired.
\end{proof}

\begin{cor}
Let $n\geq 3$ and $\Mfd$ be a parallelized connected $n$-manifold with finite dimensional cohomology, such that 
\[
\bar H^k := \bar H^k(\Mfd) = 0  \quad \text{for $k< \frac n3$ or $k> \frac n3-1$.}
\]
Let $\bar\Delta\in \bar H \otimes \bar H$ be the diagonal element.
Then the Maurer-Cartan element $Z\in \GGC_{\bar H,n}$ of Corollary \ref{cor:config space} is $Z=Z_{\bar \Delta}$, see \eqref{equ:Z bar Delta}.
\end{cor}
Note that the diagonal element can for example be determined as in Example \ref{ex:diag of config space} and is usually known or easy to obtain.

\begin{proof}
We know that $Z=Z_{\bar \Delta}+Z_1$, with $Z_1\in \GGC_{\bar H,n}$ terms with $\geq 1$ vertex, see Corollary \ref{cor:GraphsZ config space type}.
But by Lemma \ref{lem:highconn vanishing} there are no possible graphs of degree 1 that could contribute to $Z_1$, and hence $Z_1=0$.
\end{proof}

\subsection{Special case 2: Closed manifolds}
Let now $\Mfd$ be a closed connected parallelized manifold. This means that for $\Mfd$ we have Poincar\'e duality, and that the diagonal element $\Delta\in H(\Mfd)\otimes H(\Mfd)$ is non-degenerate.
By Corollary \ref{cor:HGC GC} we hence find:

\begin{cor}\label{cor:M closed}
Let $\Mfd$ be a closed connected parallelized manifold with reduced cohomology $\bar H=\bar H(M)$ and diagonal $\Delta\in  \bar H_1\otimes  \bar H_1$. Then the Maurer-Cartan element $Z$ of Corollary \ref{cor:config space} can be taken to be of the form 
$Z=Z_\Delta+Z_1$ with $Z_1\in \MC(\GC_{\bar H,n})\subset \MC(\GGC_{\bar H,n}^{Z_\Delta})$\footnote{Here we consider $\GC_{\bar H,n}$ as a dg Lie subalgebra of $\GGC_{\bar H,n}^{Z_\Delta}$ via the embedding $\delta_1$ as in Proposition \ref{prop:GC well defined}.} and $Z_\Delta$ as in \eqref{equ:ZDeltapic}.
\end{cor}

The Maurer-Cartan element $Z$ can be further restricted if we impose suitable degree bounds on the cohomology.

\begin{lemma}\label{lem:GC deg bound}
  Suppose that $V$ is concentrated in degrees $\geq 2$.
  Let $\gamma\in \GC_{V,n}$ be a graph of loop order $g$ with $D$ decorations in $V^*$. Then the cohomological degree of $\gamma$ satisfies
  \[
    |\gamma| 
    \leq 
    -(n-3)(g-1) - D +1.
  \]
  In particular, if $n\geq 4$ and $g\geq 1$ then
  \[
    |\gamma| \leq 0.
  \]
\end{lemma}
\begin{proof}
Let $\gamma\in \GC_{V,n}$ be a graph with $v$ vertices, of loop order $g$, with $e$ edges, $v=e-g+1$ vertices and $D$ decorations in $V^*$. Then the trivalence condition implies that 
\[
3v = 3(e-g+1) \leq 2e + D \Leftrightarrow e+3 \leq D +3g.
\]
Furthermore, the (negative) degree of decorations is $\leq -2D$, since by assumption $V$ is concentrated in degrees $\geq 2$.
The cohomological degree of $\gamma$ is hence 
\[
|\gamma|=nv - (n-1) e + (\text{decoration degree}) + 1 
\leq 
 e-ng + n - 2D +1 
 \leq 
 -(n-3)(g-1) - D +1.
\]
The second inequality is clear for $g\geq 2$. For $g=1$ the graph must have at least one decoration to satisfy trivalence, hence $D\geq 1$. 
\end{proof}

\begin{cor}\label{cor:Z1 tree}
 In the setting of Corollary \ref{cor:M closed}, if the dimension $n$ of $\Mfd$ is at least 4 and $H^1(\Mfd)=0$, then the Maurer-Cartan element $Z_1\in \GC_{\bar H,n}$ is a (possibly infinite) linear combination of tree graphs only.
\end{cor}
\begin{proof}
  By the previous Lemma the only graphs in $\GC_{\bar H,n}$ of degree $+1$ are tree graphs.
\end{proof}

Note that the tree part $Z_1$ of the Corollary can usually be determined relatively easily.
Namely, as in section \ref{sec:complexity filtration} the tree part of the Maurer-Cartan element $Z\in \GGC_{V,n}$ determines the homotopy type of the dg commutative algebra 
$\Graphs_{n,V}^Z(1)$. But this has to be the same as the rational homotopy type of $\Mfd$.
Moreover, one can see that this condition already determines $Z_1$ uniquely up to gauge equivalence, and rescaling. We shall not elaborate further on this point but refer to the discussions in \cite[section 8]{CW} and \cite{HamiltonLazarev}, where it is shown that the cyclic rational homotopy type (encoded by $Z_1$) is determined by the rational homotopy type.

\subsection{Special case 3: Manifolds with spherical boundary}


The last case we consider are parallelizable manifolds $\Mfd$ that are obtained from orientable (but possibly not parallelizable) closed manifolds by removing a disk.
These manifolds do not support Poincar\'e duality. However, due to Poincar\'e duality on the closed manifold the reduced diagonal element 
\[
\bar \Delta \in \bar H\otimes \bar H,
\] 
with $\bar H=\bar H(\Mfd)$, is still non-degenerate.

\begin{cor}\label{cor:M with bnd}
Let $\Mfd$ be a parallelized manifold obtained by removing a disk from a closed connected manifold of dimension $n\geq 3$. Let $\bar \Delta\in \bar H \otimes \bar H$ be the non-degenerate reduced diagonal element.
Then the Maurer-Cartan element $Z$ in Corollary \ref{cor:config space} may be taken to be 
$Z = Z_{\bar \Delta} + Z_1$ with $Z_1\in \GCp_{\bar H,n}$
and $Z_{\bar \Delta}$ as in \eqref{equ:Z bar Delta}.

If an addition $H^1(\Mfd)=0$, then $Z_1$ is a (possibly infinite) linear combination of tree graphs only.
\end{cor}
\begin{proof}
Let $Z$ be the Maurer-Cartan element provided by Corollary \ref{cor:config space}.
Then we have that 
\[
X := Z - Z_{\bar \Delta}\in \mF^1_{vert}\GGC_{\bar H, n}
\]
is a Maurer-Cartan element in $\mF^1_{vert}\GGC_{\bar H, n}^{Z_{\bar \Delta}}$. This is true since $Z$ consists of graphs with positive numbers of edges, and the only graphs with positive numbers of edges and no vertices are the line graphs. But they are accounted for by subtracting $Z_{\bar \Delta}$.
Next, we know from the second part of Proposition \ref{prop:GCp GGC comparison} that the inclusion
\begin{equation}\label{equ:GCpGC cor proof}
  \GCp_{\bar H,n}\to \mF^1_{vert}\GGC_{\bar H, n}^{Z_{\bar \Delta}}
\end{equation}
induces a bijection on $\pi_0$ of the associated Maurer-Cartan spaces. Hence the Maurer-Cartan element $X$ is gauge equivalent to some Maurer-Cartan element $Z_1\in \GCp_{V,n}$.

For the final statement of the corollary note that $\GCp_{\bar H,n}$ is a quotient of $\GC_{\bar H,n}$, and hence Lemma \ref{lem:GC deg bound} is a fortiori applicable to $\GCp_{\bar H,n}$. We even have $D\geq 1$ (as in the Lemma) since graphs in $\GCp_{\bar H,n}$ all have at least one decoration. Hence the Lemma states that there are no graphs in $\GCp_{\bar H,n}$ of degree 1 and positive genus if $H^1(\Mfd)=0$.
\end{proof}
\subsection{Automorphisms of rationalized configuration spaces}

If we have identified the correct Maurer-Cartan element $Z\in \GGC_{\bar H,n}$ that describes the rational homotopy type of the right $\FM_n$ $\La$ module $\FM_{\Mfd}$ as in Corollary \ref{cor:config space}, we can apply Theorem \ref{thm:main aut} to compute the homotopy automorphisms of the rationalization of $\FM_{\Mfd}$. In the case that the diagonal element is non-dengenerate, the answer may be simplified by passing to the smaller graph complexes of section \ref{sec:GC from HGC}. In this subsection we shall list the relevant results.

\begin{cor}\label{cor:ExpAut GC}
  Let $n$ be an integer and let $V$ be a finite dimensional positively graded vector space. Let $\Delta\in V_1\otimes V_1$ be a non-degenerate and $(-1)^n$-symmetric element of degree $n$.
Let $Z_1\in \GC_{V,n}$ be a Maurer-Cartan element, so that $Z=Z_\Delta + Z_1\in \GGC_{V,n}$ is a Maurer-Cartan element.
Then there is a weak equivalence of simplicial monoids
\[
  \Aut^h(\Graphs_{V,n}^{Z})_{[1]} \cong \Exp_\bullet((\GCex_{V, n})^{Z_1}),
\] 
with $\GCex_{V, n}$ defined as in \eqref{equ:GCex def}.
\end{cor}
\begin{proof}
By Corollary \ref{cor:ExpAut} we know that
\begin{equation}\label{equ:ExpAut GC 1}
\Aut^h(\Graphs_{V,n}^{Z})_{[1]} \cong \Exp_\bullet(\GGC_{V, n}^{nil,Z})
=
\MC_\bullet((\GGC_{V, n}^{nil,Z})_{ab}) .
\end{equation}
By twisting with $Z_1$ we obtain from the inclusion $\GCex_{V,n}\to \GGC_{V, n}^{nil,Z_\Delta}$ (see Proposition \ref{prop:GCx def}) the map of dg Lie algebras
$(\GCex_{V,n})^{Z_1}\to \GGC_{V, n}^{nil,Z}$.
Hence we obtain a map of simplicial groups 
\begin{equation}\label{equ:ExpAut GC 2}
  \MC_\bullet(((\GCex_{V, n})^{Z_1})_{ab})
  =
  \Exp_\bullet((\GCex_{V, n})^{Z_1})
  \to
  \Exp_\bullet(\GGC_{V, n}^{nil,Z})
  =
  \MC_\bullet((\GGC_{V, n}^{nil,Z})_{ab})
\end{equation}
For an abelian $\SL_\infty$ algebra $\fg$ and all $k=0,1,2,\dots$ we have that $\pi_k (\MC_\bullet(\fg))=H^{-k}(\fg)$ (for any choice of basepoint).
Hence to check that \eqref{equ:ExpAut GC 2} is a weak equivalence we have to check that the inclusion
\begin{equation}\label{equ:GC GGC cor}
(\GCex_{V,n})^{Z_1}\to \GGC_{V, n}^{nil,Z}
\end{equation}
is a quasi-isomorphism in non-positive degrees. But that is a consequence of Proposition \ref{prop:GC HGC comparison} as follows:
We take on both sides of \eqref{equ:GC GGC cor} the spectral sequences from the filtration by the numbers of vertices. The associated graded morphism is the same as the associated graded morphism of $\GCex_{V,n}\to \GGC_{V, n}^{nil,Z_\Delta}$.
But Proposition \ref{prop:GC HGC comparison} implies that this associated graded morphism is a quasi-isomorphism 
up to the (span of the) classes $H_k$ and $L_{\omega\omega}$.
But these classes are all of degree $\geq 2$. Hence by standard spectral sequence comparison results we find that \eqref{equ:GC GGC cor} is indeed a quasi-isomorphism in non-positive degrees.
\end{proof}

\begin{cor}\label{cor:ExpAut GCp}
  Let $n$ be an integer, $V$ a finite dimensional positively graded vector space. Let $\bar \Delta\in V\otimes V$ be a non-degenerate and $(-1)^n$-symmetric element of degree $n$.
Let $Z_1\in \GCp_{V,n}$ be a Maurer-Cartan element, and $Z=Z_{\bar \Delta} + Z_1\in \GGC_{V,n}$ the corresponding Maurer-Cartan element in $\GGC_{V,n}$ as above.
Then there is a weak equivalence of simplicial monoids
\[
  \Aut^h(\Graphs_{V,n}^{Z})_{[1]} \cong \Exp_\bullet((\GCpex_{V, n})^{Z_1}),
\] 
with $\GCpex_{V, n}$ defined as in \eqref{equ:GCpex def}.
\end{cor}
\begin{proof}
The proof is parallel to that of Corollary \ref{cor:ExpAut GC}, but using Proposition \ref{prop:GCp GGC comparison} instead of Proposition \ref{prop:GC HGC comparison}.
Concretely, proceeding as in the proof of Corollary \ref{cor:ExpAut GC}, the result is shown if we can show that the inclusion 
\begin{equation}\label{equ:GCp GGC cor}
  (\GCpex_{V,n})^{Z_1}\to \GGC_{V, n}^{nil,Z}
\end{equation}
is a quasi-isomorphism in non-positive degrees.
We endow the mapping cone of \eqref{equ:GCp GGC cor} with the filtration by the number of vertices and consider the associated spectral sequence.
Then by the proof of Proposition \ref{prop:GCp GGC comparison}
the $E^1$-page is spanned by those graphs in $\GGC_{V, n}^{nil,Z}$ that have no decoration in $\bar H^*$ and all legs are decorated by $e_0=1$.
A priori, the differential on the $E^2$-page of the cone of \eqref{equ:GCp GGC cor} might differ from that considered in Proposition \ref{prop:GCp GGC comparison}, because the Lie bracket with the one-vertex part of $Z_1$ contributes to the differential on the $E^1$-page.
However, note that any one-vertex graph in $\GCpex_{V,n}$ must have at least 3 decorations in $\bar H^*$, and at most one of them can be the top class $\omega$ by degree reasons. Hence the bracket of that 1-vertex graph with any  graph with no decoration and only $e_0$-decorated legs must have at least two decorations, and hence be exact on the $E_1$ page of our spectral sequence.
Hence these terms in fact do not contribute to the differential on the $E^1$-page.
The $E^1$-page thus agrees with that considered in Proposition \ref{prop:GCp GGC comparison}. By that Proposition we then conclude that the $E^2$-page of our spectral sequence is zero, and hence \eqref{equ:GCp GGC cor} is a quasi-isomorphism as required.
\end{proof}

We may apply these corollaries to compute the automorphisms of rationalized configuration spaces.
As is often the case in rational homotopy theory this application requires additional finiteness assumptions.
Generally, we say that a right operadic $\FM_n$ module $\MOp$ is $\Q$-good if the map $\MOp\to \MOp^{\Q}$ induces an isomorphism on rational chomology
\[
H(\MOp) := H^\bullet(\MOp,Q) \cong H(\MOp^{\Q}).
\] 
(Recall that the rationalization is defined as $\MOp^{\Q}=L\G (R\Omega_\sharp(\hat \MOp))$. )
This means that each space $\MOp(r)$ is $\Q$-good.

For example, let $\Mfd$ be an $n$-dimensional parallelizable manifold.
Then $\MOp=\FM_{\MOp}$ is $\Q$-good if $n\geq 3$ and $\Mfd$ is compact and simply connected, because then each $\FM_{\MOp}(r)$ is simply connected and compact as well.
However, configuration spaces in dimension $n=2$ are not $\Q$-good, not even for $\Mfd=\R^2$, see \cite{IvanovMikhailov}. Hence we restrict to $n\geq 3$ in the following.

For a $\Q$-good cofibrant $\La$ $\FM_n$ module $\MOp$ one can show analogously to \cite[Proposition 16]{FWAut} that
\[
\Aut^h(\MOp^{\Q}) \simeq \Aut^h(\Omega_\sharp(\MOp)),
\]
where on the left-hand side one takes the homotopy automorphisms in the category of $\La$ $\FM_n^\Q$ modules and on the right-hand side in the category of $\La$ Hopf $\Omega_\sharp(\FM_n)$ comodules.
Combining these observations with the results above we find the following Corollary.

\begin{cor}\label{cor:aut mfds}
  Let $n\geq 3$ and let $\Mfd$ be an $n$-dimensional parallelizable manifold such that the the right operadic $\La$ $\FM_n$ module $\FM_\Mfd$ is $\Q$-good. For example, this is satisfied if $\Mfd$ is simply connected and compact, or the interior of a compact manifold.
  Let $\bar H = \bar H(\Mfd)$ be the reduced rational cohomology. 
  Let $Z\in\mG^1 \GGC_{\bar H,n}$ be the Maurer-Cartan element of Corollary \ref{cor:config space}.
  Then we have a weak equivalence of simplicial monoids 
  \[
    \Aut^h(\FM_{\Mfd}^\Q)_{[1]} \simeq 
    \Exp_\bullet(\GGC_{\bar H,n}^{nil,Z}).
  \]
  \begin{itemize}
\item Suppose furthermore that $\Mfd$ is closed so that the diagonal element $\Delta\in \bar H_1\otimes \bar H_1$ is non-degenerate, and let $Z_1\in \GC_{\bar H,n}$ be the Maurer-Cartan element of Corollary \ref{cor:M closed} or Corollary \ref{cor:Z1 tree}. Then we have a weak equivalence of simplicial monoids 
\[
\Aut^h(\FM_\Mfd^\Q)_{[1]} \simeq \Exp_\bullet((\GCex_{\bar H, n})^{Z_1}).
\]
\item Alternatively, suppose that $\Mfd$ is obtained by removing a disk from a closed manifold, so that the reduced diagonal $\bar\Delta\in \bar H\otimes\bar H$ is non-degenerate.  Let $Z_1\in\GCp_{\bar H,n}$ be the Maurer-Cartan element of Corollary \ref{cor:M with bnd}. Then there is a weak equivalence of simplicial monoids
  \[
  \Aut^h(\FM_\Mfd^\Q)_{[1]} \simeq \Exp_\bullet((\GCpex_{\bar H, n})^{Z_1}).
  \]
  \end{itemize}
\end{cor}

As a particular application of the last assertion of the Corollary \ref{cor:aut mfds}, let us consider for $d\geq 2$ the
$n=2d$-dimensional highly connected manifolds 
\begin{align*}
  W_{g,1} &:= W_g \setminus [0,1]^n. 
& &\text{with}
&
W_g &:= \underbrace{(S^d\times S^d) \# \cdots \#(S^d\times S^d)}_{g\times},
\end{align*}
that are parallelizable \cite{KRWframings} and have recently received significant attention in the literature \cite{KRWdiffeom, KRWalgebraic}. 
In this case the Maurer-Cartan element $Z_1\in\GCp_{\bar H,n}$ must be zero, since there are no graphs of degree 1 in $\GCp_{\bar H,n}$:
Indeed, by Lemma \ref{lem:GC deg bound} the only potential graphs of degree $+1$ are tree graphs. 
But a tree graph $\gamma\in \GC_{\bar H, n}$ with $v$ vertices must have at least $v+2$ decorations, and hence has degree 
\[
|\gamma| 
\leq nv -(n-1)(v-1) - \frac n2 (v+2) +1
= -(\frac n2-1)v <0.
\]
Hence we conclude from the last assertion of Corollary \ref{cor:aut mfds} that there is a weak equivalence of simplicial monoids
\[
    \Aut^h(\FM_{W_{g,1}}^\Q)_{[1]} \simeq \Exp_\bullet(\GCp_{\bar H(W_{g,1}), n}).
\]
Note that we used here that $\GCp_{\bar H(W_{g,1}),n}=\GCpex_{\bar H(W_{g,1}),n}$, since $\bar H(W_{g,1})$ is concentrated in the single degree $\frac n2$ and has hence no linear endomorphisms of negative degree.




\appendix 

\section{A technical variant of the Goldman-Millson Theorem}
We will need a somewhat subtle variation of the Goldman-Millson theorem. It is similar to Theorem 1.1 of \cite{DolRog}, proven by Dolgushev-Rogers, however, we need to cope with filtered complete $\SL_\infty$-algebras with filtrations that start at the zeroth filtration degree rather than the first, hence leading to certain convergence issues. To remedy this, we consider two different filtrations, one for ensuring convergence, and one for the cohomology computation.

\begin{thm}[Version of Goldmann-Millson Theorem]\label{thm:GM tech}
  Let
  \[
    F=\{F_k\}_{k\geq 1} : \fg \to \fh   
  \]
  be an $\SL_\infty$-morphism between $\SL_\infty$-algebras $\fg$ and $\fh$ satisfying the following list of conditions.

  \begin{enumerate}
\item $\fg$ is abelian (i.e., just a dg vector space) and is equipped with two descending complete filtrations 
\begin{align*}
  \fg &= \mF^1\fg \supset  \mF^2\fg \supset \cdots
  \\ 
  \fg &= \mG^0\fg \supset  \mG^1\fg \supset \cdots 
  \end{align*}
  compatible with the $\SL_\infty$-structure (i.e., the differential) and such that 
  \begin{equation}\label{equ:FGg deg one}
  \fg^0  = \mG^1\fg^0.   
  \end{equation}
\item $\fh$ is equipped with two descending complete filtrations 
\begin{align*}
  \fh &= \mF^1\fh \supset  \mF^2\fh \supset \cdots
  \\ 
  \fh &= \mG^0\fh \supset  \mG^1\fh \supset \cdots 
  \end{align*}
  compatible with the $\SL_\infty$-structure.
  \item The $\SL_\infty$ operations $\mu_k$ satisfy
  \begin{equation}\label{equ:mu in G1}
  \mu_k(\fh,\dots,\fh)\subset \mG^1\fh  
  \end{equation}
  for all $k\geq 2$.
  \item\label{it:F G compat}
  The two filtrations $\mF$ and $\mG$ are compatible in the sense that 
  $\mG^p\fg\subset \mF^{q(p)}\fg$ with $q(p)\to \infty$ as $p\to \infty$, and similarly for $\fh$. In other words, the identity maps $\fg^{\mG}\to \fg^{\mF}$ and $\fh^{\mG}\to \fh^{\mF}$ are continuous, using the topology induced by $\mG$ on the left-hand side and the topology induced by $\mF$ on the right-hand side.
  \item\label{it:F compat} The morphism $F$ is compatible with both the filtrations $\mF$ and $\mG$ in the sense that 
  \begin{align*}
  F_k(\mF^{p_1}\fg, \dots, \mF^{p_k}\fg) &\subset \mF^{p_1+\dots+p_k}\fh,
  &
  F_k(\mG^{p_1}\fg, \dots, \mG^{p_k}\fg) &\subset \mG^{p_1+\dots+p_k}\fh\, .
  \end{align*}
  \item\label{it:gr qiso} The linear component $F_1$ induces an isomorphism
  \[
 H^p(\gr_{\mG}\fg) \to H^p(\gr_{\mG}\fh)
  \]
  between the associated graded complexes with respect to the $\mG$-filtrations for each $p$.
  \end{enumerate}

Then the morphism $F$ induces a weak equivalence of the Maurer-Cartan spaces 
\begin{equation}\label{equ:MC induced}
\MC_\bullet(\fg) \to \MC_\bullet(\fh),  
\end{equation}
where we use the filtrations $\mF^\bullet$ to complete tensor products in the definition of the Maurer-Cartan spaces.
\end{thm}

To prefer for the proof of the Theorem, we first need to show some technical results about the Maurer-Cartan space.

\subsection{1-cells and concatenation}
\label{sec:onecell}
Let us recall a few well known technical results about the Maurer-Cartan space $\MC_\bullet(\fg)$ of a filtered complete $\SL_\infty$-algebra $\fg$.
0-cells in the simplicial set $\MC_\bullet(\fg)$ are just Maurer-Cartan elements.
1-cells are Maurer-Cartan elements $A\in \MC(\fg\hotimes \Omega(\Delta^1))$.
Concretely, we can separate $A$ into its 0- and 1-form parts and write
\[
A = \alpha(t) + a(t) dt.
\]
Then the Maurer-Cartan equation for $A$ reads, explicitly,
\begin{align}\label{equ:onecell MC}
\sum_{k\geq 1} \frac 1 {k!} \mu_k(\alpha(t),\dots,\alpha(t))&=0
&
\frac d {dt} \alpha(t)
&=
D_\alpha a(t) 
:= \sum_{k\geq 0}
\frac 1 {k!}
\mu_{k+1}(\alpha(t),\dots,\alpha(t),a(t)).
\end{align}
If there is a one-cell $A$ as above, we say that it connects the Maurer-Cartan (MC) elements $\alpha(0)$ and $\alpha(1)$. We also say that the latter MC elements are gauge equivalent, and call $A$ a gauge transformation, slightly abusing notation.
We shall use below the following well-known Lemmas, that we shall present with brief proof sketches. (See, e.g., \cite{DolRog} for more detailed discussions.)

\begin{lemma}\label{lem:IVP}
Let $a(t)\in \fg\hotimes \Omega(\Delta^1)$ be an arbitrary degree 0 element and let $\alpha_0\in \MC(\fg)$ be a Maurer-Cartan element. Then there is a unique one-cell $A = \alpha(t) + a(t) dt.$ such that $\alpha(0)=\alpha_0$.
\end{lemma}
\begin{proof}
The second equation of \eqref{equ:onecell MC} yields an initial value problem for $\alpha(t)$ that can for example be solved by fixed-point iteration on the integral form of the equation
\[
\alpha(t) = \int_0^t D_{\alpha(s)} a(s)ds.
\]
The resulting element $A$ then automatically satisfies the MC equation if so does the initial condition $\alpha_0$, as a small computation shows.
\end{proof}
The next Lemma is just a special case of the fact that $\MC_\bullet(\fg)$ is a Kan complex.
\begin{lemma}\label{lem:onecell concat}
Suppose that $A,B\in \MC_1(\fg)$ are two one-cells, such that $A=\alpha(t)+a(t)dt$ connects the Maurer-Cartan elements $\alpha_0=\alpha(0)$ and $\alpha_1=\alpha(1)$ and $B=\beta(t)+b(t)dt$ connects $\alpha_1=\beta(0)$ and $\alpha_2=\beta(1)$. Then there is a one-cell $C\in \MC_1(\fg)$ connecting $\alpha_0$ to $\alpha_2$, such that $A,B,C$ form the boundary of a 2-cell $X\in \MC_2(\fg)$.
\end{lemma}
\begin{proof}
To get a rather explicit formula, we may apply Lemma \ref{lem:IVP} for $\fg\hotimes \Omega(\Delta^1)$ in place of $\fg$ and obtain a Maurer-Cartan element $Y\in \MC(\fg\hotimes \Omega(\Delta^1\times \Delta^1))$ such that
\[
Y = \alpha(t,s) + a(t,s) dt
+ b(s) ds
\]
and $\alpha(t,0)=\alpha(t)$ and $a(t,0)=a(t)$. Then automatically $\alpha(0,s)+b(s)ds$ is a one-cell such that the $s=0$-initial value agrees with $\beta(0)$, and hence $\alpha(0,s)=\beta(s)$ by uniqueness of solutions of the initial value problem.
Now the desired horn filler $X$ can be obtained by restricting $Y$ to the half $\{0\leq s\leq t\leq 1\}\subset \Delta^1\times \Delta^1$ of the square. In particular $C$ then has the form 
\[
C = \alpha(t,t) + (a(t,t)+b(t) ) dt.
\]
Note in particular that $a(t,t)=a(t)+b(t)+(\cdots)$, where $(\cdots)$ are correction terms that are computed from $\alpha,a,b$ by applying nontrivial natural operations in $\fg$, i.e., just $\SL_\infty$ operations or integrals on the polynomial coefficients.
\end{proof}

\begin{lemma}\label{lem:onecell gauge}
  Let $A=\alpha(t)+a(t)dt$ be a 1-cell and let $c(t)\in \fg\hotimes \Omega^0(\Delta^1)$ be an arbitrary degree $-2$ element.
  Then 
  \[
    \tilde A= \alpha(t)+(a(t)+D_{\alpha(t)} c(t) )dt
  \]
  is still a 1-cell, where $D_{\alpha(t)}$ is the $\alpha$-twisted differential as in \eqref{equ:onecell MC}.
\end{lemma}
\begin{proof}
We need to check the MC equations \eqref{equ:onecell MC}. The first is obviously satisfied, and for the second we just use that for Maurer-Cartan elements $\alpha$ we have
\[
D_\alpha \circ D_\alpha = 0.
\]
\end{proof}

We also cite the following result.
\begin{lemma}[{\cite[Lemma B.2]{DolRog}}]\label{lem:onecell rectifiable}
  Let $A=\alpha(t)+a(t)dt\in \MC_1(\fg)$ be a 1-cell. Then there is a 1-cell $\tilde A=\tilde \alpha(t)+\tilde a dt$ with $\tilde a \in \fg^{-1}$ constant, such that $\alpha(0)=\tilde \alpha(0)$ and $\alpha(1)=\tilde \alpha(1)$.
\end{lemma}

\subsection{Proof of Theorem \ref{thm:GM tech}}

We will show the Theorem in several steps.
First note that the conditions \ref{it:F G compat} guarantees that the identity map induces morphisms of simplicial sets 
\begin{align*}
  \MC^\mG(\mG^1\fg)
  &\to \MC^\mF(\fg) \\
  \MC^\mG(\mG^1\fh)
  &\to \MC^\mF(\fh),
\end{align*}
where we use the filtrations $\mG$ to define the Maurer-Cartan spaces on the left, as indicated in the notation $\MC^\mG_\bullet(-)$, and the filtrations $\mF$ on the right-hand side.
Together with condition \ref{it:F compat} this means that we have a commutative square of simplicial sets
\begin{equation}\label{equ:GM square}
\begin{tikzcd}
  \MC^\mG(\mG^1\fg) \ar{r}{\sim} \ar{d} & \MC^\mG(\mG^1\fh) \ar{d} \\
  \MC^\mF(\fg) \ar{r}{\eqref{equ:MC induced}}& \MC^\mF(\fh).
\end{tikzcd}
\end{equation}

Here the upper arrow is a weak equivalence by the standard Goldman-Millson Theorem \cite[Theorem 1.1]{DolRog}, which is applicable to the restricted morphism of $\SL_\infty$-algebras $F: \mG^1\fg\to \mG^1\fh$.

\begin{lemma}
  Under the conditions of Theorem \ref{thm:GM tech} each Maurer-Cartan element in $\fg$ (resp. in $\fh$) is gauge equivalent to one in $\mG^1\fg$ (respectively $\mG^1\fh$).
\end{lemma}
\begin{proof}
The statement for $\fg$ is empty by \eqref{equ:FGg deg one}.
To see the statement for $\fh$, note first that by condition \ref{it:gr qiso} and \eqref{equ:FGg deg one} we have that $H^1(\gr_{\mG}\fh)=0$.
Let $\beta \in \fh$ be a Maurer-Cartan element. Then by the Maurer-Cartan equation and \eqref{equ:mu in G1} we have 
\[
d\beta \in \mG^1\fh.
\]
and hence $\beta+db \in \mG^1\fh$ for some $b\in \fh$ of degree -1.
We may then use Lemma \ref{lem:IVP} to build a 1-cell (i.e., Maurer-Cartan element in $\fh\hotimes \Omega(\Delta^1)$)
\[
B = \beta(t) + bdt
\]
such that $\beta(0)=\beta$, by solving the MC equation (an initial value problem) for $\beta(t)$. (Note that existence of the solution implicitly uses the presence of the filtration $\mF$ for convergence, but has nothing to do with the filtration $\mG$ or any of its properties.)
The MC equation for $B$ states in particular that 
\[
\frac d {dt} \beta(t) = D_B b = db +O(\mG^1\fh),
\]
using again \eqref{equ:mu in G1}. Hence 
\[
\beta(1) = \underbrace{\beta +db}_{\in \mG^1\fh} + O(\mG^1\fh) 
\]
is the desired Maurer-Cartan element in $\mG^1\fh$ gauge equivalent to $\beta$.
\end{proof}

\begin{cor}
The map \eqref{equ:MC induced} induces a surjection on connected components, i.e., on $\pi_0$.
\end{cor}
\begin{proof}
  Consider $\pi_0$ of the commutative square \eqref{equ:GM square}. The upper horizontal arrow induces a bijection on $\pi_0$ and by the preceding Lemma the two vertical arrows induce surjections on $\pi_0$. Hence the lower horizontal arrow must induce a surjection on $\pi_0$ as well.
\end{proof}

\begin{lemma}
The morphism \eqref{equ:MC induced} induces an injective map on $\pi_0$.
\end{lemma}
\begin{proof}
Let $\alpha\in \fg^0=\mG^1\fg^0$ be a Maurer-Cartan element such that $F(\alpha)\in \MC(\fh)$ is gauge equivalent to 0. Using Lemma \ref{lem:onecell rectifiable} we may assume that this gauge equivalence is realized by a 1-cell
\[
B:=\beta(t) + b dt \in \fh\hotimes \Omega(\Delta^1)
\]
such that $\beta(0)=F(\alpha)$, $\beta(1)=0$ and $b\in \fh^{-1}$ is constant.
From the Maurer-Cartan equation for $B$ in integral form and \eqref{equ:mu in G1} we see
\[
\mG^1\fh \ni \beta(1)-\beta(0)
=
\int_0^1 D_{\beta(t)} b dt
=db + O(\mG^1\fh),
\]
and hence 
\[
db \in \mG^1\fh.
\]
Hence by assumption \ref{it:gr qiso} of our Goldman-Millson Theorem there is some $a\in \fg^{-1}$ such that $da\in \mG^1\fg$ and 
$F_1(a)= b+dc$ for some $c\in \fh^{-2}$.
We consider a 1-cell 
\[
  A= (\alpha+(1-t)da) -adt
\] 
connecting the MC element $\alpha+da$ to the gauge equivalent MC element $\alpha$.
The one-cell $F(A)$ has the form 
\[
F(A) = F(\alpha+(1-t)da) - F^{\alpha+(1-t)da}(a) dt
=
F(\alpha+(1-t)da) - (F_1(a) +O(\mG^1\fh) )dt.
\]
We then glue the 1-cells $F(A)$ and $B$ (see Lemma \ref{lem:onecell concat}) to obtain a one-cell
\[
C=\beta'(t) + b'(t)dt  
\]
connecting $\beta'(0)=F(\alpha+da)$ to $\beta'(1)=0$. Furthermore, by the proof of Lemma \ref{lem:onecell concat} we can see that $b'(t)=-F_1(a)+b+O(\mG^1\fh)$, because higher bracket terms are in $\mG^1\fh$ by \eqref{equ:mu in G1}.
Next we can invoke Lemma \ref{lem:onecell gauge} to apply a further gauge transformation by $c$ to $C$, to obtain a 1-cell
\[
D=\beta'(t) + b''(t)dt
\]
such that
\[
  b''(t)=
  b'(t) + D_{\beta'(t)}c
  =
  -F_1(a)+b+dc +O(\mG^1\fh)\in \mG^1\fh.
\]
We have hence constructed a 1-cell in $\mG^1\fh$ connecting $U(\alpha+da)$ to 0.
Hence, by the standard Goldman-Millson Theorem applied to the morphism $F:\mG^1\fg\to \mG^1\fh$ we find that the Maurer-Cartan element (i.e., cocycle) $\alpha+da$ is gauge trivial, i.e., exact, and hence so is $\alpha$.

This argument shows that the preimage under 
\[
[F]: \pi_0\MC_\bullet(\fg) \to \pi_0\MC_\bullet(\fh)
\]
of (the class of) the Maurer-Cartan element $0\in \fh$ consists of a single element. Consider next some Maurer-Cartan element $\alpha\in \fg$ and set $\beta=F(\alpha)\in \mG^1\fh$.
Then the preimage of the class of $\beta$ under $[F]$ above is the same as the preimage of $0$ in the twisted $\SL_\infty$-algebra $\fh^\alpha$ under the twisted morphism $F^\alpha$.
Hence applying the statement already shown to this twisted morphism shows the general statement of the Lemma.
\end{proof}

\begin{proof}[Proof of Theorem \ref{thm:GM tech}]
To end the proof of Theorem \ref{thm:GM tech} it remains to check that on each connected component of the Maurer-Cartan spaces the map \eqref{equ:MC induced} induces a bijection on homotopy groups.
So let $\alpha\in \MC(\fg)$ be a Maurer-Cartan element and let $\beta=F(\alpha)\in \MC(\fh)$.
Then by \cite[Theorem 5.5]{Berglund} the homotopy groups are computed by\footnote{We note that stricly speaking Berglund's result does not cover the case of $\fg$ since our filtration starts at $\mF^0$. However, given that $\fg$ is abelian, the statement can still easily be checked under our conditions.}
\begin{align*}
\pi_k(\MC_\bullet(\fg))&=
H^{-k}(\fg^\alpha)
&
\pi_k(\MC_\bullet(\fh))&=
H^{-k}(\fh^\beta)
\end{align*}
for $k=1,2,\dots$. We hence have to check that the twisted morphism 
\[
F^\alpha : \fg^\alpha \to \fh^\beta
\]
is a quasi-isomorphism (in non-positive degrees). But this readily follows from condition \ref{it:gr qiso} and standard spectral sequence arguments.

\end{proof}



\end{document}